\documentclass{memo-l}

\usepackage{amsmath,amssymb,amsthm,amscd,dsfont,paralist}

\newtheorem{thm}{Theorem}[chapter]
\newtheorem{lem}[thm]{Lemma}
\newtheorem{cor}[thm]{Corollary}
\newtheorem{prop}[thm]{Proposition}

\theoremstyle{definition}

\theoremstyle{remark}
\newtheorem{rem}[thm]{Remark}

\numberwithin{section}{chapter}
\numberwithin{equation}{chapter}

\newcommand{\dist}{\operatorname{dist}} 
\renewcommand{\div}{\operatorname{div}} 

\newcommand{\im}{\operatorname{im}}

\newcommand{\lin}{\operatorname{span}}  
\newcommand{\0}{\mspace{0mu}_0}       
\renewcommand{\Re}{\operatorname{Re}}
\newcommand{\supp}{\operatorname{supp}} 

\newcommand{\tr}{\operatorname{tr}}

\newcommand{\Jump}[1]{[\![ #1]\!]}

\makeindex

\begin{document}

\frontmatter

\title{Rayleigh-Taylor instability for the two-phase Navier-Stokes equations with surface tension\\in cylindrical domains}


\author{Mathias Wilke}
\address{Universit\"at Regensburg, Fakult\"at f\"ur Mathematik, 93040 Regensburg, Germany}
\curraddr{}
\email{mathias.wilke@ur.de}
\thanks{First of all I would like to thank Jan Pr\"{u}ss for all the support and the motivation during the last years. Furthermore I am grateful to Dieter Bothe for proposing this topic. }

\author{}
\address{}
\curraddr{}
\email{}
\thanks{}





\maketitle

\tableofcontents

\begin{abstract}
This article is concerned with the dynamic behaviour of two immiscible and incompressible fluids in a cylindrical domain, which are separated by a sharp interface. In case that the heavy fluid is situated on top of the light fluid, one expects that the fluid on top sags down into the lower phase. This effect is known as the \emph{Rayleigh-Taylor-Instability}. We present a rather complete analysis of the corresponding free boundary problem which involves a \emph{contact angle}.
Our main result implies the existence of a critical surface tension with the following property. In case that the surface tension of the interface separating the two fluids is smaller than the critical surface tension, one has {Rayleigh-Taylor-Instability}. On the contrary, if the interface has a greater surface tension than the critical value, the instability effect does not occur and one has exponential stability of a flat interface. The last part of this article is concerned with the bifurcation of nontrivial equilibria in multiple eigenvalues. The invariance of the corresponding bifurcation equation with respect to rotations and reflections yields the existence of bifurcating subcritical equilibria. Finally it is proven that the bifurcating equilibria are unstable.
\end{abstract}

%

\chapter*{Introduction}

In a wider sense, this article is concerned with the mathematical analysis of the dynamics of fluids. To be more precise, the behavior of two fluids inside a bounded container, separated by a sharp interface is investigated. 

Let $u=u(t,x)$ and $p=p(t,x)$ denote the velocity field and the pressure field of a single incompressible fluid in a domain $\Omega$. By saying that the fluid is incompressible, we mean that its density $\rho>0$ is constant. Then the dynamics of the fluid are described by the Navier-Stokes equations
\begin{align}\label{int:NS}
\begin{split}
\partial_t(\rho u)-\mu\Delta u+\rho(u\cdot\nabla)u+\nabla p=\rho f,&\quad t>0,\ x\in\Omega,\\
\div u=0,&\quad t>0,\ x\in\Omega,
\end{split}
\end{align}
where $\mu>0$ represents the viscosity of the fluid and $f$ is some external force (e.g.\ gravity). The first equation reflects the balance of momentum, while the second equation states the conservation of mass.

Let us consider a more comprehensive situation, where the domain $\Omega$ is occupied by two incompressible and immiscible fluids, \emph{fluid 1} and \emph{fluid 2}, which are separated by a sharp interface $\Gamma(t)$ for each $t\ge0$. We denote by $\Omega_j(t)$ the subset of $\Omega$ which is filled with \emph{fluid j}, $j\in\{1,2\}$ with $\rho_j$,$\mu_j$ being the density and viscosity, respectively, of \emph{fluid j}. If $u^j$ and $p^j$ are the velocity fields and the pressure fields of \emph{fluid j}, respectively, then, for $t\ge 0$, one sets
$$u(t,x):=
\begin{cases}
u^1(t,x),&\ x\in\Omega_1(t),\\
u^2(t,x),&\ x\in\Omega_2(t),
\end{cases}\quad
p(t,x):=
\begin{cases}
p^1(t,x),&\ x\in\Omega_1(t),\\
p^2(t,x),&\ x\in\Omega_2(t).
\end{cases}
$$
Assuming that $(u^j,p^j)$ satisfies the Navier-Stokes equations in each of the phases $\Omega_j(t)$, then we may conclude that $(u,p)$ satisfies \eqref{int:NS} for all $t>0$ and $x\in\Omega\backslash\Gamma(t)$, where $\rho$ and $\mu$ are defined by
$$\rho(x):=
\begin{cases}
\rho_1,&\ x\in\Omega_1(t),\\
\rho_2,&\ x\in\Omega_2(t),
\end{cases}\quad
\mu(x):=
\begin{cases}
\mu_1,&\ x\in\Omega_1(t),\\
\mu_2,&\ x\in\Omega_2(t).
\end{cases}
$$
Clearly one expects that the two fluids should affect each other in their dynamics. Therefore, it is natural to ask for relations that describe the coupling of the two fluids across the interface $\Gamma(t)$. If one neglects effects of phase transitions between the phases $\Omega_1(t)$ and $\Omega_2(t)$ (e.g.\ the exchange of mass) then the motion of the moving boundary $\Gamma(t)$ should only be caused by the velocity fields of the both fluids. Therefore it is natural to propose that $u^2|_{\Gamma(t)}=u^1|_{\Gamma(t)}$. Then the \emph{normal velocity} $V_\Gamma$ of $\Gamma(t)$ is given by
\begin{equation}\label{int:NS2}
V_\Gamma=u\cdot\nu_\Gamma,
\end{equation}
where $\nu_{\Gamma}$ denotes the unit normal field on $\Gamma(t)$ pointing from $\Omega_1(t)$ to $\Omega_2(t)$.
We call the quantity $\Jump{u}:=u^2|_{\Gamma(t)}-u^1|_{\Gamma(t)}$ the \emph{jump of $u$ across $\Gamma(t)$}. Note that $\Jump{u}=0$ if and only if the velocity field $u$ is continuous across the interface $\Gamma(t)$. Another condition on $\Gamma(t)$ reads
\begin{equation}\label{int:NS3}
-\Jump{\mu(\nabla u+\nabla u^{\sf T})}\nu_\Gamma+\Jump{p}\nu_\Gamma=\sigma H_\Gamma\nu_\Gamma,
\end{equation}
where $\sigma>0$ denotes the (constant) \emph{surface tension} of $\Gamma(t)$ and $H_\Gamma:=-\div_\Gamma\nu_\Gamma$ is the \emph{mean curvature} of $\Gamma(t)$ with $\div_\Gamma$ being the surface divergence on $\Gamma(t)$. Condition \eqref{int:NS3} describes the balance of forces on the interface. To be precise, there is no contribution to the rate of change of the momentum coming from the interface $\Gamma(t)$.

If the fixed boundary $\partial\Omega$ of $\Omega$ is not empty, then the system \eqref{int:NS}-\eqref{int:NS3} with $\Jump{u}=0$ has to be equipped with appropriate boundary conditions on $\partial\Omega$ as well as some initial conditions on $u(0)=u_0$ and $\Gamma(0)=\Gamma_0$. There is a vast literature concerning the mathematical treatment of free boundary problems for the Navier-Stokes equations with or without surface tension. To this end we can only give a subjective selection and refer the reader to \cite{Bea84, BP08, BP08a, Deni91, Deni93, Deni94, DS91, DS95, DS02, KPW10, Kull91, PrSi09, PrSi09a, PrSi09b, PrSi09c, Ray1883, ShSh07, ShSh07a, Sol83, So84, So86, So90, So91, So99, SoTa91, SoTa92, SoTa90, Tanaka93b, Tanaka93, Ta96, TT95, Tay50}. For a derivation of \eqref{int:NS}-\eqref{int:NS3} we refer to \cite{Ish06}.

To describe the effect of what is called \emph{Rayleigh-Taylor instability}, let us consider the case that $\Omega=\mathbb{R}^n$ consists of two phases $\Omega_1(t)$ and $\Omega_2(t)$ which are separated by an interface $\Gamma(t)$, given by a graph of a height function $h$ over $\mathbb{R}^{n-1}$, i.e.\
$$\Gamma(t):=\{x=(x',x_n)\in\Omega:x_n=h(t,x'),\ x'\in\mathbb{R}^{n-1}\}.$$
Assume furthermore, that $\Omega_2(t)$ is the upper phase, hence $$\Omega_2(t)=\{x=(x',x_n)\in\Omega:x_n>h(t,x'),\ x'\in\mathbb{R}^{n-1}\}.$$
Both phases are filled with two fluids with possibly different densities which are accelerated in the direction of $-e_n$ by the gravitational force.

Taking a close look at the system \eqref{int:NS}-\eqref{int:NS3} it turns out that the vanishing velocity fields, constant pressure fields and the flat interfaces belong to the set of \emph{equilibria}, i.e.\ the set of all solutions, which are constant with respect to $t$. Henceforth we will speak of the trivial equilibrium, when $u=0$, $p$ is constant  and $h=0$. Heuristically one expects that the stability behavior of the trivial equilibrium is being influenced by the densities $\rho_2>0$ and $\rho_1>0$ of the fluids. Indeed, if $\Jump{\rho}=\rho_2-\rho_1>0$, i.e.\ if the heavier fluid is placed above the lighter fluid, then one expects that the trivial equilibrium is unstable while in case that $\Jump{\rho}\le0$, the trivial equilibrium should be stable. Indeed, if $\Jump{\rho}>0$ then the upper phase, which is the heavier one, should sack down into the lower phase, see Figure 1. This effect is called Rayleigh-Taylor instability and it goes back to the pioneering works of \textsc{Rayleigh} \cite{Ray1883} and \textsc{Taylor} \cite{Tay50}. For more information concerning Rayleigh-Taylor instability we refer the interested reader to \textsc{Chandrasekhar} \cite{Cha61} \& \textsc{Kull} \cite{Kull91} and the references cited therein.
A rigorous proof of Rayleigh-Taylor instability for the two-phase Navier-Stokes equations in the above setting has been given by \textsc{Pr\"{u}ss} \& \textsc{Simonett} \cite{PrSi09c}. The basic strategy is to consider the full linearization of the quasilinear problem \eqref{int:NS}-\eqref{int:NS3} at the equilibrium and to compute the spectrum of the linearization. Due to the lack of compactness, there is a portion of approximate eigenvalues in the spectrum of the linearization. In addition, there is no spectral gap which would allow to apply classical tools to carry over the linear stability properties to the nonlinear case. To this end the authors in \cite{PrSi09c} apply Henry's instability theorem \cite[Theorem 5.1.5]{Hen81} which does not require a spectral gap.

In the periodic framework, i.e.\ if $\Omega=\mathbb{T}^2\times\mathbb{R}$, where $\mathbb{T}=\mathbb{R}/\mathbb{Z}$ is the 1-torus a rigorous proof of Rayleigh-Taylor instability has been given by \textsc{Tice} \& \textsc{Wang} in \cite{TiWa12}. Note that if $\Jump{\rho}>0$, then the result in \cite{PrSi09c} states that the trivial equilibrium is always unstable, no matter what the remaining parameters $\mu>0$ and $\sigma>0$ are. However, in the periodic setting considered in \cite{TiWa12}, the stability properties of the trivial equilibrium do also depend on the surface tension. To be more precise, there exists a critical surface tension $\sigma_c>0$ such that if $\sigma>\sigma_c$, then the trivial equilibrium is stable, while if $0<\sigma<\sigma_c$, it is unstable. In other words, even if $\Jump{\rho}>0$, a sufficiently large surface tension $\sigma>0$ of $\Gamma(t)$ prevents the heavier phase of sacking down into the lower phase.

An advantage of the approach via maximal regularity of type $L_p$ which has been used in \cite{PrSi09c} is that one obtains a semi-flow for the free boundary problem in a natural phase space. In particular, there is no loss of regularity. With the help of functional calculus for sectorial operators and harmonic analysis it is then shown that there exists $\lambda_\infty>0$ such that the interval $[0,\lambda_\infty]$ is the unstable part of the spectrum of the linearization. The functional analytic setting used in \cite{PrSi09c} then allows to apply Henry's instability theorem \cite[Theorem 5.1.5]{Hen81} to conclude instability for the nonlinear problem. In contrast to the result in \cite{PrSi09c}, the authors in \cite{TiWa12} construct so-called growing mode solutions (horizontal Fourier modes growing exponentially in time) for the linearized problem and use several energy estimates to study the spectrum of the full linearization. The passage from linear to nonlinear (in-)stability follows from a Guo-Strauss bootstrap procedure, which has been introduced by \textsc{Guo} \& \textsc{Strauss} in \cite{GuoStr95}. Due to the higher order energy estimates, the regularity of the initial values is considerably high and therefore not optimal, when one compares with the assumptions in \cite{PrSi09c}. However, the authors in \cite{TiWa12} obtain a clear picture of the stability properties of the trivial equilibrium in dependence of $\Jump{\rho}$ and $\sigma>0$. Concerning further results on Rayleigh-Taylor instability for different problems, we refer the reader to the selection \cite{BSW14,EMM12,GuoTic11,JTW16,JJW14,JJW16,WX16}.

So far, there seem to be no results on Rayleigh-taylor instability for \eqref{int:NS} in bounded domains. It is one purpose of this article to extend the results obtained in \cite{PrSi09c} to the framework of bounded cylindrical domains. To be precise, we assume that $\Omega=G\times (H_1,H_2)$, where $G\subset\mathbb{R}^{n-1}$, $n\in\{2,3\}$ is a bounded domain with smooth boundary and $H_1<0<H_2$. Suppose furthermore that there is a family of hypersurfaces $\{\Gamma(t)\}_{t\ge 0}$ given as a graph of some height function $h$ over $G$, i.e.
$$\Gamma(t)=\{(x',x_n)\in\Omega:x_n=h(t,x'),\ x'\in G\},\quad t>0,$$
such that for each $t\ge 0$ the interface $\Gamma(t)$ divides $\Omega$ into two subdomains $\Omega_1(t)$ and $\Omega_2(t)$ which are filled with two fluids, respectively. Let us make the convention that $\Omega_2(t)$ is the upper phase. Assuming that the equations \eqref{int:NS}-\eqref{int:NS3} together with the condition $\Jump{u}=0$ are satisfied, we are led in a natural way to the problem of finding suitable boundary conditions on the vertical part $S_1:=\partial G\times (H_1,H_2)$ and the horizontal part $S_2:=(G\times\{H_1\})\cup (G\times \{H_2\})$ of the boundary $\partial\Omega$ of $\Omega$. This turns out to be a delicate question, since within the above setting we are on the one side concerned with two parts $S_1$ and $S_2$ of the boundary such that $\partial S_1=\partial S_2$. Therefore the boundary conditions on $S_1$ and $S_2$ have to be chosen in such a way that they are compatible to each other. On the other side we have to deal with a \emph{contact angle problem}, as $\partial\Gamma(t)$ is a moving contact line on $S_1$. At this point we want to emphasize that the choice of the periodic setting in \cite{TiWa12} allows to circumvent the formation of a contact angle.

The theory of contact angle problems, in particular with a dynamic contact angle which depends on $t$, is yet not well understood. In fact, there exist different point of views about how to model such a problem. One party supposes that the dynamic contact angle is determined by an additional equation, while the other party assumes that the contact angle will be determined by the dynamic equations for the interface and the fluid, hence the equation for the contact angle should be redundant. We refer to \cite{Bil06} \& \cite{Shik97} and to the references given therein.

Therefore, in order to avoid this lack of clarity, we assume throughout this article, that the contact angle is constant and equal to 90 degrees. One can interpret this ansatz as a kind of idealization. It is possible to translate the condition on the contact angle to a condition on the height function $h$ from above. Indeed, if $h$ is sufficiently smooth, then the unit normal on $\Gamma(t)$ with respect to $\Omega_1(t)$ is given by
$$\nu_\Gamma=\frac{1}{\sqrt{1+|\nabla_{x'}h|^2}}
\begin{pmatrix}
-\nabla_{x'} h\\ 1
\end{pmatrix}.$$
Since the outer unit normal on $S_1$ is given by $\nu_{S_1}=(\nu_{\partial G},0)^{\sf T}$, the condition on the contact angle reads $\nu_\Gamma\cdot\nu_{S_1}=0$ or equivalently $\partial_{\nu_{\partial G}} h=0$ at the contact line. Concerning $S_1$ it is not possible to propose Dirichlet boundary conditions, the so-called \emph{no-slip} boundary conditions, since this leads to a paradoxon for the moving contact line, see e.g.\ \cite{Sol83}. The next canonical choice are the so-called \emph{Navier} boundary conditions or \emph{partial-slip} boundary conditions
$$u\cdot\nu_{S_1}=0,\quad P_{S_1}(\mu(\nabla u+\nabla u^{\sf T})\nu_{S_1})+\alpha u=0,$$
where $P_{S_1}:=I-\nu_{S_1}\otimes\nu_{S_1}$ denotes the projection to the tangent space on $S_1$. The parameter $\alpha>0$ has the physical meaning of a friction coefficient. However, it turns out that this kind of boundary condition does not allow the interface to move along $S_1$ which is not very reasonable, as numerical simulations show. To see this, consider for simplicity the case $n=2$. The equation \eqref{int:NS2} in terms of $h$ then reads
\begin{equation}\label{int:NS4}
\partial_t h=u_2-u_1\partial_{1} h,
\end{equation}
where $u=(u_1,u_2)$. Observe that for $n=2$ the partial slip conditions read as follows
$$u_1=0,\quad \mu(\partial_1 u_2+\partial_2 u_1)+\alpha u_2=0.$$
Therefore it holds that $\mu\partial_1 u_2+\alpha u_2=0$, which is a Robin boundary condition for $u_2$ on $S_1$. Differentiating \eqref{int:NS4} with respect to $x_1$, and taking into account that $\partial_1 h=0$ at $S_1$ (by the contact angle condition) we obtain
$\partial_1 u_2=0$, hence $u_2=0$ if $\alpha>0$. Consequently it holds that $\partial_t h=0$ at $S_2$ and therefore $h(t)$ is constant with respect to $t$.

In order to circumvent this problem, we will consider the case $\alpha=0$, the so-called \emph{pure-slip} boundary conditions. From a physical point of view this means that there is no friction on the boundary $S_1$. Having fixed the boundary conditions on $S_1$ we may choose suitable boundary conditions on $S_2$, having in mind that these conditions have to match those on $S_1$. It turns out that homogeneous Dirichlet boundary conditions are a good choice, since they are compatible with the pure-slip boundary conditions on $S_1$ and furthermore they allow to apply \emph{Korn's inequality} for $Du:=\nabla u+\nabla u^{\sf T}$, see Theorem \ref{thm:korn}. Note that the no-slip boundary conditions on $S_2$ do not cause any problems with the moving interface, since we will always have $\Gamma(t)\cap S_2=\emptyset$ for all $t\ge 0$. We are thus led to the problem
\begin{align}\label{int:NScap}
\begin{split}
\partial_t(\rho u)-\mu\Delta u+\rho (u\cdot\nabla)u+\nabla p&=-\rho\gamma_a e_n,\quad \text{in}\ \Omega\backslash\Gamma(t),\\
\div u&=0,\quad \text{in}\ \Omega\backslash\Gamma(t),\\
-\Jump{\mu(\nabla u+\nabla u^{\sf T})}\nu_\Gamma+\Jump{p}\nu_\Gamma&=\sigma H_\Gamma\nu_\Gamma,\quad \text{on}\ \Gamma(t),\\
\Jump{u}&=0,\quad \text{on}\ \Gamma(t),\\
V_\Gamma&=u\cdot\nu_\Gamma,\quad \text{on}\ \Gamma(t),\\
P_{S_1}\left(\mu(\nabla u+\nabla u^{\sf T})\nu_{S_1}\right)&=0,\quad \text{on}\ S_1\backslash\partial\Gamma(t),\\
u\cdot\nu_{S_1}&=0,\quad \text{on}\ S_1\backslash\partial\Gamma(t),\\
u&=0,\quad \text{on}\ S_2,\\
\nu_\Gamma\cdot\nu_{S_1}&=0,\quad \text{on}\ \partial\Gamma(t),\\
u(0)&=u_0,\quad \text{in}\ \Omega\backslash\Gamma(0),\\
\Gamma(0)&=\Gamma_0,
\end{split}
\end{align}
where we denote by $\gamma_a>0$ the acceleration constant due to gravity.

With this article, we present a rather complete analysis of \eqref{int:NScap}. In Chapter \ref{chptr:redmodprbl} we will first transform \eqref{int:NScap} to a fixed domain which does not vary in time. This will be done by means of a height function $h$, assuming that $\Gamma(t)$ is given as the graph of $h$ over the domain $G$. By means of local charts the transformed problem can be drawn back to certain model problems. As the analysis of two types of these model problems, namely the Stokes equations in quarter-spaces and the two-phase Stokes equations in half spaces is not known, we will provide a systematic treatment of these problems subsequently. At this point we want to emphasize that the analysis of the latter problems is more involved than the usual model problems in half spaces. This is due to the fact that one has to deal with mixed boundary conditions meeting at the contact line. However, our assumption on the contact angle enables us to use reflection techniques in order to draw back the quarter space to a half space with dirichlet boundary conditions and the two-phase half space to a two-phase full space with a flat interface.

In Chapter \ref{chptr:localization} we use the results from Chapter \ref{chptr:redmodprbl} combined with a localization procedure to prove existence and uniqueness of a solution of the principal linearization having maximal regularity of type $L_p$. To be precise, if $u$ and $p$ denote the (transformed) velocity field and pressure field, respectively, we will show that $(u,p,\Jump{p},h)$ enjoys the regularity
$$u\in H_p^1(J;L_p(\Omega)^n)\cap L_p(J;H_p^2(\Omega\backslash\Sigma)^n),\quad p\in L_p(J;\dot{H}_p^1(\Omega)),$$
$$\Jump{p}\in W_p^{1/2-1/2p}(J;L_p(\Sigma))\cap L_p(J;W_p^{1-1/p}(\Sigma)).$$
and
$$h\in W_p^{2-1/2p}(J;L_p(\Sigma))\cap H_p^1(J;W_p^{2-1/p}(\Sigma))\cap L_p(J;W_p^{3-1/p}(\Sigma)),$$
where $J=[0,T]$ is some nonempty bounded interval. This optimal regularity result in turn allows to apply the contraction mapping principle in Chapter \ref{chptr:LWP} to obtain a unique solution of the nonlinear problem having optimal regularity as well. In particular, the problem \eqref{int:NScap} generates a local semiflow in a naural phase space.

Chapter \ref{QualBeh} is devoted to the investigation of the stability properties of the trivial equilibrium, i.e.\ $u=0$, $h=0$ and $p$ is constant. It turns out that if $\Jump{\rho}>0$ then there exists a critical surface tension $\sigma_c:=\Jump{\rho}\gamma_a/\lambda_1>0$, where $\lambda_1>0$ denotes the first nontrivial eigenvalue of the Neumann Laplacian in $L_2(G)$. If $\sigma>\sigma_c$ then the trivial equilibrium is exponentially stable in the natural phase space, while in case $\sigma\in(0,\sigma_c)$ it will be unstable. If $\Jump{\rho}\le 0$, then the trivial equilibrium is always exponentially stable. Specializing to the case $G=B_R(0)$, we obtain as a corollary that for fixed surface tension $\sigma>0$ and if $\Jump{\rho}>0$ there exists a critical radius
$$R_c:=\left(\frac{\sigma\lambda_1^*}{\Jump{\rho}\gamma_a}\right)^{1/2},$$
such that if $R<R_c$, then the trivial equilibrium is exponentially stable, while for $R>R_c$ it will be unstable. Here $\lambda_1^*>0$ denotes the first nontrivial eigenvalue of the Neumann Laplacian in $L_2(B_1(0))$, given by $\lambda_1^*=(j_{1,1}')^2$, where $j_{1,1}'$ is the first zero of the derivative $J_1'$ of the Bessel function $J_1$ (see \cite{Abra64}). The proof of the stability result requires some effort, since after the transformation to a fixed domain one has to pay the price that in particular the (transformed) velocity field is no longer divergence free. Therefore, one has to split the solution into two parts in a suitable way such that one part is divergence free while the other part, whose divergence does not vanish, satisfies a nonlinear problem, which can be handled by the implicit function theorem.

The results in Chapter \ref{QualBeh} suggest that if $\sigma$ descreases from $\sigma>\sigma_c$ to $\sigma<\sigma_c$, then an eigenvalue of the full linearization will cross the imaginary axis. Therefore it is natural to ask for possible bifurcations from the trivial equilibrium. In Chapter \ref{chptr:bifurc} we will see that the eigenvalue which crosses the imaginary axis through zero is, unfortunately, not simple if $n=3$. Therefore it is not possible to apply the bifurcation results of Crandall-Rabinowitz directly. By the choice of the boundary conditions, the equilibria of the transformed problem are $u=0$, $p$ is constant and the height function $h$ satisfies the \emph{capillary equation}
\begin{align}\label{int:capeq}
\begin{split}
\sigma \div_{x'}\left(\frac{\nabla_{x'} {h}}{\sqrt{1+|\nabla_{x'} {h}|^2}}\right)+\Jump{\rho}\gamma_a {h}&=0,\quad x'\in B_R(0),\\
\partial_{\nu_{B_R(0)}}{h}&=0,\quad x'\in\partial B_R(0).
\end{split}
\end{align}
This equation for $h$ exhibits certain symmetry properties, in particular we will show that it is invariant under the group action of the orthogonal group $O(2)$. This fact enables us to reduce the bifurcation equation to a one dimensional equation and to apply the implicit function theorem which yields the existence of \emph{subcritical} bifurcating branches from the trivial solution. The remaining part of Chapter \ref{chptr:bifurc} deals with the proof that the bifurcating equilibria are unstable. To this end we compute the full linearization in these equilibria and show that there is at least one eigenvalue in the unstable part of the spectrum of the linearization. The passage to nonlinear instability follows the same lines as in Chapter \ref{QualBeh}.

Finally we decided to collect all technical results which are needed for the execution of the above program in an appendix. Several results concerning extension operators, auxiliary elliptic and parabolic problems in quarter spaces and two-phase half spaces but also in bounded cylindrical domains are provided. In addition, we state the divergence theorem for bounded Lipschitz domains as well as Korn's inequality for functions having a vanishing trace on some nontrivial part of the boundary of the domain.

\medskip

\noindent\textbf{Notation:} The symbols $H_p^s$, $W_p^s$, $s\ge 0$ refer to the Bessel potential spaces and Sobolev-Slobodeckii spaces, respectively (Sobolev spaces for $s\in\mathbb{N}$ with $H=W$). If $J=[0,T]$ is some interval and $X$ a suitable Banach space, then $_0W_p^s(J;X)$ denotes the subspace of $W_p^s(J;X)$ consisting of all functions having a vanishing trace at $t=0$, whenever it exists. Finally we denote by $\dot{W}_p^k(\Omega)=\dot{H}_p^k(\Omega)$ the homogeneous Sobolev space of order $k\in\mathbb{N}$, where $\Omega\subset\mathbb{R}^n$ is some domain.

\aufm{Mathias Wilke}


\mainmatter
%
%
%

\chapter{Preliminaries and model problems}\label{chptr:redmodprbl}

For the sake of readability we will assume throughout this article that the space dimension $n$ is equal to $3$. This is the most important case from a viewpoint of applications. 
Furthermore we will assume from now on that $p>n+2$. In Chapter \ref{chptr:LWP} about the well-posedness of the nonlinear model, this condition on $p$ is a result of some Sobolev embeddings which are needed for the proof. Observe that in case $n=3$ this yields $p>5$.

It is convenient to introduce the \emph{modified pressure} $\tilde{\pi}:=\pi+\rho\gamma_a x_3$ in \eqref{int:NScap}. Then we obtain the following problem.
\begin{align}\label{eq:NScap1mod}
\begin{split}
\partial_t(\rho u)-\mu\Delta u+\rho (u\cdot\nabla)u+\nabla \tilde{\pi}&=0,\quad \text{in}\ \Omega\backslash\Gamma(t),\\
\div u&=0,\quad \text{in}\ \Omega\backslash\Gamma(t),\\
-\Jump{\mu(\nabla u+\nabla u^{\sf T})}\nu_\Gamma+\Jump{\tilde{\pi}}\nu_\Gamma&=\sigma H_\Gamma\nu_\Gamma+\Jump{\rho}\gamma_a x_3\nu_\Gamma,\quad \text{on}\ \Gamma(t),\\
\Jump{u}&=0,\quad \text{on}\ \Gamma(t),\\
V_\Gamma&=u\cdot\nu_\Gamma,\quad \text{on}\ \Gamma(t),\\
P_{S_1}\left(\mu(\nabla u+\nabla u^{\sf T})\nu_{S_1}\right)&=0,\quad \text{on}\ S_1\backslash\partial\Gamma(t),\\
u\cdot\nu_{S_1}&=0,\quad \text{on}\ S_1\backslash\partial\Gamma(t),\\
u&=0,\quad \text{on}\ S_2,\\
\nu_\Gamma\cdot\nu_{S_1}&=0,\quad \text{on}\ \partial\Gamma(t),\\
u(0)&=u_0,\quad \text{in}\ \Omega\backslash\Gamma(0),\\
\Gamma(0)&=\Gamma_0.
\end{split}
\end{align}
Here $\Omega=G\times(H_1,H_2)$, $H_1<0<H_2$, is a cylindrical domain where $G\subset\mathbb{R}^{2}$ is an open bounded domain with a smooth boundary $\partial G$. The compact free boundary $\Gamma(t)$ divides $\Omega$ into two unbounded disjoint phases $\Omega_j(t)$, $j=1,2$, so that $\Omega=\Omega_1(t)\cup\Gamma(t)\cup\Omega_2(t)$. The convention is that $\Omega_2(t)$ is the upper phase while $\Omega_1(t)$ is the lower one with the unit normal $\nu_\Gamma$ at $x\in\Gamma(t)$ pointing from $\Omega_1(t)$ to $\Omega_2(t)$. We denote by $\nu_{S_1}$ the outer unit normal at the fixed boundary $S_1$. The operator $P_{S_1}:=I-\nu_{S_1}\otimes \nu_{S_1}$ stands for the projection to the tangential space on $S_1$

\section{Reduction to a flat interface}

In this section we transform the time-dependent domain $\Omega\backslash\Gamma(t)$ to a fixed domain. To this end,
we assume that
$$\Gamma(t)=\{x\in G\times (H_1,H_2):x_3=h(t,x'),\ x'=(x_1,x_2)\in G,\ t\ge 0\}.$$
Let $\varphi\in C^\infty(\mathbb{R};[0,1]))$ such that $\varphi(s)=1$ if $|s|\le\delta/2$ and $\varphi(s)=0$ if $|s|\ge \delta$, where $\delta<\min\{-H_1,H_2\}/2$. Define a mapping
$$\Theta_h(t,\bar{x}):=\bar{x}
+\varphi(\bar{x}_3)h(t,\bar{x}')e_3=:\bar{x}
+\theta_h(t,\bar{x}),$$
where $\bar{x}:=(\bar{x}',\bar{x}_3)$ and for fixed $t>0$ set $x=\Theta_h(t,\bar{x})$. An easy computation shows 
$$\theta_h'^{\sf T}=
\begin{pmatrix}
0 & 0 & \partial_1h\varphi\\
0 & 0 & \partial_2h\varphi\\
0 & 0 & h\varphi'
\end{pmatrix},
$$
It follows that $\Theta_h'$ is invertible if $\|h\|_{\infty,\infty}<1/(2|\varphi'|_\infty)$ and
$$(\Theta_h')^{-\sf T}=(I+\theta_h'^{\sf T})^{-1}=\frac{1}{1+h\varphi'}\begin{pmatrix}
1+h\varphi' & 0 & -\partial_1h\varphi\\
0 & 1+h\varphi' & -\partial_2h\varphi\\
0 & 0 & 1
\end{pmatrix}.
$$
In the sequel, let $\|h\|_{\infty,\infty}<\eta$ with $0<\eta\le1/(2|\varphi'|_\infty)$ being sufficiently small. Then the inverse $\Theta_h^{-1}:\Omega\to\Omega$ is well defined and it transforms the free interface $\Gamma(t)$ to the flat and fixed interface $\Sigma:=G\times\{0\}$. Now we define the transformed quantities
\begin{align*}
\bar{u}(t,\bar{x})&:=u(t,\Theta_h(t,\bar{x}))\\
\bar{\pi}(t,\bar{x})&:=\tilde{\pi}(t,\Theta_h(t,\bar{x}))
\end{align*}
and compute $\nu_\Gamma=(-\nabla_{x'}h,1)^{\sf T}/\sqrt{1+|\nabla_{x'} h|^2}$,
\begin{align*}\nabla \tilde{\pi}&=\nabla\bar{\pi}-M_0(h)\nabla\bar{\pi}\\
\div u&=\div\bar{u}-(M_0(h)\nabla|\bar{u})\\
\Delta u&=\Delta\bar{u}-M_1(h):\nabla^2\bar{u}-M_2(h)\nabla\bar{u}\\
\partial_t u&=\partial_t\bar{u}-\varphi\partial_t h(1+\varphi'h)^{-1}\partial_{3}\bar{u},
\end{align*}
where $M_0(h):=\theta_h'^{\sf T}(I+\theta_h'^{\sf T})^{-1}$,
$$M_1(h):\nabla^2\bar{u}:=\left[2\operatorname{sym}(\theta_h'^{\sf T}[I+\theta_h']^{-\sf T})-
[I+\theta_h']^{-1}\theta_h'\theta_h'^{\sf T}[I+\theta_h']^{-\sf T}\right]:\nabla^2\bar{u},$$
and
$$M_2(h)\nabla\bar{u}:=\left([\Delta\Theta_h^{-1}]\circ\Theta_h|\nabla\right)\bar{u}.$$
Furthermore it holds that $V_\Gamma=(\partial_t\Theta_h|\nu_\Gamma)=\partial_t h(e_3|\nu_\Gamma)=\partial_t h/\sqrt{1+|\nabla_{x'} h|^2}$.
This yields the following transformed problem for $\bar{u}$ and $\bar{\pi}$ (for convenience we drop the bars in the sequel).
\begin{align}\label{eq:NScap2}
\begin{split}
\partial_t(\rho u)-\mu\Delta u+\nabla \pi&=F(u,\pi,h),\quad \text{in}\ \Omega\backslash\Sigma,\\
\div u&=F_d(u,h),\quad \text{in}\ \Omega\backslash\Sigma,\\
-\Jump{\mu \partial_3 v}-\Jump{\mu\nabla_{x'} w}&=G_v(u,h),\quad \text{on}\ \Sigma,\\
-2\Jump{\mu \partial_3 w}+\Jump{\pi}-\sigma\Delta_{x'} h-\Jump{\rho}\gamma_a h&=G_w(u,h),\quad \text{on}\ \Sigma,\\
\Jump{u}&=0,\quad \text{on}\ \Sigma,\\
\partial_t h-w&=H_1(u,h),\quad \text{on}\ \Sigma,\\
P_{S_1}\left(\mu(\nabla u+\nabla u^{\sf T})\nu_{S_1}\right)&=H_2(u,h),\quad \text{on}\ S_1\backslash\partial\Sigma,\\
u\cdot\nu_{S_1}&=0,\quad \text{on}\ S_1\backslash\partial\Sigma,\\
u&=0,\quad \text{on}\ S_2,\\
\partial_{\nu_{\partial G}}h&=0,\quad \text{on}\ \partial\Sigma,\\
u(0)&=u_0,\quad \text{in}\ \Omega\backslash\Sigma\\
h(0)&=h_0,\quad \text{on}\ \Sigma.
\end{split}
\end{align}
Here
\begin{align*}
F(u,p,h)&:=\rho\varphi\partial_t h(1+\varphi'h)^{-1}\partial_{3}u-\mu(M_1(h):\nabla^2 u+M_2(h)\nabla u)+M_0(h)\nabla \pi\\
F_d(u,h)&:=(M_0(h)\nabla|u)\\
G_v(u,h)&:=-\Jump{\mu(\nabla v+\nabla v^{\sf T})}\nabla h+|\nabla h|^2\Jump{\mu\partial_3 v}\\
&\hspace{3cm}+\left((1+|\nabla h|^2)\Jump{\mu\partial_3 w}-(\nabla h|\Jump{\mu\nabla w})\right)\nabla h\\
G_w(u,h)&:=-(\nabla h|\Jump{\mu\nabla w})-(\nabla h|\Jump{\mu\partial_3 v})+|\nabla h|^2\Jump{\mu \partial_3 w}+\sigma G_\kappa(h)\\
G_\kappa(h)&:=\div\left(\frac{\nabla h}{\sqrt{1+|\nabla h|^2}}\right)-\Delta h\\
H_1(u,h)&:=-(v|\nabla h)\\
H_2(u,h)&:=P_{S_1}(\mu(M_0(h)\nabla u+\nabla u^{\sf T}M_0(h)^{\sf T})\nu_{S_1}),
\end{align*}
where we have set $v:=(u_1,u_2)$, $w:=u_3$ and $\nabla w=\nabla_{x'} w$, $\nabla v=\nabla_{x'} v$, $\nabla h=\nabla_{x'}h$ for the sake of readability.

\section{Linearization, regularity and compatibility conditions}

We consider first the linear part, defined by the left side of \eqref{eq:NScap2}, that is
\begin{align}\label{eq:NScap3}
\begin{split}
\partial_t(\rho u)-\mu\Delta u+\nabla \pi&=f,\quad \text{in}\ \Omega\backslash\Sigma,\\
\div u&=f_d,\quad \text{in}\ \Omega\backslash\Sigma,\\
-\Jump{\mu \partial_3 v}-\Jump{\mu\nabla_{x'} w}&=g_v,\quad \text{on}\ \Sigma,\\
-2\Jump{\mu \partial_3 w}+\Jump{\pi}-\sigma\Delta_{x'} h&=g_w,\quad \text{on}\ \Sigma,\\
\Jump{u}&=u_\Sigma,\quad \text{on}\ \Sigma,\\
\partial_t h-m[w]&=g_h,\quad \text{on}\ \Sigma,\\
P_{S_1}\left(\mu(\nabla u+\nabla u^{\sf T})\nu_{S_1}\right)&=P_{S_1}g_1,\quad \text{on}\ S_1\backslash\partial\Sigma,\\
u\cdot\nu_{S_1}&=g_2,\quad \text{on}\ S_1\backslash\partial\Sigma,\\
u&=g_3,\quad \text{on}\ S_2,\\
\partial_{\nu_{\partial G}}h&=g_4,\quad \text{on}\ \partial\Sigma,\\
u(0)&=u_0,\quad \text{in}\ \Omega\backslash\Sigma\\
h(0)&=h_0,\quad \text{on}\ \Sigma,
\end{split}
\end{align}
where $m[w]:=(w_++w_-)/2$ is the arithmetic mean of the directional traces $w_\pm$ of $w$ to $\Sigma$ from $\Omega_2$ and $\Omega_1$.
Note that we neglected the term $\Jump{\rho}\gamma_a h$ in the jump of the stress tensor, as it is of lower order compared to $\Delta_{x'} h$.

Let $J=[0,T]$ with $T\in (0,\infty)$. We are looking for solutions $(u,\pi)$ of the Stokes equation with
$$u\in H_p^1(J;L_p(\Omega)^3)\cap L_p(J;H_p^2(\Omega\backslash\Sigma)^3),\quad \pi\in L_p(J;\dot{H}_p^1(\Omega)),$$
and
$$\Jump{\pi}\in W_p^{1/2-1/2p}(J;L_p(\Sigma))\cap L_p(J;W_p^{1-1/p}(\Sigma)).$$
Note that the latter regularity condition on $\Jump{\pi}$ is determined by the regularity of the Neumann trace of $u$ on $\Sigma$. For the height function $h$ this yields
$$\Delta_{x'}h\in W_p^{1/2-1/2p}(J;L_p(\Sigma))\cap L_p(J;W_p^{1-1/p}(\Sigma))$$
and
$$\partial_t h\in W_p^{1-1/2p}(J;L_p(\Sigma))\cap L_p(J;W_p^{2-1/p}(\Sigma)),$$
hence the optimal regularity class for $h$ is given by
$$h\in W_p^{2-1/2p}(J;L_p(\Sigma))\cap H_p^1(J;W_p^{2-1/p}(\Sigma))\cap L_p(J;W_p^{3-1/p}(\Sigma)).$$
In the sequel we will always assume that $p>2$. Let us discuss the necessary regularity and compatibility conditions on the data $(f,f_d,g_v,g_w,g_h,g_1,g_2,g_3,g_4,u_\Sigma,u_0,h_0)$. If $(u,\pi,\Jump{\pi},h)$ is a solution of \eqref{eq:NScap3} in the regularity classes stated above, then it holds that $f\in L_p(J;L_p(\Omega)^3)$, $f_d\in L_p(J;H_p^1(\Omega\backslash\Sigma))$
$$(g_v,g_w)\in W_p^{1/2-1/2p}(J;L_p(\Sigma)^3)\cap L_p(J;W_p^{1-1/p}(\Sigma)^3),$$
$$u_\Sigma\in W_p^{1-1/2p}(J;L_p(\Sigma)^3)\cap L_p(J;W_p^{2-1/p}(\Sigma)^3),$$
$$g_h\in W_p^{1-1/2p}(J;L_p(\Sigma))\cap L_p(J;W_p^{2-1/p}(\Sigma)),$$
$$P_{S_1}g_1\in W_p^{1/2-1/2p}(J;L_p(S_1)^3)\cap L_p(J;W_p^{1-1/p}(S_1\backslash\partial\Sigma)^3),$$
$$g_2\in W_p^{1-1/2p}(J;L_p(S_1))\cap L_p(J;W_p^{2-1/p}(S_1\backslash\partial\Sigma)),$$
$$g_3\in W_p^{1-1/2p}(J;L_p(S_2)^3)\cap L_p(J;W_p^{2-1/p}(S_2)^3),$$
$$g_4\in W_p^{3/2-1/p}(J;L_p(\Sigma))\cap H_p^1(J;W_p^{1-2/p}(\Sigma))\cap L_p(J;W_p^{2-2/p}(\Sigma)),$$
$$u_0\in W_p^{2-2/p}(\Omega\backslash\Sigma)^3,\quad h_0\in W_p^{3-2/p}(\Sigma).$$
Concerning compatibility conditions at $t=0$ we have $\div u_0=f_d|_{t=0}$, $$g_v|_{t=0}=-\Jump{\mu \partial_3v_0}-\Jump{\mu\nabla_{x'}w_0},$$
$\Jump{u_0}=u_\Sigma|_{t=0}$, $u_0\cdot \nu_{S_1}=g_2|_{t=0}$, $u_0=g_3|_{t=0}$,
$\partial_{\nu_{\partial G}} h_0=g_4|_{t=0}$ and
$$P_{S_1}(\mu(\nabla u_0+\nabla u_0^{\sf T})\nu_{S_1})=P_{S_1}g_1|_{t=0}.$$
Since $\partial\Sigma\subset S_1\neq\emptyset$ and $\partial S_1\cap\partial S_2\neq\emptyset$, there are additional compatibility conditions which have to be satisfied.

If $(u,\pi,\Jump{\pi},h)$ is a solution of \eqref{eq:NScap3} with the above regularity, then the following compatibility conditions at $\partial\Sigma$ and $\partial S_2$ have necessarily to be satisfied.
\begin{itemize}
\item $\Jump{g_2}=u_\Sigma\cdot \nu_{S_1}$, $\Jump{(g_1\cdot e_3)/\mu-\partial_3 g_2}=\partial_{\nu_{S_1}}(u_\Sigma\cdot e_3),$ at $\partial\Sigma$,
\item $P_{\partial G}[(D'v_\Sigma)\nu']=\Jump{P_{\partial G}g_1'/\mu},$
$\partial_t g_4-m[(g_1\cdot e_3)/\mu-\partial_3 g_2]=\partial_{\nu_{\partial G}}g_h,$ at $\partial\Sigma$,
\item $(g_v|\nu_{\partial G})=-\Jump{g_1\cdot e_3}$ at $\partial \Sigma$, $(g_3|\nu_{S_1})=g_2$ at $\partial S_2$,
\item $P_{\partial G}[\mu(D'g_3')\nu']=(P_{\partial G}{g}_{1}'),$ $\mu\partial_{\nu_{S_1}} (g_3\cdot e_3)+\mu\partial_3 g_2=g_1\cdot e_3$ at $\partial S_2$.
\end{itemize}
Here we use the notation $\nu'=\nu_{\partial G}$, $P_{\partial G}:=I-\nu'\otimes\nu'$, $D'v:=\operatorname{sym}[\nabla_{x'}v]$ and $g':=\sum_{k=1}^{2}(g\cdot e_k)e_k$. These conditions follow easily by comparing the equations $\eqref{eq:NScap3}_{3}$ and $\eqref{eq:NScap3}_{5-10}$ with each other.

There is another compatibility and regularity condition hidden in the system, which stems from the divergence equation. Multiply $\div u=f_d$ by $\phi\in H_{p'}^1(\Omega)$, $p'=p/(p-1)$ and integrate by parts (see Proposition \ref{pro:divthmcyldom}) to the result
\begin{multline}\label{eq:regdiv}
\int_\Omega f_d\phi dx-\int_{S_1} g_2\phi|_{S_1}dS_1-\int_{S_2}(g_3\cdot \nu_{S_1})\phi|_{S_2}dS_2+\int_\Sigma(u_\Sigma\cdot\nu_{\Sigma})\phi|_{\Sigma}d\Sigma\\
=-\int_\Omega u\cdot\nabla\phi dx,
\end{multline}
where $\nu_{S_2}(x',H_2)=e_3$, $\nu_{S_2}(x',H_1)=-e_3$ for $x'\in G$ and $\nu_\Sigma=e_3$. It follows that the functional $[\phi\mapsto\langle(f_d,g_2,g_3,u_\Sigma),\phi\rangle]$ defined by the left side of \eqref{eq:regdiv} is continuous on $H_{p'}^1(\Omega)$ with respect to the semi-norm $\|\nabla\cdot\|_{L_{p'}(\Omega)}$. Since $H_{p'}^1(\Omega)$ is dense in the homogeneous Sobolev space $\dot{H}_{p'}^1(\Omega)$ (the constants are already factorized) with respect to $\|\nabla\cdot\|_{L_{p'}(\Omega)}$ for all domains $\Omega$ which are considered in this article, it follows that $(f_d,g_2,g_3,u_\Sigma)$ determines a functional on $\dot{H}_{p'}^1(\Omega)$, i.e.\ $(f_d,g_2,g_3,u_\Sigma)\in \hat{H}_p^{-1}(\Omega):=(\dot{H}_{p'}^1(\Omega))^*$. The norm of $(f_d,g_2,g_3,u_\Sigma)$ in $\hat{H}_p^{-1}(\Omega)$ is then given by
$$
\|(f_d,g_2,g_3,u_\Sigma)\|_{\hat{H}_p^{-1}}=\sup\{\langle(f_d,g_2,g_3,u_\Sigma),\phi\rangle/
\|\nabla\phi\|_{L_{p'}}:\phi\in H_{p'}^1(\Omega)\}.
$$
Moreover, if $u\in H_p^1(J;L_p(\Omega)^n)$, then $\frac{d}{dt}(f_d,g_2,g_3,u_\Sigma)$ is well defined by the computation above, hence
$$(f_d,g_2,g_3,u_\Sigma)\in H_p^1(J;\hat{H}_p^{-1}(\Omega))$$
is another necessary compatibility and regularity condition on the data. In particular, if $\Omega$ is bounded, then we may choose $\phi=1$ in \eqref{eq:regdiv} to obtain
$$\int_\Omega f_d dx-\int_{S_1} g_2dS_1-\int_{S_2}(g_3\cdot \nu_{S_1})dS_2+\int_\Sigma(u_\Sigma\cdot\nu_{\Sigma})d\Sigma=0.$$

\section{Model problems}

The proof of existence and uniqueness of a solution $(u,\pi,\Jump{\pi},h)$ to \eqref{eq:NScap3} is based on a localization procedure. We will obtain six different types of charts, which yield six different types of model problems. These are
\begin{itemize}
\item the full space Stokes equations (without any boundary- or interface conditions)
\item the two-phase Stokes equations with a flat interface and without any boundary condition
\item the Stokes equations with pure slip boundary conditions in a half-space and no interface
\item the Stokes equations with no-slip boundary conditions in a half-space and no interface
\item the Stokes equations in a quarter space with pure slip boundary conditions on one part of the boundary and no-slip boundary conditions on the other part
\item the two-phase Stokes equations with pure slip boundary conditions in a half-space, a flat interface and a contact angle of 90 degrees.
\end{itemize}
While the first four of these problems are well understood, there seem to be no results on well-posedness of the last two problems.

\subsection{The Stokes equations in quarter-spaces}\label{QS}

Consider the problem
\begin{align}\label{eq:NSquart1}
\begin{split}
\partial_t(\rho u)-\mu\Delta u+\nabla \pi&=f,\quad x_1\in\mathbb{R},\ x_2>0,\ x_3>0,\\
\div u&=f_d,\quad x_1\in\mathbb{R},\ x_2>0,\ x_3>0,\\
\mu[\partial_2 u_1+\partial_1 u_2,\partial_3 u_2+\partial_2 u_3]^{\sf T}&=g_1,\quad x_1\in\mathbb{R},\ x_2=0,\ x_3>0,\\
u_2&=g_2,\quad x_1\in\mathbb{R},\ x_2=0,\ x_3>0,\\
u&=g_3,\quad x_1\in\mathbb{R},\ x_2>0,\ x_3=0,\\
u(0)&=u_0,\quad x_1\in\mathbb{R},\ x_2>0,\ x_3>0.
\end{split}
\end{align}
For convenience, let $\Omega:=\mathbb{R}\times\mathbb{R}_+\times\mathbb{R}_+$, $S_1:=\mathbb{R}\times\{0\}\times\mathbb{R}_+$ and $S_2:=\mathbb{R}\times\mathbb{R}_+\times\{0\}$.

In a first step we extend $u_0\in W_p^{2-2/p}(\Omega)^3$ with respect to $x_2$ via the reflection
$$\tilde{u}_0(x_1,x_2,x_3)=\begin{cases}
u_0(x_1,x_2,x_3),\quad&\text{if}\ x_2>0,\\
-u_0(x_1,-2x_2,x_3)+2u_0(x_1,-x_2/2,x_3),\quad&\text{if}\ x_2<0.
\end{cases}$$
Then $\tilde{u}_0\in W_p^{2-2/p}(\mathbb{R}\times\mathbb{R}\times\mathbb{R}_+)^3$. Applying the same method to
$$g_3 \in W_p^{1-1/2p}(J;L_p(S_2)^3)\cap L_p(J;W_p^{2-1/p}(S_2)^3)$$
yields an extension
$$\tilde{g}_3\in W_p^{1-1/2p}(J;L_p(\mathbb{R}\times\mathbb{R})^3)\cap L_p(J;W_p^{2-1/p}(\mathbb{R}\times\mathbb{R})^3).$$
Furthermore it holds that $\tilde{g}_3|_{t=0}=\tilde{u}_0|_{x_3=0}$, since $g_3|_{t=0}=u_0|_{S_2}$. Then we solve the half-space problem
\begin{align}\label{eq:NSquart2}
\begin{split}
\partial_t\tilde{u}-\Delta \tilde{u}&=0,\quad (x_1,x_2,x_3)\in\mathbb{R}^2\times\mathbb{R}_+,\\
\tilde{u}|_{x_3=0}&=\tilde{g}_3,\quad (x_1,x_2)\in\mathbb{R}^2,\ x_3=0,\\
\tilde{u}(0)&=\tilde{u}_0,\quad (x_1,x_2,x_3)\in\mathbb{R}^2\times\mathbb{R}_+,
\end{split}
\end{align}
to obtain a unique solution $$\tilde{u}\in H_p^1(J;L_p(\mathbb{R}_+^3)^3)\cap L_p(J;H_p^2(\mathbb{R}_+^3)^3).$$
If $(u,\pi)$ is a solution of \eqref{eq:NSquart1}, then the (restricted) function $(u-\tilde{u},\pi)$ solves \eqref{eq:NSquart1} with $u_0=g_3=0$ and some modified data $(f,g_1,g_2)$ (not to be relabeled) in the right regularity classes having a vanishing trace at $t=0$ whenever it exists.

In a next step we extend
$$g_1\in\!_0W_p^{1/2-1/2p}(J;L_p(S_1)^2)\cap L_p(J;W_p^{1-1/p}(S_1)^2),$$
and
$$g_2\in\!_0W_p^{1-1/2p}(J;L_p(S_1))\cap L_p(J;W_p^{2-1/p}(S_1)),$$
w.r.t.\ $x_3$ to some functions
$$\tilde{g}_1\in\!_0W_p^{1/2-1/2p}(J;L_p(\mathbb{R}^2)^2)\cap L_p(J;W_p^{1-1/p}(\mathbb{R}^2)^2),$$
and
$$\tilde{g}_2\in\!_0W_p^{1-1/2p}(J;L_p(\mathbb{R}^2))\cap L_p(J;W_p^{2-1/p}(\mathbb{R}^2)),$$
and solve the half-space problem
\begin{align}\label{eq:NSquart3}
\begin{split}
\partial_t\bar{u}-\Delta \bar{u}&=0,\quad x_1,x_3\in\mathbb{R},\ x_2>0,\\
\mu[\partial_2 \bar{u}_1+\partial_1 \bar{u}_2,\partial_3 \bar{u}_2+\partial_2 \bar{u}_3]^{\sf T}&=\tilde{g}_1,\quad x_1,x_3\in\mathbb{R},\ x_2=0,\\
\bar{u}_2&=\tilde{g}_2,\quad x_1,x_3\in\mathbb{R},\ x_2=0,\\
\bar{u}(0)&=0,\quad x_1,x_3\in\mathbb{R},\ x_2>0,
\end{split}
\end{align}
to obtain a unique solution
$$\bar{u}\in\!_0H_p^1(J;L_p(\mathbb{R}\times\mathbb{R}_+\times\mathbb{R})^3)\cap L_p(J;H_p^2(\mathbb{R}\times\mathbb{R}_+\times\mathbb{R})^3).$$
If $(u,\pi)$ is a solution of \eqref{eq:NSquart1} it follows that the (restricted) function $(u-\tilde{u}-\bar{u},\pi)$ solves the problem
\begin{align}\label{eq:NSquart4}
\begin{split}
\partial_t(\rho u)-\mu\Delta u+\nabla \pi&=f,\quad (x_1,x_2,x_3)\in \Omega,\\
\div u&=f_d,\quad (x_1,x_2,x_3)\in \Omega,\\
\mu[\partial_2 u_1+\partial_1 u_2,\partial_3 u_2+\partial_2 u_3]^{\sf T}&=0,\quad (x_1,x_2,x_3)\in S_1,\\
u_2&=0,\quad (x_1,x_2,x_3)\in S_1,\\
u&=g_3,\quad (x_1,x_2,x_3)\in S_2,\\
u(0)&=0,\quad (x_1,x_2,x_3)\in \Omega,
\end{split}
\end{align}
with some modified data $(f,f_d,g_3)$ in the right regularity classes having a vanishing trace at $t=0$ whenever it exists. Note that $g_3:=\bar{u}|_{x_3=0}$ and the compatibility conditions
$(g_3)_2=\partial_2 (g_3)_{1,3}=0$ hold if $x_1\in\mathbb{R}$, $x_2=0$ and $x_3=0$. We will now extend $(f_1,f_3,f_d,(g_3)_{1,3})$ by even reflection and $(f_2,(g_3)_{2})$ by odd reflection to $\{x_2<0\}$. Then we consider the (reflected) half-space problem
\begin{align}\label{eq:NSquart5}
\begin{split}
\partial_t(\rho \hat{u})-\mu\Delta \hat{u}+\nabla \hat{\pi}&=\tilde{f},\quad x_1,x_2\in\mathbb{R},\ x_3>0,\\
\div \hat{u}&=\tilde{f}_d,\quad x_1,x_2\in\mathbb{R},\ x_3>0,\\
\hat{u}&=\tilde{g}_3,\quad x_1,x_2\in\mathbb{R},\ x_3=0,\\
\hat{u}(0)&=0,\quad x_1,x_2\in\mathbb{R},\ x_3>0,
\end{split}
\end{align}
which has a unique solution
$$\hat{u}\in\!_0H_p^1(J;L_p(\mathbb{R}_+^3)^3)\cap L_p(J;H_p^2(\mathbb{R}_+^3)^3),$$
$$\hat{\pi}\in L_p(J;\dot{H}_p^1(\mathbb{R}_+^3)),$$
by \cite[Theorem 6.1]{BP07}.

The (restricted) pair $(u,\pi):=(\tilde{u}+\bar{u}+\hat{u},\hat{\pi})$ is the desired unique solution to \eqref{eq:NSquart1}. We have thus proven the following
\begin{thm}\label{thm:NSquart}
Let $n=3$, $p>5$, $T>0$, $\rho,\mu>0$ and $J=[0,T]$. Then there exists a unique solution
$$u\in H_p^1(J;L_p(\Omega)^3)\cap L_p(J;H_p^2(\Omega)^3)$$
$$\pi\in L_p(J;\dot{H}_p^1(\Omega))$$
of \eqref{eq:NSquart1} if and only if the data satisfy the following regularity and compatibility conditions.
\begin{enumerate}
\item $f\in L_p(J;L_p(\Omega)^3)$;
\item $f_d\in L_p(J;H_p^1(\Omega))$;
\item $g_1\in  W_p^{1/2-1/2p}(J;L_p(S_1)^2)\cap L_p(J;W_p^{1-1/p}(S_1)^2)$;
\item $g_2\in W_p^{1-1/2p}(J;L_p(S_1))\cap L_p(J;W_p^{2-1/p}(S_1))$;
\item $g_3\in W_p^{1-1/2p}(J;L_p(S_2)^3)\cap L_p(J;W_p^{2-1/p}(S_2)^3)$;
\item $u_0\in W_p^{2-2/p}(\Omega)^3$;
\item $\div u_0=f_d|_{t=0}$, $\mu[\partial_2 (u_0)_1+\partial_1 (u_0)_2,\partial_3 (u_0)_2+\partial_2 (u_0)_3]|_{x_2=0}^{\sf T}=g_1|_{t=0}$;
\item $(u_0)_2|_{x_2=0}=g_2|_{t=0}$, $u_0|_{x_3=0}=g_3|_{t=0}$;
\item $(g_3)_2|_{x_2=0}=g_2|_{x_3=0}$, $\mu[\partial_2 (g_3)_1+\partial_1 (g_3)_2,\partial_3 g_2|_{x_3=0}+\partial_2 (g_3)_3]|_{x_2=0}^{\sf T}=g_1|_{x_3=0}$;
\item $(f_d,g_2,g_3)\in H_p^1(J;\hat{H}_p^{-1}(\Omega))$.
\end{enumerate}
\end{thm}

\subsection{The Stokes equations in bent quarter-spaces}\label{bentQS}

Let $\theta\in BC^{3}(\mathbb{R})$ such that
$$G_\theta:=\{(x_1,x_2)\in\mathbb{R}^2:x_2>\theta(x_1)\}\quad \text{and}\quad \Omega_\theta=G_\theta\times\mathbb{R}_+.$$
We assume furthermore that $|\theta'|_\infty\le \eta$ and $|\theta^{(j)}|_\infty\le M$, $j\in\{2,3\}$, where we may choose $\eta>0$ as small as we wish. Let $S_{1,\theta}:=\partial G_\theta\times\mathbb{R}_+$ and $S_{2,\theta}:=G_\theta\times\{0\}$. Furthermore, let $\nu_{S_{1,\theta}}=(\nu_{G_\theta},0)^{\textsf{T}}$ with $\nu_{G_\theta}:=\frac{1}{\sqrt{1+\theta'(x_1)^2}}(\theta'(x_1),-1)^{\textsf{T}}$ denote the outer unit normal to $S_{1,\theta}$ at $(x_1,\theta(x_1),x_3)$, $(x_1,x_3)\in\mathbb{R}\times\mathbb{R}_+$ and let $P_{S_{1,\theta}}$ be the tangential projection to $S_{1,\theta}$. Consider the problem
\begin{align}\label{eq:NSbendquart1}
\begin{split}
\partial_t(\rho u)-\mu\Delta u+\nabla \pi&=f,\quad x\in\Omega_\theta,\\
\div u&=f_d,\quad x\in\Omega_\theta,\\
P_{S_{1,\theta}}[\mu (Du)\nu_{S_{1,\theta}}]&=P_{S_1}g_1,\quad x\in S_{1,\theta},\\
(u|\nu_{S_{1,\theta}})&=g_2,\quad x\in S_{1,\theta},\\
u&=g_3,\quad x\in S_{2,\theta}\\
u(0)&=u_0,\quad x\in\Omega_\theta.
\end{split}
\end{align}
Here $\rho,\mu>0$ are given constants. Note that since $\nu_{S_1}=(\nu_{\partial G},0)^{\sf T}$ it holds that
\begin{equation}\label{eq:NSbendquart1a}
P_{S_{1,\theta}}[\mu (Du)\nu_{S_{1,\theta}}]=
\begin{pmatrix}
P_{\partial G_\theta}[\mu(D'v)\nu_{G_\theta}]\\
\mu\partial_3 (v|\nu_{G_\theta})+\mu\partial_{\nu_{G_\theta}} w,
\end{pmatrix}
\end{equation}
where $D'=D_{(x_1,x_2)}$ and $u=(v,w)$. Therefore, the given data $(f,f_d,g_1,g_2,g_3,u_0)$ is subject to the compatibility conditions $(g_3|\nu_{S_{1,\theta}})=g_2|_{S_{2,\theta}}$,
$$P_{\partial G_\theta}[\mu (D'g_3')\nu_{S_{1,\theta}}]=P_{\partial G_\theta}g_{1}',$$
and $\mu(\partial_3 g_2+\partial_{\nu_{G_\theta}} (g_3\cdot e_3))=g_{1}\cdot e_3$
at the contact line $\{(x_1,\theta(x_1),0):x_1\in\mathbb{R}\}$, where
$$g_{j}':=\sum_{k=1}^2(g_j\cdot e_k) e_k$$
for $j\in\{1,3\}$. Furthermore, at $t=0$ we have $\div u_0=f_d|_{t=0}$, $u_0|_{S_{2,\theta}}=g_3|_{t=0}$, $(u_0|\nu_{S_{1,\theta}})=g_2|_{t=0}$ and $P_{S_{1,\theta}}[\mu (Du_0)\nu_{S_{1,\theta}}]=P_{S_{1,\theta}}g_1|_{t=0}$. Lastly, $(f_d,g_2,g_3)\in H_p^1(J;\hat{H}_p^{-1}(\Omega_\theta))$.

For convenience we shall reduce \eqref{eq:NSbendquart1} to the case $u_0=f=g_3=0$. To this end we first extend $u_0$ and $f$ to some $\tilde{u}_0\in W_p^{2-2/p}(\mathbb{R}^3)^3$ and $f\in L_p(J;L_p(\mathbb{R}^3)^3)$ and solve the full-space problem
\begin{align*}
\partial_t(\rho\tilde{u})-\mu\Delta\tilde{u}&=\tilde{f},\quad\ \text{in}\ \mathbb{R}^3,\\
\tilde{u}(0)&=\tilde{u}_0,\quad\ \text{in}\ \mathbb{R}^3,
\end{align*}
to obtain a unique solution
$$\tilde{u}\in H_p^1(J;L_p(\mathbb{R}^3)^3)\cap L_p(J;H_p^1(\mathbb{R}^3)^3).$$
Let now $\tilde{g}_3:=g_3-\tilde{u}|_{S_2}$. Then $\tilde{g}_3|_{t=0}=0$ by construction and we may extend $\tilde{g}_3$ to some
$$\hat{g}_3\in\! _0W_p^{1-1/2p}(J;L_p(\mathbb{R}^2)^3)\cap L_p(J;W_p^{2-1/p}(\mathbb{R}^2)^3).$$
With $\hat{g}_3$ at hand we solve the half-space problem
\begin{align*}
\partial_t(\rho\hat{u})-\mu\Delta\hat{u}&=0,\quad\ \text{in}\ \mathbb{R}_+^3,\\
\hat{u}&=\hat{g}_3,\quad\text{on}\ \mathbb{R}^2,\\
\hat{u}(0)&=0,\quad\ \text{in}\ \mathbb{R}_+^3,
\end{align*}
to obtain a unique solution
$$\hat{u}\in H_p^1(J;L_p(\mathbb{R}_+^3)^3)\cap L_p(J;H_p^1(\mathbb{R}_+^3)^3).$$
If $u$ is a solution of \eqref{eq:NSbendquart1}, it follows that the (restricted) function $\bar{u}:=u-\tilde{u}-\hat{u}$ solves \eqref{eq:NSbendquart1} with $f=u_0=g_3=0$ and some modified functions $\bar{f}_d,\bar{g}_j$, $j\in\{1,2\}$ in the correct regularity classes satisfying the compatibility conditions $\bar{g}_2|_{S_{2,\theta}}=0$, $P_{\partial G_\theta}\bar{g}_{1}'=0$ and $\bar{g}_{1}\cdot e_3=\partial_3 \bar{g}_2$ at the contact line. Moreover, $(\bar{f}_d,\bar{g}_2,0)\in\!_0H_p^1(J;\hat{H}_p^{-1}(\Omega_\theta))$.

Observe that by the identity $(P_{S_{1,\theta}}w|\nu_{S_{1,\theta}})=0$, $w\in\mathbb{R}^3$, the second component of $P_{S_{1,\theta}}w$ is redundant (it can always be calculated from the first one). Therefore we may replace the term $P_{S_{1,\theta}}[\mu(Du)\nu_{S_{1,\theta}}]$ by its first and last component, i.e.\ we consider the two equations
$$P_{S_{1,\theta}}[\mu(Du)\nu_{S_{1,\theta}}]\cdot e_j=P_{S_{1,\theta}}g_1\cdot e_j$$
for $j\in\{1,3\}$. Observe also that $P_{S_{1,\theta}}g_1\cdot e_3=g_1\cdot e_3$, since the last component of $\nu_{S_{1,\theta}}$ is identically zero..

In what follows we will transform the domain $G_\theta$ to $G:=\mathbb{R}\times\mathbb{R}_+$, the boundaries $S_{1,\theta}$ and $S_{2,\theta}$ to $S_1:=\partial G\times\mathbb{R}_+$ and $S_2:=G\times\{0\}$, respectively, and, hence, $\Omega_\theta$ to $\Omega:=G\times\mathbb{R}_+$.
To this end we introduce the new variables $\bar{x}_1=x_1$, $\bar{x}_2=x_2-\theta(x_1)$ and $\bar{x}_3=x_3$ for $x\in\Omega_\theta=G_\theta\times\mathbb{R}_+$. Suppose that $(u,\pi)$ is a solution of \eqref{eq:NSbendquart1} and define the new functions $$\bar{u}(\bar{x}):=u(\bar{x}_1,\bar{x}_2+\theta(\bar{x}_1),\bar{x}_3)$$
and
$$\bar{\pi}(\bar{x}):=\pi(\bar{x}_1,\bar{x}_2+\theta(\bar{x}_1),\bar{x}_3),$$
where $\bar{x}:=(\bar{x}_1,\bar{x}_2,\bar{x}_3)$.
In the same way we transform the data $(f_d,g_1,g_2)$ to $(\bar{f}_d,\bar{g}_1,\bar{g}_2)$. It holds that $\partial_{\bar{k}}^j\bar{u}=\partial_k^j u$ for $k\in\{2,3\}$, $j\in\{1,2\}$,
$$\partial_{1}u=\partial_{1}\bar{u}-\theta'(\bar{x}_1)\partial_{2}\bar{u}$$
and
$$\partial_{1}^2u=\partial_{1}^2\bar{u}-2\theta'(\bar{x}_1)\partial_{1}
\partial_{2}\bar{u}-\theta''(\bar{x}_1)\partial_{2}\bar{u}
+\theta'(\bar{x}_1)^2\partial_{2}^2\bar{u}.$$
Therefore, the pair $(\bar{u},\bar{\pi})$ solves the following problem
\begin{align}\label{eq:NSbendquart2}
\begin{split}
\partial_t(\rho \bar{u})-\mu{\Delta} \bar{u}+{\nabla} \bar{\pi}&=M_1(\theta,\bar{u},\bar{\pi}),\quad \bar{x}\in\Omega,\\
{\div}\ \bar{u}&=M_2(\theta,\bar{u})+\bar{f}_d,\quad \bar{x}\in\Omega,\\
\mu(\partial_{1} \bar{u}_2+\partial_{2} \bar{u}_1)&=M_3(\theta,\bar{u})-\sqrt{1+\theta'^2}^3[P_{S_{1,\theta}}\bar{g}_1\cdot e_1],\quad \bar{x}\in S_1,\\
\mu(\partial_{2} \bar{u}_3+\partial_{3} \bar{u}_2)
&=M_4(\theta,\bar{u})-\sqrt{1+\theta'^2}[\bar{g}_1\cdot e_3],\quad \bar{x}\in S_1,\\
\bar{u}_2&=M_5(\theta,\bar{u})-\sqrt{1+\theta'^2}\bar{g}_2,\quad \bar{x}\in S_1,\\
\bar{u}&=0,\quad \bar{x}\in S_2,\\
\bar{u}(0)&=0,\quad \bar{x}\in \Omega,
\end{split}
\end{align}
where the functions $M_j$ are given by
$$M_1(\theta,\bar{u},\bar{\pi}):=2\theta'(\bar{x}_1)\partial_{1}
\partial_{2}\bar{u}+\theta''(\bar{x}_1)\partial_{2}\bar{u}
-\theta'(\bar{x}_1)^2\partial_{2}^2\bar{u}
+\theta'(\bar{x}_1)\partial_{2}\bar{\pi}e_1,$$
$$M_2(\theta,\bar{u}):=\theta'(\bar{x}_1)\partial_{2}\bar{u}_1,$$
$$M_3(\theta,\bar{u}):=\mu\theta'(\bar{x}_1)[2\partial_{1}\bar{u}_1+
\theta'(\bar{x}_1)(\partial_{1}\bar{u}_2-\partial_{2}\bar{u}_1)-(1+
\theta'(\bar{x}_1)^2)\partial_{2}\bar{u}_2],$$
$$M_4(\theta,\bar{u}):=\mu\theta'(\bar{x}_1)(\partial_{1} \bar{u}_3-\theta'(\bar{x}_1)\partial_{2}\bar{u}_3+
\partial_{3}\bar{u}_1),$$
$$M_5(\theta,\bar{u}):=\theta'(\bar{x}_1)\bar{u}_1.$$
Now we want to go back from \eqref{eq:NSbendquart2} to \eqref{eq:NSbendquart1}. To this end we consider the functions on the right hand side of \eqref{eq:NSbendquart2} as given data in the right regularity classes. Our aim is the to interpret \eqref{eq:NSbendquart2} as a perturbation of \eqref{eq:NSquart1}, provided $|\theta'|_{\infty}<\eta$ and $\eta>0$ is sufficiently small. It is therefore reasonable to solve \eqref{eq:NSbendquart2} by a Neumann series argument. To this end let
$$_0\mathbb{E}_u(T):=\{u\in\!_0H_p^1(J;L_p(\Omega)^3)\cap L_p(J;H_p^2(\Omega)^3):u|_{S_2}=0\},$$
$$\mathbb{E}_\pi(T):=L_p(J;\dot{H}_p^1(\Omega)),$$
$_0\mathbb{E}(T):=\!_0\mathbb{E}_u(T)\times \mathbb{E}_\pi(T)$,
$$\tilde{\mathbb{F}}(T):=\mathbb{F}_1(T)\times \mathbb{F}_2(T)\times_{j=3}^5\, _0\mathbb{F}_j(T),$$ where
$$\mathbb{F}_1(T):=L_p(J;L_p(\Omega)^3),$$
$$\mathbb{F}_2(T):=L_p(J;H_p^1(\Omega)),$$
$$_0\mathbb{F}_3(T):=\!_0
W_p^{1/2-1/2p}(J;L_p(S_1))\cap
L_p(J;W_p^{1-1/p}(S_1)),$$
$_0\mathbb{F}_4(T):=\!_0\mathbb{F}_3(T)$, and
$$_0\mathbb{F}_5(T):=
\!_0W_p^{1-1/2p}(J;L_p(S_1))\cap
L_p(J;W_p^{2-1/p}(S_1)).$$
Finally, we set
$$_0\mathbb{F}(T):=\{(f_1,\ldots,f_5)\in\tilde{\mathbb{F}}(T):(9)\ \&\ (10)\ \text{in Theorem \ref{thm:NSquart} are satisfied}\}.$$
Define an operator $L:\,_0\mathbb{E}(T)\to\!_0\mathbb{F}(T)$ by
$$L(\bar{u},\bar{\pi}):=\begin{bmatrix}
\partial_t(\rho \bar{u})-\mu\Delta \bar{u}+\nabla \bar{\pi}\\
\div \bar{u}\\
\mu(\partial_2 \bar{u}_1+\partial_1 \bar{u}_2)|_{S_1}\\
\mu(\partial_3 \bar{u}_2+\partial_2 \bar{u}_3)|_{S_1}\\
\bar{u}_2|_{S_1}
\end{bmatrix}
$$
and note that $L:\,_0\mathbb{E}(T)\to\!_0\mathbb{F}(T)$ is an isomorphism by Theorem \ref{thm:NSquart}. Define $$M(\theta,\bar{u},\bar{\pi}):=(M_1(\theta,\bar{u},\bar{\pi}),{M}_2(\theta,\bar{u}),
{M}_3(\theta,\bar{u})
,{M}_4(\theta,\bar{u}),{M}_5(\theta,\bar{u}))^{\sf T}$$
and
$$F:=(0,f_2,f_3,f_4,f_5)^{\sf T},$$
with $f_2:=\bar{f}_d$,
$$f_3:=-\sqrt{1+\theta'^2}^3[P_{S_{1,\theta}}\bar{g}_1\cdot e_1],\ f_4:=-\sqrt{1+\theta'^2}[\bar{g}_1\cdot e_3]$$
and $f_5:=-\sqrt{1+\theta'^2}\bar{g}_2$. By the smoothness of $\theta$, it follows that $F\in\tilde{\mathbb{F}}(T)$. So it remains to check that the compatibility conditions (9) \& 10 in Theorem \ref{thm:NSquart} are satisfied. Since $\bar{g}_2|_{S_2}=0$, $P_{\partial G_\theta}\bar{g}_{1,v}=0$ and $\bar{g}_{1,w}=\partial_3 \bar{g}_2$ at the contact line, the compatibility conditions in Theorem \ref{thm:NSquart} (9) are easily verified. To verify (10) in Theorem \ref{thm:NSquart} we have to show that $(f_2,f_5,0)\in\!_0H_p^1(J;\hat{H}_p^{-1}(\Omega))$. Note that for the reduced data from above we have $(f_d,g_2,0)\in\!_0H_p^1(J;\hat{H}_p^{-1}(\Omega_\theta))$, hence for a.e.\ $t\in J$ the functional $\Psi(t):H_{p'}^1(\Omega_\theta)\to\mathbb{R}$ defined by
$$\langle\Psi(t),\phi\rangle:=\int_{\Omega_\theta} f_d(t)\phi\ dx-\int_{S_{1,\theta}} g_2(t)\phi|_{S_{1,\theta}}\ dS_\theta$$
as well as its derivative with respect to $t$ are continuous with respect to the norm $\|\nabla\cdot\|_{L_{p'}(\Omega_\theta)}$. Transforming $\Omega_\theta$ to the quarter space $\Omega$ and $S_{1,\theta}$ to ${S}_1$ via the above diffeomorphism $\Phi(x_1,x_2,x_3)=(x_1,x_2-\theta(x_1),x_3)$ yields
\begin{multline*}
\int_{\Omega_\theta} f_d(t)\phi\ dx-\int_{S_{1,\theta}} g_2(t)\phi|_{S_{1,\theta}}\ dS_\theta=\\
=\int_\Omega \bar{f}_d(t)\bar{\phi}\ d\bar{x}-\int_{{S}_1} \sqrt{1+\theta'(\bar{x})^2}\bar{g}_2(t)\bar{\phi}|_{S_1}\ d{S},
\end{multline*}
where $\bar{\phi}(\bar{x}_1,\bar{x}_2,\bar{x}_3):=\phi(x_1,x_2-\theta(x_1),x_3)$. This shows that for a.e.\ $t\in J$ the functional $\bar{\Psi}(t):H_{p'}^1(\Omega)\to\mathbb{R}$ given by
$$\langle\bar{\Psi}(t),\bar{\phi}\rangle:=\int_\Omega \bar{f}_d(t)\bar{\phi}\ d\bar{x}-\int_{{S}_1} \sqrt{1+\theta'(\bar{x})^2}\bar{g}_2(t)\bar{\phi}|_{S_{1}}\ d{S}$$
and its derivative with respect to $t$ are continuous with respect to the norm $\|\nabla\cdot\|_{L_{p'}(\Omega)}$, hence $(f_2,f_5,0)\in\!_0H_p^1(J;\hat{H}_p^{-1}(\Omega))$. This implies $F\in\!_0\mathbb{F}(T)$.

Concerning $M(\theta,\bar{u},\bar{\pi})$, we observe that for $\bar{u}\in\!_0\mathbb{E}_u(T)$ we have $\bar{u}=0$ as well as $\partial_j \bar{u}=0$ at $S_2$ for $j\in\{1,2\}$ and therefore also at the line
$$\partial S_1=\partial S_2=\overline{S_1}\cap\overline{S_2}=\mathbb{R}\times\{0\}\times\{0\},$$
by continuity of $\bar{u}$ and $\partial_j \bar{u}$ in $\overline{\Omega}$. Therefore $M_3(\theta,\bar{u})=M_5(\theta,\bar{u})=0$ at $\overline{S_1}\cap\overline{S_2}$. Moreover,
$$M_4(\theta,\bar{u})=\mu\theta'(\bar{x}_1)\partial_3 \bar{u}_1$$
at $\overline{S_1}\cap\overline{S_2}$, hence $\mu\partial_3 M_5(\theta,\bar{u})=M_4(\theta,\bar{u})$. It remains to verify the condition
$$(M_2(\theta,\bar{u}),-M_5(\theta,\bar{u}),0)\in \!_0H_p^1(J;\hat{H}_p^{-1}(\Omega))$$
for $\bar{u}\in \!_0\mathbb{E}_u(T)$. We compute
\begin{align*}
\int_\Omega M_2(\theta,\bar{u})\phi\ d\bar{x}-&\int_{S_1}(-M_5(\theta,\bar{u}))\phi|_{S_1}\ dS\\
&=\int_\Omega \theta'(\bar{x}_1)(\partial_2\bar{u}_1)\phi\ d\bar{x}+\int_{S_1}\theta'(\bar{x}_1)\bar{u}_1\phi|_{S_1}\ dS\\
&=-\int_\Omega\theta'(\bar{x}_1) \bar{u}_1\partial_2\phi\ d\bar{x},
\end{align*}
for each $\phi\in H_{p'}^1(\Omega)$, where we integrated by parts with respect to the variable $\bar{x}_2$. This yields the claim.

It follows that $M(\theta,\bar{u})\in\!_0\mathbb{F}(T)$ for each $(\bar{u},\bar{\pi})\in\!_0\mathbb{E}(T)$ and therefore we may rewrite \eqref{eq:NSbendquart2} shortly as $(\bar{u},\bar{\pi})=L^{-1}M(\theta,\bar{u},\bar{\pi})+L^{-1}F$ in $_0\mathbb{E}(T)$. We intend to show that for each $\varepsilon>0$ there exist $T_0>0$ and $\eta_0>0$ such that
\begin{equation}\label{eq:NSbendquart4}
\|M(\theta,\bar{u},\bar{\pi})\|_{\mathbb{F}(T)}\le\varepsilon \|(\bar{u},\bar{\pi})\|_{\mathbb{E}(T)},
\end{equation}
provided that $T\in (0,T_0)$ and $\eta\in (0,\eta_0)$.

The above computation for $(M_2,M_5,0)$ readily yields that
$$\|(M_2(\theta,\bar{u}),M_5(\theta,\bar{u}),0)\|_{H_p^1(J;\hat{H}_p^{-1}(\Omega))}\le \|\theta'\|_\infty\|\bar{u}\|_{\mathbb{E}_u(T)}.$$
Moreover, it holds that
\begin{align*}
\|M_2(\theta,\bar{u})\|_{L_p(J;H_p^1(\Omega))}&\le \|\theta'\|_\infty\|\bar{u}\|_{\mathbb{E}_u(T)}
+\|\theta''\|_\infty\|\bar{u}\|_{L_p(J;H_p^1(\Omega))}\\
&\le\|\theta'\|_\infty\|\bar{u}\|_{\mathbb{E}_u(T)}
+T^{1/2p}\|\theta''\|_\infty\|\bar{u}\|_{L_{2p}(J;H_p^1(\Omega))}\\
&\le (\|\theta'\|_\infty+T^{1/2p}C\|\theta''\|_\infty)\|\bar{u}\|_{\mathbb{E}_u(T)},
\end{align*}
where the constant $C>0$ stems from the embeddings
$$_0H_p^1(J;L_p(\Omega))\cap L_p(J;H_p^2(\Omega))\hookrightarrow\!_0H_p^{1/2}(J;H_p^1(\Omega))\hookrightarrow L_{2p}(J;H_p^1(\Omega)),$$
valid for each $p>1$. Note that $C>0$ does not depend on $T>0$, since $\bar{u}|_{t=0}=0$. The estimate for $M_1$ is very easy. Indeed, by H\"{o}lder's inequality we obtain
$$\|M_1(\theta,\bar{u},\bar{\pi})\|_{L_p(J;L_p(\Omega))}\le C\left[
\|\theta'\|_\infty(1+\|\theta'\|_\infty)+T^{1/2p}\|\theta''\|_\infty\right]
\|(\bar{u},\bar{\pi})\|_{\mathbb{E}(T)}.$$
Again, $C>0$ does not depend on $T>0$. The estimates for $M_3,M_4$ are nearly the same. So we just concentrate on $M_4$.
\begin{align*}
\|M_4(\theta,\bar{u})\|_{\mathbb{F}_4(T)}&\le
\|M_4(\theta,\bar{u})\|_{W_p^{1/2-1/2p}(J;L_p(S_1))}
+\|M_4(\theta,\bar{u})\|_{L_p(J;W_p^{1-1/p}(S_1))}\\
&\le C\left[\|\theta'\|_\infty\|\bar{u}\|_{\mathbb{E}_u(T)}+
\|M_4(\theta,\bar{u})\|_{L_p(J;W_p^{1-1/p}(S_1))}\right].
\end{align*}
To estimate last term, it suffices to consider a term of the form $\theta'\partial_j \bar{u}$ in $L_p(J;W_p^{1-1/p}(S_1))$ for some $j\in\{1,2,3\}$. Making use of the embedding
$$L_p(J;H_p^1(\Omega))\hookrightarrow L_p(J;W_p^{1-1/p}(S_1))$$
we obtain
$$
\|\theta'\partial_j \bar{u}\|_{L_p(J;W_p^{1-1/p}(S_1))}\le C\|\theta'\partial_j\bar{u}\|_{L_p(J;H_p^1(\Omega))}\le C\left[\|\theta'\|_\infty+T^{1/2p}\|\theta''\|_\infty\right]\|\bar{u}\|_{\mathbb{E }_u(T)},
$$
with $C>0$ being independent of $T>0$. Finally, it remains to estimate $M_5$ in $\mathbb{F}_5(T)$. We employ the embedding
$$L_p(J;H_p^2(\Omega))\hookrightarrow L_p(J;W_p^{2-1/p}(S_1))$$
to the result
\begin{align*}
\|\theta' \bar{u}_1\|_{L_p(J;W_p^{2-1/p}(S_1))}&\le C\|\theta' \bar{u}_1\|_{L_p(J;H_p^{2}(\Omega))}\\
&\le C\left[\|\theta'\|_\infty+T^{1/2p}(\|\theta''\|_\infty
+\|\theta'''\|_\infty)\right]\|\bar{u}\|_{\mathbb{E}_u(T)}.
\end{align*}
Collecting everything together, we have shown that
$$\|M(\theta,\bar{u},\bar{\pi})\|_{\mathbb{F}(T)}\lesssim \left[\|\theta'\|_\infty+T^{1/2p}(\|\theta''\|_\infty+\|\theta'''\|_\infty)\right]
\|(\bar{u},\bar{\pi})\|_{\mathbb{E}(T)}.$$
Recall that $\|\theta'\|_\infty<\eta$. Therefore, choosing first $\eta>0$, then $T>0$ small enough, we obtain the desired estimate \eqref{eq:NSbendquart4}. A Neumann series argument in $_0\mathbb{E}(T)$ finally implies that there exists a unique solution $(\bar{u},\bar{\pi})\in\!_0\mathbb{E}(T)$ of the equation $L(\bar{u},\bar{\pi})=M(\theta,\bar{u},\bar{\pi})+F$ or equivalently a solution $(u,\pi)$ of \eqref{eq:NSbendquart1}, provided that the data satisfy all relevant compatibility conditions at the contact line $\overline{S_1}\cap\overline{S_2}$.

This in turn yields a solution operator $S_{QS}:\mathbb{F}_{QS}\to\mathbb{E}_{QS}$ for \eqref{eq:NSbendquart1}, where $\mathbb{E}_{QS}$ and $\mathbb{F}_{QS}$ are the solution space and data space, respectively, for the bent quarter-space and the data in $\mathbb{F}_{QS}$ satisfy all relevant compatibility conditions at the contact line $\{(x_1,\theta(x_1),0):x_1\in\mathbb{R}\}$.

\subsection{The two-phase Stokes equations in half-spaces}\label{HS}

Consider the problem
\begin{align}\label{eq:NScap4}
\begin{split}
\partial_t(\rho u)-\mu\Delta u+\nabla \pi&=f,\quad x_1\in\mathbb{R},\ x_2>0,\ x_3\in\dot{\mathbb{R}},\\
\div u&=f_d,\quad x_1\in\mathbb{R},\ x_2>0,\ x_3\in\dot{\mathbb{R}},\\
-\Jump{\mu \partial_3 v}-\Jump{\mu\nabla_{x'} u_3}&=g_v,\quad x_1\in\mathbb{R},\ x_2>0,\ x_3=0,\\
-2\Jump{\mu \partial_3 u_3}+\Jump{\pi}-\sigma\Delta_{x'} h&=g_w,\quad x_1\in\mathbb{R},\ x_2>0,\ x_3=0,\\
\Jump{u}&=u_\Sigma,\quad x_1\in\mathbb{R},\ x_2>0,\ x_3=0,\\
\partial_t h-m[u_3]&=g_h,\quad x_1\in\mathbb{R},\ x_2>0,\ x_3=0,\\
\mu[\partial_2 u_1+\partial_1 u_2,\partial_3 u_2+\partial_2 u_3]^{\sf T}&=g_1,\quad x_1\in\mathbb{R},\ x_2=0,\ x_3\in\dot{\mathbb{R}},\\
u_2&=g_2,\quad x_1\in\mathbb{R},\ x_2=0,\ x_3\in\dot{\mathbb{R}},\\
\partial_2 h&=g_3,\quad x_1\in\mathbb{R},\ x_2=0,\ x_3=0,\\
u(0)&=u_0,\quad x_1\in\mathbb{R},\ x_2>0,\ x_3\in\dot{\mathbb{R}},\\
h(0)&=h_0\quad x_1\in\mathbb{R},\ x_2>0,\ x_3=0.
\end{split}
\end{align}
Here $m[w]:=(w_+ + w_-)/2$, where $w_\pm$ denote the traces of $w$ at $x_3=0$ from above and below, respectively. Note that $m[w]=w|_{x_3=0}$ if $w$ is continuous at $x_3=0$, that is, if $\Jump{w}=0$. Furthermore $x':=(x_1,x_2)$.

For convenience we set $\Omega:=\mathbb{R}\times\mathbb{R}_+\times\mathbb{R}$, $S_1:=\mathbb{R}\times\{0\}\times\mathbb{R}$, $\Sigma:=\mathbb{R}\times\mathbb{R}_+\times\{0\}$ and $\partial\Sigma:=\mathbb{R}\times\{0\}\times\{0\}$.
We will prove the following existence and uniqueness result.
\begin{thm}\label{thm:linWP1}
Let $n=3$, $p>5$, $T>0$, $\rho_j,\mu_j>0$, $j=1,2$, $J=[0,T]$. The problem \eqref{eq:NScap4} has a unique solution $(u,\pi,h)$
with regularity
$$u\in H_p^1(J;L_p(\Omega)^3)\cap L_p(J;H_p^2(\Omega\backslash\Sigma)^3),\quad \pi\in L_p(J;\dot{H}_p^1(\Omega\backslash\Sigma)),$$
$$\Jump{\pi}\in W_p^{1/2-1/2p}(J;L_p(\Sigma))\cap L_p(J;W_p^{1-1/p}(\Sigma)),$$
$$h\in W_p^{2-1/2p}(J;L_p(\Sigma))\cap H_p^1(J;W_p^{2-1/p}(\Sigma))\cap L_p(J;W_p^{3-1/p}(\Sigma)),$$
if and only if the data satisfy the following regularity and compatibility conditions.
\begin{enumerate}
\item $f\in L_p(J;L_p(\Omega)^3)$,
\item $f_d\in L_p(J;H_p^1(\Omega\backslash\Sigma))$,
\item $g=(g_v,g_w)\in W_p^{1/2-1/2p}(J;L_p(\Sigma)^3)\cap L_p(J;W_p^{1-1/p}(\Sigma))^3$,
\item $u_\Sigma\in W_p^{1-1/2p}(J;L_p(\Sigma)^3)\cap L_p(J;W_p^{2-1/p}(\Sigma)^3)$;
\item $g_h\in W_p^{1-1/2p}(J;L_p(\Sigma))\cap L_p(J;W_p^{2-1/p}(\Sigma))$,
\item $g_1\in W_p^{1/2-1/2p}(J;L_p(S_1))^2\cap L_p(J;W_p^{1-1/p}(S_1\backslash\partial\Sigma))^2$,
\item $g_2\in W_p^{1-1/2p}(J;L_p(S_1))\cap L_p(J;W_p^{2-1/p}(S_1\backslash\partial\Sigma))$,
\item $g_3\in W_p^{3/2-1/p}(J;L_p(\partial\Sigma))\cap H_p^1(J;W_p^{1-2/p}(\partial\Sigma))\cap L_p(J;W_p^{2-2/p}(\partial\Sigma))$;
\item $u_0=(v_0,w_0)\in W_p^{2-2/p}(\Omega)^3,\ h_0\in W_p^{3-2/p}(\Sigma)$
\item $\div u_0=f_d|_{t=0}$, $\Jump{u_0}=u_\Sigma|_{t=0}$,
\item $\mu[\partial_2 (u_0)_1+\partial_1 (u_0)_2,\partial_3 (u_0)_2+\partial_2 (u_0)_3]|_{x_2=0}^{\sf T}=g_1|_{t=0}$,
\item $(u_{0})_2|_{x_2=0}=g_2|_{t=0}$, $\partial_2 h_0|_{x_2=0}=g_3|_{t=0}$, $-\Jump{\mu \partial_3 v_0}-\Jump{\mu\nabla_{x'} (u_0)_3}=g_v|_{t=0}$,
\item $(g_v)_2+\Jump{(g_1)_2}=0,\ \Jump{(g_1)_1/\mu}=\partial_2(u_\Sigma)_1+\partial_1(u_\Sigma)_2$ at $\partial\Sigma$;
\item $\Jump{(g_1)_2/\mu-\partial_3 g_2}=\partial_2(u_\Sigma)_3$, $\Jump{g_2}=(u_\Sigma)_2$ at $\partial\Sigma$,
\item $\partial_t g_3-m[(g_1)_2/\mu-\partial_3 g_2]=\partial_2 g_h$ at $\partial\Sigma$,
\item $(f_d,u_\Sigma\cdot e_3,g_2)\in H_p^1(J;\hat{H}_p^{-1}(\Omega))$.
\end{enumerate}
\end{thm}
\begin{proof}
In a first step we will show that without loss of generality we may assume $u_0=0$ and $h_0=0$. We start with $h_0$. For that purpose we extend $h_0$ and $g_h$ with respect to $x_2$ to some functions $\tilde{h}_0\in W_p^{3-2/p}(\mathbb{R}^2)$ and
$$\tilde{g}_h\in W_p^{1-1/2p}(J;L_p(\mathbb{R}^2))\cap L_p(J;W_p^{2-1/p}(\mathbb{R}^2)),$$
respectively. Furthermore, we extend $u_0$ with respect to $x_2$ to some function $\tilde{u}_0\in W_p^{2-2/p}(\mathbb{R}^2\times\dot{\mathbb{R}})^3$, where $\dot{\mathbb{R}}:=\mathbb{R}\backslash\{0\}$. The extensions for $u_0$ and $g_h$ can be achieved by applying a higher order reflection method as in Section \ref{QS}. In general, for the extension of $h_0$, one cannot apply the reflection technique from Section \ref{QS}, since for large $p$ one has $W_p^{3-2/p}\hookrightarrow C^2$. However, the extension for $\tilde{h}_0$ exists due to the results in \cite{Tri83,Tri95}.
Let now
\begin{multline*}
\tilde{h}(t) = [2e^{-(I-\Delta_{x'})^{1/2} t}-e^{-2(I-\Delta_{x'})^{1/2} t}]\tilde{h}_0 +\\           [e^{-(I-\Delta_{x'}) t}-e^{-2(I-\Delta_{x'}) t}](I-\Delta_{x'})^{-1}\{m[\tilde{u}_0\cdot e_3]+\tilde{g}_h|_{t=0}\},
    \quad t\geq0,
\end{multline*}
where $\Delta_{x'}$ denotes the Laplace operator with respect to the variables $x'=(x_1,x_2)\in \mathbb{R}^2$. Since $\tilde{h}_0\in W_p^{3-2/p}(\mathbb{R}^2)$ and
$m[\tilde{u}_0\cdot e_3],\tilde{g}_h|_{t=0}\in W_p^{2-3/p}(\mathbb{R}^2)$, it follows from elementary semigroup theory that
$$\tilde{h}\in W_p^{2-1/2p}(J;L_p(\mathbb{R}^2))\cap H_p^1(J;W_p^{2-1/p}(\mathbb{R}^2))\cap L_p(J;W_p^{3-1/p}(\mathbb{R}^2))$$
with $\tilde{h}(0)=\tilde{h}_0$ and $\partial_t\tilde{h}(0)=m[\tilde{u}_0\cdot e_3]+\tilde{g}_h|_{t=0}$.

Let us turn to $u_0$. Consider the extension $\tilde{u}_0\in W_p^{2-2/p}(\mathbb{R}^2\times\dot{\mathbb{R}})^3$ from above and let $\tilde{u}_0^\pm:=\tilde{u}_0|_{x_3\gtrless0}\in W_p^{2-2/p}(\mathbb{R}^2\times{\mathbb{R}_\pm})^3$. Extend $\tilde{u}_0^+$ with respect to the variable $x_3$ to $\hat{u}_0^+\in W_p^{2-2/p}(\mathbb{R}^3)^3$. Then we solve the full space problem
\begin{align*}
\partial_t \hat{u}^+-\Delta\hat{u}^+&=0,\quad x\in\mathbb{R}^3,\\
\hat{u}^+(0)&=\hat{u}_0^+,\quad x\in\mathbb{R}^3,
\end{align*}
to obtain a unique solution
$$\hat{u}^+\in H_p^1(J;L_p(\mathbb{R}^3)^3)\cap L_p(J;H_p^2(\mathbb{R}^3)^3).$$
Extending $\tilde{u}_0^-$ with respect to $x_3$ to some $\hat{u}_0^-\in W_p^{2-2/p}(\mathbb{R}^3)^3$ and solving the latter full space problem with $\hat{u}_0^+$ being replaced by $\hat{u}_0^-$ yields a unique solution
$$\hat{u}^+\in H_p^1(J;L_p(\mathbb{R}^3)^3)\cap L_p(J;H_p^2(\mathbb{R}^3)^3).$$
Then we define
$$\hat{u}:=\begin{cases}\hat{u}^+|_{\Omega},&\ x_3>0,\\ \hat{u}_-|_{\Omega},&\ x_3<0.\end{cases}$$
Then $\hat{u}\in H_p^1(J;L_p(\Omega)^3)\cap L_p(J;H_p^2(\Omega)^3)$ and $\hat{u}|_{t=0}=u_0$ in $\Omega\backslash\Sigma$. If $(u,\pi,\Jump{\pi},h)$ is a solution of \eqref{eq:NScap4}, then $(u-\hat{u},\pi,\Jump{\pi},h-\tilde{h})$ solves \eqref{eq:NScap4} with $u_0=0$, $h_0=0$ and some modified data (not to be relabeled) $(f,f_d,g_v,g_w,u_\Sigma,g_h,g_1,g_2,g_3)$ in the right regularity classes, having  vanishing traces at $t=0$ and satisfying the compatibility conditions at $\partial\Sigma$ stated in Theorem \ref{thm:linWP1}. Note also that by construction $\partial_t(h-\tilde{h})|_{t=0}=0$.

By Proposition \ref{prop:app2} we may also assume that $g_3=0$. Indeed, there exists $$h_*\in\!_0W_p^{2-1/2p}(J;L_p(\Sigma))\cap\!_0H_p^1(J;W_p^{2-1/p}(\Sigma))\cap L_p(J;W_p^{3-1/p}(\Sigma))$$
such that $\partial_2 h_{*}|_{x_2=0}=g_3$. Replacing $h$ by $h-h_*$ it follows that $\partial_2 (h-h_*)|_{x_2=0}=0$. The functions $g_h$ and $g_w$ have to be replaced by $g_h-\partial_t h_*$ and $g_w+\sigma\Delta_{x'}h_*$, respectively.

Next we extend
$$g_1^+:=g_1|_{x_3>0}\in \,_0W_p^{1/2-1/2p}(J;L_p(\mathbb{R}_+^2)^2)\cap L_p(J;W_p^{1-1/p}(\mathbb{R}_+^2)^2)$$
by even reflection and
$$g_2^+:=g_2|_{x_3>0}\in \,_0W_p^{1-1/2p}(J;L_p(\mathbb{R}_+^2))\cap L_p(J;W_p^{2-1/p}(\mathbb{R}_+^2))$$
by means of the reflection
$$\tilde{g}_2^+(t,x_1,x_3)=\begin{cases}
g_2^+(t,x_1,x_3),\quad&\text{if}\ x_3>0,\\
-g_2^+(t,x_1,-2x_3)+2g_2^+(t,x_1,-x_3/2),\quad&\text{if}\ x_3<0.
\end{cases}$$
to functions
$$\tilde{g}_1^+\in \,_0W_p^{1/2-1/2p}(J;L_p(\mathbb{R}^2)^2)\cap L_p(J;W_p^{1-1/p}(\mathbb{R}^2)^2)$$
and
$$\tilde{g}_2^+\in \,_0W_p^{1-1/2p}(J;L_p(\mathbb{R}^2))\cap L_p(J;W_p^{2-1/p}(\mathbb{R}^2)).$$
Let $\mu^+:=\mu|_{x_3>0}$ and solve the parabolic system
\begin{equation}\label{eq:NShalfsp1}
\begin{array}{rclll}
\partial_t u_{*}-\Delta u_{*}&=&0,&\quad (x_1,x_3)\in\mathbb{R}^2,\ x_2>0,\\
\mu^+[\partial_2 (u_{*})_1+\partial_1 (u_{*})_2,\partial_3 (u_{*})_2+\partial_2 (u_{*})_3]^{\sf T} & =& \tilde{g}_1^+,&\quad (x_1,x_3)\in\mathbb{R}^2,\ x_2=0,\\
(u_{*})_2&=&\tilde{g}_2^+,&\quad (x_1,x_3)\in\mathbb{R}^2,\ x_2=0,\\
u_{*}(0)&=&0,&\quad (x_1,x_3)\in\mathbb{R}^2,\ x_2>0,
\end{array}
\end{equation}
by \cite{DHP07}, to obtain a solution
$$u_{*}\in\,_0H_p^1(J;L_p(\mathbb{R}_+^2\times\mathbb{R}))^3\cap L_p(J;H_p^2(\mathbb{R}_+^2\times\mathbb{R}))^3.$$
Then we repeat the same procedure for $g_j^-:=g_j|_{x_3<0}$ to obtain a function
$$u_{**}\in\,_0H_p^1(J;L_p(\mathbb{R}_+^2\times\mathbb{R}))^3\cap L_p(J;H_p^2(\mathbb{R}_+^2\times\mathbb{R}))^3$$
as a solution of \eqref{eq:NShalfsp1} with $\tilde{g}_j^+$ being replaced by the extensions $\tilde{g}_j^-$ of $g_j^-$ and $\mu^+$ being replaced by $\mu^-:=\mu|_{x_3<0}$.

Define
$$v:=\begin{cases}u_*,&\ x_3>0,\\ u_{**},&\ x_3<0.\end{cases}$$
It follows that the function $\bar{u}:=u-v$ satisfies $\bar{u}|_{t=0}=0$, $\Jump{\bar{u}}=u_\Sigma-\Jump{v}:=k$ and
$$\mu[\partial_2 \bar{u}_1+\partial_1 \bar{u}_2,\partial_3 \bar{u}_2+\partial_2 \bar{u}_3]=0,\quad \bar{u}_2=0$$
at $S_1\backslash\overline{\Sigma}$. In order to remove the jump of $\bar{u}$, we note that by the compatibility conditions it holds that $k_2=0$ and $\partial_2 k_1=\partial_2 k_3=0$ on at $\partial\Sigma$. Therefore it is possible to extend
$$k\in \,_0 W_p^{1-1/2p}(J;L_p(\mathbb{R}^2_+))^3\cap L_p(J;W_p^{2-1/p}(\mathbb{R}^2_+))^3$$
to a function
$$\tilde{k}\in\,_0 W_p^{1-1/2p}(J;L_p(\mathbb{R}^2))^3\cap L_p(J;W_p^{2-1/p}(\mathbb{R}^2))^3$$
by even reflection of $k_1,k_3$ and odd reflection of $k_2$. Then we solve the Dirichlet problem
\begin{equation}\label{eq:NScap5a}
\begin{array}{rclll}
\partial_t w-\Delta w&=&0,&\quad (x_1,x_2)\in\mathbb{R}^2,\ x_3>0,\\
\tr_{x_3=0} w&=&\tilde{k},&\quad (x_1,x_2)\in\mathbb{R}^2,\ x_3=0,\\
w(0)&=&0,&\quad (x_1,x_2)\in\mathbb{R}^2,\ x_3>0,
\end{array}
\end{equation}
to obtain a unique solution
$$w\in \,_0H_p^1(J;L_p(\mathbb{R}^3_+))\cap L_p(J;H_p^2(\mathbb{R}_+^3)).$$
Note that by symmetry the function
$$\bar{w}(t,x)=
\begin{bmatrix}
w_1(t,x_1,-x_2,x_3),\\ -w_2(t,x_1,-x_2,x_3),\\w_3(t,x_1,-x_2,x_3)
\end{bmatrix}
$$
is a solution of \eqref{eq:NScap5a} too, hence $w=\bar{w}$ and therefore it holds that $w_2=0$ as well as $\partial_2 w_1+\partial_1 w_2=\partial_3 w_2+\partial_2 w_3=0$ at $S_1\backslash\overline{\Sigma}$. Let $\bar{u}_{\pm}:=\bar{u}|_{x_3\gtrless0}$ and define $$u^*:=
\begin{cases}
\bar{u}_+ -w,\quad&\text{if}\ x_3>0,\\
\bar{u}_-,\quad&\text{if}\ x_3<0.
\end{cases}$$
Then $\Jump{u^*}=0$ and
$$\mu[\partial_2 u^*_1+\partial_1 u^*_2,\partial_3 u^*_2+\partial_2 u^*_3]=0,\quad u^*_2=0$$
on $S_1\backslash\overline{\Sigma}$. We arrive at the problem
\begin{align}\label{eq:NScap6}
\begin{split}
\partial_t(\rho u)-\mu\Delta u+\nabla \pi&=f,\quad x\in\Omega\backslash\Sigma,\\
\div u&=f_d,\quad x\in\Omega\backslash\Sigma,\\
-\Jump{\mu \partial_3 v}-\Jump{\mu\nabla_{x'} w}&=g_v,\quad x\in\Sigma,\\
-2\Jump{\mu \partial_3 u_3}+\Jump{\pi}-\sigma\Delta_{x'} h&=g_w,\quad x\in\Sigma,\\
\Jump{u}&=0,\quad x\in\Sigma,\\
\partial_t h-u_3&=g_h,\quad x\in\Sigma,\\
\mu[\partial_2 u_1+\partial_1 u_2,\partial_3 u_2+\partial_2 u_3]^{\sf T}&=0,\quad x\in S_1\backslash\partial\Sigma,\\
u_2&=0,\quad x\in S_1\backslash\partial\Sigma,\\
\partial_2 h&=0,\quad x\in\partial\Sigma,\\
u(0)&=0,\quad x\in\Omega\backslash\Sigma,\\
h(0)&=0\quad x\in\Sigma,
\end{split}
\end{align}
with modified data $f\in L_p(J;L_p(\Omega))^3$,
$$f_d\in L_p(J;H_p^1(\Omega\backslash\Sigma)),$$
$$(g_v,g_w)\in\,_0W_p^{1/2-1/2p}(J;L_p(\Sigma)^3)\cap L_p(J;W_p^{1-1/p}(\Sigma)^3),$$
and
$$g_h\in\,_0 W_p^{1-1/2p}(J;L_p(\Sigma))\cap L_p(J;W_p^{2-1/p}(\Sigma)),$$
satisfying the compatibility conditions $(g_v)_2=\partial_2 g_h=0$ at $\partial\Sigma$ and $(f_d,0,0)\in\!_0H_p^1(J;\hat{H}_p^{-1}(\Omega))$.

Therefore it is possible to extend $(f_1,f_3,f_d,(g_v)_1,g_w,g_h)$ by even reflection to $\{x_2<0\}$. On the other side we may extend $(f_2,(g_v)_2)$ by odd reflection to $\{x_2<0\}$. In a next step we consider the (reflected) problem
\begin{align}\label{eq:NScap7}
\begin{split}
\partial_t(\rho \tilde{u})-\mu\Delta \tilde{u}+\nabla \tilde{\pi}&=\tilde{f},\quad (x_1,x_2)\in\mathbb{R}^2,\ x_3\in\dot{\mathbb{R}},\\
\div \tilde{u}&=\tilde{f}_d,\quad (x_1,x_2)\in\mathbb{R}^2,\ x_3\in\dot{\mathbb{R}},\\
-\Jump{\mu \partial_3 \tilde{v}}-\Jump{\mu\nabla_{x'} \tilde{u}_3}&=\tilde{g}_v,\quad (x_1,x_2)\in\mathbb{R}^2,\ x_3=0,\\
-2\Jump{\mu \partial_3 \tilde{u}_3}+\Jump{\tilde{\pi}}-\sigma\Delta_{x'} \tilde{h}&=\tilde{g}_w,\quad (x_1,x_2)\in\mathbb{R}^2,\ x_3=0,\\
\Jump{\tilde{u}}&=0,\quad (x_1,x_2)\in\mathbb{R}^2,\ x_3=0,\\
\partial_t \tilde{h}-\tilde{u}_3&=\tilde{g}_h,\quad (x_1,x_2)\in\mathbb{R}^2,\ x_3=0,\\
\tilde{u}(0)&=0,\quad (x_1,x_2)\in\mathbb{R}^2,\ x_3\in\dot{\mathbb{R}},\\
\tilde{h}(0)&=0,\quad (x_1,x_2)\in\mathbb{R}^2,\ x_3=0,
\end{split}
\end{align}
with given reflected data $\tilde{f}\in L_p(J;L_p(\mathbb{R}^2\times\mathbb{R}))^3$,
$$\tilde{f}_d\in L_p(J;H_p^1(\mathbb{R}^2\times\dot{\mathbb{R}})),$$
$$(\tilde{g}_v,\tilde{g}_w)\in\,_0W_p^{1/2-1/2p}(J;L_p(\mathbb{R}^2)^3)\cap L_p(J;W_p^{1-1/p}(\mathbb{R}^2)^3),$$
and
$$\tilde{g}_h\in\,_0 W_p^{1-1/2p}(J;L_p(\mathbb{R}^2))\cap L_p(J;W_p^{2-1/p}(\mathbb{R}^2)),$$
where $(\tilde{f}_d,0)\in\!_0H_p^1(J;\hat{H}_p^{-1}(\mathbb{R}^2\times\mathbb{R}))$.

By \cite[Theorem 5.1]{PrSi09a} there exists a unique solution $(\tilde{u},\tilde{\pi},\Jump{\tilde{\pi}},\tilde{h})$ of \eqref{eq:NScap7} with regularity
$$\tilde{u}\in\!_0H_p^1(J;L_p(\mathbb{R}^3))^3\cap L_p(J;H_p^2(\dot{\mathbb{R}}^3))^3,$$
$$\tilde{\pi}\in L_p(J;\dot{H}_p^1(\dot{\mathbb{R}}^3)),$$
$$\Jump{\tilde{\pi}}\in\!_0W_p^{1/2-1/2p}(J;L_p(\mathbb{R}^2))\cap L_p(J;W_p^{1-1/p}(\mathbb{R}^2)),$$
and
$$\tilde{h}\in\!_0W_p^{2-1/2p}(J;L_p(\mathbb{R}^2))\cap\!_0H_p^1(J;W_p^{2-1/p}(\mathbb{R}^2))\cap L_p(J;W_p^{3-1/p}(\mathbb{R}^2)).$$

Note that by symmetry the function $(\bar{u},\bar{\pi},\bar{h})$ with $\bar{u}_j(x):=\tilde{u}_j(x_1,-x_2,x_3)$, $j\in\{1,3\}$, $\bar{u}_2(x):=-\tilde{u}_2(x_1,-x_2,x_3)$, $\bar{\pi}(x):=\tilde{\pi}(x_1,-x_2,x_3)$ and $\bar{h}(x'):=\tilde{h}(x_1,-x_2)$ is a solution of \eqref{eq:NScap7} too. Therefore, by uniqueness, it follows that
$$\tilde{u}_j(x_1,-x_2,x_3)=\tilde{u}_j(x_1,x_2,x_3),\ j\in\{1,3\},$$ $$\tilde{u}_2(x_1,x_2,x_3)=-\tilde{u}_2(x_1,-x_2,x_3),\ \tilde{\pi}(x_1,x_2,x_3)=\tilde{\pi}(x_1,-x_2,x_3)$$
and $\tilde{h}(x_1,x_2)=\tilde{h}(x_1,-x_2)$. This necessarily yields
$$\tilde{u}_2=(\partial_2 \tilde{u}_1+\partial_1\tilde{u}_2)=(\partial_3 \tilde{u}_2+\partial_2\tilde{u}_3)=0,$$
as well as $\partial_2\tilde{h}=0$ at $S_1\backslash\overline{\Sigma}$. Hence the restriction $(\tilde{u},\tilde{\pi},\Jump{\tilde{\pi}},\tilde{h})|_{\Omega}$ is the unique strong solution of \eqref{eq:NScap6}.
\end{proof}

\subsection{The two-phase Stokes equations in bent half-spaces}\label{bentHS}

Let $\theta\in BC^{3}(\mathbb{R})$ such that
$$G_\theta:=\{(x_1,x_2)\in\mathbb{R}^2:x_2>\theta(x_1)\}\quad \text{and}\quad \Omega_\theta=G_\theta\times\mathbb{R}.$$
We assume furthermore that $|\theta'|_\infty\le \eta$ and $|\partial_x^j\theta|_\infty\le M$, $j\in\{2,3\}$, where we may choose $\eta>0$ as small as we wish. Let $S_{1,\theta}:=\partial G_\theta\times\mathbb{R}$. Furthermore, let $\nu_{S_{1,\theta}}=(\nu_{G_\theta},0)^{\textsf{T}}$ with $\nu_{G_\theta}:=\frac{1}{\sqrt{1+\theta'(x_1)^2}}(\theta'(x_1),-1)^{\textsf{T}}$ denote the outer unit normal to $S_{1,\theta}$ at $(x_1,\theta(x_1),x_3)$, $(x_1,x_3)\in\mathbb{R}\times\mathbb{R}$ and let $P_{S_{1,\theta}}$ be the tangential projection to $S_{1,\theta}$. Furthermore, let $\Sigma_\theta:=G_\theta\times\{0\}$ and $\partial\Sigma_\theta:=\partial G_{\theta}\times\{0\}$.

Consider the problem
\begin{align}\label{eq:NScap8}
\begin{split}
\partial_t(\rho u)-\mu\Delta u+\nabla \pi&=f,\quad x\in\Omega_\theta\backslash\Sigma_\theta,\\
\div u&=f_d,\quad x\in\Omega_\theta\backslash\Sigma_\theta,\\
-\Jump{\mu \partial_3 v}-\Jump{\mu\nabla_{x'} w}&=g_v,\quad x\in\Sigma_\theta,\\
-2\Jump{\mu \partial_3 w}+\Jump{\pi}-\sigma\Delta_{x'} h&=g_w,\quad x\in\Sigma_\theta,\\
\Jump{u}&=u_\Sigma,\quad x\in\Sigma_\theta,\\
\partial_t h-m[w]&=g_h,\quad x\in\Sigma_\theta,\\
P_{S_{1,\theta}}\left(\mu(\nabla u+\nabla u^{\sf T})\nu_{S_{1,\theta}}\right)&=P_{S_{1,\theta}}g_1,\quad x\in S_{1,\theta}\backslash\partial\Sigma_\theta,\\
u\cdot\nu_{S_{1,\theta}}&=g_2,\quad x\in S_{1,\theta}\backslash\partial\Sigma_\theta,\\
\partial_{\nu_{G_\theta}} h&=g_3,\quad x\in\partial\Sigma_\theta\\
u(0)&=u_0,\quad x\in\Omega_\theta\backslash\Sigma_\theta,\\
h(0)&=h_0,\quad x\in\Sigma_\theta,
\end{split}
\end{align}
where $u=(v,w)$ and $v=(u_1,u_2)$, $w=u_3$. Without loss of generality we may consider $u_0=0$ and $h_0=0$ in \eqref{eq:NScap8}. Literally, this can be seen as in Section \ref{HS}, we will not go into the details. The remaining modified data (not to be relabeled) belong to the right regularity classes and they have vanishing traces at $t=0$.

Next, we will show that we may assume $u_\Sigma=0$. For that purpose, extend $u_\Sigma$ with respect to $x_2$ to some function
$$\tilde{u}_\Sigma\in\!_0W_p^{1-1/2p}(J;L_p(\mathbb{R}^2)^3)\cap L_p(J;W_p^{2-1/p}(\mathbb{R}^2)^3),$$
and solve the half space problem
\begin{align*}
\partial_t u_*-\Delta u_*&=0,\quad x\in\mathbb{R}^2\times\mathbb{R}_+,\\
u_*&=\tilde{u}_\Sigma,\quad x\in\mathbb{R}^2\times\{0\},\\
u_*(0)&=0,\quad x\in\mathbb{R}^2\times\mathbb{R}_+,
\end{align*}
by \cite{DHP07} to obtain a unique solution
$$u_*\in\!_0H_p^1(J;L_p(\mathbb{R}^2\times\mathbb{R}_+)^3)\cap L_p(J;H_p^2(\mathbb{R}^2\times\mathbb{R}_+)^3).$$
If $(u,\pi,\Jump{\pi},h)$ is a solution of \eqref{eq:NScap8} with $u_0=0$ and $h_0=0$, and $$u_{**}:=\begin{cases}
\bar{u}_+ -u_*,\quad&\text{if}\ x_3>0,\\
\bar{u}_-,\quad&\text{if}\ x_3<0.
\end{cases}$$
where $u^\pm:=u|_{x_3\gtrless 0}$, then $\Jump{u_{**}}=0$. Again the remaining modified data have the correct regularity and vanishing traces at $t=0$. Note that in this case $m[w_{**}]=w_{**}$, where $u_{**}=(v_{**},w_{**})$.

Let us show that we may reduce \eqref{eq:NScap8} with $u_0=0$, $h_0=0$ and $u_\Sigma=0$ to the case $g_v=0$, $g_w=0$ and $g_h=0$. To this end we extend the data $(g_v,g_w,g_h)$ with respect to $x_2$ to some functions
$$(\tilde{g}_v,\tilde{g}_w)\in\!_0W_p^{1/2-1/2p}(J;L_p(\mathbb{R}^2)^3)\cap L_p(J;W_p^{1-1/p}(\mathbb{R}^2)^3),$$
and
$$\tilde{g}_h\in\!_0W_p^{1-1/2p}(J;L_p(\mathbb{R}^2))\cap L_p(J;W_p^{2-1/p}(\mathbb{R}^2)).$$
Then we consider the two-phase problem
\begin{align}\label{eq:NScap8a}
\begin{split}
\partial_t\tilde{u}-\Delta \tilde{u}&=0,\quad x\in\mathbb{R}^2\times\dot{\mathbb{R}},\\
-\Jump{\mu \partial_3 \tilde{v}}-\Jump{\mu\nabla_{x'} \tilde{w}}&=\tilde{g}_v,\quad x\in\mathbb{R}^2\times\{0\},\\
-2\Jump{\mu \partial_3 \tilde{w}}-\sigma\Delta_{x'} \tilde{h}&=\tilde{g}_w,\quad x\in\mathbb{R}^2\times\{0\},\\
\Jump{\tilde{u}}&=0,\quad x\in\mathbb{R}^2\times\{0\},\\
\partial_t \tilde{h}-\tilde{w}&=\tilde{g}_h,\quad x\in\mathbb{R}^2\times\{0\},\\
\tilde{u}(0)&=0,\quad x\in\mathbb{R}^2\times\dot{\mathbb{R}},\\
\tilde{h}(0)&=0,\quad x\in\mathbb{R}^2\times\{0\},
\end{split}
\end{align}
for the unknowns $(\tilde{u},\tilde{h})$. Interestingly, the equations for $\tilde{v}$ and $\tilde{w}$ decouple. Therefore we study for the moment the problem
\begin{align}\label{eq:NScap8b}
\begin{split}
\partial_t\tilde{w}-\Delta \tilde{w}&=0,\quad x\in\mathbb{R}^2\times\dot{\mathbb{R}},\\
-2\Jump{\mu \partial_3 \tilde{w}}-\sigma\Delta_{x'} \tilde{h}&=\tilde{g}_w,\quad x\in\mathbb{R}^2\times\{0\},\\
\Jump{\tilde{w}}&=0,\quad x\in\mathbb{R}^2\times\{0\},\\
\partial_t \tilde{h}-\tilde{w}&=\tilde{g}_h,\quad x\in\mathbb{R}^2\times\{0\},\\
\tilde{w}(0)&=0,\quad x\in\mathbb{R}^2\times\dot{\mathbb{R}},\\
\tilde{h}(0)&=0,\quad x\in\mathbb{R}^2\times\{0\},
\end{split}
\end{align}
for the unknowns $(\tilde{w},\tilde{h})$. Assume that $(\tilde{w},\tilde{h})$ are already known. Then, $\tilde{w}$ is explicitly given by
\begin{equation}\label{eq:NScap8c}
\tilde{w}(x_3)=\frac{1}{2(\mu_++\mu_-)}\begin{cases}
e^{-L x_3}L^{-1}(\sigma\Delta_{x'}\tilde{h}+\tilde{g}_w),\quad&\text{if}\ x_3>0,\\
e^{-L (-x_3)}L^{-1}(\sigma\Delta_{x'}\tilde{h}+\tilde{g}_w),\quad&\text{if}\ x_3<0,
\end{cases}
\end{equation}
where $L:=(\partial_t-\Delta_{x'})^{1/2}$. Therefore,
$$\tilde{w}|_{x_3=0}=\frac{1}{2(\mu_++\mu_-)}L^{-1}(\sigma\Delta_{x'}\tilde{h}+\tilde{g}_w)$$
and it follows that we may reduce \eqref{eq:NScap8b} to a single equation for $\tilde{h}$ which reads
\begin{equation}\label{eq:NScap8ca}
\partial_t\tilde{h}-\frac{\sigma}{2(\mu_++\mu_-)}L^{-1}\Delta_{x'}\tilde{h}=
\frac{1}{2(\mu_++\mu_-)}L^{-1}\tilde{g}_w+\tilde{g}_h,
\end{equation}
and which is subject to the initial condition $\tilde{h}(0)=0$. Making use of the $\mathcal{R}$-boundedness of the operator $\Delta_{x'}$ in $K_p^s(\mathbb{R}^2)$, $K\in\{W,H\}$, the operator-valued $H^\infty$-calculus for $\partial_t$ in $_0H_p^r(J;K_p^s(\mathbb{R}^2))$ and real interpolation one can show as in \cite[Section 5]{PrSi09a} that the operator $\partial_t-\frac{\sigma}{2(\mu_++\mu_-)}L^{-1}\Delta_{x'}$ is invertible in $_0W_p^{1-1/2p}(J;L_p(\mathbb{R}^2))\cap L_p(J;W_p^{2-1/p}(\mathbb{R}^2))$ with domain
$$_0W_p^{2-1/2p}(J;L_p(\mathbb{R}^2))\cap\!_0H_p^1(J;W_p^{2-1/p}(\mathbb{R}^2))\cap L_p(J;W_p^{3-1/p}(\mathbb{R}^2)).$$
Hence there exists a unique solution
$$\tilde{h}\in\!_0W_p^{2-1/2p}(J;L_p(\mathbb{R}^2))\cap \!_0H_p^1(J;W_p^{2-1/p}(\mathbb{R}^2))\cap L_p(J;W_p^{3-1/p}(\mathbb{R}^2))$$
of \eqref{eq:NScap8ca}. Then $\tilde{w}$ is given by \eqref{eq:NScap8c} and, finally, $\tilde{v}$ is the unique solution of the two-phase problem
\begin{align*}
\partial_t\tilde{v}-\Delta \tilde{v}&=0,\quad x\in\mathbb{R}^2\times\dot{\mathbb{R}},\\
-\Jump{\mu \partial_3 \tilde{v}}&=\Jump{\mu\nabla_{x'} \tilde{w}}+\tilde{g}_v,\quad x\in\mathbb{R}^2\times\{0\},\\
\Jump{\tilde{v}}&=0,\quad x\in\mathbb{R}^2\times\{0\},\\
\tilde{v}(0)&=0,\quad x\in\mathbb{R}^2\times\dot{\mathbb{R}}.
\end{align*}
In summary, we have shown that we may reduce \eqref{eq:NScap8} to the problem
\begin{align}\label{eq:NScap8d}
\begin{split}
\partial_t(\rho u)-\mu\Delta u+\nabla \pi&=f,\quad x\in\Omega_\theta\backslash\Sigma_\theta,\\
\div u&=f_d,\quad x\in\Omega_\theta\backslash\Sigma_\theta,\\
-\Jump{\mu \partial_3 v}-\Jump{\mu\nabla_{x'} w}&=0,\quad x\in\Sigma_\theta,\\
-2\Jump{\mu \partial_3 w}+\Jump{\pi}-\sigma\Delta_{x'} h&=0,\quad x\in\Sigma_\theta,\\
\Jump{u}&=0,\quad x\in\Sigma_\theta,\\
\partial_t h-w&=0,\quad x\in\Sigma_\theta,\\
P_{S_{1,\theta}}\left(\mu(\nabla u+\nabla u^{\sf T})\nu_{S_{1,\theta}}\right)&=P_{S_{1,\theta}}g_1,\quad x\in S_{1,\theta}\backslash\partial\Sigma_\theta,\\
u\cdot\nu_{S_{1,\theta}}&=g_2,\quad x\in S_{1,\theta}\backslash\partial\Sigma_\theta,\\
\partial_{\nu_{G_\theta}} h&=g_3,\quad x\in\partial\Sigma_\theta\\
u(0)&=0,\quad x\in\Omega_\theta\backslash\Sigma_\theta,\\
h(0)&=0,\quad x\in\Sigma_\theta,
\end{split}
\end{align}
with given data $(f,g_1,g_2,g_3)$ having vanishing traces at $t=0$ and which satisfy the compatibility conditions
$$\Jump{g_2}=0,\ \Jump{g_1\cdot e_3}=0,\ \Jump{P_{S_{1,\theta}}g_1\cdot e_1/\mu}=0,\
\Jump{\partial_3 g_2-g_1\cdot e_3/\mu}=0,$$
and
$$\partial_t g_3+\partial_3 g_2-g_1\cdot e_3/\mu=0$$
at the contact line $\{(x_1,\theta(x_1),0):x_1\in\mathbb{R}\}$. To see this, one can apply the representation \eqref{eq:NSbendquart1a} from Section \ref{bentQS}. Note also that the second component of $P_{S_{1,\theta}}w$ is redundant, as it can always be reproduced from the first component. Finally, it holds that $(f_d,0,g_2)\in\!_0H_p^1(J;\hat{H}_p^{-1}(\Omega_\theta))$.

We will now transform $\Omega_\theta$ to $\Omega:=\mathbb{R}\times\mathbb{R}_+\times\mathbb{R}$, $S_{1,\theta}$ to $S_1:=\mathbb{R}\times\{0\}\times\mathbb{R}$, $\Sigma_\theta$ to $\Sigma:=\mathbb{R}\times\mathbb{R}_+\times\{0\}$ and $\partial\Sigma_\theta$ to $\partial\Sigma:=\mathbb{R}_+\times\{0\}\times\{0\}$. To this end we introduce the new variables $\bar{x}_1=x_1$, $\bar{x}_2=x_2-\theta(x_1)$ and $\bar{x}_3=x_3$ for $x\in\Omega_\theta$. Suppose that $(u,\pi,h)$ is a solution of \eqref{eq:NScap8} and define the new functions
$$\bar{u}(\bar{x}):=u(\bar{x}_1,\bar{x}_2+\theta(\bar{x}_1),\bar{x}_3)$$
$$\bar{\pi}(\bar{x}):=\pi(\bar{x}_1,\bar{x}_2+\theta(\bar{x}_1),\bar{x}_3)$$
and
$$\bar{h}(\bar{x}'):=h(\bar{x}_1,\bar{x}_2+\theta(\bar{x}_1)),$$
where $\bar{x}':=(\bar{x}_1,\bar{x}_2)$. In the same way we transform all of the data. Then, as in Section \ref{bentQS}, $(\bar{u},\bar{\pi},\bar{h})$ satisfies the problem
\begin{align}\label{eq:NScap9}
\begin{split}
\partial_t(\rho \bar{u})-\mu\Delta \bar{u}+\nabla \bar{\pi}&=M_1(\theta,\bar{u},\bar{\pi})+\bar{f},\quad \bar{x}\in\Omega\backslash\Sigma,\\
\div \bar{u}&=M_2(\theta,\bar{u})+\bar{f}_d,\quad \bar{x}\in\Omega\backslash\Sigma,\\
-\Jump{\mu \partial_{3} \bar{v}}-\Jump{\mu\nabla_{\bar{x}'} \bar{w}}&=M_3(\theta,\bar{u}),\quad \bar{x}\in\Sigma\\
-2\Jump{\mu \partial_{3} \bar{w}}+\Jump{\bar{\pi}}-\sigma\Delta_{\bar{x}'} \bar{h}&=M_4(\theta,\bar{h}),\quad \bar{x}\in\Sigma,\\
\Jump{\bar{u}}&=0,\quad \bar{x}\in\Sigma\\
\partial_t \bar{h}-\bar{w}&=0,\quad \bar{x}\in\Sigma,\\
\mu(\partial_{1} \bar{u}_2+\partial_{2} \bar{u}_1)&=M_5(\theta,\bar{u})-\sqrt{1+\theta'^2}^3[P_{S_{1,\theta}}\bar{g}_1\cdot e_1],\quad \bar{x}\in S_1\backslash\partial\Sigma,\\
\mu(\partial_{2} \bar{u}_3+\partial_{3} \bar{u}_2)
&=M_6(\theta,\bar{u})-\sqrt{1+\theta'^2}[\bar{g}_1\cdot e_3],\quad \bar{x}\in S_1\backslash\partial\Sigma,\\
\bar{u}_2&=M_7(\theta,\bar{u})-\sqrt{1+\theta'^2}\bar{g}_2,\quad \bar{x}\in S_1\backslash\partial\Sigma,\\
\partial_{2} \bar{h}&=M_8(\theta,\bar{h})-\sqrt{1+\theta'^2}\bar{g}_3,\quad \bar{x}\in\partial\Sigma,\\
\bar{u}(0)&=0,\quad \bar{x}\in\Omega\backslash\Sigma\\
\bar{h}(0)&=0,\quad \bar{x}\in\Sigma,
\end{split}
\end{align}
where $\bar{u}=(\bar{v},\bar{v})$.
The functions $M_j$ are given by
$$M_1(\theta,\bar{u},\bar{\pi}):=2\theta'(\bar{x}_1)\partial_{1}
\partial_{2}\bar{u}+\theta''(\bar{x}_1)\partial_{2}\bar{u}
-\theta'(\bar{x}_1)^2\partial_{2}^2\bar{u}
+\theta'(\bar{x}_1)\partial_{2}\bar{\pi}e_1,$$
$$M_2(\theta,\bar{u}):=\theta'(\bar{x}_1)\partial_{2}\bar{u}_1,$$
$$M_3(\theta,\bar{u})=[-\theta'(\bar{x}_1)\Jump{\mu\partial_{2}\bar{w}},0]^{\sf T},$$
$$M_4(\theta,\bar{h})=\sigma\left(-2\theta'(\bar{x}_1)\partial_{1}
\partial_{2}\bar{h}-\theta''(\bar{x}_1)\partial_{2}\bar{h}
+\theta'(\bar{x}_1)^2\partial_{2}^2\bar{h}\right),$$
$$M_5(\theta,\bar{u}):=\mu\theta'(\bar{x}_1)[2\partial_{1}\bar{u}_1+
\theta'(\bar{x}_1)(\partial_{1}\bar{u}_2-\partial_{2}\bar{u}_1)-(1+
\theta'(\bar{x}_1)^2)\partial_{2}\bar{u}_2],$$
$$M_6(\theta,\bar{u}):=\mu\theta'(\bar{x}_1)(\partial_{1} \bar{u}_3-\theta'(\bar{x}_1)\partial_{2}\bar{u}_3+
\partial_{3}\bar{u}_1),$$
$$M_7(\theta,\bar{u}):=\theta'(\bar{x}_1)\bar{u}_1.$$
and
$$M_8(\theta,\bar{h})=\theta'(\bar{x}_1)\left(\partial_{1}\bar{h}
-\theta'(\bar{x}_1)\partial_{2}\bar{h}\right).$$
Let us define the function spaces
$$_0\mathbb{E}_u(T):=\{u\in\!_0H_p^1(J;L_p(\Omega)^3)\cap L_p(J;H_p^2(\Omega\backslash\Sigma)^3):\Jump{u}=0,\ \text{on}\ \Sigma\},$$
$$\mathbb{E}_\pi(T):=L_p(J;\dot{H}_p^1(\Omega\backslash\Sigma)),$$
$$_0\mathbb{E}_q(T):=\!_0W_p^{1/2-1/2p}(J;L_p(\Sigma))\cap L_p(J;W_p^{1-1/p}(\Sigma)),$$
$$_0\mathbb{E}_h(T):=\!_0W_p^{2-1/2p}(J;L_p(\Sigma))\cap \!_0H_p^1(J;W_p^{2-1/p}(\Sigma))\cap L_p(J;W_p^{3-1/p}(\Sigma))$$
\begin{multline*}
_0\mathbb{E}(T):=\{(u,\pi,q,h)\in\!_0\mathbb{E}_u(T)\times \mathbb{E}_\pi(T)\times
\!_0\mathbb{E}_q(T)\times\!_0\mathbb{E}_h(T):\\
q=\Jump{\pi},\ \partial_t h-u\cdot e_3=0\ \text{on}\ \Sigma\},
\end{multline*}
$$\tilde{\mathbb{F}}(T):=\mathbb{F}_1(T)\times \mathbb{F}_2(T)\times_{j=3}^8\, _0\mathbb{F}_j(T),$$ where
$$\mathbb{F}_1(T):=L_p(J;L_p(\Omega)^3),$$
$$\mathbb{F}_2(T):=L_p(J;H_p^1(\Omega\backslash\Sigma)),$$
$$_0\mathbb{F}_3(T):=\!_0
W_p^{1/2-1/2p}(J;L_p(\Sigma)^2)\cap
L_p(J;W_p^{1-1/p}(\Sigma)^2),$$
$$_0\mathbb{F}_4(T):=\!_0
W_p^{1/2-1/2p}(J;L_p(\Sigma))\cap
L_p(J;W_p^{1-1/p}(\Sigma)),$$
$$_0\mathbb{F}_5(T):=\!_0
W_p^{1/2-1/2p}(J;L_p(S_1))\cap
L_p(J;W_p^{1-1/p}(S_1\backslash\partial\Sigma)),$$
$_0\mathbb{F}_6(T):=\!_0\mathbb{F}_5(T)$,
$$_0\mathbb{F}_7(T):=
\!_0W_p^{1-1/2p}(J;L_p(S_1))\cap
L_p(J;W_p^{2-1/p}(S_1)),$$
and
$$_0\mathbb{F}_8(T):=\!_0W_p^{3/2-1/p}(J;L_p(\Sigma))\cap \!_0H_p^1(J;W_p^{1-2/p}(\Sigma))\cap L_p(J;W_p^{2-2/p}(\Sigma)).$$
Finally, we set
$$_0\mathbb{F}(T):=\{(f_1,\ldots,f_8)\in\tilde{\mathbb{F}}(T):(13)\ \&\ (16)\ \text{in Theorem \ref{thm:linWP1} are satisfied}\}.$$
Define an operator $L:\,_0\mathbb{E}(T)\to\!_0\mathbb{F}(T)$ by
$$L(\bar{u},\bar{\pi},\bar{q},\bar{h}):=\begin{bmatrix}
\partial_t(\rho \bar{u})-\mu\Delta \bar{u}+\nabla \bar{\pi}\\
\div \bar{u}\\
-\Jump{\mu\partial_3\bar{v}}-\Jump{\mu\nabla_{\bar{x}'}\bar{w}}\\
-2\Jump{\mu\partial_3\bar{w}}+\bar{q}-\sigma\Delta_{\bar{x}'}\bar{h}\\
\mu(\partial_2 \bar{u}_1+\partial_1 \bar{u}_2)|_{S_1}\\
\mu(\partial_3 \bar{u}_2+\partial_2 \bar{u}_3)|_{S_1}\\
\bar{u}_2|_{S_1}\\
\partial_2\bar{h}|_{\partial\Sigma}
\end{bmatrix}
$$
and note that $L:\,_0\mathbb{E}(T)\to\!_0\mathbb{F}(T)$ is an isomorphism by Theorem \ref{thm:linWP1}. Define
$$M(\theta,\bar{u},\bar{\pi},\bar{h}):=(M_1,{M}_2,
{M}_3,{M}_4,{M}_5,M_6,M_7,M_8)^{\sf T}(\theta,\bar{u},\bar{\pi},\bar{h})$$
and
$$F:=(f_1,f_2,0,0,f_5,f_6,f_7,f_8)^{\sf T},$$
with $f_1:=\bar{f}$, $f_2:=\bar{f}_d$,
$$f_5:=-\sqrt{1+\theta'^2}^3[P_{S_{1,\theta}}\bar{g}_1\cdot e_1],\ f_6:=-\sqrt{1+\theta'^2}[\bar{g}_1\cdot e_3],$$
$f_7:=-\sqrt{1+\theta'^2}\bar{g}_2$ and $f_8:=-\sqrt{1+\theta'^2}\bar{g}_3$.
It can be readily checked that the components of $F$ satisfy the compatibility conditions (13)-(16) in Theorem \ref{thm:linWP1}. In fact, this can be seen as in Section \ref{bentQS}. Since $\theta\in BC^3(\mathbb{R})$ this implies that $F\in\!_0\mathbb{F}(T)$. In the same way one can show that the components of $M(\theta,\bar{u},\bar{\pi},\bar{h})$ satisfy the compatibility conditions (14)-(16) as well as the second compatibility condition in (13) in Theorem \ref{thm:linWP1}. Unfortunately the first condition in Theorem \ref{thm:linWP1} (13) for $M_6$, which reads
$$\Jump{M_6(\theta,\bar{u})}=0\ \text{on}\ \Sigma,$$
is in general not satisfied. To circumvent this problem, we modify $M_3(\theta,\bar{u})$ as follows
$$\bar{M}_3(\theta,\bar{u}):=\theta'(\bar{x}_1)\left[\Jump{\mu\partial_2\bar{w}},-
\operatorname{ext}_{\Sigma}\Big(\Jump{\mu(\partial_{1} \bar{w}-\theta'(\bar{x}_1)\partial_{2}\bar{w}+
\partial_3\bar{u}_1)|_{S_1\backslash\partial\Sigma}}\Big)\right]^{\sf T}.$$
Here $\operatorname{ext}_\Sigma$ is a suitable bounded and linear extension operator from $$_0W_p^{1/2-1/p}(J;L_p(\partial\Sigma))\cap L_p(J;W_p^{1-2/p}(\partial\Sigma))$$
to
$$_0W_p^{1/2-1/2p}(J;L_p(\Sigma))\cap L_p(J;W_p^{1-1/p}(\Sigma)),$$
such that $[\operatorname{ext}_\Sigma z]|_{\partial\Sigma}=z$ for all $z\in\!_0W_p^{1/2-1/p}(J;L_p(\partial\Sigma))\cap L_p(J;W_p^{1-2/p}(\partial\Sigma))$, which exists due to Proposition \ref{app:propext}.
Note that if we have a solution $(u,\pi,q,h)\in\!_0\mathbb{E}(T)$ of \eqref{eq:NScap9} with $M_3(\theta,\bar{u})$ replaced by $\bar{M}_3(\theta,\bar{u})$, then, by the first component of the third line in \eqref{eq:NScap9}, we obtain that
$$\Jump{\mu(\partial_1 \bar{w}-\theta'(\bar{x}_1)\partial_2\bar{w}+\partial_3 \bar{u}_1)}=0$$
on $\Sigma$, hence $\bar{M}_3(\theta,\bar{u})=M_3(\theta,\bar{u})$ in this case.

Let us define
$$\bar{M}(\theta,\bar{u},\bar{\pi},\bar{h}):=(M_1,{M}_2,
\bar{M}_3,{M}_4,{M}_5,M_6,M_7,M_8)^{\sf T}(\theta,\bar{u},\bar{\pi},\bar{h}).$$
Since the modification in $M_3$ does not affect the other compatibility conditions in Theorem \ref{thm:linWP1}, it follows readily that $\bar{M}(\theta,\bar{u},\bar{\pi},\bar{h})\in\!_0\mathbb{F}(T)$ for each $(\bar{u},\bar{\pi},\bar{q},\bar{h})\in\!_0\mathbb{E}(T)$.
Therefore, we may rewrite \eqref{eq:NScap9}, with $M_3$ replaced by $\bar{M}_3$, in the more condensed form
\begin{equation}\label{eq:NScap9a}
(\bar{u},\bar{\pi},\bar{q},\bar{h})=L^{-1}\bar{M}(\theta,\bar{u},\bar{\pi},\bar{h})
+L^{-1}F
\end{equation}
in the space $_0\mathbb{E}(T)$. As in Section \ref{bentQS} we will apply a Neumann series argument to show that \eqref{eq:NScap9a} has a unique solution $(\bar{u},\bar{\pi},\bar{q},\bar{h})\in\!_0\mathbb{E}(T)$. For that purpose we need to show the following property for $\bar{M}$. For each $\varepsilon>0$ there exist $T_0>0$ and $\eta_0>0$ such that
\begin{equation*}
\|\bar{M}(\theta,\bar{u},\bar{\pi},\bar{h})\|_{\mathbb{F}(T)}\le\varepsilon \|(\bar{u},\bar{\pi},\bar{q},\bar{h})\|_{\mathbb{E}(T)},
\end{equation*}
provided that $T\in (0,T_0)$ and $\eta\in (0,\eta_0)$. Mimicking the estimates of Section \ref{bentQS} for the components of $\bar{M}$ and taking into account that the operator $\operatorname{ext}_\Sigma$ is linear and bounded, one obtains an estimate of the form
$$\|\bar{M}(\theta,\bar{u},\bar{\pi},\bar{h})\|_{\mathbb{F}(T)}\le C \left[\|\theta'\|_\infty+T^{1/2p}(\|\theta''\|_\infty+\|\theta'''\|_\infty)\right]
\|(\bar{u},\bar{\pi},\bar{q},\bar{h})\|_{\mathbb{E}(T)},$$
with a uniform constant $C>0$. Since $\|\theta'\|_\infty\le\eta$, we may first choose $\eta>0$ sufficiently small and then $T>0$ sufficiently small, to obtain the desired estimate for the function $\bar{M}$.

Then we may apply a Neumann series argument in $_0\mathbb{E}(T)$ to conclude that there exists a unique solution $(\bar{u},\bar{h},\bar{\pi})\in\!_0\mathbb{E}(T)$ of the equation $$L(\bar{u},\bar{\pi},\bar{q},\bar{h})=\bar{M}(\theta,\bar{u},\bar{\pi},\bar{h})+F$$
or equivalently a unique solution $(u,\pi,q,h)$ of \eqref{eq:NScap8} as explained above.

This in turn yields a solution operator $S_{HS}:\mathbb{F}_{HS}\to\mathbb{E}_{HS}$ for \eqref{eq:NSbendquart1}, where $\mathbb{E}_{HS}$ and $\mathbb{F}_{HS}$ are the solution space and data space, respectively, for the bent half-space and the data in $\mathbb{F}_{HS}$ satisfy all relevant compatibility conditions at the contact line $\partial\Sigma_\theta$.


%
%
%

\chapter{General bounded cylindrical domains}\label{chptr:localization}

Let $n=3$ and $p>5$. In this section we will prove that system \eqref{eq:NScap3} admits a unique solution. To this end we apply the method of localization. We want to emphasize that this localization procedure cannot be simply carried over from standard parabolic systems. This is due to the divergence equation and the presence of the pressure in \eqref{eq:NScap3}. Let
$$\mathbb{E}_u(J):=H_p^1(J;L_p(\Omega)^3)\cap L_p(J;H_p^2(\Omega\backslash\Sigma)^3),\quad \mathbb{E}_\pi(J):=L_p(J;\dot{H}_p^1(\Omega)),$$
$$\mathbb{E}_q(J):=W_p^{1/2-1/2p}(J;L_p(\Sigma))\cap L_p(J;W_p^{1-1/p}(\Sigma)).$$
$$\mathbb{E}_h(J):=W_p^{2-1/2p}(J;L_p(\Sigma))\cap H_p^1(J;W_p^{2-1/p}(\Sigma))\cap L_p(J;W_p^{3-1/p}(\Sigma)),$$
and
$$\mathbb{E}(J):=\{(u,\pi,q,h)\in \mathbb{E}_u(J)\times\mathbb{E}_\pi(J)\times\mathbb{E}_q(J)\times \mathbb{E}_h(J):q=\Jump{\pi}\}.$$

\section{Regularity of the pressure}

Let $(u)_\Omega:=u-\frac{1}{|\Omega|}\int_\Omega u dx$ denote the part of $u\in L_1(\Omega)$ with mean value zero. We start with an auxiliary lemma which provides some additional regularity for the pressure.
\begin{lem}\label{lem:RegPressure}
Let $(u,\pi,\Jump{\pi},h)\in\mathbb{E}(J)$ be a solution of \eqref{eq:NScap3} with
$$f_d=u_0=h_0=g_2=u_\Sigma\cdot\nu_\Sigma=g_3\cdot\nu_{\partial\Omega}=0,$$
and $f\in\!_0W_p^{\alpha}(J;L_p(\Omega)^3)$ for some $\alpha\in (0,1/2-1/2p)$.
Then the following assertions hold.
\begin{enumerate}
\item If $\Omega$ is bounded, then $(\pi)_\Omega\in\!_0W_p^{\alpha}(J;L_p(\Omega))$ and the estimate
$$\|(\pi)_\Omega\|_{W_p^{\alpha}(L_p)}\le C\left(\|u\|_{\mathbb{E}_u}+\|\Jump{\pi}\|_{\mathbb{E}_q}+\|f\|_{W_p^\alpha(L_p)}\right)$$
is valid, where $C>0$ does not depend on the length of the interval $J$.
\item If $\Omega$ is a full space, a (bent) quarter space or a (bent) half space, then $(\pi)_K\in\!_0W_p^{\alpha}(J;L_p(K))$ for each bounded set $K\subset\overline{\Omega}$. Furthermore there exists a constant $C_K>0$ which does not depend on the length of the interval $J$ such that the estimate
    $$\|(\pi)_K\|_{W_p^{\alpha}(L_p(K))}\le C_K\left(\|u\|_{\mathbb{E}_u}+\|\Jump{\pi}\|_{\mathbb{E}_q}+\|f\|_{W_p^\alpha(L_p)}\right)$$
    is valid.
\end{enumerate}
\end{lem}
\begin{proof}
1. Let $g\in L_{p^\prime}(\Omega)$ be given and solve the problem
\begin{align}\label{reducedf2}
\Delta\psi&=g-(g|\mathds{1})\quad \mbox{ in } \; \Omega\backslash\Sigma,\nonumber\\
[\![\rho\psi]\!]&=0\quad \mbox{ on } \;\Sigma\nonumber,\\
[\![\partial_{\nu_\Sigma}\psi]\!]&=0\quad \mbox{ on } \;\Sigma,\\
\partial_{\nu_{\partial\Omega}}\psi&=0\quad \mbox{ on }\; \partial\Omega\backslash\partial\Sigma=(S_1\backslash\partial\Sigma)\cup S_2,\nonumber
\end{align}
by Lemma \ref{lem:appaux0} and define $\phi:=\rho\psi$.
Since $((\pi)_\Omega|\mathds{1})=(u|\nabla\phi)=0$ we obtain by integration by parts
\begin{align*}
((\pi)_\Omega|g)&=((\pi)_\Omega|(g)_\Omega) \\
&=\left(\frac{(\pi)_\Omega}{\rho}|\Delta \phi\right) = -\int_\Sigma[\![\frac{(\pi)_\Omega}{\rho}\partial_{\nu_{\Sigma}} \phi]\!]d\Sigma - \left(\frac{\nabla\pi}{\rho}|\nabla \phi\right)\\
&= -\int_\Sigma[\![\pi]\!]\frac{\partial_{\nu_{\Sigma}} \phi}{\rho}d\Sigma - \left(\frac{\mu}{\rho}\Delta u|\nabla \phi\right)-(f|\nabla\phi)\\
&= \int_\Omega \frac{\mu}{\rho}\nabla u:\nabla^2 \phi dx -\int_{\partial\Omega}\frac{\mu\partial_{\nu_{\partial\Omega}} u}{\rho}\nabla\phi d\sigma+\int_\Sigma \{[\![\frac{\mu\partial_{\nu_{\Sigma}} u}{\rho}\nabla\phi]\!]-[\![\pi]\!]\frac{\partial_{\nu_{\Sigma}} \phi}{\rho}\}d\Sigma\\
&\hspace{1cm}-(f|\nabla\phi).
\end{align*}
Note that there exists a constant $C>0$ such that $\|\phi\|_{W_{p'}^2}\le C\|g\|_{L_{p'}}$. Hence, taking the supremum of the left hand side over all functions $g\in L_{p'}(\Omega)$ with norm less or equal to one, we obtain
\begin{multline*}
\|(\pi)_\Omega(t)\|_{L_p(\Omega)}\le C\Big(\|\nabla u(t)\|_{L_p(\Omega)}+\|\partial_{\nu_{\partial\Omega}}u(t)\|_{L_p(\partial\Omega)}\\
+\|\left(\partial_{\nu_{\Sigma}}u(t)\right)_\pm\|_{L_p(\Sigma)}+
\|\Jump{\pi(t)}\|_{L_p(\Sigma)}+\|f(t)\|_{L_p(\Omega)}\Big),
\end{multline*}
for almost all $t\in J$. The same strategy yields the estimate
\begin{multline*}
\|(\pi)_\Omega(t)-(\pi)_\Omega(s)\|_{L_p(\Omega)}\le C\Big(\|\nabla (u(t)-u(s))\|_{L_p(\Omega)}
+\|\partial_{\nu_{\partial\Omega}}(u(t)-u(s))\|_{L_p(\partial\Omega)}\\
+\|\left(\partial_{\nu_{\Sigma}}(u(t)-u(s))\right)_\pm\|_{L_p(\Sigma)}+
\|\Jump{\pi(t)}-\Jump{\pi(s)}\|_{L_p(\Sigma)}+\|f(t)-f(s)\|_{L_p(\Omega)}\Big),
\end{multline*}
for almost all $s,t\in J$.

By the mixed derivative theorem and trace theory it holds that $\partial_k u_l\in\!_0H_p^{1/2}(J;L_p(\Omega))$,
$$\left(\partial_k u_l\right)_\pm|_{\Sigma}\in\!_0W_p^{1/2-1/2p}(J;L_p(\Sigma))$$
and
$$\partial_k u_l|_{\partial\Omega}\in\!_0W_p^{1/2-1/2p}(J;L_p(\partial\Omega)),$$
for $k,l\in \{1,2,3\}$. Moreover, $\Jump{\pi}\in\!_0W_p^{1/2-1/2p}(J;L_p(\Sigma))$. Since $H_p^s\hookrightarrow W_p^{s-\varepsilon}$ for each $s>0$, $\varepsilon\in (0,s)$, the claim follows.

2. The proof of the second assertion follows essentially the lines of the proof of the first assertion. We fix a bounded set $K\subset\overline{\Omega}$. Let $g\in L_p(K)$ and define $(g)_K:=g-\frac{1}{|K|}(g|1)_K$, where $(u|v)_K:=\int_K uv dx$. Extend $(g)_K$ by zero to $\tilde{g}\in L_p(\Omega)$. Then $\tilde{g}\in \hat{W}_p^{-1}(\Omega)\cap L_p(\Omega)$ and we may solve the elliptic problem \eqref{reducedf2} with $\tilde{g}$ as an inhomogeneity in the first equation by Lemma \ref{lem:appaux0}. This yields a solution $\psi\in \dot{H}^1_p(\Omega\backslash\Sigma)\cap \dot{H}_p^2(\Omega\backslash\Sigma)$ satisfying the estimate
$$\|\nabla\psi\|_{L_p(\Omega)}+\|\nabla^2\psi\|_{L_p(\Omega)}\le C\|\tilde{g}\|_{L_p(\Omega)}\le C_K\|g\|_{L_p(K)}.$$
We have $((\pi)_K|g)_K=((\pi)_K|(g)_K)_K=((\pi)_K|\tilde{g})_\Omega:=\int_\Omega(\pi)_K\tilde{g} dx$. We are now in a position to imitate the steps in the proof of the first assertion. This yields the validity of the second assertion.
\end{proof}

\section{Reduction of the data}\label{sec:RedofData}

It is convenient to reduce the data in \eqref{eq:NScap3} to the special case
$$f=f_d=u_0=h_0=g_2=u_\Sigma\cdot\nu_\Sigma=g_3\cdot\nu_{\partial\Omega}=0.$$
Extend $h_0\in W_p^{3-2/p}(\Sigma)$ and $g_h|_{t=0},m[u_0\cdot e_3]\in W_p^{2-3/p}(\Sigma)$ to some functions $\tilde{h}_0\in W_p^{3-2/p}(\mathbb{R}^2)$ and $\tilde{g}_h^0,\tilde{m}_0\in W_p^{2-3/p}(\mathbb{R}^2)$, respectively, and define
\begin{multline*}
    \tilde{h}_*(t) = [2e^{-(I-\Delta_{x'})^{1/2} t}-e^{-2(I-\Delta_{x'})^{1/2} t}]\tilde{h}_0 +\\
    [e^{-(I-\Delta_{x'}) t}-e^{-2(I-\Delta_{x'}) t}](I-\Delta_{x'})^{-1}\left(\tilde{m}_0+\tilde{g}_h^0\right),
    \quad t\geq0.
\end{multline*}
Then
$$\tilde{h}_*\in W^{2-1/2p}_p(J;L_p(\mathbb{R}^2))\cap H^1_p(J;W^{2-1/p}_p(\mathbb{R}^2))
\cap L_p(J;W^{3-1/p}_p(\mathbb{R}^2))$$
and it holds that $\tilde{h}_*(0)=\tilde{h}_0$ as well as $\partial_t\tilde{h}_*(0)=\tilde{m}_0+\tilde{g}_h^0$. Defining $h_*:=\tilde{h}_*|_{\Sigma}$ it follows that $h_*(0)=h_0$ and $\partial_t h_*(0)=m[u_0]+g_h|_{t=0}$. Setting $h_1:=h-h_*$ we have $h_1|_{t=0}=\partial_t h_1|_{t=0}=0$.

Next, let $u_0=(v_0,w_0)$ and $q_0:=2\Jump{\mu\partial_3 w_0}+\sigma\Delta_{x'}h_0+g_w|_{t=0}\in W_p^{1-3/p}(\Sigma)$. Extend $q_0$ to some
$\tilde{q}_0\in W_p^{1-3/p}(\mathbb{R}^2)$ and define $\tilde{q}_*(t):=e^{\Delta_{x'}t}\tilde{q}_0$. Then
$$\tilde{q}_*\in W_p^{1/2-1/2p}(J;L_p(\mathbb{R}^2))\cap L_p(J;W_p^{1-1/p}(\mathbb{R}^2)).$$
Setting $q_*:=\tilde{q}_*|_{\Sigma}$ it follows that
$$q_*\in W_p^{1/2-1/2p}(J;L_p(\Sigma))\cap L_p(J;W_p^{1-1/p}(\Sigma))$$
and $q_*|_{t=0}=q_0$. Given $q_*$, we solve the weak elliptic transmission problem
\begin{align*}
(\nabla\pi_*|\nabla\phi)&=0,\quad \phi\in H_{p'}^1(\Omega) ,\\
[\![\pi_*]\!]&=q_*,\quad \mbox{ on }\;\Sigma
\end{align*}
to obtain a unique solution $\pi_*\in L_p(J;\dot{H}_p^1(\Omega\backslash\Sigma))$ by Lemma \ref{lem:appauxlemweak}.

Next we solve the parabolic transmission problem
\begin{align}
\begin{split}\label{auxprbl1}
\partial_t(\rho u_*)-\mu\Delta u_*&=-\nabla\pi_*+\rho f,\quad \text{in}\ \Omega\backslash\Sigma,\\
-\Jump{\mu \partial_3 v_*}-\Jump{\mu\nabla_{x'} w_*}&=g_v,\quad \text{on}\ \Sigma,\\
-2\Jump{\mu \partial_3 w_*}&=g_w-q_*+\sigma\Delta_{x'}h_*,\quad \text{on}\ \Sigma,\\
\Jump{u_*}&=u_\Sigma,\quad \text{on}\ \Sigma,\\
P_{S_1}\left(\mu(\nabla u_*+\nabla u_*^{\sf T})\nu_{S_1}\right)&=P_{S_1}g_1,\quad \text{on}\ S_1\backslash\partial\Sigma,\\
u_*\cdot\nu_{S_1}&=g_2,\quad \text{on}\ S_1\backslash\partial\Sigma,\\
u_*&=g_3,\quad \text{on}\ S_2,\\
u_*(0)&=u_0,\quad \text{in}\ \Omega\backslash\Sigma.
\end{split}
\end{align}
to obtain a solution $u_*\in H_p^1(J;L_p(\Omega)^3)\cap L_p(J;H_p^2(\Omega\backslash\Sigma)^3)$ by Lemma \ref{lem:appaux2}. Note that all relevant compatibility conditions of the data are satisfied by assumption. Setting $u_1=u-u_*$ and $\pi_1=\pi-\pi_*$ we see that w.l.o.g. we may assume that $u_0=h_0=f=0$. To remove $f_d$ we solve the transmission problem
\begin{equation}\label{auxprbl2}
\begin{aligned}
\Delta \psi &= f_d-\div u_* &\ \mbox{in}\quad &\Omega\backslash\Sigma,\\
\mbox{}[\![\rho\psi]\!] &= 0 &\ \mbox{on}\quad &\Sigma,\\
[\![\partial_{e_3}\psi]\!] &= 0 &\ \mbox{on}\quad &\Sigma,\\
\partial_{\nu_{\partial\Omega}}\psi &= 0 &\ \mbox{on}\quad &\partial\Omega\backslash\partial\Sigma=(S_1\backslash\partial\Sigma)\cup S_2,
\end{aligned}
\end{equation}
by Lemma \ref{lem:appauxhighreg}. We remark that $\int_\Omega(f_d-\div u_*)dx=0$ by the compatibility conditions on $(f_d,u_\Sigma,g_2,g_3)$ and
$$f_d-\div u_*\in\!_0H_p^1(J;\hat{H}_p^{-1}(\Omega))\cap L_p(J;H_p^1(\Omega\backslash\Sigma)).$$
Therefore we obtain a solution $\nabla\psi\in\!_0\mathbb{E}_u(J)$. Setting $u_2:=u_1-\nabla\psi$, $\pi_2:=\pi_1+\rho\partial_t\psi-\mu\Delta\psi$ and $h_2:=h_1$ we see that we may assume that $f_d=g_2=u_\Sigma\cdot e_3=g_3\cdot e_3=0$. The time trace of all the remaining data at $t=0$ vanishes.

\section{Localization procedure}

Before we can state the main result of this section, we introduce some function spaces. Let
$$\mathbb{F}_1(J):=L_p(J;L_p(\Omega)^3),\quad \mathbb{F}_2(J):=L_p(J;H_p^1(\Omega\backslash\Sigma)).$$
$$\mathbb{F}_3(J):=W_p^{1/2-1/2p}(J;L_p(\Sigma)^2)\cap L_p(J;W_p^{1-1/p}(\Sigma)^2),$$
$$\mathbb{F}_4(J):=W_p^{1/2-1/2p}(J;L_p(\Sigma))\cap L_p(J;W_p^{1-1/p}(\Sigma)),$$
$$\mathbb{F}_5(J):=W_p^{1-1/2p}(J;L_p(\Sigma)^3)\cap L_p(J;W_p^{2-1/p}(\Sigma)^3),$$
$$\mathbb{F}_6(J):=W_p^{1-1/2p}(J;L_p(\Sigma))\cap L_p(J;W_p^{2-1/p}(\Sigma)),$$
$$\mathbb{F}_7(J):=W_p^{1/2-1/2p}(J;L_p(S_1)^3)\cap L_p(J;W_p^{1-1/p}(S_1\backslash\partial\Sigma)^3),$$
$$\mathbb{F}_8(J):=W_p^{1-1/2p}(J;L_p(S_1))\cap L_p(J;W_p^{2-1/p}(S_1\backslash\partial\Sigma)),$$
$$\mathbb{F}_9(J):=W_p^{1-1/2p}(J;L_p(S_2))\cap L_p(J;W_p^{2-1/p}(S_2)),$$
$$\mathbb{F}_{10}(J):=W_p^{3/2-1/p}(J;L_p(\partial\Sigma))\cap H_p^1(J;W_p^{1-2/p}(\partial\Sigma))\cap L_p(J;W_p^{2-2/p}(\partial\Sigma)),$$
and $\tilde{\mathbb{F}}(J):=\times_{j=1}^{10}\mathbb{F}_j(J)$ as well as
$$\mathbb{F}(J):=\{(f_1,\ldots,f_{10})\in \tilde{\mathbb{F}}(J):(f_2,f_5,f_8,f_{9})\in H_p^1(J;\hat{H}_p^{-1}(\Omega))\}.$$
Furthermore, we set $X_{\gamma}:=X_{\gamma,u}\times X_{\gamma,h}$, where $X_{\gamma,u}:=W_p^{2-2/p}(\Omega\backslash\Sigma)^3$ and $X_{\gamma,h}:=W_p^{3-2/p}(\Sigma)$.

The main result of this section reads as follows.
\begin{thm}\label{thm:linmaxreg}
Let $\mu_j,\rho_j,H_j,\sigma>0$, $n=3$, $p>5$ and let $G\in\mathbb{R}^{n-1}$ be open and bounded with $\partial G\in C^4$. Define $\Omega:=G\times (H_1,H_2)$ and let $\Sigma:=G\times\{0\}$. Let $S_1:=\partial G\times (H_1,H_2)$ and $S_2:=(G\times\{H_1\})\cup (G\times\{H_2\})$ be the vertical and horizontal parts of the boundary of $\Omega$, respectively. Then there exists a unique solution
$$u\in H_p^1(J;L_p(\Omega)^3)\cap L_p(J;H_p^2(\Omega\backslash\Sigma)^3),\quad \pi\in L_p(J;\dot{H}_p^1(\Omega\backslash\Sigma)),$$
$$\Jump{\pi}\in W_p^{1/2-1/2p}(J;L_p(\Sigma))\cap L_p(J;W_p^{1-1/p}(\Sigma))$$
$$h\in W_p^{2-1/2p}(J;L_p(\Sigma))\cap H_p^1(J;W_p^{2-1/p}(\Sigma))\cap L_p(J;W_p^{3-1/p}(\Sigma)),$$
of \eqref{eq:NScap3} if and only if the data are subject to the following regularity and compatibility conditions.
\begin{enumerate}
\item $(f,f_d,g_v,g_w,u_\Sigma,g_h,g_1,g_2,g_3,g_4)\in\mathbb{F}(J)$,
\item $(u_0,h_0)\in X_\gamma$,
\item $\div u_0=f_d|_{t=0}$, $-\Jump{\mu\nabla_{x'} w_0}-\Jump{\mu\partial_3 v_0}=g_v|_{t=0}$, $\Jump{u_0}=u_\Sigma|_{t=0}$,
\item $P_{S_1}(\mu(\nabla u_0+\nabla u_0^{\sf T})\nu_{S_1})=P_{S_1}g_1|_{t=0}$, $u_0\cdot \nu_{S_1}=g_2|_{t=0}$, $u_0=g_3|_{t=0}$,
\item $\partial_{\nu_{\partial G}}h_0=g_4|_{t=0}$,
\item $\Jump{g_2}=u_\Sigma\cdot \nu_{S_1}$,
\item $\Jump{(g_1\cdot e_3)/\mu-\partial_3 g_2}=\partial_{\nu_{S_1}}(u_\Sigma\cdot e_3),$
\item $P_{\partial G}[(D'v_\Sigma)\nu']=\Jump{P_{\partial G}g_1'/\mu},$
\item $\partial_t g_4-m[(g_1\cdot e_3)/\mu-\partial_3 g_2]=\partial_{\nu_{S_1}}g_h,$
\item $(g_v|\nu_{S_1})=-\Jump{g_1\cdot e_3}$, $(g_3|\nu_{S_1})=g_2$,
\item $P_{\partial G}[\mu(D'g_3')\nu']=(P_{\partial G}{g}_{1}'),$
\item $\mu\partial_{\nu_{S_1}} (g_3\cdot e_3)+\mu\partial_3 g_2=g_1\cdot e_3$,
\end{enumerate}
where $\nu':=\nu_{\partial G}$.
\end{thm}
\begin{proof} We will split the proof in two parts.

\noindent\textbf{(I) Existence of a left inverse}

Let $(u,\pi,\Jump{\pi},h)$ be a solution of \eqref{eq:NScap3}. By the results of the last subsection there exists $(\bar{u},\bar{\pi},\Jump{\bar{\pi}},\bar{h})$ such that $(\tilde{u},\tilde{\pi},\Jump{\tilde{\pi}},\tilde{h}):=(u,\pi,\Jump{\pi},h)-(\bar{u},\bar{\pi},\Jump{\bar{\pi}},\bar{h})$ solves the problem
\begin{align}\label{eq:loc1}
\begin{split}
\partial_t(\rho \tilde{u})-\mu\Delta \tilde{u}+\nabla \tilde{\pi}&=0,\quad \text{in}\ \Omega\backslash\Sigma,\\
\div \tilde{u}&=0,\quad \text{in}\ \Omega\backslash\Sigma,\\
-\Jump{\mu \partial_3 \tilde{v}}-\Jump{\mu\nabla_{x'} \tilde{w}}&=\tilde{g}_v,\quad \text{on}\ \Sigma,\\
-2\Jump{\mu \partial_3 \tilde{w}}+\Jump{\tilde{\pi}}-\sigma\Delta_{x'} \tilde{h}&=\tilde{g}_w,\quad \text{on}\ \Sigma,\\
\Jump{\tilde{u}}&=\tilde{u}_\Sigma,\quad \text{on}\ \Sigma,\\
\partial_t \tilde{h}-m[\tilde{w}]&=\tilde{g}_h,\quad \text{on}\ \Sigma,\\
P_{S_1}\left(\mu(\nabla \tilde{u}+\nabla \tilde{u}^{\sf T})\nu_{S_1}\right)&=P_{S_1}\tilde{g}_1,\quad \text{on}\ S_1\backslash\partial\Sigma,\\
\tilde{u}\cdot\nu_{S_1}&=0,\quad \text{on}\ S_1\backslash\partial\Sigma,\\
\tilde{u}&=\tilde{g}_3,\quad \text{on}\ S_2,\\
\partial_{\nu_{\partial G}}\tilde{h}&=\tilde{g}_4,\quad \text{on}\ \partial\Sigma,\\
\tilde{u}(0)&=0,\quad \text{in}\ \Omega\backslash\Sigma\\
\tilde{h}(0)&=0,\quad \text{on}\ \Sigma,
\end{split}
\end{align}
and $(\tilde{g}_3|e_3)=(\tilde{u}_\Sigma|e_3)=0$. Choose open sets $U_k=B_{r}(x_k)$ with
\begin{itemize}
\item $\partial\Sigma \subset\bigcup_{k=7}^{N_1}U_k$,
\item $\partial S_2\subset\bigcup_{k=N_1+1}^{N}U_k$,
\end{itemize}
and choose $r>0$ sufficiently small such that the corresponding solution operators from Sections \ref{bentQS} \& \ref{bentHS} are well-defined. According to Proposition \ref{partitionone} there exist open and connected sets
\begin{itemize}
\item $U_0\cap\Sigma\neq\emptyset$, $U_0\cap \partial\Omega=\emptyset$;
\item $U_k\subset\Omega_k$, $k=1,2$;
\item $U_k\cap S_1\neq\emptyset$, $U_k\cap (\Sigma\cup S_2)=\emptyset$, $k=3,4$;
\item $U_k\cap S_2\neq\emptyset$, $U_k\cap (\Sigma\cup S_1)=\emptyset$, $k=5,6$,
\end{itemize}
and a family of functions $\{\varphi\}_{k=0}^N\subset C_c^3(\mathbb{R}^3;[0,1])$ such that $\overline{\Omega}\subset\bigcup_{k=0}^N U_k$, $\operatorname{supp}\varphi_k\subset U_k$, $\sum_{k=0}^N\varphi_k=1$ and $\partial_{\nu_{S_1}}\varphi_k(x)=\partial_{e_3}\varphi_k(x)=0$ for $x\in U_k\cap(\partial\Sigma\cup\partial S_2)$, $k\ge 7$.

Multiplying each equation in \eqref{eq:loc1} by $\varphi_k$ we obtain the following local problems
\begin{align}\label{eq:loc2}
\begin{split}
\partial_t(\rho \tilde{u}_k)-\mu\Delta \tilde{u}_k+\nabla \tilde{\pi}_k&=F_k(\tilde{u},\tilde{\pi}),\quad \text{in}\ \Omega^k\backslash\Sigma^k,\\
\div \tilde{u}_k&=F_{dk}(\tilde{u}),\quad \text{in}\ \Omega^k\backslash\Sigma^k,\\
-\Jump{\mu \partial_3 \tilde{v}_k}-\Jump{\mu\nabla_{x'} \tilde{w}_k}&=\tilde{g}_{vk}+G_{vk}(\tilde{u}),\quad \text{on}\ \Sigma^k,\\
-2\Jump{\mu \partial_3 \tilde{w}_k}+\Jump{\tilde{\pi}_k}-\sigma\Delta_{x'} \tilde{h}_k&=\tilde{g}_{wk}+G_{wk}(\tilde{u},\tilde{h}),\quad \text{on}\ \Sigma^k,\\
\Jump{\tilde{u}_k}&=\tilde{u}_{\Sigma k},\quad \text{on}\ \Sigma^k,\\
\partial_t \tilde{h}_k-m[\tilde{w}_k]&=\tilde{g}_{hk},\quad \text{on}\ \Sigma^k,\\
P_{S_1^k}\left(\mu(\nabla \tilde{u}_k+\nabla \tilde{u}_k^{\sf T})\nu_k\right)&=P_{S_1^k}\tilde{g}_{1k}+G_{1k}(\tilde{u}),\quad \text{on}\ S_1^k\backslash\partial\Sigma^k,\\
\tilde{u}_k\cdot\nu_{k}&=0,\quad \text{on}\ S_1^k\backslash\partial\Sigma^k,\\
\tilde{u}_k&=\tilde{g}_{3k},\quad \text{on}\ S_2^k,\\
\partial_{\nu_k}\tilde{h}_k&=\tilde{g}_{4k},\quad \text{on}\ \partial\Sigma^k,\\
\tilde{u}_k(0)&=0,\quad \text{in}\ \Omega^k\backslash\Sigma^k\\
\tilde{h}_k(0)&=0,\quad \text{on}\ \Sigma^k,
\end{split}
\end{align}
where
$$F_k(\tilde{u},\tilde{\pi}):=[\nabla,\varphi_k]\tilde{\pi}
-\mu[\Delta,\varphi_k]\tilde{u},$$
$$F_{dk}(\tilde{u}):=\tilde{u}\cdot\nabla\varphi_k,$$
$$G_{vk}(\tilde{u}):=(I-e_3\otimes e_3)G_k(\tilde{u},\tilde{h}),$$
$$G_{wk}(\tilde{u},\tilde{h}):=G_k(\tilde{u},\tilde{h})e_3,$$
$$G_k(\tilde{u},\tilde{h}):=[\![-\mu(\nabla \varphi_k\otimes \tilde{u}+ \tilde{u}\otimes\nabla\varphi_k)]\!] e_3
-\sigma[\Delta_\Sigma,\varphi_k]\tilde{h}e_3,$$
and
$$G_{1k}(\tilde{u}):=(I-\nu_k\otimes\nu_k)(\mu(\nabla \varphi_k\otimes \tilde{u}+ \tilde{u}\otimes\nabla\varphi_k)) \nu_k.$$
Furthermore we have set $P_{S_1^k}:= I-\nu_k\otimes\nu_k$.

For $k=0$ we obtain a pure two-phase problem with a flat interface in $\mathbb{R}^n$. This case has been treated in \cite{PrSi09a}. If $k\in\{1,2\}$ then we are lead to one-phase Stokes equations in $\mathbb{R}^n$. An analysis of these problems can be found in \cite{BP07}. If $k\in\{7,\ldots,N_1\}$ and $k\in\{N_1+1,\ldots,N\}$ then we rotate the coordinate system (with respect to the $x_3$ axis) and translate it to obtain two-phase Stokes equations in bent half-spaces and one-phase Stokes equations in bent quarter-spaces, respectively. These problems have been treated in Sections \ref{bentQS} and \ref{bentHS}. Hence, the solution operators for the charts $U_k$, $k\ge 7$ are well defined by the results in \ref{bentQS} and \ref{bentHS}. Finally, if $k\in\{3,4\}$ then we obtain the Stokes equations in bent half-spaces with pure-slip conditions, while for $k\in\{5,6\}$ we are lead to the Stokes equations in half-spaces with no-slip boundary condition, see e.g.\ \cite{BP07} for the theory of the last two type of problems. We denote the corresponding solution operators for each chart $U_k$ by $\mathcal{S}_k$.

Note that all functions $F_j$, $G_j$ carry additional time regularity (take into account Lemma \ref{lem:RegPressure}) with exception of $F_{dk}$. To circumvent this problem we will reduce \eqref{eq:loc2} to the case $F_{dk}=0$. For this purpose we apply Lemma \ref{lem:appauxhighreg} and solve the transmission problem
\begin{equation*}
\begin{aligned}
\Delta \phi_k &= F_{dk}(\tilde{u}) &\ \mbox{in}\quad &\Omega^k\backslash\Sigma^k,\\
\mbox{}[\![\rho\phi_k]\!] &= 0 &\ \mbox{on}\quad &\Sigma^k,\\
[\![\partial_{e_3}\phi_k]\!] &= 0 &\ \mbox{on}\quad &\Sigma^k,\\
\partial_{\nu_{k}}\phi_k &= 0 &\ \mbox{on}\quad &\partial\Omega^k\backslash\partial\Sigma^k.
\end{aligned}
\end{equation*}
This yields a solution
$$\nabla\phi_k\in\!_0H_p^1(J;H_p^1(\Omega^k\backslash\Sigma^k)^3)\cap L_p(J;H_p^3(\Omega^k\backslash\Sigma^k)^3)=:\!_0Z(J)$$
satisfying the estimate
\begin{equation}\label{eq:locphik}\|\nabla\phi_k\|_{Z(J)}\le C_N\|\tilde{u}\|_{\mathbb{E}_u(J)}.
\end{equation}
The constant $C_N>0$ depends on $N$ but not on the length of $J$. We define $\hat{u}_k:=\tilde{u}_k-\nabla\phi_k$ and $\hat{\pi}_k:=\tilde{\pi}_k+\rho\partial_t\phi_k-\mu\Delta\phi_k$. With $\hat{h}=\tilde{h}$ we obtain the system
\begin{align}\label{eq:loc2a}
\begin{split}
\partial_t(\rho \hat{u}_k)-\mu\Delta \hat{u}_k+\nabla \hat{\pi}_k&=F_k(\tilde{u},\tilde{\pi}),\quad \text{in}\ \Omega^k\backslash\Sigma^k,\\
\div \hat{u}_k&=0,\quad \text{in}\ \Omega^k\backslash\Sigma^k,\\
-\Jump{\mu \partial_3 \hat{v}_k}-\Jump{\mu\nabla_{x'} \hat{w}_k}&=\tilde{g}_{vk}+\hat{G}_{vk}(\tilde{u}),\quad \text{on}\ \Sigma^k,\\
-2\Jump{\mu \partial_3 \hat{w}_k}+\Jump{\hat{\pi}_k}-\sigma\Delta_{x'} \hat{h}_k&=\tilde{g}_{wk}+\hat{G}_{wk}(\tilde{u},\tilde{h}),\quad \text{on}\ \Sigma^k,\\
\Jump{\hat{u}_k}&=\tilde{u}_{\Sigma k}-\Jump{\nabla\phi_k},\quad \text{on}\ \Sigma^k,\\
\partial_t \hat{h}_k-m[\hat{w}_k]&=\tilde{g}_{hk}+m[\partial_3\phi_k],\quad \text{on}\ \Sigma^k,\\
P_{S_1^k}\left(\mu(\nabla \hat{u}_k+\nabla \hat{u}_k^{\sf T})\nu_k\right)&=P_{S_1^k}\tilde{g}_{1k}+\hat{G}_{1k}(\tilde{u}),\quad \text{on}\ S_1^k\backslash\partial\Sigma^k,\\
\hat{u}_k\cdot\nu_{k}&=0,\quad \text{on}\ S_1^k\backslash\partial\Sigma^k,\\
\hat{u}_k&=\tilde{g}_{3k}-\nabla\phi_k,\quad \text{on}\ S_2^k,\\
\partial_{\nu_k}\hat{h}_k&=\tilde{g}_{4k},\quad \text{on}\ S_1^k\cap\Sigma^k,\\
\hat{u}_k(0)&=0,\quad \text{in}\ \Omega^k\backslash\Sigma^k\\
\hat{h}_k(0)&=0,\quad \text{on}\ \Sigma^k,
\end{split}
\end{align}
where
$$\hat{G}_k(\tilde{u},\tilde{h}):=G_k(\tilde{u},\tilde{h})+2\Jump{\mu\nabla^2\phi_k}e_3-\Jump{\mu\Delta\phi_k}e_3,$$
$\hat{G}_{kv}$, $\hat{G}_{kw}$ defined as above and
$$\hat{G}_{1k}(\tilde{u}):=G_{1k}(\tilde{u})-2\mu(I-\nu_k\otimes\nu_k)\nabla^2\phi_k\nu_k.$$
With the help of the solution operators $\mathcal{S}_k$, we may rewrite \eqref{eq:loc2a} as
\begin{equation}\label{eq:loc3}
(\hat{u}_k,\hat{\pi}_k,\hat{h}_k)
=\mathcal{S}_k\left(\tilde{H}_k+H_k(\tilde{u},\tilde{\pi},\tilde{h})\right),
\end{equation}
where $\tilde{H}_k$ stands for the set of given data and $H_k(\tilde{u},\tilde{\pi},\tilde{h})$ denotes the remaining part on the right hand side of \eqref{eq:loc2a}. Let $\{\theta_k\}_{k=0}^N\subset C_c^\infty(U_k)$ such that $\theta_k|_{\operatorname{supp}\varphi_k}=1$ and multiply \eqref{eq:loc3} by $\theta_k$.
By Lemma \ref{lem:RegPressure} it holds that $(\tilde{\pi}_k\nabla^j\theta_k),(\hat{\pi}_k\nabla^j\theta_k)
\in\!_0W_p^\alpha(J;L_p(\Omega^k))$ for each $j\in\{0,1,2\}$ and $k\in\{0,\ldots,N\}$, since $\operatorname{supp}\theta_k\subset U_k$ is bounded. In addition, the estimate
\begin{multline*}
\|\tilde{\pi}_k\nabla^j\theta_k\|_{W_p^\alpha(J;L_p(\Omega^k))}
+\|\hat{\pi}_k\nabla^j\theta_k\|_{W_p^\alpha(J;L_p(\Omega^k))}\\
\le C\left(\|\tilde{u}\|_{\mathbb{E}_u(J)}+\|\tilde{h}\|_{\mathbb{E}_u(J)}+
\|\tilde{H}\|_{\mathbb{F}(J)}\right)
\end{multline*}
is valid, where $C>0$ does not depend on $T>0$. This implies
\begin{align*}
\|(\nabla^j\theta_k)(\rho\partial_t\phi_k-\mu\Delta\phi_k)\|_{W_p^\alpha(J;L_p(\Omega^k))}&=
\|(\nabla^j\theta_k)(\hat{\pi}_k-\tilde{\pi}_k)\|_{W_p^\alpha(J;L_p(\Omega^k))}\\
&\le C\left(\|\tilde{u}\|_{\mathbb{E}_u(J)}+\|\tilde{h}\|_{\mathbb{E}_u(J)}+
\|\tilde{H}\|_{\mathbb{F}(J)}\right)
\end{align*}
and since $\Delta\phi_k=F_{dk}(\tilde{u})\in\!_0\mathbb{E}_u(J)$, it follows that
$$\|(\nabla^j\theta_k)\partial_t\phi_k\|_{_0W_p^\alpha(J;L_p(\Omega^k))}\le C\left(\|\tilde{u}\|_{\mathbb{E}_u(J)}+\|\tilde{h}\|_{\mathbb{E}_u(J)}+
\|\tilde{H}\|_{\mathbb{F}(J)}\right)$$
for each $j\in\{0,1,2\}$ and $k\in\{0,\ldots,N\}$. Hence, by H\"{o}lder's inequality and Sobolev embedding
$$\|(\nabla^j\theta_k)\partial_t\phi_k\|_{L_p(J;L_p(\Omega^k))}\le
T^{1/2p}\|(\nabla^j\theta_k)\partial_t\phi_k\|_{_0W_p^\alpha(J;L_p(\Omega^k))}.$$
Next, we apply H\"{o}lder's inequality, Sobolev embeddings and the mixed derivative theorem to obtain
\begin{align*}
\|\theta_k\partial_t\phi_k\|_{L_p(J;H_p^1(\Omega^k))}&\le T^{1/2p}\|\theta_k\partial_t\phi_k\|_{L_{2p}(J;H_p^1(\Omega^k))}\\
&\le CT^{1/2p}\|\theta_k\partial_t\phi_k\|_{W_p^{\alpha/2-\varepsilon}(J;H_p^1(\Omega^k))}\\
&\le CT^{1/2p}\|\theta_k\partial_t\phi_k\|_{H_p^{\alpha/2-\varepsilon/2}(J;H_p^1(\Omega^k))}\\
&\le CT^{1/2p}\|\theta_k\partial_t\phi_k\|_{H_p^{\alpha-\varepsilon}(J;L_p(\Omega^k))\cap L_p(J;H_p^2(\Omega^k))}\\
&\le CT^{1/2p}\|\theta_k\partial_t\phi_k\|_{W_p^{\alpha}(J;L_p(\Omega^k))\cap L_p(J;H_p^2(\Omega^k))}
\end{align*}
for some $\alpha\in (0,1/2-1/2p)$ and a sufficiently small $\varepsilon>0$. Note that $$\|\nabla\partial_t\phi_k\|_{L_p(J;L_p(\Omega^k))}+\|\nabla^2\partial_t\phi_k\|_{L_p(J;L_p(\Omega^k))}\le C\|\tilde{u}\|_{\mathbb{E}_u(J)},$$
by \eqref{eq:locphik}, hence
$$\|\theta_k\partial_t\phi_k\|_{L_p(J;H_p^1(\Omega^k))}\le CT^{1/2p}\left(\|\tilde{u}\|_{\mathbb{E}_u(J)}+\|\tilde{h}\|_{\mathbb{E}_u(J)}+
\|\tilde{H}\|_{\mathbb{F}(J)}\right).$$
In particular, this implies
\begin{align*}
\|\theta_k\partial_t\nabla\phi_k\|_{L_p(J;L_p(\Omega^k))}&\le
\|\theta_k\partial_t\phi_k\|_{L_p(J;H_p^1(\Omega^k))}
+\|(\nabla\theta_k)\partial_t\phi_k\|_{L_p(J;L_p(\Omega^k))}\\
&\le CT^{1/2p}\left(\|\tilde{u}\|_{\mathbb{E}_u(J)}+\|\tilde{h}\|_{\mathbb{E}_u(J)}+
\|\tilde{H}\|_{\mathbb{F}(J)}\right).
\end{align*}
Moreover, by Sobolev embedding and the mixed derivative theorem, we obtain
$$\|\theta_k\nabla\phi_k\|_{L_p(J;H_p^2(\Omega^k))}\le CT^{1/2p}\|\nabla\phi_k\|_{_0H_p^{1/2}(J;H_p^2(\Omega^k))}\le CT^{1/2p}\|\tilde{u}\|_{\mathbb{E}_u(J)}.$$
Since all terms in $H_k(\tilde{u},\tilde{\pi},\tilde{h})$ carry additional time regularity, there exists some $\gamma>0$ such that
$$\|H_k(\tilde{u},\tilde{\pi},\tilde{h})\|_{\mathbb{F}(J)}\le CT^\gamma\|(\tilde{u},\tilde{\pi},\tilde{h})\|_{\mathbb{E}(J)}.$$
We may now replace $\theta_k\hat{u}_k$ by $\theta_k(\tilde{u}_k-\nabla\phi_k)$ and $\theta_k\hat{\pi}_k$ by $\theta_k(\tilde{\pi}_k+\rho\partial_t\phi_k-\mu\Delta\phi_k)$ in \eqref{eq:loc3} to obtain the estimate
\begin{equation}\label{eq:loc4}
\|\theta_k(\tilde{u}_k,\tilde{\pi}_k,\tilde{h}_k)\|_{\mathbb{E}(J)}\le C\left(\|\theta_k\tilde{H}_k\|_{\mathbb{F}(J)}+
T^{\tilde{\gamma}}\|(\tilde{u},\tilde{\pi},\tilde{h})\|_{\mathbb{E}(J)}\right),
\end{equation}
with a constant $C>0$ being independent of $T>0$. Here $\tilde{\gamma}:=\max\{1/2p,\gamma\}$. Since $\theta_k(\tilde{u}_k,\tilde{\pi}_k,\tilde{h}_k)=(\tilde{u}_k,\tilde{\pi}_k,\tilde{h}_k)$ we may take the sum over all charts to obtain
$$\|(\tilde{u},\tilde{\pi},\tilde{h})\|_{\mathbb{E}(J)}\le C_N\left(\|\tilde{H}\|_{\mathbb{F}(J)}+
T^{\tilde{\gamma}}\|(\tilde{u},\tilde{\pi},\tilde{h})\|_{\mathbb{E}(J)}\right).
$$
Therefore, choosing $T>0$ sufficiently small, we obtain the a priori estimate
$$\|(\tilde{u},\tilde{\pi},\tilde{h})\|_{\mathbb{E}(J)}
\le C_N\|\tilde{H}\|_{\mathbb{F}(J)}$$
for the solution of \eqref{eq:loc1}. A successive application of the above argument yields the estimate on each finite interval $J=[0,T]$. It follows that the solution-to-data operator $L:\!_0\mathbb{E}(J)\to\!_0\mathbb{F}(J)$, defined by the left hand side of \eqref{eq:loc1} is injective with closed range. In particular, there exists a left inverse $S$ for $L$, that is $SLz=z$ for all $z\in\!_0\mathbb{E}(J)$.

\medskip

\noindent\textbf{(II) Existence of a right inverse}

It remains to prove the existence of a right inverse for $L$. To this end, let the data $F:=(f,f_d,g_v,g_w,g_1,g_2,g_3,g_4,u_\Sigma,g_h)\in\mathbb{F}(J)$, $(u_0,h_0)\in X_\gamma$, subject to the conditions in Theorem \ref{thm:linmaxreg} be given.
By the results in Section \ref{sec:RedofData}, we may assume without loss of generality  that $u_0=h_0=0$. In particular this means that the time traces of all inhomogeneities at $t=0$ vanish if they exist.

Let $u_*,\nabla\psi\in\!_0\mathbb{E}_u(J)$ denote the unique solutions of \eqref{auxprbl1} and \eqref{auxprbl2}, respectively, where now $q_*=\pi_*=h_*=0$. Set $\bar{u}:=u_*-\nabla\psi$, $\bar{\pi}:=\mu\Delta\psi-\rho\partial_t\psi$ and $\bar{h}=0$. Defining
$$\bar{S}F:=(\bar{u},\bar{\pi},\Jump{\bar{\pi}},\bar{h})$$
it holds that
$$L\bar{S}F=L(\bar{u},\bar{\pi},\Jump{\bar{\pi}},\bar{h})=
\begin{pmatrix}
f\\f_d\\g_v+G_v(\psi)\\g_w+G_w(\psi)\\u_\Sigma+G_\Sigma(\psi)\\G_h(u_*,\psi)\\
g_1+G_1(\psi)\\g_2\\g_3+G_3(\psi)\\0
\end{pmatrix},$$
where
$$G_v(\psi):=2\Jump{\mu (I-e_3\otimes e_3)(\nabla^2\psi e_3)},$$ $$G_w(\psi):=2\Jump{\mu(\nabla^2\psi e_3)\cdot e_3}+\Jump{\mu\Delta\psi},$$ $G_\Sigma(\psi):=-\Jump{\nabla\psi}$, $G_h(u_*,\psi):=-m[u_*\cdot e_3-\partial_3\psi]$, $$G_1(\psi):=-2\mu(I-\nu_{S_1}\otimes\nu_{S_1})(\nabla^2\psi\nu_{S_1}),$$
and $G_3(\psi):=-\nabla\psi|_{S_2}$.

In a next step we consider the problems
\begin{align}\label{eq:loc5}
\begin{split}
\partial_t(\rho \tilde{u}_k)-\mu\Delta \tilde{u}_k+\nabla \tilde{\pi}_k&=0,\quad \text{in}\ \Omega^k\backslash\Sigma^k,\\
\div \tilde{u}_k&=0,\quad \text{in}\ \Omega^k\backslash\Sigma^k,\\
-\Jump{\mu \partial_3 \tilde{v}_k}-\Jump{\mu\nabla_{x'} \tilde{w}_k}&=G_v^k(\psi),\quad \text{on}\ \Sigma^k,\\
-2\Jump{\mu \partial_3 \tilde{w}_k}+\Jump{\tilde{\pi}_k}-\sigma\Delta_{x'} \tilde{h}_k&=G_w^k(\psi),\quad \text{on}\ \Sigma^k,\\
\Jump{\tilde{u}_k}&=G_\Sigma^k(\psi),\quad \text{on}\ \Sigma^k,\\
\partial_t \tilde{h}_k-m[\tilde{w}_k]&=G_h^k(u_*,\psi)-g_h^k,\quad \text{on}\ \Sigma^k,\\
P_{S_1^k}\left(\mu(\nabla \tilde{u}_k+\nabla \tilde{u}_k^{\sf T})\nu_{k}\right)&=G_1^k(\psi),\quad \text{on}\ S_1^k\backslash\partial\Sigma^k,\\
\tilde{u}_k\cdot\nu_{k}&=0,\quad \text{on}\ S_1^k\backslash\partial\Sigma^k,\\
\tilde{u}_k&=G_3^k(\psi),\quad \text{on}\ S_2^k,\\
\partial_{\nu_{k}}\tilde{h}_k&=-g_4^k,\quad \text{on}\ \partial\Sigma^k,\\
\tilde{u}(0)&=0,\quad \text{in}\ \Omega^k\backslash\Sigma^k\\
\tilde{h}(0)&=0,\quad \text{on}\ \Sigma^k,
\end{split}
\end{align}
where
$$G_j^k(\psi):=G_j(\psi)\varphi_k,\ j\in\{v,w,\Sigma,1,3\},\ G_h^k(u_*,\psi):=G_h(u_*,\psi)\varphi_k,$$
and $g_m^k:=g_m\varphi_k$, $m\in\{h,4\}$.
Let us check whether the right hand side in \eqref{eq:loc5} satisfies all relevant compatibility conditions at $\partial\Sigma^k$ and $\partial S_2^k$, $k\ge 7$. Consider first the case $x\in \partial S_2^k$, $k\in\{7,\ldots,N_1\}$.

We have to show that the relations $G_3^k(\psi)\cdot\nu_k=0$, $\mu\partial_{\nu_k}(G_3^k(\psi)\cdot e_3)=G_1^k(\psi)\cdot e_3$ and
$$P_{\partial G^k}[\mu (D'G_3^{k'}(\psi))\nu_k ']=-P_{\partial G^k}[\mu (D'\psi)\nu_k ']\varphi_k$$
hold at $\partial S_2^k$, where
$$G_3^{k'}(\psi):=\begin{pmatrix}
G_3^k(\psi)\cdot e_1\\
G_3^k(\psi)\cdot e_2
\end{pmatrix}
.$$
The first condition is equivalent to $\varphi_k(\nabla\psi\cdot \nu_k)=0$ at $\partial S_2^k$. Since $\nu_k=\nu_{S_1}=(\nu',0)$ on $\operatorname{supp}\varphi_k$, the claim follows from the fact that $\partial_{\nu_k}\psi=\nabla\psi\cdot \nu_k=\nabla_{x'}\psi\cdot\nu'=0$ at $x\in \partial S_2\cap\operatorname{supp}\varphi_k$, by construction of $\psi$. Next, we compute $$\partial_{\nu_k}(G_3^k(\psi)\cdot e_3)=-\partial_{\nu_k}(\varphi_k\partial_3\psi)=-
\partial_3\psi\partial_{\nu_k}\varphi_k-\varphi_k\partial_{\nu_k}\partial_3\psi=0,$$
since $\partial_{\nu_k}\varphi_k=0$ and
$$\partial_{\nu_k}\partial_3\psi=\partial_{\nu'}\partial_3\psi=
\partial_3\partial_{\nu'}\psi=0$$
at $\supp\varphi_k\cap\partial S_2$, since $\nu'$ does not depend on $x_3$ and $\partial_{\nu'}\psi(x_3)=0$ for all $x_3\in[H_1,H_2]\backslash\{0\}$ by construction of $\psi$. Furthermore we have
$$G_1^k\psi\cdot e_3=\nu_1\partial_1\partial_3\psi+\nu_2\partial_2\partial_3\psi=
\partial_3\partial_{\nu'}\psi=0$$
at $\supp\varphi_k\cap\partial S_2$. Therefore, the second compatibility condition holds. Concerning the last compatibility condition, note that
$$D'G_3^{k'}(\psi)=-D'(\varphi_k\nabla_{x'}\psi)=-2\varphi_k\nabla^2\psi-
\nabla_{x'}\varphi_k\otimes\nabla_{x'}\psi-\nabla_{x'}\psi\otimes\nabla_{x'}\varphi_k.$$
From this identity we obtain
\begin{align*}
(D'G_3^{k'}(\psi))\nu_k'&=-2\varphi_k\nabla^2\psi\nu_k'-
\nabla_{x'}\varphi_k\partial_{\nu_k'}\psi-\nabla_{x'}\psi\partial_{\nu_k'}\varphi_k\\
&=-P_{\partial G^k}[\mu (D'\psi)\nu_k '],
\end{align*}
since $\nu_k'=\nu'$ on $\supp\varphi_k$ and therefore $\partial_{\nu_k'}\varphi_k=\partial_{\nu_k'}\psi=0$ at $\partial S_2\cap\supp\varphi_k$. It follows that all compatibility conditions at $\partial S_2^k$ are satisfied.

The validity of the compatibility conditions at $\partial\Sigma^k$, $k\in\{N_1+1,\ldots,N\}$, can be checked in a very similar way, taking into account the properties of $\psi$ and the fact that $\partial_{\nu_k'}\varphi_k=0$ at $\partial\Sigma\cap\supp\varphi_k$, $k\in\{N_1+1,\ldots,N\}$.

Therefore, for each $k\in\{0,\ldots,N\}$, there exists a unique solution $(\tilde{u}_k,\tilde{\pi}_k,\tilde{h}_k)$ of \eqref{eq:loc5}. Let $\{\theta_k\}_{k=0}^N\subset C_c^\infty(U_k)$ such that $\theta_k|_{\operatorname{supp}\varphi_k}=1$. Note that the function $(\nabla\theta_k\cdot\tilde{u}_k)|_\Omega$ is mean value free, since $\tilde{u}_k$ is a divergence free vector field and $\Jump{\tilde{u}_k}\cdot e_3=0$ on $\Sigma\cap U_k$, $\tilde{u}_k\cdot \nu_k=0$ at $(S_1\backslash\partial\Sigma)\cap U_k$ as well as $\tilde{u}_k\cdot e_3=0$ at $S_2\cap U_k$. Therefore, we may solve the problems
\begin{equation}\label{eq:loc6}
\begin{aligned}
\Delta \psi_k &= (\nabla\theta_k\cdot \tilde{u}_k)|_{\Omega}&\ \mbox{in}\quad &\Omega\backslash\Sigma,\\
\mbox{}[\![\rho\psi_k]\!] &= 0 &\ \mbox{on}\quad &\Sigma,\\
[\![\partial_{e_3}\psi_k]\!] &= 0 &\ \mbox{on}\quad &\Sigma,\\
\partial_{\nu_{\partial\Omega}}\psi_k &= 0 &\ \mbox{on}\quad &\partial\Omega\backslash\partial\Sigma=(S_1\backslash\partial\Sigma)\cup S_2,
\end{aligned}
\end{equation}
by Lemma \ref{lem:appauxhighreg}. This yields unique solutions
$$\nabla\psi_k\in\!_0H_p^1(J;H_p^1(\Omega\backslash\Sigma)^3)\cap L_p(J;H_p^3(\Omega\backslash\Sigma)^3).$$
Finally, we define
$$\tilde{S}F:=\sum_{k=0}^N(\theta_k\tilde{u}_k-\nabla\psi_k,\theta_k\tilde{\pi}_k
+\rho\partial_t\psi_k-\mu\Delta\psi_k,\theta_k\tilde{h}_k),$$
and we observe that
$$L\tilde{S}F=\sum_{k=0}^N
\begin{pmatrix}
-\mu[\Delta,\theta_k]\tilde{u}_k+[\nabla,\theta_k]\tilde{\pi}_k\\
0\\ \theta_kG_v^k(\psi)+(I-e_3\otimes e_3)G(\tilde{u}_k,\tilde{h}_k)+G_v(\psi_k)\\ \theta_kG_w^k(\psi)+G(\tilde{u}_k,\tilde{h}_k)e_3+G_w(\psi_k)\\
\theta_k G_\Sigma^k(\psi)+G_\Sigma(\psi_k)\\
\theta_k(G_h^k(u_*,\psi)-g_h^k)+m[\partial_3\psi_k]\\
\theta_k G_1^k(\psi)+P_{S_1^k}[\mu(\nabla \theta_k\otimes \tilde{u}_k+ \tilde{u}_k\otimes\nabla\theta_k)\nu_k]+G_1(\psi_k)\\0\\ \theta_kG_3^k(\psi)+G_3(\psi_k)\\ \tilde{h}_k\partial_{\nu_k}\theta_k-\theta_k g_4^k
\end{pmatrix},$$
where
$$G(\tilde{u}_k,\tilde{h}_k):=[\![-\mu(\nabla \theta_k\otimes \tilde{u}_k+ \tilde{u}_k\otimes\nabla\theta_k)]\!] e_3
-\sigma[\Delta_{x'},\theta_k]\tilde{h}_ke_3.$$
Since $\theta_k|_{\operatorname{supp}\varphi_k}=1$ it follows that $\theta_kG_j^k(\psi)=G_j^k(\psi)$, $\theta_kg_m^k=g_m^k$ and $\theta_k G_h^k(u_*,\psi)=G_h^k(u_*,\psi)$ for $j\in\{v,w,\Sigma,1,3\}$, $m\in\{h,4\}$. Therefore we have
$$\sum_{k=0}^N\theta_k G_j^k(\psi)=G_j(\psi)$$
as well as $\sum_{k=0}^N \theta_k g_m^k=g_m$ and $\sum_{k=0}^N\theta_kG_h^k(u_*,\psi)=G_h(u_*,\psi)$ since $\sum_{k=0}^N\varphi_k=1$. Setting $\hat{S}F:=\bar{S}F-\tilde{S}F$, we obtain the identity
$$L\hat{S}F=L\bar{S}F-L\tilde{S}F=F-RF$$
where
$$RF:=\sum_{k=0}^N
\begin{pmatrix}
-\mu[\Delta,\theta_k]\tilde{u}_k+[\nabla,\theta_k]\tilde{\pi}_k\\
0\\ (I-e_3\otimes e_3)G(\tilde{u}_k,\tilde{h}_k)+G_v(\psi_k)\\ G(\tilde{u}_k,\tilde{h}_k)e_3+G_w(\psi_k)\\
G_\Sigma(\psi_k)\\0
\\
P_{S_1^k}[\mu(\nabla \theta_k\otimes \tilde{u}_k+ \tilde{u}_k\otimes\nabla\theta_k)\nu_k]+G_1(\psi_k)\\0\\G_3(\psi_k)\\ \tilde{h}_k\partial_{\nu_k}\theta_k
\end{pmatrix}.$$
If we can show that there exists a constant $C>0$ being independent of $T>0$ such that the estimate
$$\|RF\|_{\mathbb{F}(J)}\le CT^\gamma\|F\|_{\mathbb{F}(J)}$$
for some $\gamma>0$ holds, then, if $T>0$ is sufficiently small, the operator $(I-R)$ is invertible and the right inverse $S$ for $L$ is given by $S:=\hat{S}(I-R)^{-1}$.

We remark that all terms which involve $\tilde{u}_k$ and $\tilde{h}_k$ are of lower order and therefore these terms carry additional (time-) regularity. Furthermore the terms involving $\psi_k$ carry additional (time-) regularity as well, since $\nabla\psi_k$ is regular enough. The only difficulty that arises is the estimate of $\sum_{k=0}^N[\nabla,\theta_k]\tilde{\pi}_k$ in $L_p(J;L_p(\Omega)^3)$. However, by Lemma \ref{lem:RegPressure} we know that $\tilde{\pi}_k\in\!_0W_p^\alpha(0,T;L_{p,loc}(\Omega^k))$ for some $\alpha\in (0,1/2-1/2p)$. Since $\theta_k$ has compact support, this yields the estimate
$$\|[\nabla,\theta_k]\tilde{\pi}_k\|_{W_p^\alpha(L_p)}\le C\left(\|\tilde{u}_k\|_{\mathbb{E}_u}+\|\tilde{h}_k\|_{\mathbb{E}_h}
+\|\nabla\psi\|_{\mathbb{E}_u}\right)$$
for some constant $C>0$ which does not depend on $T>0$. In particular this implies
\begin{multline*}
\|\sum_{k=0}^N[\nabla,\theta_k]\tilde{\pi}_k\|_{L_p(J;L_p(\Omega))}
\le C_NT^\gamma\Big(\|u_*\|_{\mathbb{E}_u(J)}
+\|\nabla\psi\|_{\mathbb{E}_u(J)}\\
+\|g_h\|_{\mathbb{F}_{6}(J)}
+\|g_4\|_{\mathbb{F}_{10}(J)}\Big)
\le C_NT^\gamma\|F\|_{\mathbb{F}(J)},
\end{multline*}
for some $\gamma>0$.
\end{proof}

We shall also prove a result on well-posedness for the linear system
\begin{align}\label{eq:NScap5}
\begin{split}
\partial_t(\rho u)-\mu\Delta u+\nabla \pi&=f,\quad \text{in}\ \Omega\backslash\Sigma,\\
\div u&=f_d,\quad \text{in}\ \Omega\backslash\Sigma,\\
-\Jump{\mu \partial_3 v}-\Jump{\mu\nabla_{x'} w}&=g_v,\quad \text{on}\ \Sigma,\\
-2\Jump{\mu \partial_3 w}+\Jump{\pi}-\sigma\Delta_{x'} h-\gamma_a\Jump{\rho} h&=g_w,\quad \text{on}\ \Sigma,\\
\Jump{u}&=u_\Sigma,\quad \text{on}\ \Sigma,\\
\partial_t h-m[w]&=g_h,\quad \text{on}\ \Sigma,\\
P_{S_1}\left(\mu(\nabla u+\nabla u^{\sf T})\nu_{S_1}\right)&=P_{S_1}g_1,\quad \text{on}\ S_1\backslash\partial\Sigma,\\
u\cdot\nu_{S_1}&=g_2,\quad \text{on}\ S_1\backslash\partial\Sigma,\\
u&=g_3,\quad \text{on}\ S_2,\\
\partial_{\nu_{\partial G}}h&=g_4,\quad \text{on}\ \partial\Sigma,\\
u(0)&=u_0,\quad \text{in}\ \Omega\backslash\Sigma\\
h(0)&=h_0,\quad \text{on}\ \Sigma.
\end{split}
\end{align}
\begin{cor}\label{cor:linmaxreg}
Let $\gamma_a>0$. Under the assumptions of Theorem \ref{thm:linmaxreg}, there exists a unique solution
$$u\in H_p^1(J;L_p(\Omega)^3)\cap L_p(J;H_p^2(\Omega\backslash\Sigma)^3),\quad \pi\in L_p(J;\dot{H}_p^1(\Omega\backslash\Sigma)),$$
$$\Jump{\pi}\in W_p^{1/2-1/2p}(J;L_p(\Sigma))\cap L_p(J;W_p^{1-1/p}(\Sigma))$$
$$h\in W_p^{2-1/2p}(J;L_p(\Sigma))\cap H_p^1(J;W_p^{2-1/p}(\Sigma))\cap L_p(J;W_p^{3-1/p}(\Sigma)),$$
of \eqref{eq:NScap5} if and only if the data are subject to the conditions (1)-(12) in Theorem \ref{thm:linmaxreg}.
\end{cor}
\begin{proof}
Necessity of the conditions follows from trace theory. To prove the sufficiency part, let
$$\mathbb{E}_1(J):=H_p^1(J;L_p(\Omega)^3)\cap L_p(J;H_p^2(\Omega\backslash\Sigma)^3),\quad \mathbb{E}_2(J):=L_p(J;\dot{H}_p^1(\Omega\backslash\Sigma)),$$
$$\mathbb{E}_3(J):=W_p^{1/2-1/2p}(J;L_p(\Sigma))\cap L_p(J;W_p^{1-1/p}(\Sigma))$$
$$\mathbb{E}_4(J):=W_p^{2-1/2p}(J;L_p(\Sigma))\cap H_p^1(J;W_p^{2-1/p}(\Sigma))\cap L_p(J;W_p^{3-1/p}(\Sigma)),$$
and $\mathbb{E}(J):=\{(u,\pi,q,h)\in \times_{j=1}^4\mathbb{E}_j(J):q=\Jump{\pi}\}$. We first solve \eqref{eq:NScap3} for the given data, to obtain a unique solution $(u_*,\pi_*,q_*,h_*)\in\mathbb{E}(J)$. Then we consider the problem
\begin{align}\label{eq:NScap10}
\begin{split}
\partial_t(\rho u)-\mu\Delta u+\nabla \pi&=0,\quad \text{in}\ \Omega\backslash\Sigma,\\
\div u&=0,\quad \text{in}\ \Omega\backslash\Sigma,\\
-\Jump{\mu \partial_3 v}-\Jump{\mu\nabla_{x'} w}&=0,\quad \text{on}\ \Sigma,\\
-2\Jump{\mu \partial_3 w}+\Jump{\pi}-\sigma\Delta_{x'} h-\gamma_a\Jump{\rho} h&=\gamma_a\Jump{\rho}h_*,\quad \text{on}\ \Sigma,\\
\Jump{u}&=0,\quad \text{on}\ \Sigma,\\
\partial_t h-m[w]&=0,\quad \text{on}\ \Sigma,\\
P_{S_1}\left(\mu(\nabla u+\nabla u^{\sf T})\nu_{S_1}\right)&=0,\quad \text{on}\ S_1\backslash\partial\Sigma,\\
u\cdot\nu_{S_1}&=0,\quad \text{on}\ S_1\backslash\partial\Sigma,\\
u&=0,\quad \text{on}\ S_2,\\
\partial_{\nu_{\partial G}}h&=0,\quad \text{on}\ \partial\Sigma,\\
u(0)&=0,\quad \text{in}\ \Omega\backslash\Sigma\\
h(0)&=0,\quad \text{on}\ \Sigma.
\end{split}
\end{align}
Define $L:\!_0\mathbb{E}(J)\to\!_0\mathbb{F}(J)$ by the left side of \eqref{eq:NScap10} and $L_0:\!_0\mathbb{E}(J)\to\!_0\mathbb{F}(J)$ by the left side of \eqref{eq:NScap3} without the initial conditions. We already know that $L_0:\!_0\mathbb{E}(J)\to\!_0\mathbb{F}(J)$ is boundedly invertible, hence
$$L=L_0+(L-L_0)=L_0(I+L_0^{-1}(L-L_0)).$$
This in turn yields that $L:\!_0\mathbb{E}(J)\to\!_0\mathbb{F}(J)$ is boundedly invertible, provided that $((I+L_0^{-1}(L-L_0)):\!_0\mathbb{E}(J)\to\!_0\mathbb{E}(J)$ has this property. To this end it suffices to show that the norm of $L_0^{-1}(L-L_0)$ in $\mathbb{E}(J)$ is less than one. For $z\in\!_0\mathbb{E}(J)$ we obtain the estimate
$$\|L_0^{-1}(L-L_0)z\|_{\mathbb{E}(J)}\le M\gamma_a\Jump{\rho}\|h\|_{\mathbb{F}_4(J)}\le T^\alpha M\gamma_a\Jump{\rho}\|h\|_{\mathbb{E}_4(J)}\le T^\alpha M\gamma_a\Jump{\rho}\|z\|_{\mathbb{E}(J)},$$
for some $\alpha>0$. Here $M:=\|L_0^{-1}\|_{\mathcal{B}(_0\mathbb{F}(J_0);_0\mathbb{E}(J_0))}$ and $J=[0,T]\subset[0,T_0]=:J_0$. It follows that if $T>0$ is sufficiently small, then $L:\!_0\mathbb{E}(J)\to\!_0\mathbb{F}(J)$ is boundedly invertible. The result extends to all $T>0$ by a successive application of this argument.
\end{proof}


%
%
%

\chapter{Nonlinear well-posedness}\label{chptr:LWP}

It is the aim of this section to establish an existence and uniqueness result for the nonlinear problem \eqref{eq:NScap2}.

\section{Function spaces and regularity}

Before we go into the details, there is a remark concerning the nonlinearity
$$H_2(u,h)=P_{S_1}\left(\mu(M_0(h)\nabla u+\nabla u^{\sf T}M_0(h)^{\sf T})\nu_{S_1}\right)$$
in \eqref{eq:NScap2} in order. One readily computes
$$(M_0(h)\nabla u+\nabla u^{\sf T}M_0(h)^{\sf T})\nu_{S_1}=\frac{1}{1+h\varphi'}
\begin{pmatrix}
\varphi\partial_3 u_1\partial_{\nu_{\partial G}}h+\varphi\partial_1 h\partial_3(u\cdot\nu_{S_1})\\
\varphi\partial_3 u_2\partial_{\nu_{\partial G}}h+\varphi\partial_2 h\partial_3(u\cdot\nu_{S_1})\\
\varphi\partial_3 u_3\partial_{\nu_{\partial G}}h+\varphi' h\partial_3 (u\cdot\nu_{S_1})
\end{pmatrix},
$$
where $\nu_{S_1}=(\nu_1,\nu_2,0)^{\sf T}$. Therefore, if $u\cdot\nu_{S_1}=0$ on $S_1\backslash\partial\Sigma$ and $\partial_{\nu_{\partial G}}h=0$ on ${\partial G}$, it follows that $H_2(u,h)=0$ at $S_1\backslash\partial\Sigma$ (note that the function $h$ depends only on $x'=(x_1,x_2)$, wherefore it is constant with respect to $x_3$).

Define the solution spaces
\begin{multline*}\mathbb{E}_u(T):=\{u\in H_p^1(J;L_p(\Omega)^3)\cap L_p(J;H_p^2(\Omega\backslash\Sigma)^3):\\
\Jump{u}=0,\ u\cdot\nu_{S_1}=0,\ P_{S_1}(\mu(\nabla u+\nabla u^{\sf T})\nu_{S_1})=0,\ u|_{S_2}=0\},
\end{multline*}
$$\mathbb{E}_\pi(T):=L_p(J;\dot{H}_p^1(\Omega\backslash\Sigma)),$$
$$\mathbb{E}_{q}(T):=W_p^{1/2-1/2p}(J;L_p(\Sigma))\cap L_p(J;W_p^{1-1/p}(\Sigma)),$$
\begin{multline*}
\mathbb{E}_h(T):=\{h\in W_p^{2-1/2p}(J;L_p(\Sigma))\cap H_p^1(J;W_p^{2-1/p}(\Sigma))\cap L_p(J;W_p^{3-1/p}(\Sigma)):\\
\partial_{\nu_{\partial G}}h=0\},
\end{multline*}
and
$$\mathbb{E}(T):=\{(u,\pi,q,h)
\in\mathbb{E}_u(T)\times\mathbb{E}_\pi(T)\times\mathbb{E}_{q}(T)\times\mathbb{E}_h(T)
:q=\Jump{\pi}\}.
$$
Moreover, we define the data spaces as follows.
$$\mathbb{F}_1(T):=L_p(J;L_p(\Omega)^3),$$
$$\mathbb{F}_2(T):=H_p^1(J;\hat{H}_p^{-1}(\Omega))\cap L_p(J;H_p^1(\Omega\backslash\Sigma)),$$
$$\mathbb{F}_3(T):=\{f_3\in W_p^{1/2-1/2p}(J;L_p(\Sigma)^3)\cap L_p(J;W_p^{1-1/p}(\Sigma)^3):P_{\Sigma}(f_3)\cdot\nu_{S_1}=0\},$$
$$\mathbb{F}_4(T):=\{f_4\in W_p^{1-1/2p}(J;L_p(\Sigma))\cap L_p(J;W_p^{2-1/p}(\Sigma)):\partial_{\nu_{\partial G}}f_4=0\},$$
and
$\mathbb{F}(T):=\times_{j=1}^4\mathbb{F}_j(T)$.

Define an operator $L=(L_1,L_2,L_3,L_4)$ on $\mathbb{E}(T)$ by
\begin{align*}
L_1(u,\pi)&:=\rho\partial_t u-\mu\Delta u+\nabla\pi\\
L_2(u)&:=\div u\\
L_3(u,q,h)&:=\Jump{-\mu(\nabla u+\nabla u^{\sf T})}e_3+q e_3-(\Delta_{x'}h)e_3-\gamma_a\Jump{\rho} h e_3\\
L_4(u,h)&:=\partial_t h-(u|e_3)
\end{align*}
and a nonlinear mapping $N=(N_1,N_2,N_3,N_4)$ on $\mathbb{E}(T)$ by
\begin{align*}
N_1(u,\pi,h)&:=F(u,\pi,h)\\
N_2(u,h)&:=F_d(u,h)-\frac{1}{|\Omega|}\int_\Omega F_d(u,h)\ dx\\
N_3(u,h)&:=(G_v(u,h),0)^{\sf T}+G_w(u,h)e_3\\
N_4(u,h)&:=H_1(u,h).
\end{align*}
It follows from Corollary \ref{cor:linmaxreg} that for each fixed $T>0$ the mapping $L:\!_0\mathbb{E}(T)\to\!_0\mathbb{F}(J)$ is an isomorphism, since all compatibility conditions at the contact line $\partial \Sigma$ are satisfied by construction.

Let $U_T:=\{z=(u,\pi,q,h)\in\mathbb{E}(T):\|h\|_{L_\infty(L_\infty)}<\eta\}$, where $\eta>0$ is sufficiently small.
Concerning the nonlinearity $N(z)$ we have the following result
\begin{prop}\label{prop:regnonlin}
Let $p>n+2$. Then
\begin{enumerate}
\item $N\in C^2(U_T;\mathbb{F}(T))$ and $N(0)=0$ as well as $DN(0)=0$.
\item $DN(w)\in\mathcal{B}(U_T;\mathbb{F}(T))$ for each $w\in \mathbb{E}(T)$.
\end{enumerate}
\end{prop}
\begin{proof}
We shall show that $N(z)\in\mathbb{F}(T)$ for each $z\in U_T$. Let $z=(u,\pi,q,h)\in U_T$. Then it is easily seen that $N_1(z)=F(u,\pi,h)\in\mathbb{F}_1(T)$. Concerning $N_2(z)$, we have
$$\|N_2(z)\|_{L_p(H_p^1)}\le C(\|h\|_{L_\infty(W_\infty^2)}\|u\|_{L_p(H_p^1)}
+\|h\|_{L_\infty(W_\infty^1)}\|u\|_{L_p(H_p^2)}),$$
since $\mathbb{E}_h(T)\hookrightarrow BUC([0,T];C^2(\Sigma))$ for $p>n+2$. Furthermore, for $\phi\in \dot{H}_p^1(\Omega)$ we obtain after integration by parts ($h$ does not depend on $x_3$)
\begin{multline*}
(N_2(z)|\phi)_2=(N_2(z)|\phi-\bar{\phi})_2=-\int_\Omega\Big[(u_1\partial_1 h+u_2\partial_2 h)\partial_3\Big((\phi-\bar{\phi})\frac{\varphi}{1+h\varphi'}\Big)+\\
+u_3 h\partial_3\Big((\phi-\bar{\phi})\frac{\varphi'}{1+h\varphi'}\Big)\Big]dx,
\end{multline*}
where $\bar{\phi}:=\frac{1}{|\Omega|}\int_\Omega\phi dx$. Since $\mathbb{E}_h(T)\hookrightarrow BUC^1([0,T];C^1(\Sigma))$ for $p>n+2$, it follows from Poincar\'{e}'s inequality for functions with mean value zero that $N_2(z)\in\mathbb{F}_2(T)$.

The desired regularity property of $N_3(z)$ can be readily checked. It remains to show that $$P_\Sigma N_3(z)\cdot\nu_{S_1}=(G_v(u,h),0)^{\sf T}\cdot\nu_{S_1}=0.$$
Inserting the expression for $G_v(u,h)$ yields
\begin{multline*}
P_\Sigma N_3(z)\cdot\nu_{S_1}=-\left(\Jump{\mu(\nabla_{x'} v+\nabla_{x'} v^{\sf T})}\nabla_{x'} h|\nu_{\partial G}\right)\\
+|\nabla_{x'} h|^2\Jump{\mu\partial_3 (u|\nu_{S_1})}+\left((1+|\nabla_{x'} h|^2)\Jump{\mu\partial_3 w}-(\nabla_{x'} h|\Jump{\mu\nabla w})\right)\partial_{\nu_{\partial G}}h,
\end{multline*}
where $\nu_{S_1}=(\nu_{\partial G},0)^{\sf T}$. The last term in this equation vanishes, since $\partial_{\nu_{\partial G}}h=0$. Moreover, since $\mu(u\cdot\nu_{S_1})(x_3)=0$ for each $x_3\in (H_1,0)\cup(0,H_2)$, the second term vanishes as well. Finally, since $P_{S_1}(\mu(\nabla u+\nabla u^{\sf T})\nu_{S_1})=0$, it holds that
$$\mu(\nabla u+\nabla u^{\sf T})\nu_{S_1}=\left(\mu(\nabla u+\nabla u^{\sf T})\nu_{S_1}|\nu_{S_1}\right)\nu_{S_1}$$
on $S_1\backslash\partial\Sigma$, hence also
$$\Jump{\mu(\nabla u+\nabla u^{\sf T})}\nu_{S_1}=\left(\Jump{\mu(\nabla u+\nabla u^{\sf T}})\nu_{S_1}|\nu_{S_1}\right)\nu_{S_1}$$
at the contact line, since $\nu_{S_1}$ does not depend on $x_3$. Taking the inner product with $(\nabla_{x'} h,0)^{\sf T}$ yields
$$(\Jump{\mu(\nabla u+\nabla u^{\sf T})}\nu_{S_1}|(\nabla_{x'} h,0)^{\sf T})=\left(\mu(\nabla u+\nabla u^{\sf T})\nu_{S_1}|\nu_{S_1}\right)\partial_{\nu_{\partial G}}h=0,$$
since $\partial_{\nu_{\partial G}}h=0$. But by symmetry of the stress tensor we also have
$$(\Jump{\mu(\nabla u+\nabla u^{\sf T})}\nu_{S_1}|(\nabla h,0)^{\sf T})=(\nu_{\partial G}|\Jump{\mu(\nabla_{x'} v+\nabla_{x'} v^{\sf T})}\nabla h),$$
where $u=(v,w)$, hence $N_3(z)\in\mathbb{F}_3(T)$.

Finally, concerning $N_4(z)$, one has to observe that $(u|\nu_{S_1})=0$ and $P_{S_1}((\nabla u+\nabla u^{\sf T})\nu_{S_1})=0$ on $S_1\backslash\partial\Sigma$ if $u\in\mathbb{E}_u(T)$. For $\nu_{S_1}=(\nu_{\partial G},0)^{\sf T}$, this implies in particular that $(v|\nu_{\partial G})=0$ and $P_{\partial G}((\nabla_{x'} v+\nabla_{x'} v^{\sf T})\nu_{\partial G})=0$ on $S_1\backslash\partial\Sigma$. Since $\Jump{v}=0$ on $\overline{\Sigma}$, by continuity of $v$, we clearly have $\Jump{\nabla_{x'} v}=0$ on $\Sigma$, since the jump acts into the direction of $x_3$ which is perpendicular to both $e_1$ and $e_2$. In particular we have $(v|\nu_{\partial G})=0$ and $P_{\partial G}((\nabla_{x'} v+\nabla_{x'} v^{\sf T})\nu_{\partial G})=0$ at the contact line $\partial\Sigma$. Since in addition we know that $\partial_{\nu_{\partial G}}h=0$ at $\partial\Sigma$, it follows from Proposition \ref{prop:appaux} that $\partial_{\nu_{\partial G}} (v|\nabla_{x'} h)=0$ at $\partial\Sigma$.

The remaining assertions can be proven as in \cite[Proposition 6.2]{PrSi09a}.
\end{proof}

\section{Reduction to time trace zero}\label{sec:Redzerotrace}

Let $(u_0,h_0)\in W_p^{2-2/p}(\Omega\backslash\Sigma)^3\times W_p^{3-2/p}(\Sigma)$ such that
$$\div u_0=F_d(u_0,h_0),\quad -\Jump{\mu\partial_3 v_0}-\Jump{\mu\nabla_{x'} w_0}=G_v(v_0,h_0),$$
$\Jump{u_0}=0$ on $\Sigma$, $u_0\cdot\nu_{S_1}=0$, $P_{S_1}(\mu(\nabla u_0+\nabla u_0^{\sf T})\nu_{S_1})=0$ on $S_1\backslash\partial\Sigma$, $u_0|_{S_2}=0$ and $\partial_{\nu_{\partial G}}h_0=0$ on $\partial\Sigma$.

Let $H:=\max\{H_1,-H_2\}<0$ and $u_0^+:=u_0|_{x_3\in[0,H_2]}$. Define
$$\tilde{u}_0^+(x):=
\begin{cases}
u_0^+(x_1,x_2,x_3),\quad&\text{if}\ x_3\in[0,H_2),\\
-u_0^+(x_1,x_2,-2x_3)+2u_0^+(x_1,x_2,-x_3/2),\quad&\text{if}\ x_3\in (H/2,0)
\end{cases}
$$
as well as
$$\bar{u}_0^+(x):=
\begin{cases}
\tilde{u}_0^+(x_1,x_2,x_3),\quad&\text{if}\ x_3\in[0,H_2),\\
\tilde{u}_0^+(x_1,x_2,x_3)\psi(x_3),\quad&\text{if}\ x_3\in (H/2,0),\\
0,\quad&\text{if}\ x_3\in (H_1,H/2],
\end{cases}
$$
where $\psi\in C_c^\infty(\mathbb{R};[0,1])$ such that $\psi(s)=1$ if $|s|<-H/6$ and $\psi(s)=0$ if $|s|>-H/3$. It follows by construction that $\bar{u}_0^+\in W_p^{2-2/p}(\Omega)^3\hookrightarrow C^1(\overline{\Omega})^3$, if $p>n+2$. We then solve the parabolic problem
\begin{align}
\begin{split}\label{auxprblnonlin1}
\partial_t(u^+)-\mu^+\Delta u^+&=0,\quad \text{in}\ \Omega,\\
P_{S_1}\left(\mu^+(\nabla u^+ +\nabla (u^+)^{\sf T})\nu_{S_1}\right)&=0,\quad \text{on}\ S_1,\\
u^+\cdot\nu_{S_1}&=0,\quad \text{on}\ S_1,\\
u^+&=0,\quad \text{on}\ S_2,\\
u^+(0)&=\bar{u}_0^+,\quad \text{in}\ \Omega,
\end{split}
\end{align}
by Lemma \ref{lem:appaux1}, where $\mu^+:=\mu|_{x_3\in (0,H_2)}>0$ is constant.

Let us check whether $\bar{u}_0^+$ satisfies the relevant compatibility conditions at $S_1$ and $S_2$. It is easy to see that $\bar{u}_0^+=0$ at $S_2$. Furthermore we have ${u}_0^+\cdot\nu_{S_1}=0$ for all $x_3\in (0,H_2)$ by the assumption on $u_0$. From the definition of $\tilde{u}_0^+$ we obtain that $\tilde{u}_0^+\cdot \nu_{S_1}=0$ for all $x_3\in (H/2,0)$, hence also $\bar{u}_0^+\cdot\nu_{S_1}=0$ for $x_3\in (H_1,0)$ by the definition of $\bar{u}_0^+$. Since $\bar{u}_0^+\in C^1(\overline{\Omega})^3$ we also have $\bar{u}_0^+\cdot\nu_{S_1}=0$ for $x_3=0$. It remains to prove that
\begin{equation}\label{eq:LWP1}
P_{S_1}\left(\mu^+(\nabla \bar{u}_0^+ +\nabla (\bar{u}_0^+)^{\sf T})\nu_{S_1}\right)=0
\end{equation}
on $S_1$. Again, this is true for $x_3\in (0,H_2)$, by the assumption on $u_0$. Since the first two components of this tangential projection do only contain derivatives with respect to the $(x_1,x_2)$-variables, it follows from the definition of $\bar{u}_0^+$ that
$$P_{S_1}\left(\mu^+(\nabla \bar{u}_0^+ +\nabla (\bar{u}_0^+)^{\sf T})\nu_{S_1}\right)\cdot e_j=0$$
for $j\in\{1,2\}$ and $x_3\in (H_1,0)$. The third component of the projection is given by
$$\partial_{\nu_{S_1}}(\bar{u}_0^+\cdot e_3)+\partial_3(\bar{u}_0^+\cdot\nu_{S_1}).$$
Evidently, it holds that $\partial_{\nu_{S_1}}(\bar{u}_0^+\cdot e_3)=0$ by the same reasons as above, since the last component of $\nu_{S_1}$ vanishes. Furthermore, we have
$$\partial_3(\bar{u}_0^+\cdot\nu_{S_1})=
\begin{cases}
\psi\partial_3(\tilde{u}_0^+\cdot\nu_{S_1})+\psi'(\tilde{u}_0^+\cdot\nu_{S_1}),
\quad&\text{if}\ x_3\in (H/2,0),\\
0,\quad&\text{if}\ x_3\in (H_1,H/2].
\end{cases}
$$
Since $u_0^+\cdot \nu_{S_1}=0$ for all $x_3\in (0,H_2)$ it follows that $\partial_3 (u_0^+\cdot \nu_{S_1})=0$ for $x_3\in (0,H_2)$. From the identity
$$\partial_3(\tilde{u}_0^+\cdot \nu_{S_1})=-\partial_3[u_0^+(x_1,x_2,-2x_3)\cdot\nu_{S_1}]+2\partial_3 [u_0^+(x_1,x_2,-x_3/2)\cdot\nu_{S_1}]$$
for $x_3\in (H/2,0)$, we readily obtain that $\partial_3(\bar{u}_0^+\cdot\nu_{S_1})=0$ for $x_3\in (H_1,0)$. Finally, since $\bar{u}_0^+\in C^1(\overline{\Omega})^3$, it follows that \eqref{eq:LWP1} holds on all of $S_1$.

Solving \eqref{auxprblnonlin1} by Lemma \ref{lem:appaux1} yields a unique solution
$$u^+\in H_p^1(J;L_p(\Omega)^3)\cap L_p(J;H_p^2(\Omega)^3)$$
satisfying the estimate
$$\|u^+\|_{H_p^1(L_p)\cap L_p(H_p^2)}\le M\|\bar{u}_0^+\|_{W_p^{2-2/p}},$$
where $M>0$ does not depend on $u_0^+$.

Applying the same procedure to $u_0^-:=u_0|_{x_3\in [H_1,0]}$ (with a suitable cut-off function $\psi$) yields a $C^1$-extension $\bar{u}_0^-$ of $u_0^-$. Therefore, we obtain a unique solution
$$u^-\in H_p^1(J;L_p(\Omega)^3)\cap L_p(J;H_p^2(\Omega)^3)$$
of \eqref{auxprblnonlin1} with $\mu^+$ and $\bar{u}_0^+$ replaced by $\mu^-$ and $\bar{u}_0^-$, respectively, satisfying the estimate
$$\|u^-\|_{H_p^1(L_p)\cap L_p(H_p^2)}\le M\|\bar{u}_0^-\|_{W_p^{2-2/p}},$$
where $M>0$ does not depend on $u_0^-$. We then define
$$\bar{u}:=
\begin{cases}
u^+,\quad&\text{if}\ x_3\in(0,H_2),\\
u^-,\quad&\text{if}\ x_3\in (H_1,0).
\end{cases}
$$
Note that in general $\bar{u}\in H_p^1(J;L_p(\Omega)^3)\cap L_p(J;H_p^2(\Omega\backslash\Sigma)^3)$, since $\Jump{\bar{u}}$ is not necessarily zero.

In a next step we solve the two phase problem
\begin{align}
\begin{split}\label{auxprblnonlin2}
\partial_t(\rho \tilde{u})-\mu\Delta \tilde{u}&=0,\quad \text{in}\ \Omega\backslash\Sigma,\\
\Jump{\mu \partial_3 \tilde{v}}+\Jump{\mu\nabla_{x'} \tilde{w}}&=\Jump{\mu \partial_3 \bar{v}}+\Jump{\mu\nabla_{x'} \bar{w}},\quad \text{on}\ \Sigma,\\
\Jump{\mu \partial_3 \tilde{w}}&=\Jump{\mu \partial_3 \bar{w}},\quad \text{on}\ \Sigma,\\
\Jump{\tilde{u}}&=0,\quad \text{on}\ \Sigma,\\
P_{S_1}\left(\mu(\nabla \tilde{u}+\nabla \tilde{u}^{\sf T})\nu_{S_1}\right)&=0,\quad \text{on}\ S_1\backslash\partial\Sigma,\\
\tilde{u}\cdot\nu_{S_1}&=0,\quad \text{on}\ S_1\backslash\partial\Sigma,\\
\tilde{u}&=0,\quad \text{on}\ S_2,\\
\tilde{u}(0)&=u_0,\quad \text{in}\ \Omega\backslash\Sigma,
\end{split}
\end{align}
by Lemma \ref{lem:appaux2}, where $\tilde{u}=(\tilde{v},\tilde{w})$ and $\bar{u}=(\bar{v},\bar{w})$. The compatibility conditions at $t=0$ are satisfied, since $\bar{u}(0)=u_0$. Let us check that the compatibility condition
$$\Jump{\mu \partial_3 (\bar{u}|\nu_{S_1})}+\Jump{\mu\partial_{\nu_{S_1}}\bar{w}}=0$$ holds at the contact line $\partial\Sigma$. Since by construction of $\bar{u}$ we have
$$P_{S_1}\left(\mu(\nabla \bar{u} +\nabla \bar{u}^{\sf T})\nu_{S_1}\right)=0,$$
at $S_1\backslash\partial\Sigma$, the third component yields $\mu\left(\partial_{\nu_{S_1}}\bar{w}+\partial_3(\bar{u}\cdot\nu_{S_1})\right)=0$ at $S_1\backslash\partial\Sigma$. This in turn implies that $\Jump{\mu \partial_3 (\bar{u}|\nu_{S_1})}+\Jump{\mu\partial_{\nu_{S_1}}\bar{w}}=0$. Note that for the third equation in \eqref{auxprblnonlin2} there has no compatibility condition at $\partial\Sigma$ to be satisfied. Therefore we obtain a unique solution $\tilde{u}\in\mathbb{E}_u(T)$ by Lemma \ref{lem:appaux2}.

Define $f_d^*:=\div \tilde{u}\in\mathbb{F}_2(T)$, $g^*:=\Jump{-\mu(\nabla \tilde{u}+\nabla \tilde{u}^{\sf T})e_3}\in\mathbb{F}_3(T)$ and $g_h^*:=e^{-At}(v_0|_{\Sigma}\cdot\nabla h_0)$, with $A:=(I-\Delta_N)$, where $\Delta_N$ is the Neumann-Laplacian and $e^{-At}$ denotes the $C_0$-semigroup, generated by $-A$ in $L_p(\Sigma)$. Then, since  $(v_0|_{\Sigma}\cdot\nabla h_0)\in W_p^{2-3/p}(\Sigma)$ with $\partial_{\nu_{\partial G}} (v_0|_{\Sigma}\cdot\nabla h_0)=0$ by Proposition \ref{prop:appaux} at $\partial\Sigma$, it follows that $e^{-At}g_h\in \mathbb{F}_4(T)$. The fact that $P_\Sigma(\Jump{-\mu(\nabla \tilde{u}+\nabla \tilde{u}^{\sf T})e_3})\cdot\nu_{S_1}=0$ holds by construction of $\tilde{u}$.

By Corollary \ref{cor:linmaxreg} there exists a unique solution $z_*=(u_*,\pi_*,q_*,h_*)\in\mathbb{E}(T)$ of the initial value problem $Lz_*=(0,f_d^*,g^*,g_h^*)$, $(u_*,h_*)|_{t=0}=(u_0,h_0)$, since the compatibility conditions at $t=0$ in the second and third component are satisfied by construction. We remark that $z_*$ satisfies the estimate
$$\|z_*\|_{\mathbb{E}(T)}\le C_0\|(u_0,h_0)\|_{X_\gamma},$$
and $C_0>0$ does not depend on $(u_0,h_0)$.

\section{Nonlinear well-posedness}

Define the mapping $K(z):=N(z+z_*)-Lz_*$, where $z\in\!_0\mathbb{E}(T)$. By Proposition \ref{prop:regnonlin} it holds that $K(z)\in\!_0\mathbb{F}(T)$ for each $z\in\!_0\mathbb{E}(T)$, wherefore, we may consider the mapping $\mathcal{K}(z):=L^{-1}K(z)$. We intend to show that this mapping has a fixed point in $_0\mathbb{E}(T)$.

The main result of this section reads as follows.
\begin{thm}\label{nonlinwellposedness}
Let $n=3$, $p>5$. For each given $T>0$ there exists a number $\eta=\eta(T)>0$ such that for all initial values $(u_0,h_0)\in W_p^{2-2/p}(\Omega\backslash\Sigma)^3\times W_p^{3-2/p}(\Sigma)$ satisfying the compatibility conditions
$$\div u_0=F_d(u_0,h_0),\quad -\Jump{\mu\partial_3 v_0}-\Jump{\mu\nabla_{x'} w_0}=G_v(v_0,h_0),$$
$\Jump{u_0}=0$, $u_0\cdot\nu_{S_1}=0$, $P_{S_1}(\mu(\nabla u_0+\nabla u_0^{\sf T})\nu_{S_1})=0$, $u_0|_{S_2}=0$ and $\partial_{\nu_{\partial G}} h_0=0$ as well as the smallness condition
$$\|u_0\|_{W_p^{2-2/p}(\Omega\backslash\Sigma)}+\|h_0\|_{W_p^{3-2/p}(\Sigma)}\le\eta,$$
there exists a unique solution $(u,\pi,q,h)\in\mathbb{E}(T)$ of \eqref{eq:NScap2}.
\end{thm}
\begin{proof}
For a given Banach space $Z$, let
$$\mathbb{B}_Z:=\{z\in Z:\|z\|_Z\le 1\}.$$
Based on Proposition \ref{prop:regnonlin}, for each $\varepsilon\in (0,1)$ there exists $\delta(\varepsilon)>0$ such that
$$\|DN(z+z_*)\|_{\mathcal{B}(\mathbb{E}(T),\mathbb{F}(T))}\le\varepsilon$$
whenever $(z+z_*)\in \delta\mathbb{B}_{\mathbb{E}(T)}\subset U_T$. Let $M:=\|L^{-1}\|_{\mathcal{B}(_0\mathbb{F}(T);_0\mathbb{E}(T))}>0$ and $C:=\|L\|_{\mathcal{B}(\mathbb{E}(T);\mathbb{F}(T))}>0$.
We assume that $\varepsilon>0$ from above is chosen sufficiently small, such that $\varepsilon\in (0,1/(2M))$. Suppose furthermore that $z\in\frac{\delta}{2}\mathbb{B}_{_0\mathbb{E}(T)}$ and $(u_0,h_0)\in \frac{\delta}{4MC_0(1+C)}\mathbb{B}_{X_\gamma}$. This yields
$$\|z+z_*\|_{\mathbb{E}(T)}\le \delta/2+\delta/(4M(1+C))<\delta$$
and therefore
\begin{align*}
\|\mathcal{K}(z)\|_{\mathbb{E}(T)}&\le M\|K(z)\|_{\mathbb{F}(T)}\le M(\|N(z+z_*)\|_{\mathbb{F}(T)}+\|Lz_*\|_{\mathbb{F}(T)})\\
&\le M[\varepsilon(\|z\|_{\mathbb{E}(T)}+\|z_*\|_{\mathbb{E}(T)})+C\|z_*\|_{\mathbb{E}(T)}]\\
&\le M(\varepsilon \|z\|_{\mathbb{E}(T)}+C_0(1+C)\|(u_0,h_0)\|_{X_\gamma})\\
&\le M\varepsilon \frac{\delta}{2}+\frac{\delta}{4}\le\delta/2
\end{align*}
hence $\mathcal{K}:\frac{\delta}{2}\mathbb{B}_{_0\mathbb{E}(T)}\to \frac{\delta}{2}\mathbb{B}_{_0\mathbb{E}(T)}$ is a self-mapping.
Furthermore we obtain
$$\|\mathcal{K}(z_1)-\mathcal{K}(z_2)\|_{\mathbb{E}(T)}\le M\varepsilon\|z_1-z_2\|_{\mathbb{E}(T)}\le \frac{1}{2}\|z_1-z_2\|_{\mathbb{E}(T)},$$
valid for all $z_1,z_2\in\frac{\delta}{2}\mathbb{B}_{_0\mathbb{E}(T)}$ and all initial values $(u_0,h_0)\in \frac{\delta}{4MC_0(1+C)}\mathbb{B}_{X_\gamma}$. The contraction mapping principle yields a unique fixed point $\tilde{z}\in\frac{\delta}{2}\mathbb{B}_{_0\mathbb{E}(T)}$ of $\mathcal{K}(z)$, i.e. $\tilde{z}=\mathcal{K}(\tilde{z})$. Equivalently this means $L\tilde{z}=N(\tilde{z}+z_*)-Lz_*$, hence $\bar{z}:=\tilde{z}+z_*$ solves $L\bar{z}=N(\bar{z})$. To show that $\bar{z}=(\bar{u},\bar{\pi},\bar{q},\bar{h})$ is a solution of \eqref{eq:NScap2}, it remains to prove that $F_d(\bar{u},\bar{h})$ is mean value free. Indeed, let $t\in[0,T]$ be fixed and set $\hat{u}(t,x):=\bar{u}(t,\Theta^{-1}_{\bar{h}}(t,x))$ it follows that $\hat{u}\in H_p^1(\Omega)$ with $(\hat{u}|\nu_{S_1})=0$ at $S_1\backslash\partial\Gamma(t)$, $\hat{u}=0$ at $S_2$ and
$$\div\hat{u}=(\div\bar{u}-F_d(\bar{u},\bar{h}))\circ\Theta_{\bar{h}}^{-1}.$$
The divergence theorem and the transformation formula yield
\begin{align*}
0&=\int_{\Omega\backslash\Gamma(t)}\div\hat{u}\ dx\\
&=\int_{\Omega\backslash\Sigma}
\left(\div\bar{u}-F_d(\bar{u},\bar{h})\right)\det\Theta_{\bar{h}}'\ d\bar{x}\\
&=-\frac{1}{|\Omega|}\int_{\Omega\backslash\Sigma}F_d(\bar{u},\bar{h})\ d\bar{x}
\int_{\Omega\backslash\Sigma}\det\Theta_{\bar{h}}'\ d\bar{x},
\end{align*}
where $\bar{x}:=\Theta_{\bar{h}}^{-1}(x)$. Since $\det\Theta_{\bar{h}}'>0$, the claim follows.
\end{proof}


%
%
%

\chapter{Rayleigh-Taylor instability}\label{QualBeh}

\section{Equilibria and spectrum of the linearization}\label{EquilLin}

In this subsection we compute the equilibria of \eqref{eq:NScap1mod} as well as the spectrum of the linearization of \eqref{eq:NScap1mod} around the trivial equilibrium.

Assume that we have a time independent solution of \eqref{eq:NScap1mod}. Then multiplying $\eqref{eq:NScap1mod}_1$ by $u$ and integrating by parts yields the identity
$$\|\mu^{1/2} Du\|_{L_2(\Omega)}^2=0,$$
hence $u=0$ on $\partial\Omega$ and therefore $u=0$ in all of $\Omega$, by Korn's inequality. If $u=0$, then $\pi$ must be constant, with possibly different values in different phases. Hence, $\eqref{eq:NScap1mod}_3$ yields that
$$\sigma H_\Gamma+\Jump{\rho}\gamma_a x_3=const,$$
on $\Gamma$. In particular, if $H_\Gamma=0$ then $x_3$ must be constant, hence flat interfaces belong to the set of equilibria. Assume that $\Gamma$ is given by the graph of a height function $h$, that is
$$\Gamma=\{x\in\Omega:x_3=h(x_1,x_2),\ (x_1,x_2)\in G\}.$$
Then the normal $\nu_\Gamma$ on $\Gamma$, pointing from $\Omega_1$ $(x_3<h(x_1,x_2))$ into $\Omega_2$ $(x_3>h(x_1,x_2))$ is given by
$$\nu_\Gamma(x',h(x'))=\frac{1}{\sqrt{1+|\nabla_{x'} h(x')|^2}}[-\nabla_{x'} h(x'),1]^{\sf T},\quad x'=(x_1,x_2)\in B_R(0).$$
Since $H_\Gamma=-\div_\Gamma\nu_\Gamma$, we obtain the quasilinear elliptic problem
\begin{align}\label{eq:stab1}
\begin{split}
\sigma \div_{x'}\left(\frac{\nabla_{x'} {h}}{\sqrt{1+|\nabla_{x'} {h}|^2}}\right)+\Jump{\rho}\gamma_a {h}&=c,\quad x'\in G,\\
\partial_{\nu_{\partial G}}{h}&=0,\quad x'\in\partial G,
\end{split}
\end{align}
where $c:=\frac{1}{|G|}\int_{G}h dx'$. All admissible height functions which solve \eqref{eq:stab1} belong to the set of equilibria,

We are interested in the stability properties of the flat interface $\Sigma=G\times\{0\}$ in $\Omega:=G\times(H_1,H_2)$. After transformation of \eqref{eq:NScap1mod} to the fixed domain $\Omega\backslash\Sigma$ and linearization around the equilibrium $(0,\Sigma)$, we obtain the linear problem
\begin{align}\label{eq:stab10}
\begin{split}
\partial_t(\rho u)-\mu\Delta u+\nabla p&=0,\quad \text{in}\ \Omega\backslash\Sigma,\\
\div u&=0,\quad \text{in}\ \Omega\backslash\Sigma,\\
-\Jump{\mu(\nabla u+\nabla u^{\sf T})}e_3+\Jump{p}e_3&=\sigma(\Delta_{x'} h) e_3+\Jump{\rho}\gamma_a h e_3,\quad \text{on}\ \Sigma,\\
\Jump{u}&=0,\quad \text{on}\ \Sigma,\\
\partial_t h-u_3&=0,\quad \text{on}\ \Sigma,\\
P_{S_1}\left(\mu(\nabla u+\nabla u^{\sf T})\nu_{S_1}\right)&=0,\quad \text{on}\ S_1\backslash\partial\Sigma,\\
(u|\nu_{S_1})&=0,\quad \text{on}\ S_1\backslash\partial\Sigma,\\
u&=0,\quad \text{on}\ S_2,\\
\partial_{\nu_{\partial G}}h&=0,\quad \text{on}\ \partial\Sigma,\\
u(0)&=u_0,\quad \text{in}\ \Omega\backslash\Sigma,\\
h(0)&=h_0,\quad \text{on}\ \Sigma.
\end{split}
\end{align}
Define a linear operator $L:X_1\to X_0$ by
$$L(u,h):=[(\mu/\rho)\Delta u-(1/\rho)\nabla p,u\cdot e_3],$$
where $X_0:=L_{p,\sigma}(\Omega)\times \{h\in W_p^{2-1/p}(\Sigma):\int_G h\ dx'=0,\ \partial_{\nu_{\partial G}}h=0\}$,
$$L_{p,\sigma}(\Omega):=\overline{\{u\in C_c^\infty(\Omega)^3:\div u=0\}}^{\|\cdot\|_{L_p}},\ \bar{X}_1=H_p^2(\Omega\backslash \Sigma)^3\times W_p^{3-1/p}(\Sigma)$$
and
\begin{multline}\label{X1}
X_1:=D(L)=\{(u,h)\in X_0\cap \bar{X}_1:P_{\Sigma}(\Jump{\mu(\nabla u+\nabla u^{\sf T})}e_3)=0,\ \Jump{u}=0,\\
P_{S_1}\left(\mu(\nabla u+\nabla u^{\sf T})\nu_{S_1}\right)=0,\ (u|\nu_{S_1})=0,\ \partial_{\nu_{\partial G}}h=0\}.
\end{multline}
The function $p\in \dot{H}_p^1(\Omega\backslash \Sigma)$ in the definition of $L$ is determined as the solution of the weak transmission problem
\begin{align*}
\left(\frac{1}{\rho}\nabla p|\nabla\phi\right)_{L_2(\Omega)}&=\left(\frac{\mu}{\rho}\Delta u|\nabla\phi\right)_{L_2(\Omega)}\\
\Jump{p}&=\sigma\Delta_{x'} h+\Jump{\rho}\gamma_a h+(\Jump{\mu(\nabla u+\nabla u^{\sf T})}e_3|e_3),\quad \text{on}\ \Sigma,
\end{align*}
where $\phi\in H_{p'}^1(\Omega)$ and $p'=p/(p-1)$, which is well-defined thanks to Lemma \ref{lem:appauxlemweak}. We will sometimes make use of the notation via solution operators, i.e.
\begin{equation}\label{formulaPi}
\frac{1}{\rho}\nabla p=T_1[(\mu/\rho)\Delta u]+T_2[\sigma\Delta_{x'} h+\Jump{\rho}\gamma_a h+(\Jump{\mu(\nabla u+\nabla u^{\sf T})}e_3|e_3)],
\end{equation}
where $T_1:L_p(\Omega)^3\to L_p(\Omega)^3$ and $T_2:W_p^{1-1/p}(\Sigma)\to L_p(\Omega)^3$ are bounded linear operators.

In what follows we will analyze the spectrum of the operator $L$. Note that $L$ has a compact resolvent. This implies that the spectrum of $L$ is discrete and it consists solely of eigenvalues with finite multiplicity. Consider the eigenvalue problem $\lambda (u,h)=L(u,h)$, that is
\begin{align}\label{eq:stab11}
\begin{split}
\lambda\rho u-\mu\Delta u+\nabla p&=0,\quad \text{in}\ \Omega\backslash\Sigma,\\
\div u&=0,\quad \text{in}\ \Omega\backslash\Sigma,\\
-\Jump{\mu(\nabla u+\nabla u^{\sf T})}e_3+\Jump{p}e_3&=\sigma(\Delta_{x'} h) e_3+\Jump{\rho}\gamma_a h e_3,\quad \text{on}\ \Sigma,\\
\Jump{u}&=0,\quad \text{on}\ \Sigma,\\
\lambda h-u_3&=0,\quad \text{on}\ \Sigma,\\
P_{S_1}\left(\mu(\nabla u+\nabla u^{\sf T})\nu_{\partial\Omega}\right)&=0,\quad \text{on}\ S_1\backslash\partial\Sigma,\\
(u|\nu_{S_1})&=0,\quad \text{on}\ S_1\backslash\partial\Sigma,\\
u&=0,\quad \text{on}\ S_2,\\
\partial_{\nu_{\partial G}}h&=0,\quad \text{on}\ \partial\Sigma,\\
\end{split}
\end{align}
We test the first equation with $u$ and integrate by parts to obtain
\begin{equation}\label{L2spectrum}
\lambda|\rho^{1/2}u|_{L_2(\Omega)}^2+\frac{1}{2}|\mu^{1/2} Du|_{L_2(\Omega)}^2
+\bar{\lambda}\left[\sigma|\nabla_{x'} {h}|_{L_2(G)}^2-\Jump{\rho}\gamma_a|{h}|_{L_2(G)}^2\right]=0,
\end{equation}
The above identity for $\lambda=0$ implies $u=0$, by Korn's inequality, hence $p$ as well as $\Jump{p}$ are constant. Therefore $h$ is a solution of the linear elliptic problem
\begin{align}\label{eq:stab12}
\begin{split}
\Delta_{x'} h+\frac{\Jump{\rho}\gamma_a}{\sigma} h&=0,\quad x'\in G,\\
\partial_{\nu_{\partial G}}h&=0,\quad x'\in\partial G,
\end{split}
\end{align}
since $h$ is mean value free. Let $\sigma(-\Delta_N)\subset (0,\infty)$ denote the spectrum of the negative Neumann-Laplacian in the space
$$X:=\left\{h\in W_p^{1-1/p}(G):\int_G h\ dx'=0\right\}$$
and let $E(\eta)$ denote the eigenspace corresponding to the eigenvalue $\eta\in\sigma(-\Delta_N)$. It follows that $h=0$ is the unique solution of \eqref{eq:stab12} if and only if
$$\frac{\Jump{\rho}\gamma_a}{\sigma} \notin\sigma(-\Delta_N)$$
and there exists $0\neq h\in E(\eta)$ if and only if
$$\eta:=\frac{\Jump{\rho}\gamma_a}{\sigma} \in\sigma(-\Delta_N).$$
This shows that
$$0\in\sigma(L)\quad \text{if and only if}\quad \frac{\Jump{\rho}\gamma_a}{\sigma} \in\sigma(-\Delta_N).$$
Suppose that $0\neq \lambda\in\sigma(L)$ with $\Re\lambda=0$. Taking real parts in \eqref{L2spectrum} it follows that $u=0$ by Korn's inequality, hence $h$ must be nontrivial. By equation $\eqref{eq:stab11}_5$ it follows that $\lambda=0$. This shows that $\lambda=0$ is the only eigenvalue of $L$ on the imaginary axis.

In particular, if
$$\frac{\Jump{\rho}\gamma_a}{\sigma}<\lambda_1,$$
$\lambda_1>0$ being the first nontrivial eigenvalue of $-\Delta_N$ in $X$, then $$\sigma(L)\subset\{\lambda\in \mathbb{C}:\Re\lambda\le-\omega<0\},$$
for some $\omega>0$, since
$$|\nabla_{x'} h|_{L_2(G)}^2-
\frac{\Jump{\rho}\gamma_a}{\sigma}|h|_{L_2(G)}^2\ge 0,$$
by the Poincar\'{e} inequality for functions $h$ with mean value zero. Note that there exists $\kappa>0$ such that $\kappa-L$ is a sectorial operator, since $L$ has maximal $L_p$-regularity. In particular, it holds that $\sigma(L-\kappa)\subset\Sigma_{\pi/2+\delta}$ or equivalently $\sigma(L)\subset\Sigma_{\pi/2+\delta}+\kappa$ for some $\delta\in (0,\pi/2)$. This concludes the proof of existence of the number $\omega>0$ above.

We aim to show that $\sigma(L)\cap\mathbb{C}_+\neq \emptyset$ whenever $\frac{\Jump{\rho}\gamma_a}{\sigma}>\lambda_1$. To this end, for $\lambda\ge0$ and given $g\in W_p^{1-1/p}(G)$, $p>2$, we solve the elliptic two-phase Stokes problem
\begin{align}\label{N2D}
\begin{split}
\lambda\rho u-\mu\Delta u +\nabla p&=0,\quad \text{in}\ \Omega\backslash\Sigma,\\
\div u&=0,\quad \text{in}\ \Omega\backslash\Sigma,\\
-\Jump{\mu(\nabla u+\nabla u^{\sf T})}e_3+\Jump{p}e_3&=ge_3,\quad \text{on}\ \Sigma,\\
\Jump{u}&=0,\quad \text{on}\ \Sigma,\\
P_{S_1}\left(\mu(\nabla u+\nabla u^{\sf T})\nu_{S_1}\right)&=0,\quad \text{on}\ S_1\backslash\partial\Sigma,\\
(u|\nu_{S_1})&=0,\quad \text{on}\ S_1\backslash\partial\Sigma,\\
u&=0,\quad \text{on}\ S_2,\\
\end{split}
\end{align}
by Theorem \ref{thm:AppEllStokes} to obtain a unique solution $u\in H_p^2(\Omega\backslash\Sigma)\cap H_p^1(\Omega)$. Define the (reduced) \emph{Neumann-to-Dirichlet operator} $N_\lambda:W_p^{1-1/p}(G)\to W_p^{2-1/p}(G)$ by $N_\lambda g:=(u|e_3)$. With the compact operator $N_\lambda$ at hand we may rewrite the eigenvalue problem \eqref{eq:stab11} as follows
\begin{equation}\label{eq:h}
\lambda h+N_\lambda (A_*h)=0,
\end{equation}
where $A_*h:= -\sigma\Delta_N h-\Jump{\rho}\gamma_a h$ is the shifted Neumann Laplacian with domain
$$D(A_*)=\left\{h\in W_p^{3-1/p}(G):\int_{G}h\ dx'=0,\ \partial_{\nu_{\partial G}} h=0\ \text{on}\ \partial G\right\}.$$
We remark that for $\lambda\ge0$ problems \eqref{eq:stab11} and \eqref{eq:h} are equivalent. Therefore it suffices to show that for $\frac{\Jump{\rho}\gamma_a}{\sigma}>\lambda_1$ there exists $\lambda>0$ such that equation \eqref{eq:h} has a nontrivial solution $h\in D(A_*)$.

Concerning $N_\lambda$ we have the following result.
\begin{prop}\label{prop:Dir2NeuProp}
The Neumann-to-Dirichlet operator $N_\lambda$ of the Stokes-problem \eqref{N2D} admits a compact self-adjoint extension to $L_2(G)$ which has the following properties.
\begin{enumerate}
\item If $u$ denotes the solution of \eqref{N2D}, then
$$(N_\lambda g|g)_2=\lambda|\rho^{1/2}u|_{L_2(\Omega)}^2+\frac{1}{2}|\mu^{1/2} Du|_{L_2(\Omega)}^2$$ for all $g\in W_p^{1-1/p}(G)$ and $\lambda\ge 0$.
\item For each $\alpha\in(0,1/2)$ there is a constant $C>0$ such that
$$(N_\lambda g|g)_2\geq \frac{(1+\lambda)^\alpha}{C}|N_\lambda g|_{L_2(G)}^2,$$
for all $g\in L_2(G)$ and $\lambda\ge 0$. In particular,
$$ |N_\lambda|_{\mathcal{B}(L_2(G))}\leq \frac{C}{(1+\lambda)^\alpha}$$ for all $\lambda\geq0$.
\item $N_\lambda g$ has mean value zero for all $\lambda\ge 0$ and each $g\in L_2(G)$.
\end{enumerate}
\end{prop}
\begin{proof}
The first assertion follows from integration by parts, while for the proof of the second assertion one uses trace theory, interpolation theory and Korn's inequality. To show the third assertion, observe that for each $\lambda\ge 0$ we have
$$\int_G N_\lambda g\ dx'=\int_G (u|e_3)\ dx'=\int_{\Omega_1} \div u_1\ dx=0,$$
by the divergence theorem, where $u_1:=u|_{\Omega_1}$.
\end{proof}
Proposition \ref{prop:Dir2NeuProp} combined with Korn's inequality imply that whenever $N_\lambda g=0$, then $u=0$, hence $g$ must be constant. Therefore, the restriction of $N_\lambda$ to functions with mean value zero is injective. Therefore we may rewrite equation \eqref{eq:h} as
\begin{equation}\label{eq:h2}
\lambda N_\lambda^{-1}h+A_* h=0,
\end{equation}
for each $h\in D(A_*)$. Let us consider \eqref{eq:h2} in $L_2^{(0)}(G)$, the subspace of $L_2(G)$ consisting of functions with vanishing mean value. Define $B_\lambda:=\lambda N_\lambda^{-1}+A_*$ with
$$D(B_\lambda)=D(A_*)=\left\{h\in W_2^{2}(G)\cap L_2^{(0)}(G) :\partial_{\nu_{\partial G}} h=0\ \text{on}\ \partial G\right\},$$
since $N_\lambda^{-1}$ is a relatively compact perturbation of $A_*$. We will show that the operator $B_\lambda$ is positive definite provided $\lambda>0$ is large enough. Let $\mu_j>0$ be an eigenvalue of $N_\lambda^{-1}$ in $L_2^{(0)}(G)$ with corresponding eigenfunction $e_j$. Then
$$\frac{1}{\mu_j} |e_j|_2=|N_\lambda e_j|_2\le \frac{C}{(1+\lambda)^\alpha}|e_j|_2,$$
hence $\mu_j\ge \frac{1}{C}>0$ for each $\lambda\ge 0$. It follows that
$$(B_\lambda h|h)_2=\lambda (N_\lambda^{-1} h|h)_2+(A_*h|h)_2\ge \left(\lambda/C-\Jump{\rho}\gamma_a\right) |h|_2^2>0$$
for each $h\in D(A_*)$, if $\lambda>0$ is sufficiently large.

On the other hand, let $0\neq h_*\in D(A_*)$ be an eigenfunction of $-\Delta_N$ to the first nontrivial eigenvalue $\lambda_1>0$ of $-\Delta_N$, hence $-\Delta_N h_*=\lambda_1 h_*$. This yields
$$(B_\lambda h_*|h_*)_2=\lambda (N_\lambda^{-1} h_*|h_*)_2-\sigma\left(\frac{\Jump{\rho}\gamma_a}{\sigma}-\lambda_1\right)|h_*|_2^2.$$
Since $\lim_{\lambda\to 0_+}\lambda (N_\lambda^{-1} h_*|h_*)=0$ it follows that $(B_\lambda h_*|h_*)_2<0$ provided $\lambda>0$ is sufficiently small and $\frac{\Jump{\rho}\gamma_a}{\sigma}>\lambda_1$. Let $\frac{\Jump{\rho}\gamma_a}{\sigma}>\lambda_1$ and define
$$\lambda_*:=\sup\{\lambda>0:B_\mu\ \text{is not positive semi-definite for each}\ \mu\in(0,\lambda] \}.$$
Then $\lambda_*>0$ by what we have shown above and $B_{\lambda}$ has a negative eigenvalue for each $\lambda<\lambda_*$, since the resolvent of $B_\lambda$ is compact. It follows that $0\in\sigma(B_{\lambda_*})$, hence there exists a solution $0\neq h\in D(A_*)$ in $L_2^{(0)}(G)$ of \eqref{eq:h2}. A bootstrap argument finally shows that $h\in D(A_*)\cap W_p^{3-1/p}(G)$. This in turn yields that $\sigma(L)\cap\mathbb{C}_+\neq\emptyset$ whenever $\frac{\Jump{\rho}\gamma_a}{\sigma}>\lambda_1$. We have proven the following result.
\begin{prop}\label{spectrumLR}
The operator $L$ defined above has the following spectral properties.
\begin{enumerate}
\item $\sigma(L)\cap i\mathbb{R}\subset\{0\}$ and $0\in\sigma(L)$ if and only if $\Jump{\rho}\gamma_a/\sigma\in \sigma(-\Delta_N)$.
\item If $\Jump{\rho}\le 0$ then $\sigma(L)\subset\mathbb{C}_-$.
\item If $\Jump{\rho}>0$ and $\frac{\Jump{\rho}\gamma_a}{\sigma}<\lambda_1$, then $\sigma(L)\subset\mathbb{C}_-$.
\item If $\Jump{\rho}>0$ and $\frac{\Jump{\rho}\gamma_a}{\sigma}>\lambda_1$, then $\sigma(L)\cap\mathbb{C}_+\neq\emptyset$.
\end{enumerate}
\end{prop}

\section{Parametrization of the nonlinear phase manifold}\label{paraphase}

We have already seen that after a Hanzawa transformation, the transformed velocity field is no longer divergence free. Moreover, the jump condition of the stress tensor as well as the divergence condition are transformed into some nonlinear terms. It is the aim of this subsection, to parameterize the nonlinear phase manifold
\begin{multline*}
\mathcal{P}\mathcal{M}:=\{(u,h)\in W_p^{2-2/p}(\Omega\backslash\Sigma)^3\times [W_p^{3-2/p}(\Sigma)\cap X]:\\
u|_{S_2}=0,\ u|_{S_1}\cdot\nu_{S_1}=0,\ P_{S_1}(\mu(\nabla u+\nabla u^{\sf T})\nu_{S_1})=0,\ \Jump{u}=0,\\
P_\Sigma(\mu(\nabla u+\nabla u^{\sf T})e_3)=(G_v(u,h),0),\ \partial_{\nu_{\partial G}}h=0,\
\div u=F_d(u,h)\},
\end{multline*}
as a subset of $X_\gamma:=W_p^{2-2/p}(\Omega\backslash\Sigma)^3\times W_p^{3-2/p}(\Sigma)$, near the trivial equilibrium $(u_*,h_*)=(0,0)$ over the linear phase manifold
\begin{multline*}
X_\gamma^0:=\{(u,h)\in [W_p^{2-2/p}(\Omega\backslash\Sigma)^3\times W_p^{3-2/p}(\Sigma)]\cap X_0:u|_{S_2}=0,\ u|_{S_1}\cdot\nu_{S_1}=0,\\
P_{S_1}(\mu(\nabla u+\nabla u^{\sf T})\nu_{S_1})=0,\ \Jump{u}=0,\ P_\Sigma(\mu(\nabla u+\nabla u^{\sf T})e_3)=0,\ \partial_{\nu_{\partial G}}h=0\}.
\end{multline*}
Let $\mathbb{E}_\pi:=\dot{W}_p^{1-2/p}(\Omega\backslash\Sigma)$,
$\mathbb{E}_q:=W_p^{1-3/p}(\Sigma)$,
\begin{multline*}
\mathbb{E}_u:=\{u\in W_p^{2-2/p}(\Omega\backslash\Sigma)^3:\\
\Jump{u}=0,\ u|_{S_1}\cdot\nu_{S_1}=0,\
u|_{S_2}=0,\ P_{S_1}(\mu(\nabla {u}+\nabla {u}^{\sf T})\nu_{S_1})=0\},
\end{multline*}
$\mathbb{E}:=\{(u,\pi,q)\in \mathbb{E}_u\times\mathbb{E}_\pi\times\mathbb{E}_q:q=\Jump{\pi}\},$
and
\begin{multline*}\mathbb{F}:=\{(f_1,f_2)\in
[W_p^{1-2/p}(\Omega\backslash\Sigma)\cap \hat{H}_p^{-1}(\Omega)]\times\\
\times W_p^{1-3/p}(\Sigma)^3:
(P_\Sigma f_2)\cdot\nu_{S_1}=0\ \text{at}\ \partial\Sigma\}.
\end{multline*}
We will need the following auxilliary result for the Stokes problem
\begin{align}
\begin{split}\label{para0}
\rho\omega {u}-\mu\Delta{u}+\nabla{\pi}&=0,\quad \text{in}\ \Omega\backslash\Sigma,\\
\div {u}&=f_d,\quad\ \text{in}\ \Omega\backslash\Sigma,\\
-\Jump{\mu\partial_3 {v}}-\Jump{\mu\nabla_{x'}{w}}&=g_v,\quad \text{on}\ \Sigma,\\
-2\Jump{\mu\partial_3 {w}}+\Jump{{\pi}}&=g_w,\quad \text{on}\ \Sigma,\\
\Jump{{u}}&=0,\quad \text{on}\ \Sigma,\\
P_{S_1}(\mu(\nabla {u}+\nabla {u}^{\sf T})\nu_{S_1})&=0,\quad \text{on}\ S_1\backslash\partial\Sigma,\\
{u}\cdot\nu_{S_1}&=0,\quad \text{on}\ S_1\backslash\partial\Sigma,\\
{u}&=0,\quad \text{on}\ S_2.
\end{split}
\end{align}
\begin{prop}\label{weakstokes}
Let $n=3$, $p>5$ and $\rho_j,\mu_j>0$. If $\omega>0$ is sufficiently large, then there exists a unique solution
$(u,\pi,q)\in\mathbb{E}$ of \eqref{para0} if and only if $(f_d,g_v,g_w)\in\mathbb{F}$.
Moreover, there exists a constant $M_\omega>0$ such that
$$\|(u,\pi,q)\|_{\mathbb{E}}\le M_\omega\|(f_d,g_v,g_w)\|_{\mathbb{F}}.$$
\end{prop}
\begin{proof}
For the proof of this result one may apply the same strategy which was used in the proof of Theorem \ref{thm:AppEllStokes}. We omit the details.
\end{proof}
Let us consider the elliptic problem
\begin{align}
\begin{split}\label{para1}
\rho\omega \bar{u}-\mu\Delta\bar{u}+\nabla\bar{\pi}&=0,\quad \text{in}\ \Omega\backslash\Sigma,\\
\div \bar{u}&=P_0F_d(\bar{u}+\tilde{u},\tilde{h}),\quad\ \text{in}\ \Omega\backslash\Sigma,\\
-\Jump{\mu\partial_3 \bar{v}}-\Jump{\mu\nabla_{x'}\bar{w}}&=G_v(\bar{u}+\tilde{u},\tilde{h}),\quad \text{on}\ \Sigma,\\
-2\Jump{\mu\partial_3 \bar{w}}+\Jump{\bar{\pi}}&=G_w(\bar{u}+\tilde{u},\tilde{h}),\quad \text{on}\ \Sigma,\\
\Jump{\bar{u}}&=0,\quad \text{on}\ \Sigma,\\
P_{S_1}(\mu(\nabla \bar{u}+\nabla \bar{u}^{\sf T})\nu_{S_1})&=0,\quad \text{on}\ S_1\backslash\partial\Sigma,\\
\bar{u}\cdot\nu_{S_1}&=0,\quad \text{on}\ S_1\backslash\partial\Sigma,\\
\bar{u}&=0,\quad \text{on}\ S_2,
\end{split}
\end{align}
for $(\bar{u},\bar{\pi},\Jump{\bar{\pi}})$, where $\omega>0$ and $(\tilde{u},\tilde{h})\in rB_{X_\gamma^0}(0)$ are given. Here we have set
$$P_0f:=f-\frac{1}{|\Omega|}\int_\Omega f\ dx,$$
for $f\in L_1(\Omega)$.

Define a nonlinear mapping $N:\mathbb{E}_u\times X_\gamma^0\to \mathbb{F}$ via
$$N(\bar{u},\tilde{u},\tilde{h}):=
\begin{pmatrix}
P_0F_d(\bar{u}+\tilde{u},\tilde{h})\\
G_v(\bar{u}+\tilde{u},\tilde{h})\\
G_w(\bar{u}+\tilde{u},\tilde{h})
\end{pmatrix}.
$$
Let $S_\omega$ denote the solution operator which is induced by
Proposition \ref{weakstokes} and define a mapping $H:=\mathbb{E}\times X_\gamma^0\to\mathbb{E}$ by
$$H((\bar{u},\bar{\pi},\bar{q}),(\tilde{u},\tilde{h})):=(\bar{u},\bar{\pi},\bar{q})
-S_\omega N(\bar{u},\tilde{u},\tilde{h}).$$
Since $N(0)=0$ it follows that $H(0)=0$. Since $N\in C^2$ it holds that $H\in C^2$, too. Differentiating $H$ with respect to $(\bar{u},\bar{\pi},\bar{q})$ we obtain
$$D_{(\bar{u},\bar{\pi},\bar{q})}H(0)=I_{\mathbb{E}},$$
where we used the fact that $D_{\bar{u}}N(0)=0$. The implicit function theorem implies the existence of a $C^2$-function $\phi_0:rB_{X_\gamma^0}\to\mathbb{E}$ with $\phi_0(0)=0$ and $\phi_0'(0)=0$, such that $H(\phi_0(\tilde{u},\tilde{h}),(\tilde{u},\tilde{h}))=0$ whenever $(\tilde{u},\tilde{h})\in rB_{X_\gamma^0}(0)$. In other words, this means that $(\bar{u},\bar{\pi},\bar{q})=\phi_0(\tilde{u},\tilde{h})$ is the unique solution of \eqref{para1} for a given $(\tilde{u},\tilde{h})\in rB_{X_\gamma^0}(0)$. It can furthermore be shown that $P_0F_d(\bar{u}+\tilde{u},\tilde{h})=F_d(\bar{u}+\tilde{u},\tilde{h})$ (see proof of Theorem \ref{nonlinwellposedness}).

Let $P(\bar{u},\bar{\pi},\bar{q}):=\bar{u}$ and define $\phi(\tilde{u},\tilde{h}):=P\phi_0(\tilde{u},\tilde{h})$ as well as
$$\Phi(\tilde{u},\tilde{h}):=(\tilde{u},\tilde{h})+(\phi(\tilde{u},\tilde{h}),0).$$
It follows that $\Phi(rB_{X_\gamma^0}(0))\subset \mathcal{P}\mathcal{M}$ and that $\Phi$ is injective. We will now show that $\Phi$ is locally surjective near 0. To this end we assume that $(u,h)\in \mathcal{P}\mathcal{M}$ is given and close to 0 in $X_\gamma$. Then we solve the linear problem
\begin{align}
\begin{split}\label{para2}
\rho\omega \bar{u}-\mu\Delta\bar{u}+\nabla\bar{\pi}&=0,\quad \text{in}\ \Omega\backslash\Sigma,\\
\div \bar{u}&=P_0F_d({u},{h}),\quad\ \text{in}\ \Omega\backslash\Sigma,\\
-\Jump{\mu\partial_3 \bar{v}}-\Jump{\mu\nabla_{x'}\bar{w}}&=G_v({u},{h}),\quad \text{on}\ \Sigma,\\
-2\Jump{\mu\partial_3 \bar{w}}+\Jump{\bar{\pi}}&=G_w({u},{h}),\quad \text{on}\ \Sigma,\\
\Jump{\bar{u}}&=0,\quad \text{on}\ \Sigma,\\
P_{S_1}(\mu(\nabla \bar{u}+\nabla \bar{u}^{\sf T})\nu_{S_1})&=0,\quad \text{on}\ S_1\backslash\partial\Sigma,\\
\bar{u}\cdot\nu_{S_1}&=0,\quad \text{on}\ S_1\backslash\partial\Sigma,\\
\bar{u}&=0,\quad \text{on}\ S_2,
\end{split}
\end{align}
by Proposition \ref{weakstokes} to obtain $\bar{u}\in \mathbb{E}_u$. Define $(\tilde{u},\tilde{h}):=(u-\bar{u},h)$ and observe that
$$\div \tilde{u}=F_d(u,h)-P_0F_d(u,h)=\frac{1}{|\Omega|}\int_{\Omega}F_d(u,h)\ dx.$$
Since $\tilde{u}\in H_p^1(\Omega)^3$ with $\tilde{u}|_{S_1}\cdot\nu_{S_1}=0$, $\tilde{u}|_{S_2}=0$ and $\Jump{\tilde{u}}=0$, it follows that $P_0F_d(u,h)=F_d(u,h)$, hence $\div\tilde{u}=0$.

This in turn yields $(\tilde{u},\tilde{h})\in X_\gamma^0$ and $\phi(\tilde{u},\tilde{h})=\bar{u}$, showing that $\Phi$ is locally surjective near $0$.

\section{Main result on Rayleigh-Taylor instability}

In this subsection we are goint to prove the following main result.
\begin{thm}\label{mainresultstab}
Let $n=3$, $p>5$ and $\rho_j,\mu_j,\gamma_j,\sigma>0$. Denote by $(u_*,h_*)=(0,0)$ the trivial equilibrium and let $s(L)<0$ denote the spectral bound of $L$. Then the following assertions hold.
\begin{enumerate}
\item If $\Jump{\rho}\gamma_a/\sigma<\lambda_1$, then $(u_*,h_*)$ is exponentially stable in the following sense. There exist constants $\eta\in [0,-s(L))$ and $\delta>0$ such that whenever $(u_0,h_0)\in \mathcal{P}\mathcal{M}$ with
    $$\|(u_0,h_0)\|_{X_\gamma}\le\delta,$$
    then the estimate
    $$\|(u(t),h(t))\|_{X_\gamma}\le e^{-\eta t}\|(u_0,h_0)\|_{X_\gamma}$$
    is valid for all $t\ge 0$.
\item If $\Jump{\rho}>0$ and $\Jump{\rho}\gamma_a/\sigma>\lambda_1$, then $(u_*,h_*)$ is unstable in the following sense. There is a constant $\varepsilon_0>0$ such that for each $\delta>0$ there are initial values $(u_0,h_0)\in \mathcal{P}\mathcal{M}$ with
    $$\|(u_0,h_0)\|_{X_\gamma}\le\delta$$
    such that the solution $(u,h)$ of \eqref{eq:NScap2} satisfies
    $$\|(u(t_0),h(t_0))\|_{X_\gamma}\ge\varepsilon_0$$
    for some $t_0>0$.
\end{enumerate}
\end{thm}
\begin{proof}
1. Let $(u_0,h_0)\in X_\gamma$ be fixed such that $\|u_0\|_{W_p^{2-2/p}}+\|h_0\|_{W_p^{3-2/p}}\le\delta$ for some sufficiently small $\delta>0$ to be determined later. It follows from the results of the last subsection that $(u_0,h_0)=(\tilde{u}_0,\tilde{h}_0)+(\phi(\tilde{u}_0,\tilde{h}_0),0)$, i.e.\ $\tilde{h}_0=h_0$, where $(\tilde{u}_0,\tilde{h}_0)\in rB_{X_\gamma^0}(0)$. For $h\in L_1(\Sigma)$, we define
$$P_0^\Sigma h:=\frac{1}{|\Sigma|}\int_\Sigma h\ dx',$$
and consider the linear evolution equation
\begin{equation}\label{eq:tildeuh}
\partial_t (\tilde{u},\tilde{h})-L(\tilde{u},\tilde{h})=\omega\left((I-T_1)\bar{u},P_0^\Sigma \bar{h}\right),\quad (\tilde{u},\tilde{h})|_{t=0}=(\tilde{u}_0,\tilde{h}_0),
\end{equation}
in the space
$$X_0:=L_{p,\sigma}(\Omega)\times \left\{h\in W_p^{2-1/p}(\Sigma):\int_G h\ dx'=0,\ \partial_{\nu_{\partial G}}h=0\right\},$$
where $L$ has been defined in Subsection \ref{EquilLin} and $(\bar{u},\bar{h})\in e^{-\eta}[\mathbb{E}_u(\mathbb{R}_+)\times\mathbb{E}_h(\mathbb{R}_+)]$ are given functions. Here $\eta\in[0,-s(L))$, where $s(L)<0$ denotes the spectral bound of $L$.

By Corollary \ref{cor:linmaxreg} \& Proposition \ref{spectrumLR} it follows that the operator $L$ has the property of $L_p$-maximal regularity on $\mathbb{R}_+$ provided that $\Jump{\rho}\gamma_a/\sigma<\lambda_1$. Since
$(f,g):=\omega\left((I-T_1)\bar{u},P_0^\Sigma \bar{h}\right)\in e^{-\eta}L_p(\mathbb{R}_+;X_0)$
and $(\tilde{u}_0,\tilde{h}_0)\in X_\gamma^0$ we obtain a unique solution
$$(\tilde{u},\tilde{h})\in e^{-\eta }[H_p^1(\mathbb{R}_+;X_0)\cap L_p(\mathbb{R}_+;X_1)]=:e^{-\eta}\tilde{\mathbb{E}}(\mathbb{R}_+)$$
for each $\eta\in [0,-s(L))$, where $X_1=D(L)$ is given by \eqref{X1} . We denote by $$\Xi:=(\partial_t-L,\operatorname{tr}|_{t=0})^{-1}:e^{-\eta}L_p(\mathbb{R}_+;X_0)\times X_\gamma^0\to e^{-\eta}\tilde{\mathbb{E}}(\mathbb{R}_+)$$ the corresponding solution operator which satisfies the estimate
$$\|\Xi((f,g),(\tilde{u}_0,\tilde{h}_0))\|_{e^{-\eta}\tilde{\mathbb{E}}(\mathbb{R}_+)}\le M\|((f,g),(\tilde{u}_0,\tilde{h}_0))\|_{e^{-\eta}L_p(\mathbb{R}_+;X_0)\times X_\gamma^0}.$$
In particular, by \eqref{formulaPi} we obtain on the one hand that $\nabla\tilde{\pi}$ is given in terms of $(\bar{u},\bar{h})$ and
$$\|\nabla\tilde{\pi}\|_{e^{-\eta}L_p(\mathbb{R}_+;L_p(\Omega))}\le CM\|((f,g),(\tilde{u}_0,\tilde{h}_0))\|_{e^{-\eta}L_p(\mathbb{R}_+;X_0)\times X_\gamma^0}.$$
At this point we remark that the function $\tilde{h}$ possesses some more regularity. Indeed, it holds that
$$\partial_t \tilde{h}=\tilde{u}_3|_{\Sigma}+\omega P_0^\Sigma\bar{h}\in e^{-\eta }W_p^{1-1/2p}(\mathbb{R}_+;L_p(\Sigma)),$$
hence $\tilde{h}\in e^{-\eta }W_p^{2-1/2p}(\mathbb{R}_+;L_p(\Sigma))$ holds in addition.

Next, we consider the problem
\begin{align}\label{eq:barh}
\begin{split}
\omega\rho \bar{u}+\partial_t\rho \bar{u}-\mu\Delta \bar{u}+\nabla \bar{\pi}&=F(\tilde{u}+\bar{u},\tilde{\pi}+\bar{\pi},\tilde{h}+\bar{h}),\quad \text{in}\ \Omega\backslash\Sigma,\\
\div \bar{u}&=P_0F_d(\tilde{u}+\bar{u},\tilde{h}+\bar{h}),\quad \text{in}\ \Omega\backslash\Sigma,\\
-\Jump{\mu \partial_3 \bar{v}}-\Jump{\mu\nabla_{x'} \bar{w}}&=G_v(\tilde{u}+\bar{u},\tilde{h}+\bar{h}),\quad \text{on}\ \Sigma,\\
-2\Jump{\mu \partial_3 \bar{w}}+\Jump{\bar{\pi}}-\sigma\Delta_{x'} \bar{h}-\Jump{\rho}\gamma_a \bar{h}&=G_w(\tilde{u}+\bar{u},\tilde{h}+\bar{h}),\quad \text{on}\ \Sigma,\\
\Jump{\bar{u}}&=0,\quad \text{on}\ \Sigma,\\
\omega \bar{h}+\partial_t \bar{h}-(u|e_3)&=H_1(\tilde{u}+\bar{u},\tilde{h}+\bar{h}),\quad \text{on}\ \Sigma,\\
P_{S_1}\left(\mu(\nabla \bar{u}+\nabla \bar{u}^{\sf T})\nu_{S_1}\right)&=0,\quad \text{on}\ S_1\backslash\partial\Sigma,\\
\bar{u}\cdot\nu_{S_1}&=0,\quad \text{on}\ S_1\backslash\partial\Sigma,\\
\bar{u}&=0,\quad \text{on}\ S_2,\\
\partial_{\nu_{\partial G}}\bar{h}&=0,\quad \text{on}\ \partial\Sigma,\\
\bar{u}(0)&=\phi(\tilde{u}_0,\tilde{h}_0),\quad \text{in}\ \Omega\backslash\Sigma\\
\bar{h}(0)&=0,\quad \text{on}\ \Sigma,
\end{split}
\end{align}
where $(\tilde{u},\tilde{h})=\Xi\left((I-T_1)\bar{u},P_0^\Sigma \bar{h}\right)$ and $\nabla\tilde{\pi}$ is given by \eqref{formulaPi}, with $(u,h)$ being replaced by $(\tilde{u},\tilde{h})$.

Let
\begin{multline*}
e^{-\eta}\mathbb{E}_u(\mathbb{R}_+):=\{u\in e^{-\eta}[H_p^1(\mathbb{R}_+;L_p(\Omega)^3)\cap L_p(\mathbb{R}_+;H_p^2(\Omega\backslash\Sigma)^3)]:\\
\Jump{u}=0,\ u\cdot\nu_{S_1}=0,\ P_{S_1}(\mu(\nabla u+\nabla u^{\sf T})\nu_{S_1})=0,\ u|_{S_2}=0\},
\end{multline*}
$$e^{-\eta}\mathbb{E}_\pi(\mathbb{R}_+):=e^{-\eta}L_p(\mathbb{R}_+;\dot{H}_p^1(\Omega\backslash\Sigma)),$$
$$e^{-\eta}\mathbb{E}_{q}(\mathbb{R}_+):=e^{-\eta}[W_p^{1/2-1/2p}(\mathbb{R}_+;L_p(\Sigma))\cap L_p(\mathbb{R}_+;W_p^{1-1/p}(\Sigma))],$$
\begin{multline*}
e^{-\eta}\mathbb{E}_h(\mathbb{R}_+):=\{h\in e^{-\eta}[W_p^{2-1/2p}(\mathbb{R}_+;L_p(\Sigma))\cap\\
\cap H_p^1(\mathbb{R}_+;W_p^{2-1/p}(\Sigma))\cap L_p(\mathbb{R}_+;W_p^{3-1/p}(\Sigma))]:\partial_{\nu_{\partial G}}h=0\},
\end{multline*}
and
$$e^{-\eta}\mathbb{E}(\mathbb{R}_+):=\{(u,\pi,q,h)
\in e^{-\eta}[\mathbb{E}_u(\mathbb{R}_+)\times \mathbb{E}_\pi(\mathbb{R}_+)
\times \mathbb{E}_{q}(\mathbb{R}_+)\times \mathbb{E}_h(\mathbb{R}_+)]
:q=\Jump{\pi}\}.
$$
Moreover, we define the data spaces as follows.
$$e^{-\eta}\mathbb{F}_1(\mathbb{R}_+):=e^{-\eta}L_p(\mathbb{R}_+;L_p(\Omega)^3),$$
$$e^{-\eta}\mathbb{F}_2(\mathbb{R}_+):=e^{-\eta}[H_p^1(\mathbb{R}_+;\hat{H}_p^{-1}(\Omega))\cap L_p(\mathbb{R}_+;H_p^1(\Omega\backslash\Sigma))],$$
\begin{multline*}
e^{-\eta}\mathbb{F}_3(\mathbb{R}_+):=\{f_3\in e^{-\eta}[W_p^{1/2-1/2p}(\mathbb{R}_+;L_p(\Sigma)^3)\cap L_p(\mathbb{R}_+;W_p^{1-1/p}(\Sigma)^3)]:\\
P_{\Sigma}(f_3)\cdot\nu_{S_1}=0\},
\end{multline*}
$$e^{-\eta}\mathbb{F}_4(\mathbb{R}_+):=\{f_4\in e^{-\eta}[W_p^{1-1/2p}(\mathbb{R}_+;L_p(\Sigma))\cap L_p(\mathbb{R}_+;W_p^{2-1/p}(\Sigma))]:\partial_{\nu_{\partial G}}f_4=0\},$$
and
$e^{-\eta}\mathbb{F}(\mathbb{R}_+):=\times_{j=1}^4e^{-\eta}\mathbb{F}_j(\mathbb{R}_+)$.

Define an operator $\mathbb{L}_\omega:e^{-\eta}\mathbb{E}(\mathbb{R}_+)\to e^{-\eta}\mathbb{F}(\mathbb{R}_+)$ by
$$\mathbb{L}_\omega(\bar{u},\bar{\pi},\bar{q},\bar{h}):=
\begin{pmatrix}
\omega\rho \bar{u}+\partial_t\rho \bar{u}-\mu\Delta \bar{u}+\nabla \bar{\pi}\\
\div \bar{u}\\
-\Jump{\mu(\nabla \bar{u}+\nabla \bar{u}^{\sf T})}e_3+\bar{q}e_3-\sigma\Delta_{x'}\bar{h}e_3-\Jump{\rho}\gamma_a \bar{h} e_3\\
\omega \bar{h}+\partial_t \bar{h}-\bar{u}\cdot e_3
\end{pmatrix},
$$
where $\bar{u}=(\bar{v},\bar{w})$ and set
\begin{multline*}
\bar{X}_\gamma:=\{(u,h)\in W_p^{2-2/p}(\Omega\backslash\Sigma)^3\times W_p^{3-2/p}(\Sigma):\\
u|_{S_2}=0,\ u|_{S_1}\cdot\nu_{S_1}=0,\ P_{S_1}(\mu(\nabla u+\nabla u^{\sf T})\nu_{S_1})=0,\ \Jump{u}=0,\ \partial_{\nu_{\partial G}}h=0\}
\end{multline*}
Denote by
$$\operatorname{ext}_\eta:\bar{X}_\gamma\to e^{-\eta}[\mathbb{E}_u(\mathbb{R}_+)\times\mathbb{E}_h(\mathbb{R}_+)]$$
a linear extension operator, such that $\operatorname{ext}_\eta (\hat{u},\hat{h})|_{t=0}=(\hat{u},\hat{h})$. The existence of such an extension operator can be seen as in Section \ref{sec:Redzerotrace}, by solving the corresponding auxiliary problems in exponentially weighted spaces.

Furthermore, we define a nonlinear mapping $N:e^{-\eta}[\mathbb{E}_u(\mathbb{R}_+)\times\mathbb{E}_\pi(\mathbb{R}_+)
\times\mathbb{E}_h(\mathbb{R}_+)]\times X_\gamma^0\to e^{-\eta}\mathbb{F}(\mathbb{R}_+)$ by
$$N((\bar{u},\bar{\pi},\bar{h}),(\tilde{u}_0,\tilde{h}_0)):=
\begin{pmatrix}
\bar{F}(\bar{u},\bar{\pi},\bar{h})\\
\bar{F}_d\Big((\bar{u},\bar{h})+
\operatorname{ext}_\eta[(\phi(\tilde{u}_0,\tilde{h}_0),0)-(\bar{u}(0),\bar{h}(0))]\Big)\\
\bar{G}_v\Big((\bar{u},\bar{h})+
\operatorname{ext}_\eta[(\phi(\tilde{u}_0,\tilde{h}_0),0)-(\bar{u}(0),\bar{h}(0))]\Big)\\
\bar{G}_w(\bar{u},\bar{h})\\
\bar{H}_1(\bar{u},\bar{h})
\end{pmatrix}.
$$
Here the functions $(\bar{F},\bar{F}_d,\bar{G}_j,\bar{H}_1)$ result from $(F,F_d,G_j,H_1)$ by replacing $(\tilde{u},\tilde{h})$ and $\nabla\tilde{\pi}$ by $\Xi\left((I-T_1)\bar{u},P_0^\Sigma \bar{h}\right)$ and \eqref{formulaPi}, respectively.

Consider the equation
$$\mathbb{L}_\omega(\bar{u},\bar{\pi},\bar{q},\bar{h})=
N((\bar{u},\bar{\pi},\bar{h}),(\tilde{u}_0,\tilde{h}_0)).$$
subject to the initial condition $(\bar{u},\bar{h})|_{t=0}=(\phi(\tilde{u}_0,\tilde{h}_0),0)$.
If we can show that this problem has a unique solution $(\bar{u},\bar{\pi},\bar{q},\bar{h})\in e^{-\eta}\mathbb{E}(\mathbb{R}_+)$, then, by construction, $(\bar{u},\bar{\pi},\bar{q},\bar{h})$ is a solution of \eqref{eq:barh}.

Let $(f,f_d,g_v,g_w,g_h)\in e^{-\eta}\mathbb{F}(\mathbb{R}_+)$ and $(u_0,h_0)\in \bar{X}_\gamma$ be given such that $\div u_0=f_d|_{t=0}$ and $-\Jump{\mu\nabla_{x'}w_0}-\Jump{\mu\partial_3 v_0}=g_v|_{t=0}$, where $u_0=(v_0,w_0)$. Consider the linear problem to find a unique $w=(u,\pi,q,h)\in e^{-\eta}\mathbb{E}(\mathbb{R}_+)$ such that
$$\mathbb{L}_\omega w=F,\quad z(0)=z_0=(u_0,h_0),$$
for a sufficiently large $\omega>0$, where $F:=(f,f_d,g_v,g_w,g_h)$ and $z:=(u,h)$. By Corollary \ref{cor:AppStokes} we may assume without loss of generality that $f=u_0=0$, $f_d=g_w=0$ and $g_v=0$. The remaining problem with $\tilde{F}=(0,0,0,0,g_h)$ and $\tilde{z}_0=(0,h_0)$ can be written in the abstract form
$$\omega z+\dot{z}+Lz=(0,g_h),\ t>0,\quad z(0)=\tilde{z}_0,$$
where the operator $L$ has been defined in Section \ref{EquilLin}. If $\omega>0$ is chosen sufficiently large, then there exists a unique solution $z\in e^{-\omega}[\mathbb{E}_u(\mathbb{R}_+)\times\mathbb{E}_h(\mathbb{R}_+)]$, since $L$ has the property of maximal regularity of type $L_p$ in
$$L_{p,\sigma}(\Omega)\times \left\{h\in W_p^{2-1/p}(\Sigma):\partial_{\nu_{\partial G}}h=0\right\},$$
by Corollary \ref{cor:linmaxreg}.

Therefore it makes sense to define a function $H:e^{-\eta}[\mathbb{E}_u(\mathbb{R}_+)\times\mathbb{E}_\pi(\mathbb{R}_+)
\times\mathbb{E}_h(\mathbb{R}_+)]\times X_\gamma^0\to e^{-\eta}\mathbb{E}(\mathbb{R}_+)$ by
\begin{multline*}
H((\bar{u},\bar{\pi},\bar{h}),(\tilde{u}_0,\tilde{h}_0))\\
:=(\bar{u},\bar{\pi},\bar{q},\bar{h})-(\mathbb{L}_\omega,\operatorname{tr}|_{t=0})^{-1}
[N((\bar{u},\bar{\pi},\bar{h}),(\tilde{u}_0,\tilde{h}_0)),(\phi(\tilde{u}_0,\tilde{h}_0),0)].
\end{multline*}
Note that $H$ is well defined, since all compatibility conditions at $t=0$ as well as at $\partial\Sigma$ and $\partial S_2$ are satisfied by construction.
It follows from Proposition \ref{prop:regnonlin} and the results in Section \ref{paraphase} that $H$ is a $C^2$-mapping with $H(0)=0$ and $$D_{(\bar{u},\bar{\pi},\bar{q},\bar{h})}H(0)=I_{e^{-\eta}\mathbb{E}(\mathbb{R}_+)}.$$
Therefore, the implicit function theorem yields the existence of a $C^2$-function $\psi:X_\gamma^0\to e^{-\eta}\mathbb{E}(\mathbb{R}_+)$ with $\psi(0)=0$ and $\psi'(0)=0$ such that $H(\psi(\tilde{u}_0,\tilde{h}_0),(\tilde{u}_0,\tilde{h}_0))=0$, whenever $(\tilde{u}_0,\tilde{h}_0)\in rB_{X_\gamma^0}(0)$ for some sufficiently small $r>0$.

Let
$$(u,\pi,q,h):=(\tilde{u},\tilde{\pi},\tilde{q},\tilde{h})
+(\bar{u},\bar{\pi},\bar{q},\bar{h}).$$
As in the proof of Theorem \ref{nonlinwellposedness} one can show that $P_0F_d(u,h)=F_d(u,h)$, since $\div u=\div(\tilde{u}+\bar{u})=\div\bar{u}$.
Integrating $\tilde{w}=\tilde{u}\cdot e_3$ over $\Sigma$ yields
$$\int_\Sigma \tilde{w}\ dx'=\int_{\Omega_1}\div\tilde{u}_1\ dx=0.$$
This in turn implies that
\begin{align*}
(\omega+\frac{d}{dt})\int_\Sigma \bar{h}\ dx'&=\int_\Sigma[\bar{w}-(v|\nabla h)]\ dx'\\
&=\int_\Sigma[w-(v|\nabla h)]\ dx'\\
&=\int_\Sigma(u|\nu_{\Gamma(t)})\sqrt{1+|\nabla h|^2}\ dx'\\
&=\int_{\Gamma(t)}\left((u\circ{\Theta_h^{-1}})|\nu_{\Gamma(t)}\right)\ d\Gamma(t)\\
&=\int_{\Omega_1(t)}\div (u\circ\Theta_h^{-1})\ d\Omega_1(t)\\
&=0,
\end{align*}
since
$$\div (u\circ\Theta_h^{-1})=(\div u-F_d(u,h))\circ\Theta_h^{-1}=(\div\bar{u}-F_d(u,h))\circ\Theta_h^{-1}=0.$$
Since $\bar{h}|_{t=0}=0$, this readily yields that $\bar{h}$ is mean value free, hence $P_0^\Sigma\bar{h}=\bar{h}$ and therefore $(u,\pi,q,h)$ is a solution of \eqref{eq:NScap2} which is unique by Theorem \ref{nonlinwellposedness}. The component $(u,h)$ of the solution has the representation
$$(u,h)=\bar{\psi}(\tilde{u}_0,\tilde{h}_0)+\bar{\Xi}(\tilde{u}_0,\tilde{h}_0),$$
where $\bar{\psi}(\tilde{u}_0,\tilde{h}_0):=(\bar{u},\bar{h})$ and $\bar{\Xi}$ results by replacing $(\bar{u},\bar{h})$ by $\bar{\psi}(\tilde{u}_0,\tilde{h}_0)$ in the definition of $\Xi$. This yields the estimate
$$\|(u,h)\|_{e^{-\eta}[\mathbb{E}_u\times\mathbb{E}_h]}\le M\|(\tilde{u}_0,\tilde{h}_0)\|_{X_\gamma^0},$$
where $M>0$ does not depend on $(\tilde{u}_0,\tilde{h}_0)\in rB_{X_\gamma^0}(0)$ as long as $r>0$ is sufficiently small. This follows from the smoothness of the function $\psi$. Since $(\tilde{u}_0,\tilde{h}_0)=({u}_0,{h}_0)-\phi(\tilde{u}_0,\tilde{h}_0)$ and $\phi(0)=0$ as well as $\phi'(0)=0$, we find for each $\varepsilon>0$ a number $r(\varepsilon)>0$ such that the estimate
\begin{align*}
\|(\tilde{u}_0,\tilde{h}_0)\|_{X_\gamma}&\le\|({u}_0,{h}_0)\|_{X_\gamma}
+\|\phi(\tilde{u}_0,\tilde{h}_0)\|_{X_\gamma}\\
&\le \|({u}_0,{h}_0)\|_{X_\gamma}+\varepsilon\|(\tilde{u}_0,\tilde{h}_0)\|_{X_\gamma}
\end{align*}
is valid. This implies the final estimate
$$\|(u,h)\|_{e^{-\eta}[\mathbb{E}_u\times\mathbb{E}_h]}\le M_\varepsilon\|({u}_0,{h}_0)\|_{X_\gamma},$$
proving the first assertion.

2. Denote by $\sigma^+$ the collection of the eigenvalues of $L$ with positive real part and let $P^+$ be the spectral projection related to $\sigma^+$. Define $P^-:=I-P^+$ and $X_0^\pm:=P^\pm X_0$. Since $\sigma^+$ is finite, it follows that $X_0^+$ is finite dimensional and the decompositions
$$X_0=X_0^+\oplus X_0^-,\quad L=L^+\oplus L^-$$
hold, where $L^+$ is a bounded linear operator from $X_0^+$ to $X_0^+$. Note further that the spaces $D(L^+)$ and $X_0^+$ coincide and that
$$\|z\|:=\|P^+ z\|_{X_0}+\|P^- z\|_{X_0}$$
defines an equivalent norm in $X_0$, since $P^\pm$ are bounded and linear operators.
By spectral theory, it holds that $\sigma^\pm=\sigma(L^\pm)$ and $\sigma^-\subset\overline{\mathbb{C}_-}$. Let $\lambda_*\in\sigma^+$ denote the eigenvalue with the smallest real part and choose numbers $\kappa,\eta>0$ such that $[\kappa-\eta,\kappa+\eta]\subset (0,\operatorname{Re}\lambda_*)$. It follows that the strip
$$\{\lambda\in\mathbb{C}:\operatorname{Re}\lambda\in[\kappa-\eta,\kappa+\eta]\}$$
does not contain any eigenvalues of $L$. Therefore the restricted semigoups $e^{\mp L^\pm t}$ satisfy the estimates
\begin{equation}\label{eq:inst1}
\|e^{L^- t}\|\le M e^{(\kappa-\eta)t},\quad \|e^{-L^+ t}\|\le M e^{-(\kappa+\eta)t},\quad t\ge 0,
\end{equation}
for some constant $M>0$.

Our aim is to prove the second assertion by a contradiction argument. To this end we assume that $(u_*,h_*)=(0,0)$ is stable. Then there exists a global solution $(u(t),\pi(t),q(t),h(t))$ of \eqref{eq:NScap2} such that $(u,\pi,q,h)\in\mathbb{E}(T)$ for each finite interval $J=[0,T]\subset [0,\infty)$. Moreover, for each $\varepsilon>0$ there exists $\delta(\varepsilon)>0$ such that whenever $\|(u_0,h_0)\|_{X_\gamma}\le\delta$ then $\|(u(t),h(t))\|_{X_\gamma}\le\varepsilon$ for all $t\ge 0$. Note that the solution admits the decomposition
$$(u,\pi,q,h)=(\tilde{u},\tilde{\pi},\bar{q},\tilde{h})+(\bar{u},\bar{\pi},\bar{q},\bar{h}),$$
where $(\tilde{u},\tilde{h})$ solves \eqref{eq:tildeuh} with $\tilde{\pi}$, $\tilde{q}=\Jump{\tilde{\pi}}$ given in terms of $(\tilde{u},\tilde{h})$ (see \eqref{formulaPi}) and $(\bar{u},\bar{\pi},\bar{q},\bar{h})$ solves \eqref{eq:barh} with a given right hand side $(u,\pi,q,h)$. Observe that in this case $P_0^\Sigma\bar{h}=\bar{h}$, by integration of $\eqref{eq:barh}_6$ over $\Sigma$, since
$$\int_{\Sigma}(\bar{u}|e_n) d\Sigma=\int_{\Omega_1}\div\bar{u}^1 dx=\int_{\Omega_1}F_d(u^1,h) dx=\int_{\Omega_1}\div u^1 dx$$
and
$$\int_\Sigma H_1(u,h) d\Sigma=\int_\Sigma(\partial_t h-(u|e_3)) d\Sigma=-\int_{\Omega_1}
\div u^1 dx,$$
where $u^1:=u|_{\Omega_1}$ and where we made use of the fact that $P_0^\Sigma h=h$.

To shorten the notation we introduce the new functions $\tilde{z}:=(\tilde{u},\tilde{h})$, $\bar{z}=(\bar{u},\bar{h})$, $\tilde{w}=(\tilde{u},\tilde{\pi},\bar{q},\tilde{h})$ and $\bar{w}=(\bar{u},\bar{\pi},\bar{q},\bar{h})$. The functions $P^\pm\tilde{z}$ solve the evolutionary problem
\begin{equation}\label{eq:inst2}
\frac{d}{dt}P^\pm\tilde{z}-L^\pm P^\pm\tilde{z}=\omega P^\pm Q\bar{z},\quad P^\pm\tilde{z}|_{t=0}=P^\pm\tilde{z}_0,
\end{equation}
where $Q\bar{z}:=((I-T_1)\bar{u},\bar{h})$ and $\tilde{z}_0:=(\tilde{u}_0,\tilde{h}_0)$. In a first step we show that $P^+\tilde{z}$ is given by the formula
\begin{equation}\label{eq:inst3}
P^+\tilde{z}(t)=-\int_t^\infty e^{L^+(t-s)}\omega P^+Q\bar{z}(s)\ ds.
\end{equation}
Since $P^+$ is bounded and $X_\gamma^0\hookrightarrow X_0$, it follows from the assumption that
$$\|P^+\tilde{z}(t)\|_{X_0^+}\le \|P^+{z}(t)\|_{X_0^+}+\|P^+\bar{z}(t)\|_{X_0^+}\le C(\varepsilon+\|\bar{z}(t)\|_{X_0})$$
for all $t\ge 0$. This implies the estimate
\begin{align}\label{eq:inst4}
\begin{split}
\|e^{-\kappa t}P^+\tilde{z}\|_{L_p(0,T;X_0^+)}&\le C\left(\varepsilon\left(\int_0^Te^{-\kappa p t}\ dt\right)^{1/p}+\|e^{-\kappa t}\bar{z}\|_{L_p(0,T;X_0)}\right)\\
&\le C(\kappa,p)\left(\varepsilon+\|e^{-\kappa t}\bar{z}\|_{\tilde{\mathbb{E}}(T)}\right),
\end{split}
\end{align}
where
$$\mathbb{\tilde{E}}(T):=\mathbb{E}_u(T)\times\mathbb{E}_h(T),$$
and $\mathbb{\tilde{E}}(T)\hookrightarrow L_p(0,T;X_0)$, with an embedding constant being independent of $T>0$. Employing the relation
\begin{equation}\label{eq:inst5}
\frac{d}{dt}(e^{-\kappa t}P^+\tilde{z}(t))=(-\kappa I+L^+)e^{-\kappa t}P^+\tilde{z}(t)+e^{-\kappa t}P^+Q\bar{z}(t),
\end{equation}
we obtain that
\begin{equation}\label{eq:inst6}
\|e^{-\kappa t} P^+\tilde{z}\|_{\mathbb{Z}(T)}\le C_1(\varepsilon+\|e^{-\kappa t}\bar{z}\|_{\mathbb{\tilde{E}}(T)}),
\end{equation}
where the constant $C_1>0$ does not depend on $T>0$. Here we have set
$$\mathbb{Z}(T):=H_p^1(0,T;X_0)\cap L_p(0,T;D(L)).$$
For the function $e^{-\kappa t}P^-\tilde{z}(t)$ there holds the identity
\begin{equation}\label{eq:inst6.1}
\frac{d}{dt}(e^{-\kappa t}P^-\tilde{z}(t))=(-\kappa I+L^-)e^{-\kappa t}P^-\tilde{z}(t)+e^{-\kappa t}P^-Q\bar{z}(t).
\end{equation}
Since by \eqref{eq:inst1} the semigroup generated by $(-\kappa I+L^-)$ is exponentially stable in $X_0^-$, we obtain from $L_p$-maximal regularity theory that the estimate
\begin{align}\label{eq:inst7}
\begin{split}
\|e^{-\kappa t}P^-\tilde{z}\|_{\mathbb{Z}(T)}&\le M\left(\|P^-\tilde{z}_0\|_{X_\gamma^0}+\|e^{-\kappa t}P^-Q\bar{z}\|_{L_p(0,T;X_0)}\right)\\
&\le M \left(\|P^-\tilde{z}_0\|_{X_\gamma^0}+\|e^{-\kappa t}\bar{z}\|_{\mathbb{\tilde{E}}(T)}\right)
\end{split}
\end{align}
is valid, with some constant $M>0$ that does not depend on $T>0$. A combination of \eqref{eq:inst6} and \eqref{eq:inst7} implies
\begin{equation}\label{eq:inst8}
\|e^{-\kappa t}\tilde{z}\|_{\mathbb{Z}(T)}\le C_2\left(\varepsilon+\|P^-\tilde{z}_0\|_{X_\gamma^0}+\|e^{-\kappa t}\bar{z}\|_{\mathbb{\tilde{E}}(T)}\right),
\end{equation}
with $C_2>0$ being independent of $T>0$. In what follows, we want to reproduce the norm of $e^{-\kappa t}\tilde{z}$ in $\tilde{\mathbb{E}}(T)$ on the left hand side of \eqref{eq:inst8}. To this end we have to estimate $e^{-\kappa t}\tilde{h},e^{-\kappa t}\partial_t\tilde{h}$ in $W_p^{1-1/2p}(0,T;L_p(\Sigma))$.

To estimate $e^{-\kappa t}\tilde{h}$ in $W_p^{1-1/2p}(0,T;L_p(\Sigma))$ we cannot simply use interpolation of $H_p^1(0,T;L_p(\Sigma))$ with $L_p(0,T;L_p(\Sigma))$, since the interpolation constant would depend on $T>0$. The following proposition takes care about this problem.
\begin{prop}\label{prop:auxinst1}
Let $T\in (0,\infty)$, $\kappa>0$ and let $\tilde{z}\in\mathbb{Z}(T)$ be the unique solution to \eqref{eq:tildeuh}. Then there exists $\hat{z}\in\mathbb{Z}(\mathbb{R}_+)$ with $\hat{z}|_{[0,T]}=\tilde{z}$ such that the estimate
$$\|e^{-\kappa t}\hat{z}\|_{\mathbb{Z}(\mathbb{R}_+)}\le M\left(\|\tilde{z}_0\|_{X_\gamma^0}+\|e^{-\kappa t}\bar{z}\|_{L_p(0,T;X_0)}+\|e^{-\kappa t}\tilde{z}\|_{L_p(0,T;X_0)}\right)$$
is valid, with a constant $M>0$ being independent of $T>0$.
\end{prop}
\begin{proof}
We fix $a>0$ large enough such that the operator $L-aI$ has the property of $L_p$-maximal regularity on $\mathbb{R}_+$. Define a function $f:\mathbb{R}_+\to X_0$ by
$$f(t):=
\begin{cases}
\omega Q\bar{z}(t)+a\tilde{z}(t),&\ \text{if}\ t\in[0,T],\\
0,&\ \text{if}\ t>T.
\end{cases}
$$
Then $f\in L_p(\mathbb{R}_+;X_0)$ and we may solve the problem
\begin{equation}\label{eq:inst9}
\partial_t \hat{z}-(L-aI)\hat{z}=f,\quad \hat{z}|_{t=0}=\tilde{z}_0,
\end{equation}
to obtain a unique solution $\hat{z}\in\mathbb{Z}(\mathbb{R}_+)$. Observe that by the uniqueness of the solution of \eqref{eq:tildeuh} it holds that $\hat{z}|_{[0,T]}=\tilde{z}$.

Multiplying \eqref{eq:inst9} by $e^{-\kappa t}$, it follows that the function $e^{-\kappa t}\hat{z}(t)$ solves the initial value problem
$$
\partial_t (e^{-\kappa t}\hat{z})-(L-(a+\kappa)I)e^{-\kappa t}\hat{z}=e^{-\kappa t}f,\quad \hat{z}|_{t=0}=\tilde{z}_0.
$$
Since the operator $L-(a+\kappa)I$ has $L_p$-maximal regularity on $\mathbb{R}_+$ as well, we obtain the desired estimate. The independence of the constant $M>0$ with respect to $t$ follows from the exponential stability of the analytic semigroup which is generated by $L-(a+\kappa)I$.
\end{proof}
Since $\|e^{-\kappa t}\tilde{z}\|_{W_p^{1-1/2p}(0,T;X_0)}\le \|e^{-\kappa t}\hat{z}\|_{W_p^{1-1/2p}(\mathbb{R}_+;X_0)}$ (here we use the intrinsic norm in $W_p^{1-1/2p}$) it follows by the real interpolation method and Proposition \ref{prop:auxinst1} that the estimate
\begin{align}\label{eq:inst10}
\begin{split}
\|e^{-\kappa t}\tilde{z}\|_{W_p^{1-1/2p}(0,T;X_0)}&\le M\left(\|\tilde{z}_0\|_{X_{\gamma}^0}
+\|e^{-\kappa t}\bar{z}\|_{L_p(0,T;X_0)}+\|e^{-\kappa t}\tilde{z}\|_{L_p(0,T;X_0)}\right)\\
&\le M\left(\|\tilde{z}_0\|_{X_{\gamma}^0}+\|e^{-\kappa t}\bar{z}\|_{\tilde{\mathbb{E}}(T)}+\|e^{-\kappa t}\tilde{z}\|_{\mathbb{Z}(T)}\right)
\end{split}
\end{align}
is valid. The second equation in \eqref{eq:tildeuh} and Proposition \ref{prop:auxinst1} together with trace theory imply
\begin{align}
\begin{split}\label{eq:inst11}
\|e^{-\kappa t}\partial_t&\tilde{h}\|_{W_p^{1-1/2p}(0,T;L_p(\Sigma))}\\
&\le C_3\left(\|e^{-\kappa t}\tilde{u}\|_{W_p^{1-1/2p}(0,T;L_p(\Sigma))}+\|e^{-\kappa t}\bar{h}\|_{W_p^{1-1/2p}(0,T;L_p(\Sigma))}\right)\\
&\le C_4\left(\|\tilde{z}_0\|_{X_{\gamma}^0}+\|e^{-\kappa t}\bar{z}\|_{\tilde{\mathbb{E}}(T)}+\|e^{-\kappa t}\tilde{z}\|_{\mathbb{Z}(T)}\right).
\end{split}
\end{align}
Observe that for the estimate of $e^{-\kappa t}\bar{h}$ we have used the fact that $$\tilde{\mathbb{E}}(T)\hookrightarrow W_p^{1-1/2p}(0,T;L_p(\Sigma))$$
with an embedding constant being independent of $T>0$, since the (instrinsic) norm in the last space is a part of the norm in $\tilde{\mathbb{E}}(T)$. Combining \eqref{eq:inst8} with \eqref{eq:inst10} \& \eqref{eq:inst11} we obtain
\begin{equation}\label{eq:inst12}
\|e^{-\kappa t}\tilde{z}\|_{\tilde{\mathbb{E}}(T)}\le C_5\left(\varepsilon+\|\tilde{z}_0\|_{X_\gamma^0}+\|P^-\tilde{z}_0\|_{X_\gamma^0}
+\|e^{-\kappa t}\bar{z}\|_{\mathbb{\tilde{E}}(T)}\right),
\end{equation}
with a constant $C_5>0$ being independent of $T>0$.

We are now turning our attention to the system \eqref{eq:barh} for $\bar{w}=(\bar{u},\bar{\pi},\bar{q},\bar{h})$ which we write shortly as $\mathbb{L}_\omega\bar{w}=N(\tilde{w}+\bar{w})$ with initial condition $\bar{z}|_{t=0}=(\phi(\tilde{z}_0),0)$. It will be convenient to write $N(w)=N_1(z)+N_2(z,\pi)$, where all components of $N_2(z,\pi)$ are zero except for the first one, which is given by $M_0(h)\nabla\pi$.
\begin{prop}\label{prop:estN}
Let $\kappa\ge0$. There exists a nondecreasing function $\alpha:\mathbb{R}_+\to\mathbb{R}_+$ with $\alpha(\varepsilon)\to 0$ as $\varepsilon\to 0$ such that
\begin{enumerate}
\item if $z\in{\mathbb{Z}}(\mathbb{R}_+)$, then
$$\|e^{-\kappa t}N_1(z)\|_{\mathbb{F}(\mathbb{R}_+)}\le \alpha (\varepsilon)\|e^{-\kappa t}z\|_{\mathbb{Z}(\mathbb{R}_+)},$$
whenever $\|z(t)\|_{X_\gamma}\le \varepsilon$ for all $t\ge 0$;
\item if $\hat{z}\in\!_0{\mathbb{Z}}(T)$ and $z_*\in{\mathbb{Z}}(\mathbb{R}_+)$, then
$$\|e^{-\kappa t}N_1(\hat{z}+z_*)\|_{\mathbb{F}(T)}\le \alpha (\varepsilon)C\left(\|e^{-\kappa t}\hat{z}\|_{\mathbb{Z}(T)}+\|e^{-\kappa t}z_*\|_{\mathbb{Z}(\mathbb{R}_+)}\right),$$
whenever
$$\|\hat{z}(t)\|_{X_\gamma}\le C\varepsilon$$
for all $t\in[0,T]$
and
$$\|z_*(t)\|_{X_\gamma}\le C\varepsilon$$
for all $t\ge 0$. The constant $C>0$ does not depend on $T>0$.
\end{enumerate}
\end{prop}
\begin{proof}
The proof of the first assertion follows by similar arguments as in \cite[Proposition 9]{LPS06}.

Therefore we concentrate on the proof of the second assertion. For $\hat{z}\in \!_0\tilde{\mathbb{E}}(T)$ we define a bounded linear extension operator $E:\!_0{\mathbb{Z}}(T)\to \!_0{\mathbb{Z}}(\mathbb{R}_+)$ by
$$(E\hat{z})(t):=
\begin{cases}
\hat{z}(t),&\quad t\in[0,T],\\
\hat{z}(2T-t),&\quad t\in [T,2T],\\
0,&\quad t\ge 2T.
\end{cases}
$$
For the norm of $e^{-\kappa t}(E\hat{z})$ in $\mathbb{Z}(\mathbb{R}_+)$ we then obtain
\begin{align*}
\|e^{-\kappa t}E\hat{z}\|_{\mathbb{Z}(\mathbb{R}_+)}^p&=\int_0^T e^{-\kappa tp}\|\hat{z}(t)\|_{X_1}^p dt+\int_T^{2T}e^{-\kappa tp}\|\hat{z}(2T-t)\|_{X_1}^p dt\\
&+\int_0^T e^{-\kappa tp}\|\dot{\hat{z}}(t)\|_{X_0}^p dt+\int_T^{2T}e^{-\kappa tp}\|\dot{\hat{z}}(2T-t)\|_{X_0}^p dt\\
&=\int_0^T e^{-\kappa tp}\|\hat{z}(t)\|_{X_1}^p dt+\int_0^{T}e^{-\kappa (2T-\tau)p}\|\hat{z}(\tau)\|_{X_1}^p d\tau\\
&+\int_0^T e^{-\kappa tp}\|\dot{\hat{z}}(t)\|_{X_0}^p dt+\int_0^{T}e^{-\kappa (2T-\tau)p}\|\dot{\hat{z}}(\tau)\|_{X_0}^p d\tau\\
&\le \|e^{-\kappa t}\hat{z}\|_{\mathbb{Z}(T)},
\end{align*}
since $2T-\tau\ge\tau$ for $\tau\in [0,T]$.

In addition there holds $\|(E\hat{z})(t)\|_{W_p^{2-2/p}\times W_p^{3-2/p}}\le C\varepsilon$ for all $t\ge 0$. Then the first assertion yields
\begin{align*}
\|e^{-\kappa t}N_1(\hat{z}+z_*)\|_{\mathbb{F}(T)}&\le\|e^{-\kappa t}N_1(E\hat{z}+z_*)\|_{\mathbb{F}(\mathbb{R}_+)}\\
&\le \alpha(\varepsilon)C\|e^{-\kappa t}(E\hat{z}+z_*)\|_{\mathbb{Z}(\mathbb{R}_+)}\\
&\le \alpha(\varepsilon)C\left(\|e^{-\kappa t}\hat{z}\|_{\mathbb{Z}(T)}+\|e^{-\kappa t}z_*\|_{\mathbb{Z}(\mathbb{R}_+)}\right).
\end{align*}
\end{proof}
In order to apply this proposition to the system $\mathbb{L}_\omega\bar{w}=N(\bar{w}+\tilde{w})$, let $z_*$ be an extension of $z_0$ such that $e^{-\kappa t}z_*\in\tilde{\mathbb{E}}(\mathbb{R}_+)$ and $\|z_*\|_{\mathbb{Z}(\mathbb{R}_+)}\le C\|z_0\|_{X_\gamma}$. The existence of such an extension can be seen as in Step 1 of the proof. Then we use the representation $N(w)=N_1(z)+N_2(z,\pi)$ and the identity $N_1(z)=N_1(z-z_*+z_*)=N_1(\hat{z}+z_*)$, where $\hat{z}:=(z-z_*)\in\!_0\mathbb{Z}(T)$. Finally, note that
$$\|e^{-\kappa t}N_2(z,\pi)\|_{L_p(0,T;L_p(\Omega))}\le C\varepsilon\|e^{-\kappa t}\pi\|_{\mathbb{E}_\pi(T)}.$$
Therefore, the second assertion of Proposition \ref{prop:estN} implies the estimate
\begin{align*}
\|e^{-\kappa t}N(\bar{w}+\tilde{w})\|_{\mathbb{F}(T)}&\le \alpha(\varepsilon)C\left(\|e^{-\kappa t}\tilde{z}\|_{\mathbb{Z}(T)}+\|e^{-\kappa t}\bar{z}\|_{\mathbb{Z}(T)}+\|e^{-\kappa t}z_*\|_{\mathbb{Z}(\mathbb{R}_+)}\right)\\
&\hspace{1cm}+\varepsilon C\left(\|e^{-\kappa t}\tilde{\pi}\|_{\mathbb{E}_\pi(T)}
+\|e^{-\kappa t}\bar{\pi}\|_{\mathbb{E}_\pi(T)}\right)\\
&\le\alpha_1(\varepsilon)\left(\|e^{-\kappa t}\tilde{z}\|_{\tilde{\mathbb{E}}(T)}+\|e^{-\kappa t}\bar{z}\|_{\tilde{\mathbb{E}}(T)}+\|e^{-\kappa t}\bar{\pi}\|_{\mathbb{E}_\pi(T)}+\|z_0\|_{X_\gamma}\right),
\end{align*}
where $\alpha_1(\varepsilon):=\alpha(\varepsilon)+\varepsilon\to 0$ as $\varepsilon\to 0$. Here we have used the estimates $\|e^{-\kappa t}z_*\|_{\mathbb{Z}(\mathbb{R}_+)}\le C\|z_0\|_{X_\gamma}$ and
$$\|e^{-\kappa t}\tilde{\pi}\|_{\mathbb{E}_\pi(T)}\le C\left(\|e^{-\kappa t}\tilde{z}\|_{\tilde{\mathbb{E}}(T)}+\|e^{-\kappa t}\bar{z}\|_{\tilde{\mathbb{E}}(T)}\right),$$
which hold for some constant $C>0$ that does not depend on $T>0$. Note also that $\tilde{\mathbb{E}}(T)\hookrightarrow\mathbb{Z}(T)$ with a universal embedding constant being independent of $T>0$ and $\|\hat{z}(t)\|_{X_\gamma}\le (1+C)\varepsilon$ for all $t\in [0,T]$, $\|z_*(t)\|_{X_\gamma}\le C\varepsilon$ for all $t\ge 0$.

By the invertibility of $\mathbb{L}_\omega$ we obtain
\begin{align}
\begin{split}\label{eq:inst13}
\|e^{-\kappa t}\bar{w}\|_{{\mathbb{E}}(T)}&\le C_6\left(\|\phi(\tilde{z}_0)\|_{X_\gamma}
+\|e^{-\kappa t}N(\bar{w}+\tilde{w})\|_{\mathbb{{F}}(T)}\right)\\
&\le C_6\Big(\|\phi(\tilde{z}_0)\|_{X_\gamma}+\alpha_1(\varepsilon)(\|e^{-\kappa t}\bar{z}\|_{\tilde{\mathbb{E}}(T)}+\|e^{-\kappa t}\tilde{z}\|_{\tilde{\mathbb{E}}(T)}\\
&\hspace{1cm}+\|z_0\|_{X_\gamma})\Big).
\end{split}
\end{align}
Choose $\varepsilon>0$ sufficiently small, such that $C_6\alpha_1(\varepsilon)\le 1/2$ and note that
$$\|e^{-\kappa t}\bar{w}\|_{\mathbb{E}(T)}=\|e^{-\kappa t}\bar{z}\|_{\tilde{\mathbb{E}}(T)}+\|e^{-\kappa t}\bar{\pi}\|_{\mathbb{E}_\pi(T)}+
\|e^{-\kappa t}\Jump{\bar{\pi}}\|_{\mathbb{E}_q(T)}.$$
This implies the estimate
\begin{equation}\label{eq:inst14}
\|e^{-\kappa t}\bar{z}\|_{\tilde{\mathbb{E}}(T)}\le 2C_6\left(\|\phi(\tilde{z}_0)\|_{X_\gamma}+\alpha_1(\varepsilon)(\|e^{-\kappa t}\tilde{z}\|_{\tilde{\mathbb{E}}(T)}+\|z_0\|_{X_\gamma})\right).
\end{equation}
If $\varepsilon>0$ is sufficiently small, we obtain from \eqref{eq:inst12} and \eqref{eq:inst14} that
\begin{multline}\label{eq:inst16}
\|e^{-\kappa t}\tilde{z}\|_{\tilde{\mathbb{E}}(T)}+\|e^{-\kappa t}\bar{z}\|_{\tilde{\mathbb{E}}(T)}
\le C_7\left(\varepsilon+\|\tilde{z}_0\|_{X_\gamma^0}+\|P^-\tilde{z}_0\|_{X_\gamma^0}
+\|\phi(\tilde{z}_0)\|_{X_\gamma}\right)
\end{multline}
with $C_7>0$ being independent of $T>0$, where me made use of the fact that $z_0=\tilde{z}_0+\phi(\tilde{z}_0)$. In particular this shows that
$$e^{-\kappa t}\tilde{z},e^{-\kappa t}\bar{z}\in\tilde{\mathbb{E}}(\mathbb{R}_+).$$
This in turn yields that
\begin{multline*}
e^{-\kappa t}\int_t^\infty\|e^{L^+(t-s)}P^+\omega Q\bar{z}(s)\|_{X_0}\ ds\\
\le
M\left(\int_t^\infty e^{\eta p'(t-s)}\ ds\right)^{1/p'}\|e^{-\kappa t}\omega \bar{z}\|_{L_p(\mathbb{R}_+;X_0)}\le C(\eta,p')\|e^{-\kappa t}\omega \bar{z}\|_{\tilde{\mathbb{E}}(\mathbb{R}_+)}<\infty.
\end{multline*}
For the projection of the solution $\tilde{z}$ of \eqref{eq:tildeuh} to $X_0^+$ we have the variation of parameters formula
\begin{align*}
P^+\tilde{z}(t)&=P^+ e^{L^+ t}\tilde{z}_0+\int_0^t e^{L^+(t-s)}P^+ \omega Q\bar{z}(s) ds\\
&=P^+ e^{L^+ t}\tilde{z}_0+\int_0^\infty e^{L^+(t-s)}P^+ \omega Q\bar{z}(s) ds-\int_t^\infty e^{L^+(t-s)}P^+\omega Q\bar{z}(s) ds
\end{align*}
at our disposal. Since $e^{L^+ t}$ extends to a $C_0$-group, we obtain the identity
$$e^{-L^+ t}\left(P^+\tilde{z}(t)+\int_t^\infty e^{L^+(t-s)}P^+\omega Q\bar{z}(s) ds\right)=
P^+\tilde{z}_0+\int_0^\infty e^{-L^+s}P^+ \omega Q\bar{z}(s) ds,$$
which holds for all $t\ge 0$. The left hand side of this equation may be estimated in $X_0$ as follows.
\begin{align*}
\|e^{-L^+ t}\Big(P^+\tilde{z}(t)&+\int_t^\infty e^{L^+(t-s)}P^+\omega Q\bar{z}(s) ds\Big)\|_{X_0}\\
&\le M e^{-(\kappa+\eta)t}\left(\|\tilde{z}(t)\|_{X_0}+\int_t^\infty\|e^{L^+(t-s)}P^+\omega Q\bar{z}(s)\|_{X_0}\ ds\right)\\
&\le Me^{-\eta t}\left(\|e^{-\kappa t}\tilde{z}(t)\|_{X_0}+C\right).
\end{align*}
Here we made use of the fact that the integral does not grow faster than $e^{\kappa t}$ by the computations above. Since the function $[t\mapsto \|e^{-\kappa t}\tilde{z}(t)\|_{X_0}]$ is bounded (see above) it follows that
$$e^{-\eta t}\left(\|e^{-\kappa t}\tilde{z}(t)\|_{X_0}+C\right)\to 0$$
as $t\to\infty$. This shows in particular that $P^+\tilde{z}_0+\int_0^\infty e^{-L^+s}P^+ \omega Q\bar{z}(s) ds=0$, hence the relation \eqref{eq:inst3} holds.

From \eqref{eq:inst3} and Young's inequality we obtain the estimate
$$\|e^{-\kappa t}P^+\tilde{z}\|_{L_p(\mathbb{R}_+;X_0)}\le M(\eta)\|e^{-\kappa t}P^+\bar{z}\|_{L_p(\mathbb{R}_+;X_0)}.$$
By \eqref{eq:inst5} this yields
\begin{equation}\label{eq:inst17}
\|e^{-\kappa t}P^+\tilde{z}\|_{\mathbb{Z}(\mathbb{R}_+)}\le M(\eta)\|e^{-\kappa t}P^+\bar{z}\|_{\tilde{\mathbb{E}}(\mathbb{R}_+)}.
\end{equation}
One may now restart right after \eqref{eq:inst8} and imitate all the estimates with the interval $[0,T]$ being replaced by $\mathbb{R}_+$ to obtain the relation
\begin{equation}\label{eq:inst18}
\|e^{-\kappa t}\tilde{z}\|_{\tilde{\mathbb{E}}(\mathbb{R}_+)}+\|e^{-\kappa t}\bar{z}\|_{\tilde{\mathbb{E}}(\mathbb{R}_+)}
\le C\left(\|P^-\tilde{z}_0\|_{X_\gamma}
+\|\phi(\tilde{z}_0)\|_{X_\gamma}\right).
\end{equation}
At this point we want to emphasize that the term $\|\tilde{z}_0\|_{X_\gamma^0}$ does not appear on the right hand side of \eqref{eq:inst18}, since on $\mathbb{R}_+$ there is no need to apply Proposition \ref{prop:auxinst1}. Furthermore, since we estimate norms on the half line $\mathbb{R}_+$, we may use the first assertion of Proposition \ref{prop:estN} instead of the second one.

Formula \eqref{eq:inst3} for $t=0$ and \eqref{eq:inst18} then imply
\begin{multline*}
\|P^+\tilde{z}_0\|_{X_\gamma^0}\le M(\omega,\eta)\|e^{-\kappa t}\bar{z}\|_{L_\infty(\mathbb{R}_+;X_\gamma^0)}\le M_1(\omega,\eta)\|e^{-\kappa t}\bar{z}\|_{\tilde{\mathbb{E}}(\mathbb{R}_+)}\\
\le C\left(\|P^-\tilde{z}_0\|_{X_\gamma^0}
+\|\phi(\tilde{z}_0)\|_{X_\gamma}\right),
\end{multline*}
since $\tilde{\mathbb{E}}(\mathbb{R}_+)\hookrightarrow BUC(\mathbb{R}_+;X_\gamma^0)$. Due to the fact that $\phi(0)=0$ and $\phi'(0)=0$, we may decrease $\delta>0$ (if necessary) to achieve that
$$\|\phi(\tilde{z}_0)\|_{X_\gamma}\le \frac{1}{2}(\|P^-\tilde{z}_0\|_{X_\gamma^0}+\|P^+\tilde{z}_0\|_{X_\gamma^0}),$$
whenever $\tilde{z}_0\in \delta B_{X_\gamma^0}(0)$.
Finally, this yields the relation
$$\|P^+\tilde{z}_0\|_{X_\gamma^0}\le C\|P^-\tilde{z}_0\|_{X_\gamma^0}.$$
Choosing $\tilde{z}_0\in \delta B_{X_\gamma^0}(0)$ in such a way that $P^-\tilde{z}_0=0$ and $P^+\tilde{z}_0\neq 0$ we have a contradiction. The proof is complete.
\end{proof}
We complete this section by considering the special case $G=B_R(0)$ and give a result on stability in dependence on the radius $R>0$.
\begin{cor}
Let the conditions of Theorem \ref{mainresultstab} be satisfied and let the surface tension $\sigma>0$ be fixed. Denote by $\lambda_1^*>0$ the first nontrivial eigenvalue of the negative Neumann Laplacian on the unit ball $B_1(0)$. Then the following assertions hold.
\begin{enumerate}
\item If $R^2\Jump{\rho}\gamma_a/\sigma<\lambda_1^*$, then $(u_*,h_*)=(0,0)$ is exponentially stable in the sense of Theorem \ref{mainresultstab}.
\item If $\Jump{\rho}>0$ and $R^2\Jump{\rho}\gamma_a/\sigma>\lambda_1^*$, then $(u_*,h_*)=(0,0)$ is unstable in the sense of Theorem \ref{mainresultstab}.
\end{enumerate}
\end{cor}
\begin{proof}
The assertions follow from Theorem \ref{mainresultstab}. Indeed, denoting by $\lambda_1(R)>0$ the first nontrivial eigenvalue of the Neumann Laplacian on $B_R(0)$, Theorem \ref{mainresultstab} yields that $(0,0)$ is exponentially stable if $\Jump{\rho}\gamma_a/\sigma<\lambda_1(R)$ and unstable if  $\Jump{\rho}\gamma_a/\sigma>\lambda_1(R)$ and $\Jump{\rho}>0$. An easy computation yields that $\lambda_1(R)=\lambda_1^*/R^2$. This concludes the proof of the corollary.
\end{proof}


%
%
%

\chapter{Bifurcation at a multiple eigenvalue}\label{chptr:bifurc}

In this chapter we consider the special case $G=B_R:=B_R(0)\subset\mathbb{R}^2$ for some radius $R>0$. Proposition \ref{spectrumLR} implies that an eigenvalue of the linearization $L$ crosses the imaginary axis through zero if $\Jump{\rho}\gamma_a/\sigma=\lambda_1$, where $\lambda_1>0$ is the first nontrivial eigenvalue of the negative Neumann Laplacian in $G$. This suggests that $(\lambda_1,0)$ is a bifurcation point for the nonlinear Navier-Stokes system \eqref{eq:NScap2}. Unfortunately, the eigenvalue $\lambda_1>0$ is not simple. Indeed, it is a double eigenvalue being semi-simple. Therefore we cannot directly apply the results of Crandall \& Rabinowitz. Instead, we will use certain symmetry properties of the bifurcation equation to reduce it to a purely one dimensional bifurcation equation which then can be solved by the implicit function theorem. For a general theory concerning bifurcation at multiple eigenvalues, we refer the reader to \cite{Kie85,Sat73,Van82}.

We recall that the set of equilibria $\mathcal{E}$ for height functions $h$ with vanishing mean value is given by
$$
\mathcal{E}=\{(u_*,\pi_*,q_*,h_*):u_*=0,\ \pi_*=const., q_*=\Jump{\pi_*}=0,\ h_*\ \text{solves \eqref{eq:bifurc1}}\}.
$$
Note that if there exist nontrivial equilibria, i.e. $h_*\neq 0$, then these equilibria are determined by the nontrivial solutions of the quasilinear elliptic boundary value problem
\begin{align}\label{eq:bifurc1}
\begin{split}
\sigma \div_{x'}\left(\frac{\nabla_{x'} {h}}{\sqrt{1+|\nabla_{x'} {h}|^2}}\right)+\Jump{\rho}\gamma_a {h}&=0,\quad x'\in B_R(0),\\
\partial_{\nu_{B_R(0)}}{h}&=0,\quad x'\in\partial B_R(0).
\end{split}
\end{align}
Here the differential operators $\nabla_{x'}$ and $\div_{x'}$ act only in the variables ${x'}\in G$. We intend to show that if $\Jump{\rho}\gamma_a/\sigma=\lambda_1$, then there bifurcate nontrivial solutions $h_*$ of \eqref{eq:bifurc1} from the trivial solution $h=0$. To this end, let
\begin{align}\label{eq:bifurc1.1}
\begin{split}
X&:=\{h\in W_p^{1-1/p}(B_R):\int_{B_R}h d{x'}=0\},\\
Y&:=\{h\in W_p^{3-1/p}(B_R)\cap X:\partial_{B_R}h=0\},
\end{split}
\end{align}
and define $F:\mathbb{R}_+\times Y\to X$ by
\begin{equation}\label{eq:DefF}
F(\alpha,h):=\div_{x'}\left(\frac{\nabla_{x'} h}{\sqrt{1+|\nabla_{x'} h|^2}}\right)+\alpha h.
\end{equation}
For $h\in W_p^s(B_R)$, $s>0$, define $(\Gamma_{\mathcal{O}_\phi}h)(\bar{x}'):=h(\mathcal{O}_\phi \bar{x}')$, where
$$\mathcal{O}_\phi:=\begin{pmatrix}
\cos\phi & -\sin\phi\\
\sin\phi & \cos\phi
\end{pmatrix}$$
describes a two-dimensional rotation of $\bar{x}'\in B_R$ through the angle $\phi$. Note that $\mathcal{O}_\phi$ is an orthogonal matrix, i.e.\ $\mathcal{O}_\phi^{\sf T}=\mathcal{O}_\phi^{-1}$. Furthermore, we define $(\Gamma_\mathcal{R}h)(\bar{x}'):=h(\mathcal{R}\bar{x}')$, where $\mathcal{R}\bar{x}':=(\bar{x}_1,-\bar{x}_2)^{\sf T}$. It is easily seen that $\Gamma_j$ leaves both spaces $X$ and $Y$ invariant and one readily computes $\nabla_{\bar{x}'} (\Gamma_{\mathcal{O}_\phi} h)=\mathcal{O}_\phi^{\sf T}(\Gamma_{\mathcal{O}_\phi}\nabla_{x'} h)$, $\Delta_{\bar{x}'} (\Gamma_{\mathcal{O}_\phi} h)=\Gamma_{\mathcal{O}_\phi}\Delta_{x'} h$ and $\nabla_{\bar{x}'}^2(\Gamma_{\mathcal{O}_\phi} h)=\mathcal{O}_\phi^{\sf T}(\Gamma_{\mathcal{O}_\phi}\nabla_{x'}^2 h) \mathcal{O}_\phi$, where $\bar{x}'=\mathcal{O}_\phi^{\sf T} x'$.
Therefore, the identity
$$\div_{x'}\left(\frac{\nabla_{x'} h}{\sqrt{1+|\nabla_{x'} h|^2}}\right)=\frac{\Delta_{x'} h}{\sqrt{1+|\nabla_{x'} h|^2}}-\frac{(\nabla_{x'}^2 h\nabla_{x'} h|\nabla_{x'} h)}{\sqrt{1+|\nabla_{x'} h|^2}^3}$$
implies that $\Gamma_{\mathcal{O}_\phi} F(\alpha,h)=F(\alpha,\Gamma_{\mathcal{O}_\phi} h)$. Similarly it holds that $\Gamma_{\mathcal{R}} F(\alpha,h)=F(\alpha,\Gamma_{\mathcal{R}}h)$. This shows that $F$ is invariant with respect to the group operations of the orthogonal group $O(2)$.

\section{Lyapunov-Schmidt Reduction}\label{LSR}

By the smoothness of the mapping $[\mathbb{R}\ni s\mapsto(1+s^2)^{-1/2}]$ it holds that $F\in C^\infty(\mathbb{R}_+\times Y;X)$ and the first Fr\'{e}chet derivative of $F$ is given by
$$[D_h F(\alpha,h)]\hat{h}=\div_{x'}\left(\frac{\nabla_{x'}\hat{h}}{\sqrt{1+|\nabla_{x'} h|^2}}\right)-
\div_{x'}\left(\frac{\nabla_{x'} h(\nabla_{x'}\hat{h}|\nabla_{x'} h)}{\sqrt{1+|\nabla_{x'} h|^2}^3}\right)+\alpha\hat{h}.$$
Therefore it holds that $D_h F(\lambda_1,0)=\Delta_N+\lambda_1 I$, where $\Delta_N$ denotes the Neumann-Laplacian and $\lambda_1>0$ is the first eigenvalue of $-\Delta_N$ in $X$ (note that $0\notin\sigma(-\Delta_N)$, since all functions in $X$ have a vanishing mean value). For convenience, we set $A:=D_h F(\lambda_1,0)$. We claim that $0\in\sigma(A)$ is a semi-simple eigenvalue. Since the operator $A$ has a compact resolvent, it follows that the spectrum consists only of discrete eigenvalues having finite multiplicity. Therefore it suffices to show that $N(A)=N(A^2)$. To this end, let $0\neq v\in N(A^2)$ and $u:=Av$. Then $u\in N(A)$ and we compute
$$\|u\|_{L_2(B_R)}^2=(Av|u)_{L_2(B_R)}=(v|Au)_{L_2(B_R)}=0,$$
since $A$ is self-adjoint in $L_2(B_R)$. This shows that $u=0$, hence $v\in N(A)$ and $0\in \sigma(A)$ is semi-simple. We note on the go that this implies $X=N(A)\oplus R(A)$. Rewriting the eigenvalue problem $-\Delta_N h=\lambda h$ in polar coordinates $(r,\varphi)$, it follows that the kernel $N(A)$ of $A$ is spanned by the two linearly independent functions
\begin{equation}\label{eq:BesselFunc}
u_1^*(x'):=J_1(j_{1,1}'r/R)\cos\varphi,\quad u_2^*(x'):=J_1(j_{1,1}'r/R)\sin\varphi,
\end{equation}
$r\in [0,R],\ \varphi\in [0,2\pi)$, where $J_1$ is a Bessel function of first order and $j_{1,1}'$ denotes the first zero of the derivative $J_1'$ of $J_1$. Hence $\dim N(A)=2$ (notably, $A$ is a Fredholm operator of index zero). In particular, each $h\in X$ can be written in a unique way as $h=u+v$, where $u\in N(A)$ and $v\in R(A)$. Defining $Ph:=u$, it follows that $P:X\to N(A)$ is a projection onto $N(A)$. With $Q:=I-P$ we also have that $Q:X\to R(A)$ is onto and $Qh=v$. Moreover, it holds that $Y=U\oplus V$, where $U:=N(A)$ and $V:=R(A)\cap Y$.


Let us now split the equation $F(\alpha,h)=0$ into the two parts $PF(\alpha,u+v)=0$ and $QF(\alpha,u+v)=0$. Since the operator $D_vQF(\lambda_1,0)=QD_hF(\lambda_1,0):V\to R(A)$ is an isomorphism, we may solve the equation $QF(\alpha,u+v)=0$ in a neighborhood of $(\lambda_1,0)$ by the implicit function theorem, to obtain a unique smooth function $v_*:\mathbb{R}_+\times U\to V$ such that $QF(\alpha,u+v_*(\alpha,u))=0$ for all $(\alpha,u)$ close to $(\lambda_1,0)$. The function $v_*=v_*(\alpha,u)$ has the properties
\begin{enumerate}
\item $v_*(\alpha,0)=0$ if $\alpha>0$ is close to $\lambda_1$;
\item $D_\alpha v_*(\lambda_1,0)=0$, $D_uv_*(\lambda_1,0)=0$;
\item $\Gamma_jv_*(\alpha,u)=v_*(\alpha,\Gamma_ju)$, $j\in\{\mathcal{R},\mathcal{O}_\phi\}$ if $(\alpha,u)$ is close to $(\lambda_1,0)$.
\end{enumerate}
The first two properties follows directly from the (differentiated) equation $QF(\alpha,u+v_*,\alpha,u))=0$ and the fact that $F(\alpha,0)=0$ for each $\alpha\in\mathbb{R}_+$. The last property follows from the uniqueness of $v_*$ and the fact that $\Gamma_j QF(\alpha,u+v)=QF(\alpha,\Gamma_ju+\Gamma_j v)$, $j\in \{\mathcal{R},\mathcal{O}_\phi\}$. To see this, we differentiate the identity $\Gamma_j F(\alpha, u)=F(\alpha,\Gamma_j u)$ with respect to $u$ and evaluate the result at $(\alpha,u)=(\lambda_1,0)$ to obtain the relation
$$\Gamma_j A=A\Gamma_j.$$
In other words, $\Gamma_j$ commutes with the operator $A$. It follows readily that $\Gamma_j$ leaves $N(A)$ as well as $R(A)$ invariant, hence $\Gamma_j P=P\Gamma_j$ as well as $\Gamma_j Q= Q\Gamma_j$.

\section{Reduction to a 1-dimensional bifurcation equation}

It remains to study the equation $0=G(\alpha,u)$ for $(\alpha,u)\in\mathbb{R}_+\times U$ in some neighborhood of $(\lambda_1,0)$, where $G(\alpha,u):=PF(\alpha,u+v_*(\alpha,u))$. Let us remark that this equation is purely 2-dimensional. Similarly as above it holds that $\Gamma_j G(\alpha,u)=G(\alpha,\Gamma_j u)$ for $j\in\{\mathcal{R},\mathcal{O}_\phi\}$. Let $\Psi:U\to\mathbb{R}^2$ be defined by $\Psi(u):=(b_1,b_2)^{\sf T}$ for $u=b_1 u_1+b_2 u_2\in U$, $b_k:=(u|u_k)_{L_2(B_R)}\in \mathbb{R}$, where $u_j:=u_j^*/\|u_j^*\|_{L_2}$. It follows that $\Psi$ is an isomorphism with inverse $\Psi^{-1}$ given by $\Psi^{-1}(b_1,b_2)=b_1 u_1+b_2 u_2$. Consider now the equation
$$g(\alpha,b):=\Psi G(\alpha,\Psi^{-1}b)=0,\quad b\in\mathbb{R}^2,$$
and define $\Gamma_j^0:=\Psi\Gamma_j\Psi^{-1}$ on $\mathbb{R}^2$ for $j\in\{\mathcal{R},\mathcal{O}_\phi\}$. With these definitions it holds that $\Gamma_j^0g(\alpha,b)=g(\alpha,\Gamma_j^0 b)$ for $j\in\{\mathcal{R},\mathcal{O}_\phi\}$. A short computation also shows that
\begin{itemize}
\item $\Gamma_{\mathcal{O}_\phi}^0b=\mathcal{O}_\phi b$;
\item $\Gamma_{\mathcal{R}}^0b=\mathcal{R}b$;
\end{itemize}
hold for each $b\in\mathbb{R}^2$. We will use these two properties to reduce $g(\alpha,b)=0$ to a purely one dimensional equation. Choose $\phi$ in such a way that $\mathcal{O}_\phi b=s e_1=(s,0)^{\sf T}$ for some $s\in\mathbb{R}$ close to $0$. Then $g(\alpha,b)=0$ if and only if $g(\alpha,s e_1)=0$ by the first property. Furthermore $\mathcal{R}e_1=e_1$, hence $$g(\alpha,s e_1)=g(\alpha,s \mathcal{R}e_1)=\mathcal{R}g(\alpha,s e_1).$$
This in turn yields that $g_2(\alpha,se_1)=0$ is always satisfied and therefore we have reduced the equation $g(\alpha,b)=0$ to $g_1(\alpha,s e_1)=0$ for $(\alpha,s)\in\mathbb{R}_+\times\mathbb{R}$ close to $(\lambda_1,0)$.

Due to the fact that $D_\alpha g_1(\lambda_1,0)=0$, we cannot simply solve the equation $g_1(\alpha,se_1)=0$ for $\alpha$ in a neighborhood of $(\lambda_1,0)$ by the implicit function theorem.
To this end we define a new function
$$\tilde{g}(\alpha,s):=\begin{cases}
g_1(\alpha,se_1)/s,&\ s\neq 0,\\
D_b g_1(\alpha,0)e_1,&\ s=0.
\end{cases}
$$
Since $D_b g_1(\lambda_1,0)=0$, we have $\tilde{g}(\lambda_1,0)=0$. Moreover we compute
$$D_\alpha \tilde{g}(\lambda_1,0)=D_\alpha D_bg_1(\lambda_1,0)e_1.$$
Since $D_\alpha D_h F(\lambda_1,0)=I$ and
$$D_\alpha D_bg(\lambda_1,0)e_1=\Psi PD_\alpha D_h F(\lambda_1,0)\Psi^{-1}e_1=e_1,$$
it follows that $D_\alpha D_bg_1(\lambda_1,0)e_1=1\neq 0$. Hence, the implicit function theorem yields the existence of a smooth function $\alpha:(-\eta,\eta)\to\mathbb{R}$ with $\alpha(0)=\lambda_1$, such that $\tilde{g}(\alpha(s),s)=0$ for all $s\in (-\eta,\eta)$ and some (small) $\eta>0$. This in turn yields the following result.
\begin{thm}\label{thm:bifurc1}
Modulo the action in $O(2)$, all solutions of $F(\alpha,h)=0$ in a neighborhood $\mathcal{U}$ of $(\lambda_1,0)$ in $\mathbb{R}_+\times Y$ are given by
$$F^{-1}(0)\cap\mathcal{U}=\{(\alpha(s),su_1+y(s)):|s|<\eta\}\cup\{(\alpha,0):
(\alpha,0)\in\mathcal{U}\},$$
where $\alpha\in C^\infty((-\eta,\eta);\mathbb{R})$ with $\alpha(0)=\lambda_1>0$ and $y\in C^\infty((-\eta,\eta);R(A)\cap Y)$ with $y(0)=y'(0)=0$ are uniquely determined.
\end{thm}
\begin{proof}
Define $y(s):=v_*(\alpha(s),s u_1)$. Then the assertions for $y$ follow from the properties of the function $v_*$.
\end{proof}
Let us now show that the bifurcation in $(\lambda_1,0)$ is of \emph{subcritical type}, i.e.\ $s\alpha'(s)<0$ for $0<|s|<\delta$ and some $\delta>0$. We first prove that $\alpha'(0)=0$. To this end we differentiate the expression $F(\alpha(s),su_1+y(s))=0$ with respect to $s$ twice and evaluate at $s=0$ to obtain
$$0=\Delta_N y''(0)+\lambda_1 y''(0)+2\alpha'(0)u_1.$$
Multiplying this identity with $u_1$ in $L_2(B_R)$ and integrating by parts yields\\ $\alpha'(0)\|u_1\|_{L_2(B_R)}^2=0$, since $u_1\in N(A)$. This implies that $\alpha'(0)=0$, since $u_1\neq 0$. Differentiating $F(\alpha(s),su_1+y(s))=0$ a third time yields in $s=0$
$$0=\Delta_N y'''(0)+\lambda_1 y'''(0)-3\div (\nabla u_1|\nabla u_1|^2)+3\alpha''(0)u_1,$$
where we have used the fact that $\alpha'(0)=0$. We test the latter equation by $u_1$ in $L_2(B_R)$ and integrate by parts to the result
$$0=\alpha''(0)\|u_1\|_{L_2(B_R)}^2+\|u_1\|_{L_4(B_R)}^4,$$
hence $\alpha''(0)<0$ since $u_1\neq 0$.
\begin{cor}\label{cor:bifurc1}
The bifurcation in Theorem \ref{thm:bifurc1} at $(\lambda_1,0)$ is of subcritical type, i.e.\ $s\alpha'(s)<0$ for $0<|s|<\delta$ and some $\delta>0$.
\end{cor}

\section{Instability of the bifurcating equilibria}

This section is devoted to the proof of instability of the bifurcating equilibria for the complete system \eqref{eq:NScap2}. For the sake of readability we shall replace the functions $u_1$ and $u_2$ from the preceding sections, which span the 2D-kernel of the operator $A$, by $h_1$ and $h_2$, respectively. Hence, it holds that $F(\alpha(s),sh_1+y(s))=0$ for all $|s|<\delta$ and some small $\delta>0$, where $y(s)=v_*(\alpha(s),sh_1)$ with $y(0)=y'(0)=0$.

\textbf{Step 1}. Let us first show that $0\in \sigma(D_h F(\alpha(s),s h_1+y(s)))$ for each $s\in (-\delta,\delta)$. For convenience we set $h_1(s):=s h_1+y(s)$ for  $s\in(-\delta,\delta)$ and $\Gamma_\phi:=\Gamma_{\mathcal{O}_\phi}$. The following Proposition will be of importance.
\begin{prop}\label{prop:bifurc1}
Let $|s|<\delta$ be fixed. The mapping $[\mathbb{R}\in\phi\mapsto\Gamma_\phi h_1(s)\in Y]$ is continuously Fr\'{e}chet differentiable and its derivative is given by
$$D_\phi \Gamma_\phi h_1(s)=s(I+D_2v_*(\alpha(s),s\Gamma_\phi h_1))D_\phi\Gamma_\phi h_1,$$
with $\Gamma_\phi h_1= h_1\cos\phi- h_2\sin\phi$ and $D_\phi\Gamma_\phi h_1=-(h_1\sin\phi+ h_2\cos\phi)$.
\end{prop}
\begin{proof}
By smoothness of $v_*$ and since $\Gamma_\phi v_*(\alpha(s),s h_1)=v_*(\alpha(s),s\Gamma_\phi h_1)$ it suffices to show that the mapping
$$[\mathbb{R}\in\phi\mapsto\Gamma_\phi h_1\in N(A)]$$
is continuously Fr\'{e}chet differentiable.

Observe that $(\Gamma_\phi h_1)(x')=J_1(j_{1,1}'r/R)\cos(\varphi+\phi)$, where $x'=(r\cos\varphi,r\sin\varphi)$. Since
$$\cos(\varphi+\phi)=\cos\phi\cos\varphi-\sin\phi\sin\varphi$$
it follows that
$$\Gamma_\phi h_1= h_1\cos\phi- h_2\sin\phi,$$
hence $\Gamma_\phi h_1\in N(A)$ for each $\phi\in\mathbb{R}$. Furthermore we may differentiate the last identity with respect to $\phi$ to obtain
$$D_\phi\Gamma_\phi h_1=-(h_1\sin\phi+ h_2\cos\phi)\in N(A).$$
This completes the proof.
\end{proof}
Applying Proposition \ref{prop:bifurc1} in connection with the identity $0=F(\alpha(s),\Gamma_\phi h_1(s))$, we obtain
$$0=D_\phi F(\alpha(s),\Gamma_\phi h_1(s))=D_hF(\alpha(s),\Gamma_\phi h_1(s))D_\phi\Gamma_\phi h_1(s).$$
Evaluating the last equation at $\phi=0$ yields
$$0=D_hF(\alpha(s),h_1(s))D_\phi\Gamma_\phi h_1(s)|_{\phi=0},$$
wherefore $D_\phi\Gamma_\phi h_1(s)|_{\phi=0}\in N(D_hF(\alpha(s),h_1(s))$. By Proposition \ref{prop:bifurc1} we have
$$D_\phi\Gamma_\phi h_1(s)|_{\phi=0}=-sh_2-sD_u v_*(\alpha(s),s h_1)h_2.$$
In particular
\begin{equation}\label{eq:eigenfunctioncurve}
h_2(s):=h_2+D_u v_*(\alpha(s),s h_1)h_2\neq 0,
\end{equation}
since $0\neq h_2\in N(A)$ and $D_u v_*(\alpha(s),s h_1)h_2\in R(A)$. Moreover, the identity
$$D_hF(\alpha(s),h_1(s))h_2(s)=0$$
holds for all $s\in (-\delta,\delta)$, i.e.\ $[s\mapsto h_2(s)]$ is a smooth nontrivial eigenfunction curve for the eigenvalue $0\in \sigma(D_h F(\alpha(s),h_1(s)))$.

\textbf{Step 2}. The goal of this step is to introduce an operator $\mathcal{L}$ as an analogue of the operator $L$ in Section \ref{EquilLin}, which represents the full linearization of \eqref{eq:NScap2} in one of the bifurcating equilibria given by Theorem \ref{thm:bifurc1}. As in Section \ref{EquilLin} we will first show how to reproduce the pressure out of a given velocity field and a given height function. For that purpose we define
$$
((u|v)):=\int_{\Omega}u(x)\cdot v(x)\det D\Theta_h(x) dx,
u\in L_p(\Omega)^3,\ v\in L_{p'}(\Omega)^3,
$$
where $h$ is an admissible height function, such that $\Theta_h$ and its inverse $\Theta_h^{-1}$ are well-defined. Then, for $f\in L_p(\Omega)^3$ and $g\in W_p^{1-1/p}(\Sigma)$, we consider the weak transmission problem
\begin{align}\label{eq:bifurc7}
\begin{split}
((\mathcal{M}(h)\nabla\pi|\mathcal{M}(h)\nabla\phi))&=(f|\nabla\phi)_{L_2},\quad \phi\in W_{p'}^1(\Omega),\\
\Jump{\rho\pi}&=g,\quad\text{on}\ \Sigma,
\end{split}
\end{align}
where $\mathcal{M}(h)\in\mathbb{R}^{3\times 3}$ and $\|\mathcal{M}(h)-I\|_{L_\infty}\to 0$ as $\|h\|_{W_\infty^1}\to 0$. We may rewrite the term in brackets $((\cdot|\cdot))$ as follows
\begin{multline*}
((\mathcal{M}(h)\nabla\pi|\mathcal{M}(h)\nabla\phi))=(\nabla\pi|\nabla\phi)_{L_2}
+(h\varphi'\mathcal{M}(h)^{\sf T}\mathcal{M}(h)\nabla \pi|\nabla\phi)_{L_2}+\\
+((\mathcal{M}(h)^{\sf T}\mathcal{M}(h)-I)\nabla \pi|\nabla\phi)_{L_2},
\end{multline*}
where we used the fact that $\det D\Theta_h=1+h\varphi'$. Note, that for each $\varepsilon>0$ there exists $\eta>0$ such that
$$|((\mathcal{M}(h)\nabla\pi|\mathcal{M}(h)\nabla\phi))-(\nabla\pi|\nabla\phi)_{L_2}|
\le\varepsilon \|\nabla\pi\|_{L_p(\Omega)}\|\nabla\phi\|_{L_{p'}(\Omega)},$$
provided that $\|h\|_{W_\infty^1(\Sigma)}\le\eta$. A Neumann series argument in combination with Lemma \ref{lem:appaux0} then implies that there exists $\eta_0>0$ such that for each $h\in W_\infty^1(\Sigma)$ with $\|h\|_{W_\infty^1(\Sigma)}\le\eta_0$, problem \eqref{eq:bifurc7} admits a unique solution $\pi\in \dot{W}_p^1(\Omega\backslash\Sigma)$.

If we linearize \eqref{eq:NScap2} around a bifurcating equilibrium $(0,c,0,h_1(s))$ (given by Theorem \ref{thm:bifurc1}, where $c$ is a constant) we obtain the following linear problem.
\begin{align}\label{eq:bifurc8}
\begin{split}
\partial_t(\rho u)-\mu\Delta u+(I-M_0(h_1(s)))\nabla\pi&=\mathcal{F}_u(u,h_1(s)),\quad \text{in}\ \Omega\backslash\Sigma,\\
\div u-(M_0(h_1(s))\nabla| u)&=0,\quad \text{in}\ \Omega\backslash\Sigma,\\
-\Jump{\mu\partial_3 v}-\Jump{\mu\nabla_{x'}w}&=\mathcal{G}_v(u,h_1(s)),\quad \text{on}\ \Sigma,\\
-2\Jump{\mu\partial_3 w}+\Jump{\pi}&=\mathcal{G}_w(u,h_1(s))+
\frac{\Jump{\rho}\gamma_a}{\alpha(s)}D_hF(\alpha(s),h_1(s))h,\quad \text{on}\ \Sigma,\\
\partial_t h-(u|e_3)&=-(v|\nabla_{x'} h_1(s)),\quad \text{on}\ \Sigma,\\
P_{S_1}\left(\mu(Du)\nu_{S_1}\right)&=0,\quad \text{on}\ S_1\backslash\partial\Sigma,\\
(u|\nu_{S_1})&=0,\quad \text{on}\ S_1\backslash\partial\Sigma,\\
u&=0,\quad \text{on}\ S_2,\\
\partial_{\nu_{\partial G}}h&=0,\quad \text{on}\ \Sigma,\\
u(0)&=u_0,\quad \text{in}\ \Omega\backslash\Sigma,\\
h(0)&=h_0,\quad \text{on}\ \Sigma.
\end{split}
\end{align}
Here the surface tension $\sigma=\sigma(s)$ is given by $\sigma(s)=\frac{\Jump{\rho}\gamma_a}{\alpha(s)}$ and we have set
$$\mathcal{F}_u(u,h):=-\mu(M_1(h):\nabla^2 u+M_2(h)\nabla u),$$
\begin{multline*}
\mathcal{G}_v(u,h):=-\Jump{\mu(\nabla_{x'} v+\nabla_{x'} v^{\sf T})}\nabla_{x'} h+|\nabla_{x'} h|^2\Jump{\mu\partial_3 v}+\\
+\left((1+|\nabla_{x'} h|^2)\Jump{\mu\partial_3 w}-(\nabla_{x'} h|\Jump{\mu\nabla_{x'} w})\right)\nabla_{x'} h,
\end{multline*}
$$\mathcal{G}_w(u,h):=-(\nabla_{x'} h|\Jump{\mu\nabla_{x'} w})-(\nabla_{x'} h|\Jump{\mu\partial_3 v})+|\nabla_{x'} h|^2\Jump{\mu \partial_4 w},$$
and
$$F(\alpha,h):=\div_{x'}\left(\frac{\nabla_{x'} h}{\sqrt{1+|\nabla_{x'} h|^2}}\right)+\alpha h.$$
For fixed $s\in (-\delta,\delta)$ we define a linear operator $\mathcal{L}(s):\mathcal{X}_1\to \mathcal{X}_0$ by
$$\mathcal{L}(s)(u,h):=
\begin{pmatrix}
\frac{\mu}{\rho}\Delta u-\frac{1}{\rho}(I-M_0(h_1(s)))\nabla p+\frac{1}{\rho}\mathcal{F}_u(u,h_1(s))\\
(u|e_3)-(v|\nabla_{x'} h_1(s))
\end{pmatrix}
$$
where $\mathcal{X}_0:=\tilde{L}_{p,\sigma}(\Omega)\times \{h\in W_p^{2-1/p}(\Sigma):\int_G h\ dx'=0,\ \partial_{\nu_{\partial G}}h=0\}$,
$$\tilde{L}_{p,\sigma}(\Omega):=\{u\in L_p(\Omega)^3:u\circ\Theta_{h_1(s)}^{-1}\in L_{p,\sigma}(\Theta_{h_1(s)}\Omega)\},\ \bar{\mathcal{X}}_1=H_p^2(\Omega\backslash \Sigma)^3\times W_p^{3-1/p}(\Sigma)$$
and
\begin{multline*}
\mathcal{X}_1:=D(\mathcal{L}(s))=\{(u,h)\in \mathcal{X}_0\cap \bar{\mathcal{X}}_1:P_{\Sigma}(\Jump{\mu Du}e_3)+\mathcal{G}_v(u,h_1(s))=0,\\ \Jump{u}=0,\ P_{S_1}\left(\mu(Du)\nu_{S_1}\right)=0,\ (u|\nu_{S_1})=0,\ \partial_{\nu_{\partial G}} h=0\}.
\end{multline*}
The function $p\in \dot{W}_p^1(\Omega\backslash \Sigma)$ in the definition of $\mathcal{L}(s)$ is determined as the solution of the weak transmission problem
\begin{align}\label{eq:bifurc9}
\begin{split}
\left(\left(\frac{1}{\rho}\mathcal{M}(h_1(s))\nabla p|\mathcal{M}(h_1(s))\nabla\phi\right)\right)&=\left(\left(\frac{\mu}{\rho}\Delta u+\mathcal{F}_u(u,h_1(s))|\mathcal{M}(h_1(s))\nabla\phi\right)\right),\\
\Jump{p}&=2\Jump{\mu\partial_{3}w}+\mathcal{G}_w(u,h_1(s))+\\
&\hspace{1cm}+\frac{\Jump{\rho}\gamma_a}{\alpha(s)}D_hF(\alpha(s),h_1(s))h,\quad\text{on}\ \Sigma.
\end{split}
\end{align}
where $\mathcal{M}(h):=I-M_0(h)$. The first equation has to be satisfied for each $\phi\in W_{p'}^1(\Omega)$. Note that the solution $p\in \dot{W}_p^1(\Omega\backslash \Sigma)$ is well-defined by the above considerations.

Observe also that for $s=0$ the operator $\mathcal{L}(0)$ coincides with the operator $L$ from Section \ref{EquilLin}, since then $h_1(0)=0$. Furthermore, for $u\in D(\mathcal{L}(s))$ we note that
$$\int_\Omega [u\cdot \mathcal{M}(h_1(s))\nabla\phi]\det D\Theta_{h_1(s)} dx=0.$$
This can be seen as follows. Let $\bar{u}:=u\circ\Theta^{-1}_{h_1(s)}$ and $\bar{\phi}:=\phi\circ\Theta^{-1}_{h_1(s)}$. Then we obtain from the transformation formula, integration by parts and the boundary conditions on $\bar{u}$ that
\begin{align*}
\int_\Omega [u\cdot \mathcal{M}(h_1(s))\nabla\phi]\det D\Theta_{h_1(s)} dx&=\int_{\Omega} [\bar{u}\cdot ({\nabla}\bar{\phi})](\Theta_{h_1(s)})\det D\Theta_{h_1(s)} dx\\
&=\int_{\Theta_{h_1(s)}\Omega}\bar{u}\cdot ({\nabla}\bar{\phi}) d\bar{x}\\
&=-\int_{\Theta_{h_1(s)}\Omega}\bar{\phi}\ {\div}\bar{u}\ d\bar{x}.
\end{align*}
Since $\bar{u}=u\circ\Theta^{-1}_{h_1(s)}\in L_{p,\sigma}(\Theta_{h_1(s)}\Omega)$, the claim follows.

\textbf{Step 3}. In this step we show that for each fixed $s\neq 0$ sufficiently close to zero, the operator $\mathcal{L}(s)$, which has been introduced in Step 2, possesses a real positive eigenvalue. For that purpose, define the spaces
$$
\mathbb{Y}_1:=\{u\in H_p^2(\Omega\backslash\Sigma)^3:
u|_{S_2}=0,\ u|_{S_1}\cdot\nu_{S_1}=0,\ P_{S_1}(\mu(Du)\nu_{S_1})=0,\ \Jump{u}=0\},
$$
$$
\mathbb{Y}_2:=\dot{H}_p^1(\Omega\backslash\Sigma),\ \mathbb{Y}_3:=W_p^{1-1/p}(\Sigma),
$$
$$
\mathbb{Y}_4:=\{h\in W_p^{3-1/p}(\Sigma):\int_{\Sigma}h\ dx'=0,\ \partial_{\nu_{B_R}}h=0\}\cap R(A),
$$
$$\mathbb{Y}:=\{(u,\pi,q,h)\in\times_{j=1}^4\mathbb{Y}_j:q=\Jump{\pi}\},$$
and
\begin{multline*}
\mathbb{X}:=\{(f,f_d,g_v,g_w,g_h)\in L_p(\Omega)^3\times [H_p^1(\Omega\backslash\Sigma)\cap \hat{H}_p^{-1}(\Omega)]\times\\ \times W_p^{1-1/p}(\Sigma)^{2}\times W_p^{1-1/p}(\Sigma)\times W_p^{2-1/p}(\Sigma):\\
\int_{\Sigma}g_h\ dx'=0,\ \partial_{\nu_{B_R}}g_h=0,\ g_v|_{S_1}\cdot\nu_{B_R}=0\},
\end{multline*}
where $A$ and $R(A)$ are defined in Section \ref{LSR}. Furthermore, for some small $\delta>0$ and $|s|<\delta$, we define a linear operator family $\mathbb{L}(s):\mathbb{Y}\to\mathbb{X}$ by
$$\mathbb{L}(s)(u,\pi,q,h):=
\begin{pmatrix}
\frac{\mu}{\rho}\Delta u-\frac{1}{\rho}\nabla\pi+\frac{1}{\rho}F_u(u,\pi,h_1(s))\\
\div u-P_0F_d(u,h_1(s))\\
\Jump{\mu\partial_3 v}+\Jump{\mu\nabla_{x'}w}+G_v(u,h_1(s))\\
2\Jump{\mu\partial_3 w}-q+G_w(u,h_1(s))+\frac{\Jump{\rho}\gamma_a}{\alpha(s)}D_hF(\alpha(s),h_1(s))h\\
P_0^\Sigma[(u|e_3)-(v|\nabla_{x'} h_1(s))]
\end{pmatrix},$$
where $P_0^\Sigma k:=k-\frac{1}{|\Sigma|}\int_\Sigma k\ dx'$, $u=(v,w)$,
$$F_u(u,\pi,h):=-\mu(M_1(h):\nabla^2 u+M_2(h)\nabla u)+M_0(h)\nabla\pi,$$
$$F_d(u,h):=(M_0(h)\nabla| u),$$
\begin{multline*}
G_v(u,h):=-\Jump{\mu(\nabla_{x'} v+\nabla_{x'} v^{\sf T})}\nabla_{x'} h+|\nabla_{x'} h|^2\Jump{\mu\partial_3 v}\\
+\left((1+|\nabla_{x'} h|^2)\Jump{\mu\partial_3 w}-(\nabla_{x'} h|\Jump{\mu\nabla_{x'} w})\right)\nabla_{x'} h,
\end{multline*}
$$G_w(u,h):=-(\nabla_{x'} h|\Jump{\mu\nabla_{x'} w})-(\nabla_{x'} h|\Jump{\mu\partial_3 v})+|\nabla_{x'} h|^2\Jump{\mu \partial_3 w},$$
and
$$F(\alpha,h):=\div_{x'}\left(\frac{\nabla_{x'} h}{\sqrt{1+|\nabla_{x'} h|^2}}\right)+\alpha h.$$
In the following, let $h_2(s)$ be defined by \eqref{eq:eigenfunctioncurve}. Let $w=(u,\pi,q,h)$ and consider the function $g:(-\delta,\delta)\times\mathbb{R}\times\mathbb{R}\times\mathbb{Y}\to \mathbb{X}$ given by
$$g(s,(\beta,\gamma,w)):=\mathbb{L}(s)w-
\begin{pmatrix}
-s\alpha'(s)\beta u\\
0\\
0\\
-\frac{\Jump{\rho}\gamma_a}{s\alpha(s)}h_1(s)\\
\beta (h_1'(s)-s\alpha'(s) h)+\gamma h_2(s)
\end{pmatrix}
.$$
We remark that $h_1(s)/s=h_1+y(s)/s$, hence $h_1(s)/s\to h_1$ as $s\to 0$, since $y(0)=y'(0)=0$. Evaluating $g$ at $s=0$ yields
$$g(0,(\beta,\gamma,w))=
\begin{pmatrix}
\frac{\mu}{\rho}\Delta u-\frac{1}{\rho}\nabla\pi\\
\div u\\
\Jump{\mu\partial_3 v}+\Jump{\mu\nabla_{x'}w}\\
2\Jump{\mu\partial_3 w}-q+\frac{\Jump{\rho}\gamma_a}{\lambda_1}D_hF(\lambda_1,0)h+
\frac{\Jump{\rho}\gamma_a}{\lambda_1}h_1\\
P_0^\Sigma[(u|e_3)]-\beta h_1-\gamma h_2
\end{pmatrix}.$$
In the following we intend to find $(\beta_0,\gamma_0,w_0)\in\mathbb{R}\times\mathbb{R}\times \mathbb{Y}$ such that $g(0,(\beta_0,\gamma_0,w_0))=0$. Note that if we have such a solution, then necessarily $P_0^\Sigma[(u|e_3)]=(u|e_3)$, since
$$\int_\Sigma (u|e_3)\ dx'=\int_{\Omega_1}\div u|_{\Omega_1}\ dx=0,$$
by the boundary conditions on $u\in\mathbb{Y}_1$. Hence we consider the elliptic problem
\begin{align}\label{eq:bifurc2}
\begin{split}
-\mu\Delta u+\nabla \pi&=0,\quad \text{in}\ \Omega\backslash\Sigma,\\
\div u&=0,\quad \text{in}\ \Omega\backslash\Sigma,\\
-\Jump{\mu Du}e_3+\Jump{\pi}e_3&=\frac{\Jump{\rho}\gamma_a}{\lambda_1}
\left(D_hF(\lambda_1,0)h+h_1\right)e_3,\quad \text{on}\ \Sigma,\\
\Jump{u}&=0,\quad \text{on}\ \Sigma,\\
\beta h_1+\gamma h_2-(u|e_3)&=0,\quad \text{on}\ \Sigma,\\
P_{S_1}\left(\mu (Du)\nu_{S_1}\right)&=0,\quad \text{on}\ S_1\backslash\partial\Sigma,\\
(u|\nu_{S_1})&=0,\quad \text{on}\ S_1\backslash\partial\Sigma,\\
u&=0,\quad \text{on}\ S_2,\\
\partial_{\nu_{B_R}}h&=0,\quad \text{on}\ \Sigma,
\end{split}
\end{align}
for the unknowns $((\beta,\gamma),w)\in \mathbb{R}^2\times\mathbb{Y}$. Let us first remove the inhomogeneity $h_1$ in the equation $\eqref{eq:bifurc2}_3$. To this end we solve the elliptic transmission problem
\begin{align*}
(\nabla p|\nabla\phi)_{L_2}&=0,\quad \phi\in W_{p'}^1(\Omega),\\
\Jump{p}&=\frac{\Jump{\rho}\gamma_a}{\lambda_1}h_1,\quad \text{on}\ \Sigma,
\end{align*}
to obtain a unique solution $p\in \mathbb{Y}_2$, thanks to Lemma \ref{lem:appauxlemweak}. Therefore we may reduce \eqref{eq:bifurc2} to the elliptic problem
\begin{align}\label{eq:bifurc3}
\begin{split}
-\mu\Delta u+\nabla \pi&=-\nabla p,\quad \text{in}\ \Omega\backslash\Sigma,\\
\div u&=0,\quad \text{in}\ \Omega\backslash\Sigma,\\
-\Jump{\mu Du}e_3+\Jump{\pi}e_3&=\frac{\Jump{\rho}\gamma_a}{\lambda_1}D_hF(\lambda_1,0)he_3,\quad \text{on}\ \Sigma,\\
\Jump{u}&=0,\quad \text{on}\ \Sigma,\\
\beta h_1+\gamma h_2-(u|e_3)&=0,\quad \text{on}\ \Sigma,\\
P_{S_1}\left(\mu (Du)\nu_{S_1}\right)&=0,\quad \text{on}\ S_1\backslash\partial\Sigma,\\
(u|\nu_{S_1})&=0,\quad \text{on}\ S_1\backslash\partial\Sigma,\\
u&=0,\quad \text{on}\ S_2,\\
\partial_{\nu_{B_R}}h&=0,\quad \text{on}\ \Sigma.
\end{split}
\end{align}
With the help of the operator $L:X_1\to X_0$ from Section \ref{EquilLin} we may rewrite system \eqref{eq:bifurc3} in the abstract form
\begin{equation}\label{eq:bifurc4}
Lz-\beta z_1-\gamma z_2=f,
\end{equation}
where we have set $z=(u,h)$, $z_j:=(0,h_j)$, $j\in\{1,2\}$ and $f:=(\nabla p,0)\in X_0$ is given.

Let us compute the kernel of $L$ in $X_0$. If $L(u,h)=0$, it follows from \eqref{L2spectrum} that $Du=0$, hence $u=0$ by Korn's second inequality. This in turn implies that $h$ is a solution of the elliptic problem
\begin{align*}
\Delta_{x'} h+\lambda_1 h&=0,\quad x'\in B_R(0),\\
\partial_{\nu_{B_R}}h&=0,\quad x'\in\partial B_R(0),
\end{align*}
since $h$ is mean value free. Therefore $h\in N(A)$, with $A$ being defined as in Section \ref{LSR}, and we already know that $\dim N(A)=2$. It follows that $\dim N(L)=2$ as well with $N(L)$ being spanned by the two elements $z_1=(0,h_1)$ and $z_2=(0,h_2)$, where
$h_j\in N(A)$.

Next, we show that $0\in\sigma(L)$ is a semi-simple eigenvalue. We already know that $\sigma(L)$ consist solely of discrete eigenvalues with finite multiplicity and $0\in \sigma(L)$. Therefore, it suffices to show that $N(L)=N(L^2)$. Let $z=(z_1,z_2)\in N(L^2)$ and define $w:=Lz$. Then $w=(w_1,w_2)\in N(L)$, hence $w_1=0$ and $w_2\in N(A)$ and it remains to show that $w_2=0$. To this end, we test the first equation in $Lz=w$ by $z_1$ and integrate by parts. It follows that
$$|\mu^{1/2} Dz_1|_{L_2(\Omega)}^2
-\sigma\left[(\Delta_{x'} z_2|w_2)_{L_2(B_R)}+\lambda_1(z_2|w_2)_{L_2(B_R)}^2\right]=0.$$
Integrating by parts a second time, we see that the term in brackets $[\ldots]$ vanishes, since $w_2\in N(A)$. Therefore $Dz_1=0$, hence $z_1=0$ by Korn's second inequality. Since then also $(z_1|e_3)=0$ on $\Sigma$, the second component in $Lz=w$ implies that $w_2=0$, showing that $0\in\sigma(L)$ is semi-simple. Notably it follows that $X_0=N(L)\oplus R(L)$ and, in particular, $L$ is a Fredholm operator with index zero.

We seek to find a solution $z=(u,h)\in X_1$ of \eqref{eq:bifurc4} with the additional property $h\in R(A)$. This extra condition yields the uniqueness of the solution. First of all we find unique $f_1\in N(L)$ and $f_2\in R(L)$ such that $f=f_1+f_2$. Since $0\in\sigma(A)$ is semi-simple, too, this yields the existence of $\kappa_1,\kappa_2\in\mathbb{R}$ and $z_*=(u_*,h_*)\in X_1$ with $h_*\in R(A)$ such that $f=Lz_*-\kappa_1 z_1-\kappa_2 z_2$. Setting $z=z_*$ it follows that
$$Lz_*-\beta z_1-\gamma z_2=Lz_*-\kappa_1 z_1-\kappa_2 z_2,$$
hence $\beta=\kappa_1$ and $\gamma=\kappa_2$. Therefore, the triple
$$(\beta_0,\gamma_0,w_0):=(\kappa_1,\kappa_2,(u_*,\pi_*,\Jump{\pi_*},h_*))$$ is the unique solution of \eqref{eq:bifurc2} in $\mathbb{Y}$, where $\pi_*\in \mathbb{Y}_2$ solves
\begin{align*}
(\nabla \pi_*|\nabla\phi)_{L_2}&=(\mu\Delta u_*|\phi)_{L_2},\quad \phi\in W_{p'}^1(\Omega),\\
\Jump{\pi_*}&=(\Jump{\mu Du_*}e_3|e_3)+\frac{\Jump{\rho}\gamma_a}{\lambda_1}\left(D_hF(\lambda_1,0)h_*+h_1\right),\quad \text{on}\ \Sigma.
\end{align*}
It is noteworthy that the relation
\begin{equation}\label{eq:bifurc10}\beta_0=\frac{\lambda_1}{\Jump{\rho}\gamma_a}
\frac{\|\mu^{1/2}Du_*\|_{L_2(\Omega)}^2}{\|h_1\|_{L_2(\Sigma)}^2}
\end{equation}
holds. This follows easily by testing the first equation in \eqref{eq:bifurc2} with $u_*$ and integrating by parts. Assuming that $Du_*=0$, Korn's second inequality yields $u_*=0$, hence $\pi_*$ is constant and therefore $\Jump{\pi_*}=0$, since
$$D_h F(\lambda_1,0)h_*+h_1=\Delta_{x'} h_*+\lambda_1 h_*+h_1$$
is mean value free. But then $h_1\in R(A)\cap N(A)$ which yields $h_1=0$, since $0\in\sigma(A)$ is semi-simple, a contradiction. This in turn implies that $\beta_0>0$, a fact that will be of importance later on.

Next, we consider the derivative of $g$ with respect to $(\beta,\gamma,w)$ which we will denote by $D_2g$ in the sequel. It is given by
$$D_2g(0,(\beta_0,\gamma_0,w_0))(\hat{\beta},\hat{\gamma},\hat{w})=
\begin{pmatrix}
\frac{\mu}{\rho}\Delta \hat{u}-\frac{1}{\rho}\nabla\hat{\pi}\\
\div \hat{u}\\
\Jump{\mu\partial_3 \hat{v}}+\Jump{\mu\nabla_{x'}\hat{w}}\\
2\Jump{\mu\partial_3 \hat{w}}-\hat{q}+\frac{\Jump{\rho}\gamma_a}{\lambda_1}D_hF(\lambda_1,0)\hat{h}\\
P_0^\Sigma[(\hat{u}|e_3)]-\hat{\beta} h_1-\hat{\gamma} h_2
\end{pmatrix}.
$$
We will now show that $D_2g(0,(\beta_0,\gamma_0,w_0)):\mathbb{Y}\to\mathbb{X}$ is an isomorphism. Injectivity follows by testing the first equation with $\hat{u}$, integrating by parts and invoking the additional condition $\hat{h}\in R(A)$. Therefore it remains to prove surjectivity of $D_2g(0,(\beta_0,\gamma_0,w_0))$. For that purpose, let $g=(g_1,g_2,g_3,g_4,g_5)\in \mathbb{X}$ be given and consider the equation
\begin{equation}\label{eq:bifurc4.1}
D_2g(0,(\beta_0,\gamma_0,w_0))(\hat{\beta},\hat{\gamma},\hat{w})=g.
\end{equation}
First, let us show that it suffices to consider the special case $g=(g_1,0,0,0,g_5)$. To see this, we solve the elliptic problem
\begin{align}\label{eq:bifurc5}
\begin{split}
\rho\omega u-\mu\Delta u+\nabla \pi&=0,\quad \text{in}\ \Omega\backslash\Sigma,\\
\div u&=g_2,\quad \text{in}\ \Omega\backslash\Sigma,\\
-\Jump{\mu\partial_3 v}-\Jump{\mu\nabla_{x'}w}&=g_3,\quad \text{on}\ \Sigma,\\
-2\Jump{\mu\partial_3 w}+\Jump{\pi}&=g_4,\quad \text{on}\ \Sigma,\\
\Jump{u}&=0,\quad \text{on}\ \Sigma,\\
P_{S_1}\left(\mu (Du)\nu_{S_1}\right)&=0,\quad \text{on}\ S_1\backslash\partial\Sigma,\\
(u|\nu_{S_1})&=0,\quad \text{on}\ S_1\backslash\partial\Sigma,\\
u&=0,\quad \text{on}\ S_2,
\end{split}
\end{align}
by Theorem \ref{thm:AppEllStokes}, where $\omega>0$ is large enough, to obtain a unique solution $(u_*,\pi_*,\Jump{\pi_*})\in \times_{j=1}^3\mathbb{Y}_j$. Note that all relevant compatibility conditions at the contact line are satisfied by the definition of the space $\mathbb{X}$. Then, the (shifted) triple $$(\hat{\beta},\hat{\gamma},w):=
(\hat{\beta},\hat{\gamma},\hat{w}-(u_*,\pi_*,\Jump{\pi_*},0))$$
satisfies the problem
\begin{align}\label{eq:bifurc6}
\begin{split}
-\mu\Delta u+\nabla \pi&=\rho (g_1+\omega u_*),\quad \text{in}\ \Omega\backslash\Sigma,\\
\div u&=0,\quad \text{in}\ \Omega\backslash\Sigma,\\
-\Jump{\mu Du}e_3+\Jump{\pi}e_3-\frac{\Jump{\rho}\gamma_a}{\lambda_1}D_hF(\lambda_1,0)he_3&=0,\quad \text{on}\ \Sigma,\\
\Jump{u}&=0,\quad \text{on}\ \Sigma,\\
P_0^\Sigma[(u|e_3)]-\hat{\beta} h_1-\hat{\gamma} h_2&=g_5-P_0^\Sigma[(u_*|e_3)],\quad \text{on}\ \Sigma,\\
P_{S_1}\left(\mu (Du)\nu_{S_1}\right)&=0,\quad \text{on}\ S_1\backslash\partial\Sigma,\\
(u|\nu_{S_1})&=0,\quad \text{on}\ S_1\backslash\partial\Sigma,\\
u&=0,\quad \text{on}\ S_2,\\
\partial_{\nu_{B_R}}h&=0,\quad \text{on}\ \Sigma.
\end{split}
\end{align}
We recall that for the existence of a solution of \eqref{eq:bifurc6} it is necessary that $P_0^\Sigma[(u|e_3)]=(u|e_3)$ because of the divergence condition and the boundary conditions. Let us define $f_1:=\rho(g_1+\omega u_*)-T_1[\rho(g_1+\omega u_*)]$ and $f_2:=g_5-P_0^\Sigma[(u_*|e_3)]$, where $T_1$ is the solution operator from \eqref{formulaPi}. With the help of the operator $L:X_1\to X_0$ from Section \ref{EquilLin} we may rewrite \eqref{eq:bifurc6} as the abstract equation
$$Lz-\hat{\beta} z_1-\hat{\gamma} z_2=f,$$
where $z=(u,h)$, $f=(f_1,f_2)$ and $z_j:=(0,h_j)\in N(L)$, $j\in\{1,2\}$.
It follows from exactly the same considerations as for \eqref{eq:bifurc4} that the latter equation has a unique solution $(\hat{\beta},\hat{\gamma})\in\mathbb{R}^2$ and $z=(u,h)\in X_1$ with $h\in R(A)$. This in turn yields a solution of \eqref{eq:bifurc4.1}, showing surjectivity of the operator $D_2g(0,(\beta_0,\gamma_0,w_0)):\mathbb{Y}\to\mathbb{X}$.

The implicit function theorem yields the existence of some $\delta>0$ and smooth functions $(\beta(s),\gamma(s),w(s))\in \mathbb{R}\times\mathbb{R}\times\mathbb{Y}$ such that $g(s,(\beta(s),\gamma(s),w(s)))=0$ for all $|s|<\delta$ and $(\beta(0),\gamma(0),w(0))=(\beta_0,\gamma_0,w_0)$. It can furthermore be shown that $P_0F_d(u(s),h_1(s))=F_d(u(s),h_1(s))$ and then also
$$P_0^\Sigma[(u(s)|e_3)-(v(s)|\nabla_{x'} h_1(s))]=(u(s)|e_3)-(v(s)|\nabla_{x'} h_1(s)).$$

For $s\neq 0$ we multiply the equation $g(s,(\beta(s),\gamma(s),w(s)))=0$ by $s\alpha'(s)$ to obtain the relation
$$
\mathbb{L}(s)\tilde{w}(s)=
\begin{pmatrix}
\tilde{\beta}(s) \tilde{u}(s)\\
0\\
0\\
-\frac{\Jump{\rho}\gamma_a}{\alpha(s)}\alpha'(s)h_1(s)\\
\tilde{\beta}(s) (\tilde{h}(s)-h_1'(s))-\tilde{\gamma}(s) h_2(s)
\end{pmatrix}
,$$
where $\tilde{w}(s):=s\alpha'(s) w(s)$, $\tilde{\beta}(s):=-s\alpha'(s)\beta(s)$ and $\tilde{\gamma}(s):=-s\alpha'(s)\gamma(s)$. The identity $F(\alpha(s),h_1(s))=0$ yields furthermore
$$\alpha'(s)h_1(s)=\alpha'(s)D_\alpha F(\alpha(s),h_1(s))=-D_h F(\alpha(s),h_1(s))h_1'(s),$$
hence
$$
\mathbb{L}(s)\tilde{w}(s)=
\begin{pmatrix}
\tilde{\beta}(s) \tilde{u}(s)\\
0\\
0\\
\frac{\Jump{\rho}\gamma_a}{\alpha(s)}D_h F(\alpha(s),h_1(s))h_1'(s)\\
\tilde{\beta}(s) (\tilde{h}(s)-h_1'(s))-\tilde{\gamma}(s) h_2(s)
\end{pmatrix}
.$$
Since $D_hF(\alpha(s),h_1(s))h_2(s)=0$ it follows from the fourth equation that
\begin{multline*}
\tilde{q}(s)=2\Jump{\mu\partial_{3} \tilde{w}(s)}+G_w(\tilde{u}(s),h_1(s))+\\
\frac{\Jump{\rho}\gamma_a}{\alpha(s)}D_h F(\alpha(s),h_1(s))[\tilde{h}(s)-h_1'(s)-\frac{\gamma(s)}{\beta(s)}h_2(s)].
\end{multline*}
Setting $\hat{h}(s):=\tilde{h}(s)-h_1'(s)-\frac{\gamma(s)}{\beta(s)}h_2(s)$ and $\hat{z}(s):=(\tilde{u}(s),\hat{h}(s))$, it follows that
$$\mathcal{L}(s)\hat{z}(s)=\tilde{\beta}(s)\hat{z}(s)$$
for all $|s|<\delta$. Since $\hat{z}(0)=-(0,h_1+\frac{\gamma_0}{\beta_0}h_2)\neq (0,0)$ ($h_1$ and $h_2$ are linearly independent) it holds that $\hat{z}(s)\neq 0$ for all $|s|$ close to zero. Furthermore $\tilde{\beta}(s)=-s\alpha'(s)\beta(s)>0$ for $s\neq 0$ close to zero, by Corollary \ref{cor:bifurc1} and \eqref{eq:bifurc10}, hence $\tilde{\beta}(s)>0$ is a positive real eigenvalue of the operator $\mathcal{L}(s)$ from Step 2.

\textbf{Step 4}. In this step we prove nonlinear instability of any equilibrium $(0,h_*)$ of \eqref{eq:NScap2} which is close to zero. Since the strategy follows one-to-one along the lines of the proof of Theorem \ref{mainresultstab}, we will only give a sketch of the proof. To explain the strategy, we rewrite \eqref{eq:NScap2} in the shorter form $\mathbb{L}w=N(w)$, $z(0)=z_0$, where $w=(u,\pi,q,h)$, $q=\Jump{\pi}$ and $z=(u,h)$, $z_0=(u_0,h_0)$. Let $w_*=(0,c_*,0,h_*)$ and $z_*:=(0,h_*)$, where $c_*$ is a constant and $h_*$ is determined by Theorem \ref{thm:bifurc1} in a small neighborhood of zero. Setting $\hat{w}:=w-w_*$, this yields $\mathbb{L}_*\hat{w}=\hat{N}(\hat{w})$, $\hat{z}(0)=z_0-z_*$, where $\mathbb{L}_*:=\mathbb{L}-DN(w_*)$ and
$$\hat{N}(\hat{w}):=N(\hat{w}+w_*)-N(w_*)-DN(w_*)\hat{w},$$
since $N(w_*)=0$. Note that $\hat{N}\in C^2$ by Proposition \ref{prop:regnonlin} and $\hat{N}(0)=0$ as well as $D\hat{N}(0)=0$. The operator $\mathbb{L}_*=\mathbb{L}-DN(w_*)$ is the full linearization of \eqref{eq:NScap2} at the equilibrium $w_*$. As in the proof of Theorem \ref{mainresultstab}, we will decompose $\hat{w}$ as follows. Let $\omega>0$ and $Bw:=(u,0,\ldots,0,h)$ for $w=(u,\pi,q,h)$. Then we consider the two problems
\begin{align*}
\omega B\bar{w}+\mathbb{L}_*\bar{w}&=\hat{N}(w),\quad t>0,\\
\bar{z}(0)&=\bar{z}_0,
\end{align*}
and
\begin{align*}
\mathbb{L}_*\tilde{w}&=\omega B\bar{w},\quad t>0,\\
\tilde{z}(0)&=\tilde{z}_0.
\end{align*}
Here $\hat{z}_0=\tilde{z}_0+\bar{z}_0$, and $\bar{z}_0$ is is determined by $\tilde{z}_0$ in a similar way as in Section \ref{paraphase}. Note that there exists $s_*\in\mathbb{R}$ close to zero such that $\mathbb{L}_*\tilde{w}=\omega B\bar{w}$ is equivalent to $\partial_t \tilde{z}-\mathcal{L}(s_*)\tilde{z}=\omega \bar{z}$, where $\mathcal{L}(s_*)$ denotes the operator from Step 2. For small $|s_*|>0$ the operator $\mathcal{L}(s_*)$ can be seen as a small perturbation of the operator $L$ from Section \ref{EquilLin}, wherefore it has the property of maximal regularity of type $L_p$. By Step 3 it holds that $\mathcal{L}(s_*)$ has a positive eigenvalue. We are now in a position to apply the same strategy as in the proof of Theorem \ref{mainresultstab} to prove instability of the equilibrium $z_*$ by making use of the corresponding spectral projections. This yields the following result.
\begin{thm}\label{thm:instabbifeq}
The equilibria on the subcritical branches ($s\alpha'(s)<0$, $0<|s|<\delta$) given by Theorem \ref{thm:bifurc1} and Corollary \ref{cor:bifurc1} are unstable in the sense of Theorem \ref{mainresultstab}.
\end{thm}
\begin{rem}
The results in this chapter carry over to the two dimensional case, i.e.\ if $G=(-R,R)$. Indeed, the eigenvalue $\lambda_1>0$ of the Neumann Laplacian in $L_2(-R,R)$ is simple.  This allows to apply the results of Crandall \& Rabinowitz. It can also be shown that the full linearization $\mathcal{L}(s_*)$ has a positive eigenvalue for each $|s_*|>0$ close to zero.
\end{rem}

\appendix
%
%
%

\chapter{Appendix}

\section{Extension operators}

\begin{prop}\label{app:propext}
Let $p>2$. There exists a linear and bounded extension operator $\operatorname{ext}$ from
$$_0W_p^{1/2-1/p}(J;L_p(\mathbb{R}))\cap L_p(J;W_p^{1-2/p}(\mathbb{R}))$$
to
$$_0W_p^{1/2-1/2p}(J;L_p(\mathbb{R}\times\mathbb{R}_+))\cap L_p(J;W_p^{1-1/p}(\mathbb{R}\times\mathbb{R}_+))$$
such that $[\operatorname{ext}v]|_{\mathbb{R}\times\{0\}}=v$, for all $v\in\! _0W_p^{1/2-1/p}(J;L_p(\mathbb{R}))\cap L_p(J;W_p^{1-2/p}(\mathbb{R}))$.

Moreover, if
$$v=v(t,x,y)\in\!_0W_p^{1/2-1/2p}(J;L_p(\mathbb{R}\times\mathbb{R}_+))\cap L_p(J;W_p^{1-1/p}(\mathbb{R}\times\mathbb{R}_+))=:X,$$
then
$$\operatorname{tr}_{y=0}v\in\!_0W_p^{1/2-1/p}(J;L_p(\mathbb{R}))\cap L_p(J;W_p^{1-2/p}(\mathbb{R}))=:Y$$
and there exists a constant $C>0$ such that
$$\|\operatorname{tr}_{y=0}v\|_{Y}\le C\|v\|_{X}$$
for all $v\in X$.
\end{prop}
\begin{proof}
Let $X_0=L_p(J;L_p(\mathbb{R}))$ and consider the operator $(\partial_t-\partial_x^2)$ in $X_0$ with domain
$$_0W_p^1(J;L_p(\mathbb{R}))\cap L_p(J;W_p^2(\mathbb{R})).$$
The operator $-A:=-(\partial_t-\partial_x^2)^{1/2}$ generates an analytic semigroup $\{e^{-Ay}\}_{y\ge 0}$ in $X_0$ with domain $D(A)=[X_0,D(A^2)]_{1/2}$. Since
$$D_A(1-2/p,p)=(X_0,D(A))_{1-2/p,p}=(X_0,D(A^2))_{1/2-1/p,p},$$
by \cite[Theorem 1.15.2]{Tri95}, we obtain
$$D_A(1-2/p,p)=\,_0W_p^{1/2-1/p}(J;L_p(\mathbb{R}))\cap L_p(J;W_p^{1-2/p}(\mathbb{R})).$$
Hence, if $v\in D_A(1-2/p,p)$, then
$$[y\mapsto e^{-A y}v]\in W_p^{1-1/p}(\mathbb{R}_+;X_0)\cap L_p(\mathbb{R}_+;D_A(1-1/p,p))$$
by \cite[Theorems 3 \& 8]{diBla84}, where $D_A(1-1/p,p)=(X_0,D(A^2))_{1/2-1/2p,p}$, hence
$$D_A(1-1/p,p)=\!_0W_p^{1/2-1/2p}(J;L_p(\mathbb{R}))\cap L_p(J;W_p^{1-1/p}(\mathbb{R})).$$
Setting $[\operatorname{ext}v](y)=e^{-Ay}v$ yields the first claim, by the Fubini property of the spaces $W_p^s$.

For the proof of the second assertion, we consider $v(t,x,y)$ as a function $w(y)(t,x)$, i.e.\ $w(y)(t,x):=v(t,x,y)$. Then we have
$$w\in W_p^{1-1/p}(\mathbb{R}_+;X_0)\cap L_p(\mathbb{R}_+;D_A(1-1/p,p)),$$
where $X_0$ and $A$ are defined as above. By \cite[Lemma 4.1, (4.4)]{MeySchn12} with $\alpha=1-1/p$ and $\mu=1$ it holds that $\operatorname{tr}|_{y=0}$ is a continuous mapping from
$$W_p^{1-1/p}(\mathbb{R}_+;X_0)\cap L_p(\mathbb{R}_+;D_A(1-1/p,p))$$
to $D_A(1-1/2p,p)=D_{A^2}(1/2-1/p,p)=(X_0,D(A^2))_{1/2-1/p,p}$ with
$$(X_0,D(A^2))_{1/2-1/p,p}=\,_0W_p^{1/2-1/p}(J;L_p(\mathbb{R}))\cap L_p(J;W_p^{1-2/p}(\mathbb{R})).$$
The proof is complete.
\end{proof}

\begin{prop}\label{prop:app2}
Let $p>2$, $J=[0,T]$, $0<T<\infty$ or $J=\mathbb{R}_+$ and
$$g\in\!_0W_p^{3/2-1/p}(J;L_p(\mathbb{R}))\cap\!_0H_p^1(J;W_p^{1-2/p}(\mathbb{R}))\cap L_p(J;W_p^{2-2/p}(\mathbb{R}))=:Y.$$
Then there exists
$$h\in\!_0W_p^{2-1/2p}(J;L_p(\mathbb{R}^2_+))\cap\!_0H_p^1(J;W_p^{2-1/p}(\mathbb{R}^2_+))\cap L_p(J;W_p^{3-1/p}(\mathbb{R}^2_+))=:X,$$
such that $\partial_y h=g$ at $y=0$.

Moreover, the mapping $(\operatorname{tr}|_{y=0}\circ\partial_y):X\to Y$ is continuous.
\end{prop}
\begin{proof}
(1) Consider the operator $(\partial_t-\partial_x^2)$ in $X_0:=L_p(J;L_p(\mathbb{R}))$ with domain
$$_0W_p^{1}(J;L_p(\mathbb{R}))\cap L_p(J;W_p^2(\mathbb{R})).$$
Let $A:=(\partial_t-\partial_x^2)^{1/2}$ with domain $D(A)=[X_0,D(A^2)]_{1/2}$.
Denote by $e^{-Ay}$ the analytic $C_0$-semigroup, generated by $-A$ in $X_0$ and set $h(y):=-e^{-Ay}A^{-1}g$.
Since
$$g,\partial_t g,A^{-1}g,A^{-1}\partial_t g\in\!_0W_p^{1/2-1/p}(J;L_p(\mathbb{R}))\cap L_p(J;W_p^{1-2/p}(\mathbb{R}))$$ it follows from Proposition \ref{app:propext} that
$$h,\partial_t h,Ah,A\partial_t h\in W_p^{1-1/p}(\mathbb{R}_+;X_0)\cap L_p(\mathbb{R}_+;D_A(1-1/p,p)).$$
The operator $A^{-1}$ is an isomorphism from $(X_0,D(A^2))_{1/2-1/2p,p}$ to $(X_0,D(A^2))_{1-1/2p,p}$ by \cite[Theorem 1.15.2]{Tri95}, hence $h$ as well as $\partial_t h$ belong to
$$_0W_p^{1-1/2p}(J;L_p(\mathbb{R}_+^2))\cap L_p(J;W_p^{2-1/p}(\mathbb{R}_+^2))$$
by the Fubini property.
Furthermore $\partial_t:\!_0W_p^s(J;X)\to\!_0W_p^{s-1}(J;X)$, $s\in[1,2)$ is an isomorphism, hence
\begin{equation}\label{eq:app1}
h\in\!_0W_p^{2-1/2p}(J;L_p(\mathbb{R}_+^2))\cap\!_0W_p^1(J;W_p^{2-1/p}(\mathbb{R}_+^2)).
\end{equation}
(2) Next, we use the regularity
$$Ag\in\!_0W_p^{1-1/p}(J;L_p(\mathbb{R}))\cap L_p(J;W_p^{1-2/p}(\mathbb{R})),$$
to conclude
\begin{equation}\label{eq:app2}
-\partial_y^2 h=A^2e^{-Ay}A^{-1}g=e^{-Ay}Ag\in W_p^{1-1/p}(\mathbb{R}_+;X_0)
\end{equation}
by \cite[Theorem 8]{diBla84}, since
$$Ag\in D_A(1-2/p,p)=\!_0W_p^{1/2-1/p}(J;L_p(\mathbb{R}))\cap L_p(J;W_p^{1-2/p}(\mathbb{R})).$$
In particular, this yields that
$$h\in W_p^{3-1/p}(\mathbb{R}_+;L_p(J;L_p(\mathbb{R}))).$$
(3) It remains to show that
$$h\in L_p(\mathbb{R}_+;L_p(J;W_p^{3-1/p}(\mathbb{R}))).$$
To this end we consider the semigroup $\{e^{-A y}\}_{y\ge 0}$ in $\tilde{X}_0:=L_p(J;W_p^{1-1/p}(\mathbb{R}))$. The domain of the operator $A^2:=(\partial_t-\partial_x^2)$ in $\tilde{X}_0$ is given by
$$_0W_p^{1}(J;W_p^{1-1/p}(\mathbb{R}))\cap L_p(J;W_p^{3-1/p}(\mathbb{R})).$$
Then we have
$$[y\mapsto e^{-A y} g]\in L_p(\mathbb{R}_+;D(A)),$$
if
$$g\in D_A(1-1/p,p)=\!_0W_p^{1/2-1/2p}(J;W_p^{1-1/p}(\mathbb{R}))\cap L_p(J;W_p^{2-2/p}(\mathbb{R})).$$
Note that the assumption on $g$ implies
$$g\in\!_0H_p^{1}(J;W_p^{1-2/p}(\mathbb{R}))\cap L_p(J;W_p^{2-2/p}(\mathbb{R}))
\hookrightarrow \!_0W_p^{1/2-1/2p}(J;W_p^{1-1/p}(\mathbb{R})),
$$
which follows from \cite[Proposition 3.2]{MeySchn12}.
Replacing $g$ by $A^{-1}g$ it follows that
$$[y\mapsto e^{-A y} A^{-1}g]\in L_p(\mathbb{R}_+;D(A^2)),$$
hence
$$[y\mapsto e^{-A y} A^{-1}g]\in L_p(\mathbb{R}_+;L_p(J;W_p^{3-1/p}(\mathbb{R}))).$$
(4) For the proof of the second assertion, note first that $\partial_y$ maps $X$ continuously to
$$_0W_p^{3/2-1/2p}(J;L_p(\mathbb{R}^2_+))\cap\!_0H_p^1(J;W_p^{1-1/p}(\mathbb{R}^2_+))\cap L_p(J;W_p^{2-1/p}(\mathbb{R}^2_+)),$$
since
$$_0W_p^{2-1/2p}(J;L_p(\mathbb{R}^2_+))\cap\!_0H_p^1(J;W_p^{2-1/p}(\mathbb{R}^2_+))$$
is continuously embedded into
$$_0W_p^{3/2-1/2p}(J;H_p^1(\mathbb{R}^2_+)),$$
by \cite[Proposition 3.2]{MeySchn12}. Then the assertion follows from similar arguments as in the proof of Proposition \ref{app:propext}.
\end{proof}

\section{Partition of unity with vanishing Neumann trace}\label{sec:partitionone}

\begin{prop}\label{partitionone}
Let $G\subset\mathbb{R}^2$ be a bounded domain with boundary $\partial G\in C^{m+1}$. Then for each finite open covering $\{U_k\}_{k=1}^N$ of $\partial G$ in $\mathbb{R}^2$ there exists an open set $U_0\subset G$ with $U_0\cap \partial G=\emptyset$, $\bigcup_{k=0}^N U_k\supset \overline{G}$ and a subordinated partition of unity $\{\psi_k\}_{k=0}^N\subset C_c^m(\mathbb{R}^2)$ such that $\supp\psi_k\subset U_k$ and $\partial_{\nu}\psi_k=0$ at $\partial G$.
\end{prop}
\begin{proof}
Let $\{U_j\}_{j=1}^N$ be a finite open cover of $\partial G$. Then there exist open sets $V_j$ such that $K_j:=\overline{V}_j\subset U_j$ and $\bigcup_{k=1}^N V_j\supset\partial G$. Moreover, there exist functions $\phi_j\in C_c^\infty(U_j)$ with $0\le\phi_j\le 1$ such that $\phi_j|_{K_j}=1$. It is well-known that for sufficiently small $a>0$, the mapping $F:\partial G\times (-a,a)\to \mathbb{R}^n$, defined by $F(p,r):=p+r\nu(p)$, is a $C^2$-diffeomorphism onto its image $U_a:=\im F$. The inverse mapping $F^{-1}$ may be decomposed as $F^{-1}=(\Pi,d)$, where $\Pi\in C^m(U_a;\partial G)$ and $d\in C^m(U_a;(-a,a))$. Note that $\Pi(x)$ denotes the nearest point on $\partial G$ to $x\in U_a$ and $d(x)$ stands for the signed distance from $x\in U_a$ to $\partial G$. It can be shown that
$$U_a=\{x\in\mathbb{R}^n:\dist (x,\partial G)<a\}.$$
Choose $a>0$ small enough such that $U_a\subset\bigcup_{j=1}^N K_j$ and define new functions $\bar{\phi}_j(x):=\phi_j(\Pi(x))$ for $x\in U_a$. It follows that $\nabla\bar{\phi}_j(x)=D\Pi^{\sf T}(x)\nabla\phi_j(\Pi(x))$, hence $\partial_\nu\bar{\phi}_j(x)=(\nabla\phi_j(\Pi(x))|D\Pi(x)\nu(x))=0$ for $x\in\partial G$, since $D\Pi(x)\nu(x)$, $x\in\partial G$. Let
$$\tilde{\phi}_j(x):=\begin{cases}
\bar{\phi}_j(x)\varphi(d(x)),&\quad x\in U_a,\\
0,&\quad x\notin U_a,
\end{cases}$$
where $\varphi\in C_c^\infty(\mathbb{R};[0,1])$ such that $\varphi(s)=1$ if $|s|<a/2$ and $\varphi(s)=0$ if $|s|>3a/4$. Then we still have $\partial_\nu\tilde{\phi}(x)=0$ for $x\in\partial G$. Define $\tilde{K}_j:=K_j\cap\partial G$. Then there exists some $\delta\in (0,a/2)$ such that $F_j:=F(\tilde{K}_j,[-\delta,\delta])$ is compact, $F_j\subset U_j$ and $\bigcup_{j=1}^NF_j\supset\partial G$. It follows that $\phi_j|_{\tilde{K}_j}=1$ and therefore $\tilde{\phi}_j|_{F_j}=1$.

Consider the set $\mathcal{G}:=G\backslash \bigcup_{j=1}^NF_j$. Then $\mathcal{G}$ is a proper open subset of $G$. Choose an open set $\mathcal{U}_0\subset G$ that covers $\mathcal{G}$ and a set $\mathcal{F}_0\supset\mathcal{G}$ that is compactly contained in $\mathcal{U}_0$. Define $F_0:=\overline{\mathcal{F}_0}$. Then there exists a smooth function $\tilde{\phi}_0\in C_c^\infty(\mathcal{U}_0;[0,1])$ such that $\tilde{\phi}_0|_{F_0}=1$. In particular it holds that $\bigcup_{j=0}^N F_j\supset \overline{G}$ and $\sum_{j=0}^N\tilde{\phi}_j(x)>0$ for $x\in\overline{G}$. Finally, we set $\psi_k:=\tilde{\phi}_k/\sum_{j=0}^N\tilde{\phi}_j$, $k=0,\ldots, N$. Then $\sum_{k=0}^N\psi_k=1$ and
$$\partial_\nu\psi_k=\frac{\partial_{\nu}\tilde{\phi}_k}{\sum_{j=0}^N\tilde{\phi}_j}-
\frac{\tilde{\phi}_k\sum_{j=0}^N\partial_\nu\tilde{\phi}_j}
{\left(\sum_{j=0}^N\tilde{\phi}_j\right)^2}=0,$$
for $k\in\{0,\ldots,N\}$ at $\partial G$, since by construction also $\partial_\nu\tilde{\phi}_0=0$ at $\partial G$. The proof is complete.
\end{proof}
It is possible to extend the previous result to cylindrical domains $\Omega:=G\times(H_1,H_2)$. To this end let $S_1:=\partial G\times (H_1,H_2)$,
$$S_2:=\bigcup_{j=1}^2 G\times\{H_j\},$$
and $\Sigma:= G\times\{0\}$.
\begin{prop}\label{partitionone2}
Let $G\subset\mathbb{R}^{2}$ be a bounded domain with boundary $\partial G\in C^{m+1}$ and $\Omega:=G\times (H_1,H_2)$, $H_1<0<H_2$.
Then for each finite open covering $\{U_k\}_{k=1}^N$ of $\partial S_2\cup\partial\Sigma$
in $\mathbb{R}^n$ there exist open sets $U_j\subset\mathbb{R}^3$, $j\in\{N+1,\ldots,N+7\}$ such that
\begin{itemize}
\item $U_{N+1}\subset G\times (H_1,0)$, $U_{N+2}\subset G\times (0,H_2)$,
\item $U_{N+3}\cap U_{N+1}\cap S_1\neq\emptyset$, $U_{N+3}\cap(\Sigma\cup S_2)=\emptyset$,
\item $U_{N+4}\cap U_{N+2}\cap S_1\neq\emptyset$, $U_{N+4}\cap(\Sigma\cup S_2)=\emptyset$,
\item $U_{N+5}\cap\Sigma\neq\emptyset$, $U_{N+5}\cap(S_1\cup S_2)=\emptyset$,
\item $U_{N+6}\cap U_{N+1}\cap S_2\neq\emptyset$, $U_{N+6}\cap (S_1\cup\Sigma)=\emptyset$,
\item $U_{N+7}\cap U_{N+2}\cap S_2\neq\emptyset$, $U_{N+7}\cap (S_1\cup\Sigma)=\emptyset$,
\item $\bigcup_{j=1}^{N+7} U_j\supset \overline{\Omega}$.
\end{itemize}
Furthermore, there exists a subordinated partition of unity $\{\phi_k\}_{k=1}^{N+7}\subset C_c^m(\mathbb{R}^3)$ such that $\supp\phi_k\subset U_k$ and $\partial_{\nu_{\partial G}}\phi_k=\partial_{e_n}\phi_k=0$ at $\partial S_2\cup\partial\Sigma$.
\end{prop}
\begin{proof}
The idea of the proof is quite simple. Let $\{U_j\}_{j=1}^{N_1}$ be an open covering of $\partial\Sigma$ in $\mathbb{R}^n$ and define $\tilde{U}_j:=U_j\cap\{\mathbb{R}^{n-1}\times\{0\}\}$. Let $V_j:=\tilde{U}_j$, $j\in\{1,\ldots,N_1\}$, where we identify $V_j$ with a set in $\mathbb{R}^{n-1}$. Then, of course, $\{V_j\}_{j=1}^{N_1}$ is an open covering of $\partial\Sigma$ in $\mathbb{R}^{n-1}$. Now we are in a position to apply Proposition \ref{partitionone} to find an open set $V_0\subset\Sigma$ such that $\bigcup_{j=0}^{N_1}V_j\supset\overline{\Sigma}$. Furthermore, by Proposition \ref{partitionone}, there exists a subordinated partition of unity $\{\psi_j^\Sigma\}_{j=0}^{N_1}\subset C_c^m(\mathbb{R}^{n-1})$ with $\supp\psi_j^\Sigma\subset V_j$ and $\partial_{\nu_{\partial G}}\psi_j^{\Sigma}=0$ at $\partial\Sigma$.

Now we define $\phi_j^\Sigma(x',x_n):=\psi_j^\Sigma(x')\varphi(x_n)$, where $\varphi\in C_c^\infty(\mathbb{R};[0,1])$ such that $\varphi(s)=1$ if $|s|<\delta$ and $\varphi(s)=0$ if $|s|>2\delta$, where $\delta>0$ is sufficiently small. It follows that $\phi_j^\Sigma\in C_c^m(\mathbb{R}^n)$ and, if $\delta>0$ is sufficiently small, then $\supp\phi_j^\Sigma\subset U_j$ for $j\in\{1,\ldots,N_1\}$. Furthermore we still have $\partial_{\nu_{\partial G}}\phi_k^\Sigma=0$ and, in addition, $\partial_{e_n}\phi_j^\Sigma=0$ at $\partial\Sigma$, since $\varphi$ is constant in a neighborhood of $s=0$.

The same procedure can be applied for the charts covering $\partial S_2$. The remaining set which is a proper subset of $\Omega\backslash(\overline{S_2}\cup\overline{\Sigma})$ can be covered by finitely many open charts.
\end{proof}

\section{Auxiliary elliptic and parabolic problems}

\subsection{Elliptic problems}

The following result deals with the two-phase elliptic problem
\begin{align}\label{eq:appauxlem0}
\begin{split}
\lambda u-\Delta u&=f\quad \mbox{ in } \; \Omega\backslash\Sigma,\\
[\![\rho u]\!]&=g_1\quad \mbox{ on } \;\Sigma,\\
[\![\partial_{\nu_\Sigma} u]\!]&=g_2\quad \mbox{ on } \;\Sigma,\\
\partial_{\nu_{S_1}}u&=h_1\quad \mbox{ on }\; S_1\backslash\partial\Sigma,\\
\partial_{\nu_{S_2}}u&=h_2\quad \mbox{ on }\; S_2,
\end{split}
\end{align}
where $\Omega$ and $\Sigma$ satisfy one of the following conditions.
\begin{compactenum}[(a)]
\item $\Omega$ is either a full space, a (bent) half space or a (bent) quarter space and $\Sigma=\emptyset$,
\item $\Omega$ is either a full space or a (bent) half space with outer unit normal $-e_{n-1}$ at $x=0$ and $\Sigma=\{\mathbb{R}^{n-1}\times\{0\}\}\cap\Omega$,
\item $\Omega=G\times (H_1,H_2)$, $H_1<0<H_2$, is a cylindrical domain where $G$ is a bounded domain with boundary $\partial G\in C^4$ and $\Sigma=G\times\{0\}$.
\end{compactenum}
The sets $S_1$ and $S_2$ are the corresponding vertical and horizontal parts of the boundary of $\Omega$, respectively.

\begin{lem}\label{lem:appaux00}
Let $n=2,3$, $p\ge 2$ and assume that $\Omega$ and $\Sigma$ are subject to one of the conditions in (a)-(c) above. Then there exists $\lambda_0\ge 0$ such that for each $\lambda\ge \lambda_0$ the problem \eqref{eq:appauxlem0} has a unique solution $u\in W_p^2(\Omega\backslash\Sigma)$ if and only if the  data satisfy the following regularity and compatibility conditions.
\begin{enumerate}
\item $f\in L_p(\Omega)$,
\item $g_1\in W_p^{2-1/p}(\Sigma)$,
\item $g_2\in W_p^{1-1/p}(\Sigma)$,
\item $h_1\in W_p^{1-1/p}(S_1\backslash\partial\Sigma)$,
\item $h_2\in W_p^{1-1/p}(S_2)$,
\item $\Jump{\rho h_1}=\partial_{\nu_{\partial G}}g_1$ on $\partial\Sigma$.
\end{enumerate}
Furthermore, for each $\lambda_0>0$ there exists a constant $C=C(\lambda_0)>0$ such that for all $\lambda\ge\lambda_0$ the estimate
\begin{multline}\label{eq:MaxRegEllEst}
\lambda\|u\|_{L_p}+\|u\|_{W_p^2}
\le C\Big(\|f\|_{L_p}+\lambda^{1-1/2p}\|g_1\|_{W_p^{2-1/p}}\\
+\lambda^{1/2-1/2p}[\|g_2\|_{W_p^{1-1/p}}+\|h_1\|_{W_p^{1-1/p}}+\|h_2\|_{W_p^{1-1/p}}]\Big)
\end{multline}
for the solution of \eqref{eq:appauxlem0} is valid.
\end{lem}
\begin{proof}
For convenience we restrict ourselves to the case $n=3$. The arguments for the case $n=2$ are similar.

(a) If $\Omega$ and $\Sigma$ are subject to the first two conditions in (a), i.e.\ $\Omega$ is a full space or a half space, then the result is folklore. So let us consider the case where $\Sigma=\emptyset$ and $\Omega$ is a quarter space. To be precise, let $\Omega:=\mathbb{R}\times\mathbb{R}_+\times\mathbb{R}_+$ with $S_1:=\mathbb{R}\times\{0\}\times\mathbb{R}_+$ and $S_2:=\mathbb{R}\times\mathbb{R}_+\times\{0\}$. Therefore we have to study the problem
\begin{align}
\begin{split}\label{eq:auxell1}
\lambda u-\Delta u&=f,\quad x\in\Omega,\\
\partial_2 u&=h_1,\quad x\in S_1,\\
\partial_3 u&=h_2,\quad x\in S_2.
\end{split}
\end{align}
Extend $f$ and $h_2$ with respect to $x_2$ (by even reflection) to some functions $f\in L_p(\mathbb{R}^2\times\mathbb{R}_+)$ and $\tilde{h}_2\in W_p^{1-1/p}(\mathbb{R}^2)$ and solve the half space problem
\begin{align*}
\lambda \tilde{u}-\Delta \tilde{u}&=\tilde{f},\quad x\in\mathbb{R}^2\times\mathbb{R}_+,\\
\partial_3 \tilde{u}&=\tilde{h}_2,\quad x\in\mathbb{R}^2\times\{0\},
\end{align*}
to obtain a unique solution $\tilde{u}\in W_p^2(\mathbb{R}^2\times\mathbb{R}_+)$ for each $\lambda>0$. Note that by symmetry, the function $[x\mapsto \tilde{u}(x_1,-x_2,x_3)]$ is a solution of this problem too. Therefore, by uniqueness, it holds that $\partial_2\tilde{u}|_{S_1}=0$.

In a next step, we extend $h_1$ by even reflection and with respect to the $x_3$ variable to some $\tilde{h}_1\in W_p^{1-1/p}(\mathbb{R}^2)$ and solve the half space problem
\begin{align*}
\lambda \tilde{v}-\Delta \tilde{v}&=0,\quad x\in\mathbb{R}\times\mathbb{R}_+\times \mathbb{R},\\
\partial_2 \tilde{v}&=\tilde{h}_2,\quad x\in\mathbb{R}\times\{0\}\times \mathbb{R},
\end{align*}
to obtain a unique solution $\tilde{v}\in W_p^2(\mathbb{R}\times\mathbb{R}_+\times \mathbb{R})$ for each $\lambda>0$. As above, by symmetry and uniqueness, it holds that $\partial_3\tilde{v}|_{S_2}=0$. Therefore it follows that $u:=(\tilde{u}+\tilde{v})|_{\Omega}$ is the unique solution of \eqref{eq:auxell1}.

Finally, let $\Omega$ be a bent quarter space with $S_2$ as above and
$$S_{1,\theta}:=\{(x_1,x_2,x_3)\in\mathbb{R}^3:x_2=\theta(x_1)\},$$
where $\theta\in BC^3(\mathbb{R})$ with $\|\theta\|_\infty+\|\theta'\|_\infty\le\eta$ and $\eta>0$ can be made as small as we wish. Then the corresponding result follows from change of coordinates (set $\bar{x}_2:=x_2-\theta(x_1)$) and perturbation theory for elliptic problems. We will give a detailed proof for the case of a two-phase half space in part (b) below. The technique carries over to this case. Indeed, things are easier in (a) as there are no compatibility conditions, since $\Sigma=\emptyset$.

(b) Let $\Omega=\mathbb{R}^3$ and $\Sigma=\mathbb{R}^{2}\times\{0\}$. Then we have to solve the problem
\begin{align}
\begin{split}\label{eq:auxell2}
\lambda u-\Delta u&=f,\quad x\in\Omega\backslash\Sigma,\\
[\![\rho u]\!]&=g_1,\quad x\in\Sigma,\\
[\![\partial_3u]\!]&=g_2,\quad x\in\Sigma,
\end{split}
\end{align}
where $\rho=\rho_1\chi_{x_3<0}+\rho_2\chi_{x_3>0}$ and $\rho_j>0$. Since $f\in L_p(\mathbb{R}^3)$ we may first solve the full space problem
$$\lambda \tilde{u}-\Delta \tilde{u}=f,\quad x\in\mathbb{R}^n;$$
to obtain a unique solution $\tilde{u}\in W_p^2(\mathbb{R}^n)$ for each $\lambda>0$. Consider now the problem
\begin{align}
\begin{split}\label{eq:auxell3}
\lambda \bar{u}-\Delta \bar{u}&=0,\quad x\in\Omega\backslash\Sigma,\\
[\![\rho \bar{u}]\!]&=g_1-\Jump{\rho\tilde{u}}=:\bar{g}_1,\quad x\in\Sigma,\\
[\![\partial_3\bar{u}]\!]&=g_2,\quad x\in\Sigma.
\end{split}
\end{align}
By semigroup theory, it is easy to see that the unique solution of \eqref{eq:auxell3} is explicitly given by
$$\bar{u}(x_3):=\frac{1}{\rho_1+\rho_2}\begin{cases}
e^{-L x_3}a_+,&\quad x_3\ge 0,\\
e^{-L(-x_3)}a_-,&\quad x_3<0,
\end{cases}$$
where $L:=(\lambda-\Delta_{x'})^{1/2}$ and
$$a_+:=\bar{g}_1+\rho_2L^{-1} g_2-(\rho_1+\rho_2)L^{-1}g_2,\ a_-=-(\bar{g}_1+\rho_2L^{-1}) g_2.$$
Therefore the function $u:=\tilde{u}+\bar{u}$ is the unique solution of \eqref{eq:auxell2} which exists for each $\lambda>0$.

Let now $\Omega=\mathbb{R}\times\mathbb{R}_+\times\mathbb{R}$ and $\Sigma=\{\mathbb{R}^2\times\{0\}\}\cap\Omega$, i.e.\ we consider the case of a two-phase half space. Now we have to solve the problem
\begin{align}\label{eq:auxell4}
\begin{split}
\lambda u-\Delta u&=f\quad \mbox{ in } \; \Omega\backslash\Sigma,\\
[\![\rho u]\!]&=g_1\quad \mbox{ on } \;\Sigma,\\
[\![\partial_{3} u]\!]&=g_2\quad \mbox{ on } \;\Sigma,\\
\partial_{2}u&=h_1\quad \mbox{ on }\; S_1\backslash\partial\Sigma,
\end{split}
\end{align}
where $S_1:=\mathbb{R}\times\{0\}\times\mathbb{R}$. We will first reduce \eqref{eq:auxell4} to the case $h_1=0$. To this end we first extend $h_1^+:=h_1|_{x_3>0}$ with respect to the $x_3$ variable to some $\tilde{h}_1^+\in W_p^{1-1/p}(\mathbb{R}^2)$ and solve the half space problem
$$\lambda u^+-\Delta u^+=0,\ x_2>0,\ \partial_2 u^+=\tilde{h}_1^+,\ x_2=0,$$
to obtain a unique solution $u^+\in W_p^2(\mathbb{R}\times\mathbb{R}_+\times\mathbb{R})$. Then we repeat the same procedure for $h_1^-:=h_1|_{x_3<0}$ to obtain a unique solution $u^-\in W_p^2(\mathbb{R}\times\mathbb{R}_+\times\mathbb{R})$. Define the function
$$\bar{u}:=
\begin{cases}
u^+,&\quad x_3\ge 0,\\
u^-,&\quad x_3<0,
\end{cases}$$
and consider the problem
\begin{align}\label{eq:auxell5}
\begin{split}
\lambda \tilde{u}-\Delta \tilde{u}&=\bar{f}\quad \mbox{ in } \; \Omega\backslash\Sigma,\\
[\![\rho \tilde{u}]\!]&=\bar{g}_1\quad \mbox{ on } \;\Sigma,\\
[\![\partial_{3} \tilde{u}]\!]&=\bar{g}_2\quad \mbox{ on } \;\Sigma,\\
\partial_{2}\tilde{u}&=0\quad \mbox{ on }\; S_1\backslash\partial\Sigma,
\end{split}
\end{align}
where $\bar{f}:=f$, $\bar{g}_1:=g_1-\Jump{\rho\bar{u}}$ and $\bar{g}_2:=g_2-\Jump{\partial_3\bar{u}}$. Note that by the compatibility condition on $g_1$ and $h_1$ at $\partial\Sigma$ it holds that $\partial_2\bar{g}_1=0$ at $\partial\Sigma$. Therefore it is possible to extend $\bar{f},\bar{g}_j$ by even reflection in $x_2$ to some functions $\hat{f}\in L_p(\mathbb{R}^3)$, $\hat{g}_1\in W_p^{2-1/p}(\mathbb{R}^2)$ and $\hat{g}_2\in W_p^{1-1/p}(\mathbb{R}^2)$. Solve \eqref{eq:auxell2} with $(f,g_j)$ replaced by $(\hat{f},\hat{g}_j)$ to obtain a unique solution $\hat{u}\in W_p^2(\mathbb{R}^2\times\dot{\mathbb{R}})$. Since the function $[x\mapsto \hat{u}(x_1,-x_2,x_3)]$ is a solution of this problem too, it follows by uniqueness that $\partial_2\hat{u}=0$ at $S_1\backslash\partial\Sigma$, hence $\tilde{u}:=\hat{u}|_{\Omega}$ is the unique solution of \eqref{eq:auxell5}. Finally, $u:=\tilde{u}+\bar{u}$ solves \eqref{eq:auxell4} for each $\lambda>0$ and this solution is unique.

Consider now the case of a bent two-phase half space with outer unit normal $-e_{2}$ at $x=0$. To be precise, let
$$\Omega_\theta:=\{x\in\mathbb{R}^3:x_2>\theta(x_1)\},$$
where $\theta\in BC^3(\mathbb{R})$, with $\theta(0)=\theta'(0)=0$ and $\|\theta'\|_\infty+\|\theta\|_\infty\le\eta$, where $\eta>0$ can be made as small as we wish. Furthermore, let $S_{1,\theta}:=\partial\Omega_\theta$ and $\Sigma_\theta:=\{\mathbb{R}^2\times\{0\}\}\cap\Omega_\theta$.
We have to investigate the following problem.
\begin{align}\label{eq:auxell6}
\begin{split}
\lambda u-\Delta u&=f\quad \mbox{ in } \; \Omega_\theta\backslash\Sigma_\theta,\\
[\![\rho u]\!]&=g_1\quad \mbox{ on } \;\Sigma_\theta,\\
[\![\partial_{3} u]\!]&=g_2\quad \mbox{ on } \;\Sigma_\theta,\\
\partial_{\nu_{\partial\Sigma_\theta}}u&=h_1\quad \mbox{ on }\; S_{1,\theta}\backslash\partial\Sigma_\theta.
\end{split}
\end{align}
First of all we extend $f$ , $g_1$ and $g_2$ to some functions $\tilde{f}\in L_p(\mathbb{R}^3)$, $\tilde{g}_1\in W_p^{2-1/p}(\mathbb{R}^2)$ and $\tilde{g}_2\in W_p^{1-1/p}(\mathbb{R}^2)$, respectively. Then we solve \eqref{eq:auxell2} with $(f,g_1,g_2)$ replaced by $(\tilde{f},\tilde{g}_1,\tilde{g}_2)$ to obtain a unique solution $\tilde{u}\in W_p^2(\mathbb{R}^2\times\dot{\mathbb{R}})$. Let $\bar{h}_1:=h_1-\partial_{\nu_{\partial\Sigma_\theta}}\tilde{u}|_{S_{1,\theta}}$ and note that $\Jump{\rho \bar{h}_1}=0$ at $\partial\Sigma_{\theta}$ by the compatibility condition on $(g_1,h_1)$ at $\partial\Sigma_\theta$. We arrive at the problem
\begin{align}\label{eq:auxell7}
\begin{split}
\lambda \bar{u}-\Delta \bar{u}&=0\quad \mbox{ in } \; \Omega_\theta\backslash\Sigma_\theta,\\
[\![\rho \bar{u}]\!]&=0\quad \mbox{ on } \;\Sigma_\theta,\\
[\![\partial_{3} \bar{u}]\!]&=0\quad \mbox{ on } \;\Sigma_\theta,\\
\partial_{\nu_{\partial\Sigma_\theta}}\bar{u}&=\bar{h}_1\quad \mbox{ on }\; S_{1,\theta}\backslash\partial\Sigma_\theta.
\end{split}
\end{align}
Transforming $\Omega_\theta$, $S_{1,\theta}$ and $\Sigma_{\theta}$ to $\Omega=\mathbb{R}\times\mathbb{R}_+\times\mathbb{R}$, $S_1=\mathbb{R}\times\{0\}\times\mathbb{R}$ and $\Sigma=\{\mathbb{R}^2\times\{0\}\}\cap\Omega$ via the diffeomorphism
$$\Omega\ni (\bar{x}_1,\bar{x}_2,\bar{x}_3)\mapsto (\bar{x}_1,\bar{x}_2+\theta(\bar{x}_1),\bar{x}_3)\in\Omega_\theta$$
yields the transformed problem
\begin{align}\label{eq:auxell8}
\begin{split}
\lambda \hat{u}-\Delta \hat{u}&=M_1(\theta,\hat{u})\quad \mbox{ in } \; \Omega\backslash\Sigma,\\
[\![\rho \hat{u}]\!]&=0\quad \mbox{ on } \;\Sigma,\\
[\![\partial_{3} \hat{u}]\!]&=0\quad \mbox{ on } \;\Sigma,\\
\partial_{2}\hat{u}&=M_2(\theta,\hat{u})-\sqrt{1+\theta'^2}\hat{h}_1\quad \mbox{ on }\; S_{1}\backslash\partial\Sigma,
\end{split}
\end{align}
where $\hat{u}(\bar{x}):=\bar{u}(\bar{x}_1,\bar{x}_2+\theta(\bar{x}_1),\bar{x}_3)$, $\hat{h}_1(\bar{x}_1,\bar{x}_3):=\bar{h}(\bar{x}_1,\theta(\bar{x}_1),\bar{x}_3)$,
$$M_1(\theta,\hat{u}):=-2\theta'(\bar{x}_1)\partial_{1}
\partial_{2}\hat{u}-\theta''(\bar{x}_1)\partial_{2}\hat{u}
+\theta'(\bar{x}_1)^2\partial_{2}^2\hat{u},$$
and 
$$M_2(\theta,\hat{u}):=\theta'(\bar{x}_1)\partial_1\hat{u}|_{S_1\backslash\partial\Sigma}
-\theta'(\bar{x}_1)^2\partial_2\hat{u}|_{S_1\backslash\partial\Sigma}.$$ 
Observe that $\Jump{\rho\hat{h}_1}=0$ at $\partial\Sigma$.

Define the function spaces
$$\mathbb{E}:=\{\hat{u}\in W_p^2(\Omega\backslash\Sigma):\Jump{\rho\hat{u}}=\Jump{\partial_3 \hat{u}}=0\ \text{on}\ \Sigma\},$$
equipped with the equivalent norm $\|\hat{u}\|_{\mathbb{E},\lambda}:=\lambda\|\hat{u}\|_{L_p}+\|\hat{u}\|_{W_p^2}$, $\lambda>0$ and
$$\mathbb{F}:=\{(f_1,f_2)\in L_p(\Omega)\times W_p^{1-1/p}(S_1\backslash\partial\Sigma):\Jump{\rho f_2}=0\ \text{at}\ \partial\Sigma\}.$$
Furthermore, let a linear operator $L:\mathbb{E}\to\mathbb{F}$ be defined by
$$L\hat{u}:=\begin{pmatrix}
\lambda\hat{u}-\Delta\hat{u}\\
\partial_2\hat{u}|_{S_1\backslash\partial\Sigma}
\end{pmatrix}.
$$
It follows from our previous arguments that $L:\mathbb{E}\to\mathbb{F}$ is an isomorphism, provided $\lambda>0$.
Furthermore, by the same strategy as in \cite[Section 3.1.1]{Lun95}, there exists $\lambda_0>0$ and a constant $C>0$ such that for all $\lambda\ge\lambda_0$ and $(f_1,f_2)\in\mathbb{F}$ the estimate
\begin{equation}\label{eq:auxell9}
\|L^{-1}(f_1,f_2)\|_{\mathbb{E},\lambda}\le C\left(\|f_1\|_{L_p(\Omega)}+|\lambda|^{1/2}\|\tilde{f}_2\|_{L_p(\Omega)}+\|\nabla\tilde{f}_2\|_{L_p(\Omega)}\right),
\end{equation}
is valid, where $\tilde{f}_2$ is an extension of $f_2$ to $W_p^1(\Omega\backslash\Sigma)$.

Let now $F:=(0,-\sqrt{1+\theta'^2}\hat{h}_1)$ and $M(\theta,\hat{u}):=(M_1,M_2)(\theta,\hat{u})$.
Clearly, for each $\hat{u}\in\mathbb{E}$, it holds that $M(\theta,\hat{u})\in\mathbb{F}$, since
$$\Jump{\rho\theta'(\bar{x}_1)\partial_1\hat{u}}=\theta'(\bar{x}_1)\partial_1
\Jump{\rho\hat{u}}=0$$
at $\partial\Sigma$. Furthermore it holds that
$$\Jump{\rho\sqrt{1+\theta'^2}\hat{h}_1}=\sqrt{1+\theta'^2}
\Jump{\rho\hat{h}_1}=0$$
at $\partial\Sigma$ as well, hence $F\in\mathbb{F}$. Therefore, for $\hat{u}\in \mathbb{E}$, the expressions $L^{-1}M(\theta,\hat{u})$, $L^{-1}F$ are well defined in $\mathbb{E}$ and we may rewrite \eqref{eq:auxell8} in the shorter form
\begin{equation}\label{eq:auxell10}
\hat{u}=L^{-1}M(\theta,\hat{u})+L^{-1}F.
\end{equation}
We will now apply \eqref{eq:auxell9} to the term $L^{-1}M(\theta,\hat{u})$. To this end, note that 
$$\tilde{M}_2(\theta,\hat{u}):=\theta'(\bar{x}_1)\partial_1\hat{u}
-\theta'(\bar{x}_1)^2\partial_2\hat{u}$$
is a proper extension of $M_2(\theta,\hat{u})$ to $W_p^1(\Omega\backslash\Sigma)$. By \eqref{eq:auxell9}, this yields the estimate
\begin{multline*}
\|L^{-1}M(\theta,\hat{u})\|_{\mathbb{E},\lambda}\le \\
\le C\left(\|\theta'\|_{L_\infty(\Omega)}\|\hat{u}\|_{W_p^2(\Omega)}+[\|\theta''\|_{L_\infty(\Omega)}+\lambda^{1/2}\|\theta'\|_{L_\infty(\Omega)}]\|\hat{u}\|_{W_p^1(\Omega)}\right).
\end{multline*}
Clearly, $\|\hat{u}\|_{W_p^2(\Omega)}\le\|\hat{u}\|_{\mathbb{E},\lambda}$ and by complex interpolation we obtain furthermore
$$\|\hat{u}\|_{W_p^1(\Omega)}\le C\|\hat{u}\|_{L_p(\Omega)}^{1/2}\|\hat{u}\|_{W_p^2(\Omega)}^{1/2}\le \frac{1}{\lambda^{1/2}}\|\hat{u}\|_{\mathbb{E},\lambda}.$$
Choosing first $\|\theta'\|_\infty$ sufficiently small and then $\lambda>0$ sufficiently large, it follows that for each $\varepsilon>0$ there exist numbers $\eta_0>0$ and $\lambda_1>0$ such that $\|L^{-1}M(\theta,\hat{u})\|_{\mathbb{E},\lambda}\le\varepsilon\|\hat{u}\|_{\mathbb{E},\lambda}$, whenever $\|\theta'\|_\infty\le\eta\in (0,\eta_0)$ and $\lambda\ge\lambda_1$. Therefore, a Neumann series argument yields a unique solution of \eqref{eq:auxell10}.

(c) The proof for this assertion uses the technique of localization. By Proposition \ref{partitionone2} there exists a finite covering of $\overline\Omega$ and a subordinated partition of unity $\{\phi_k\}_{k=1}^N$ such that $\partial_{\nu_{\partial G}}\phi_k=0$ at $(\partial\Sigma\cup\partial S_2)\cap\supp\phi_k$.

Multiplying each equation in \eqref{eq:appauxlem0} by $\phi_k$, we obtain problems in local coordinates, which correspond to perturbed versions of one of the problems which have been treated in (a) \& (b). Assume that $u$ is a solution of \eqref{eq:appauxlem0}, $u_k:=u\phi_k$, $g_1^k:=g_1\phi_k$ and $h_1^k:=h_1\phi_k$, then $\Jump{\rho u_k}=g_1^k$ and
$$\partial_{\nu_{S_1}}u_k=\phi_k\partial_{\nu_{S_1}}u+u\partial_{\nu_{S_1}}\phi_k=
\phi_kh_1=h_1^k,$$
since $\nu_{S_1}=(\nu_{\partial G},0)^{\sf T}$. In particular, the commutator term in the Neumann boundary condition is identically zero. By the same reason, one has $$\partial_{\nu_{S_1}}g_{1,k}=\phi_k\Jump{\rho h_1}=\Jump{\rho h_1^k},$$
hence the local data $(g_1^k,h_1^k)$ satisfy the compatibility condition at $\partial\Sigma\cap\supp\phi_k$.

The remaining localization procedure follows along standard arguments. We refrain from giving the details and refer the reader e.g.\ to \cite{DHP03}.
\end{proof}

We shall also prove some results on the solvability of \eqref{eq:appauxlem0} in case $\lambda=0$. If $\lambda=0$ and $\Omega$ is unbounded, one cannot expect to obtain $u\in L_p(\Omega)$. Instead, we are looking for solutions $u\in \dot{W}_p^1(\Omega\backslash\Sigma)\cap\dot{W}_p^2(\Omega\backslash\Sigma)$, or equivalently $\nabla u\in W_p^1(\Omega\backslash\Sigma)$.

If $\nabla u\in W_p^1(\Omega\backslash\Sigma)$ is a solution of \eqref{eq:appauxlem0} with $g_1=0$, then, by trace theory, $f\in L_p(\Omega)$, $g_2\in W_p^{1-1/p}(\Sigma)$, $h_1\in W_p^{1-1/p}(S_1\backslash\partial\Sigma)$ and $h_2\in W_p^{1-1/p}(S_2)$. There is some hidden compatibility/regularity condition for the data $(f,g_2,h_1)$. To see this, let $\phi\in C_c^\infty(\overline{\Omega})$. We multiply $\eqref{eq:appauxlem0}_1$ by $\phi$ and integrate by parts, to obtain the identity
\begin{multline*}
\langle(f,g_2,h_1,h_2),\phi\rangle:=\int_\Omega f\phi\ dx+\int_{S_1} h_1\phi\ dS_1\\
+\int_{S_2} h_2\phi\ dS_2-\int_\Sigma g_2\phi\ d\Sigma=\int_\Omega\nabla u\cdot\nabla\phi\ dx.
\end{multline*}
It follows that the linear mapping $[\phi\mapsto \langle(f,g_2,h_1,h_2),\phi\rangle]$ is continuous on $C_c^\infty(\overline{\Omega})$ with respect to the norm $\|\nabla\cdot\|_{L_{p'}(\Omega)}$.

If $\Omega$ is a full space, a (bent) half space or a (bent) quarter space, then it is well known, that $C_c^\infty(\overline{\Omega})$ (hence also $W_{p'}^1(\Omega))$ is dense in $\dot{W}_{p'}^1(\Omega)$ with respect to the norm $\|\nabla\cdot\|_{L_{p'}(\Omega)}$. Therefore, since each functional in $$\hat{W}_p^{-1}(\Omega):=\left(\dot{W}_{p'}^1(\Omega)\right)^*,$$
is uniquely determined by its restriction to $C_c^\infty(\overline{\Omega})$, it follows that $(f,g_2,h_1,h_2)$ yields a well defined element of $\hat{W}_p^{-1}(\Omega)$ with norm given by
\begin{align*}
\|(f,g_2,h_1,h_2)\|_{\hat{W}_p^{-1}}&:=\sup\{\langle(f,g_2,h_1,h_2),\phi\rangle/
\|\nabla\phi\|_{L_{p'}}:\phi\in C_c^\infty(\overline{\Omega})\}\\
&=\sup\{\langle(f,g_2,h_1,h_2),\phi\rangle/
\|\nabla\phi\|_{L_{p'}}:\phi\in W_{p'}^1(\Omega)\}.
\end{align*}
Note that if $\Omega$ is bounded, then the above representation formula for $(f,g_2,h_1,h_2)$ holds for each $\phi\in \dot{W}_{p'}^1(\Omega)$, since $\dot{W}_q^1(\Omega)\subset{W}_q^1(\Omega)$ if $\Omega$ is bounded. This follows for example from the Poincar\'{e}-Wirtinger inequality. However, if $\Omega$ is unbounded, then the above representation for $(f,g_2,h_1,h_2)$ holds at least on the dense subspace $C_c^\infty(\overline{\Omega})$.

Furthermore, if $S_j=\emptyset$, $j\in\{1,2\}$ and/or $\Sigma=\emptyset$, then we simply neglect $h_j$ and/or $g_2$ in $(f,g_2,h_1,h_2)$.

We are now in a position to state the next auxiliary lemma.
\begin{lem}\label{lem:appaux0}
Let $n=2,3$, $p\ge 2$ and $\lambda=0$. Then the following assertions are valid.
\begin{enumerate}
\item If $\Omega$ and $\Sigma$ satisfy one of the conditions in $(a),(b)$ above, then there exists a unique solution $\nabla u\in {W}_p^{1}(\Omega\backslash\Sigma)$ of \eqref{eq:appauxlem0} with $g_1=0$ if and only if $f\in L_p(\Omega)$, $g_2\in W_p^{1-1/p}(\Sigma)$, $h_1\in W_p^{1-1/p}(S_1\backslash\partial\Sigma)$, $h_2\in W_p^{1-1/p}(S_2)$, $\Jump{\rho h_1}=0$ on $\partial\Sigma$ and $(f,g_2,h_1,h_2)\in\hat{W}_{p}^{-1}(\Omega)$.
\item If $\Omega$ and $\Sigma$ are subject to the condition $(c)$ above, then there exists a unique solution $u\in {W}_p^{2}(\Omega\backslash\Sigma)$ of \eqref{eq:appauxlem0} with $g_1=h_1=h_2=0$ if and only if
    $$f\in L_p^{(0)}(\Omega):=\{f\in L_p(\Omega):\int_\Omega f dx=0\}.$$
\end{enumerate}
\end{lem}
\begin{proof}
1. (a) If $\Omega=\mathbb{R}^n$, then we have to solve $-\Delta u=f$ for $f$ in $\hat{W}_p^{-1}(\Omega)\cap L_p(\Omega)$. It is a folkloristic result that whenever $f\in L_p(\mathbb{R}^n)$, then there is a unique solution $u\in \dot{W}_p^2(\mathbb{R}^n)$ of the equation $-\Delta u=f$. Multiplying $-\Delta u=f$ by $\phi\in C_c^\infty(\mathbb{R}^n)$ and integrating by parts, we obtain
$$\int_{\mathbb{R}^n}\nabla u\cdot\nabla\phi\ dx=-\int_{\mathbb{R}^n}\Delta u\phi\ dx=\int_{\mathbb{R}^n} f\phi\ dx.$$
Let us show that there exists a constant $C>0$ such that the estimate
\begin{equation}\label{eq:estweaklaplacian}\|\nabla u\|_{L_p(\mathbb{R}^n)}\le C\sup\left\{\frac{|\int_{\mathbb{R}^n}\nabla u\cdot\nabla\phi\ dx|}{\|\nabla\phi\|_{L_{p'}(\mathbb{R}^n)}}:\phi\in C_c^\infty(\mathbb{R}^n)\right\}
\end{equation}
is valid. Indeed, it holds that
\begin{align}
\begin{split}\label{eq:appauxlem0.02}
\sup\left\{\frac{|\int_{\mathbb{R}^n}\nabla u\cdot\nabla\phi\ dx|}{\|\nabla\phi\|_{L_{p'}(\mathbb{R}^n)}}:\phi\in C_c^\infty(\mathbb{R}^n)\right\}&\ge\frac{|\int_{\mathbb{R}^n}\nabla u\cdot\nabla\partial_j\varphi\ dx|}{\|\nabla\partial_j\varphi\|_{L_{p'}(\mathbb{R}^n)}}\\
&\ge \frac{1}{C}\frac{|\int_{\mathbb{R}^n}\partial_ju\cdot\Delta\varphi\ dx|}{\|\Delta\varphi\|_{L_{p'}(\mathbb{R}^n)}},
\end{split}
\end{align}
for all $\varphi\in C_c^\infty(\mathbb{R}^n)$, where we integrated by parts and applied the Cald\'{e}ron-Zygmund inequality $\|\nabla^2\varphi\|_{L_{p'}(\mathbb{R}^n)}\le C\|\Delta\varphi\|_{L_{p'}(\mathbb{R}^n)}$.


It is well-known that $\Delta C_c^\infty(\mathbb{R}^n)$ is dense in $L_{p'}(\mathbb{R}^n)$ with respect to the $L_{p'}$-norm. Taking the supremum on the right hand side of \eqref{eq:appauxlem0.02} over all functions
$\varphi\in C_c^\infty(\mathbb{R}^n)$ we obtain the desired inequality \eqref{eq:estweaklaplacian}. Evidently, for the solution $u\in\dot{W}_p^2(\mathbb{R}^n)$ of $-\Delta u=f$ it follows that
$$\|\nabla u\|_{L_p(\mathbb{R}^n)}\le C\sup\left\{\frac{|\int_{\mathbb{R}^n}f\phi\ dx|}{\|\nabla\phi\|_{L_{p'}(\mathbb{R}^n)}}:\phi\in C_c^\infty(\mathbb{R}^n)\right\}.$$
hence, if $f\in L_p(\mathbb{R}^n)\cap \hat{W}_p^{-1}(\mathbb{R}^n)$, then
$$\|f\|_{\hat{W}_p^{-1}}=\sup\left\{\frac{|\int_{\mathbb{R}^n}f\phi\ dx|}{\|\nabla\phi\|_{L_{p'}(\mathbb{R}^n)}}:\phi\in C_c^\infty(\mathbb{R}^n)\right\}<\infty,$$
and we obtain the estimate $\|\nabla u\|_{L_p(\mathbb{R}^n)}\le C\|f\|_{\hat{W}_p^{-1}}$. This shows that $u\in \dot{W}_p^1(\mathbb{R}^n)\cap\dot{W}_p^2(\mathbb{R}^n)$ is the unique solution.

Let $\Omega=\mathbb{R}^2\times\mathbb{R}_+$ be a half space and consider the problem
\begin{align}
\begin{split}\label{eq:auxell10aa}
-\Delta u&=f,\quad x\in\Omega,\\
\partial_3 u&=h,\quad x\in S,
\end{split}
\end{align}
where $S:=\partial\Omega=\mathbb{R}^2\times\{0\}$. By Lemma \ref{lem:appaux00} there exists some $\lambda_0>0$ such that the shifted problem
\begin{align}
\begin{split}\label{eq:auxell10ab}
\lambda_0\bar{u}-\Delta \bar{u}&=f,\quad x\in\Omega,\\
\partial_3 \bar{u}&=h,\quad x\in S,
\end{split}
\end{align}
admits a unique solution $\bar{u}\in W_p^2(\Omega)$ satisfying the estimates
$$\|\bar{u}\|_{W_p^2(\Omega)}\le C(\|f\|_{L_p(\Omega)}+\|h\|_{W_p^{1-1/p}(S)}),$$
and
$$\|\bar{u}\|_{W_p^1(\Omega)}\le C\|(f,h)\|_{\hat{W}_p^{-1}(\Omega)}.$$
To see the validity of the second estimate we use the notation from \cite[Chapter V]{Ama95} and let $A_0:=\lambda_0-\Delta$ with domain $$E_1:=D(A_0)=\{u\in W_p^2(\Omega):\partial_3 u=0\ \text{on}\ S\}$$
in $E_0:=L_p(\Omega)$. Then $A_0$ is a linear isomorphism from $E_1$ to $E_0$. Let $E_{1/2}:=[E_0,E_1]_{1/2}=W_p^1(\Omega)$ and $E_{-1/2}:=(E_{1/2}^\sharp)^*=(W_{p'}^1(\Omega))^*$, since  $A_0^\sharp=(\lambda_0-\Delta)|_{L_{p'}(\Omega)}$). Denote by $A_{-1/2}$ the $E_{-1/2}$-realization of $A_0$. By the results in \cite[Chapter V]{Ama95} it follows that $A_{-1/2}:E_{1/2}\to E_{-1/2}$ is a linear isomorphism. Moreover, since $E_1$ is dense in $E_{1/2}$, it holds that
$$\langle A_{-1/2}u,\phi\rangle=\lambda_0\int_\Omega u\phi\ dx+\int_\Omega\nabla u\cdot\nabla\phi\ dx$$
for all $\phi\in W_{p'}^1(\Omega)$ and each $u\in W_p^1(\Omega)$.

Multiply the first equation in \eqref{eq:auxell10ab} by $\phi\in W_{p'}^1(\Omega)$ and integrate by parts to the result
$$\lambda_0\int_\Omega \bar{u}\phi\ dx+\int_\Omega\nabla \bar{u}\cdot\nabla\phi\ dx=\int_\Omega f\phi\ dx-\int_S h\phi|_{S}\ dS.$$
By assumption, the right side of the last equation determines a functional $(f,h)$ on $\dot{W}_{p'}^1(\Omega)$, hence also on $W_{p'}^1(\Omega)$. Therefore it follows from the considerations above that
\begin{multline*}
\|\bar{u}\|_{W_p^1(\Omega)}\le C\|(f,h)\|_{(W_{p'}^1(\Omega))^*}=C\sup_{0\neq\phi\in W_{p'}^1(\Omega)}
\frac{|\langle(f,h),\phi\rangle|}{\|\phi\|_{W_{p'}^1(\Omega)}}\\ \le
C\sup_{0\neq\phi\in W_{p'}^1(\Omega)}
\frac{|\langle(f,h),\phi\rangle|}{\|\nabla\phi\|_{L_{p'}(\Omega)}}
=C\|(f,h)\|_{\hat{W}_p^{-1}(\Omega)}.
\end{multline*}
Therefore it suffices to study the problem
\begin{align}
\begin{split}\label{eq:auxell10ac}
-\Delta u_*&=f_*,\quad x\in\Omega,\\
\partial_3 u_*&=0,\quad x\in S,
\end{split}
\end{align}
where $f_*:=f+\Delta \bar{u}$. Observe that $f_*\in L_p(\Omega)\cap\hat{W}_p^{-1}(\Omega)$. We extend $f_*$ with respect to $x_3$ by even reflection to some $\tilde{f}$ to obtain $\tilde{f}\in L_p(\mathbb{R}^3)\cap \hat{W}_p^{-1}(\mathbb{R}^3)$. Solve the full space problem $-\Delta \tilde{u}=\tilde{f}$ to obtain a unique solution $\tilde{u}\in\dot{W}_p^1(\mathbb{R}^3)\cap\dot{W}_p^2(\mathbb{R}^3)$. By uniqueness and symmetry, it follows that $\tilde{u}(x_1,x_2,x_3)=\tilde{u}(x_1,x_2,-x_3)$, hence $\partial_3\tilde{u}=0$ on $S$. Since
$$\|\nabla u_*\|_{L_p(\Omega)}\le\|\nabla \tilde{u}\|_{L_p(\mathbb{R}^3)}\le C\|\tilde{f}\|_{\hat{W}_p^{-1}(\mathbb{R}^3)},$$
and $\|\tilde{f}\|_{\hat{W}_p^{-1}(\mathbb{R}^3)}\le 2\|f_*\|_{\hat{W}_p^{-1}(\Omega)}$ ($\tilde{f}$ is the even extension of $f_*$) it follows that
$$\|\nabla u_*\|_{L_p(\Omega)}\le C\|f_*\|_{\hat{W}_p^{-1}(\Omega)}.$$
The function $u:=\bar{u}+\tilde{u}|_{\Omega}=\bar{u}+u_*$ is the desired unique solution of \eqref{eq:auxell10a}, satisfying the estimates
$$\|\nabla^2 u\|_{L_p(\Omega)}\le C(\|f\|_{L_p(\Omega)}+\|h\|_{W_p^{1-1/p}(S)}),$$
and
$$\|\nabla u\|_{L_p(\Omega)}\le C\|(f,h)\|_{\hat{W}_p^{-1}(\Omega)}.$$
Uniqueness follows by even reflection of the solution of \eqref{eq:auxell10aa} with $f=h=0$ at $S$ and the uniqueness result for the full space.

If $\Omega=\mathbb{R}\times\mathbb{R}_+\times\mathbb{R}_+$ is a quarter space, we have to solve
\begin{align}
\begin{split}\label{eq:auxell10a}
-\Delta u&=f,\quad x\in\Omega,\\
\partial_2 u&=h_1,\quad x\in S_1,\\
\partial_3 u&=h_2,\quad x\in S_2,
\end{split}
\end{align}
where $S_1=\mathbb{R}\times\{0\}\times\mathbb{R}_+$ and $S_2=\mathbb{R}\times\mathbb{R}_+\times\{0\}$. The data satisfy $f\in L_p(\Omega)$, $h_j\in W_p^{1-1/p}(S_j)$, $j=1,2$ and $(f,h_1,h_2)\in \hat{W}_p^{-1}(\Omega)$.

By Lemma \ref{lem:appaux00} we first solve
\begin{align}
\begin{split}\label{eq:auxell10b}
\lambda_0\bar{u}-\Delta \bar{u}&=f,\quad x\in\Omega,\\
\partial_2 \bar{u}&=h_1,\quad x\in S_1,\\
\partial_3 \bar{u}&=h_2,\quad x\in S_2,
\end{split}
\end{align}
for some sufficiently large $\lambda_0>0$ to obtain a unique solution $\bar{u}\in W_p^2(\Omega)$. Note that $\bar{u}$ satisfies the estimates
$$\|\bar{u}\|_{W_p^2(\Omega)}\le C(\|f\|_{L_p(\Omega)}+\|h_1\|_{W_p^{1-1/p}(S_1)}+
\|h_2\|_{W_p^{1-1/p}(S_2)}),$$
and
$$\|\bar{u}\|_{W_p^1(\Omega)}\le C\|(f,h_1,h_2)\|_{\hat{W}_p^{-1}(\Omega)}.$$
We arrive at the problem
\begin{align}
\begin{split}\label{eq:auxell10c}
-\Delta {u}_*&={f}_*,\quad x\in\Omega,\\
\partial_2 {u}_*&=0,\quad x\in S_1,\\
\partial_3 {u}_*&=0,\quad x\in S_2,
\end{split}
\end{align}
where ${f}_*:=f+\Delta\bar{u}\in \hat{W}_p^{-1}(\Omega)\cap L_p(\Omega)$, which follows from integration by parts.
Extend $f_*$ to the half space $\mathbb{R}^3_+$ by even reflection, i.e.\ we set
$$\tilde{f}(x):=\begin{cases}
f_*(x_1,x_2,x_3),&\quad x_2\ge 0,\\
f_*(x_1,-x_2,x_3),&\quad x_2<0.
\end{cases}$$
Then $\tilde{f}\in \hat{W}_p^{-1}(\mathbb{R}^3_+)\cap L_p(\mathbb{R}^3_+)$. Next we extend $\tilde{f}$ by even reflection to the full space $\mathbb{R}^3$ by defining
$$\hat{f}(x):=\begin{cases}
\tilde{f}(x_1,x_2,x_3),&\quad x_3\ge 0,\\
\tilde{f}(x_1,x_2,-x_3),&\quad x_3<0,
\end{cases}$$
This yields that $\hat{f}\in \hat{W}_p^{-1}(\mathbb{R}^3)\cap L_p(\mathbb{R}^3)$. Solve the full space problem $-\Delta \hat{u}=\hat{f}$ to obtain a unique solution $\hat{u}\in \dot{W}_p^{1}(\mathbb{R}^3)\cap\dot{W}_p^{2}(\mathbb{R}^3)$. Since with $\hat{u}$ also $\hat{u}(-x_3)$ and $\hat{u}(-x_2)$ are solutions of $-\Delta \hat{u}=\hat{f}$, it follows from the uniqueness of the solution that $\hat{u}(x_3)=\hat{u}(-x_3)$ and $\hat{u}(x_2)=\hat{u}(-x_2)$, hence $\partial_3 \hat{u}=0$ on $S_2$ as well as $\partial_2 \hat{u}=0$ on $S_1$. Since
$$\|\nabla u_*\|_{L_p(\Omega)}\le\|\nabla \hat{u}\|_{L_p(\mathbb{R}^3)}\le C\|\hat{f}\|_{\hat{W}_p^{-1}(\mathbb{R}^3)},$$
and $\|\hat{f}\|_{\hat{W}_p^{-1}(\mathbb{R}^3)}\le C\|f_*\|_{\hat{W}_p^{-1}(\Omega)}$ it follows that
$$\|\nabla u_*\|_{L_p(\Omega)}\le C\|f_*\|_{\hat{W}_p^{-1}(\Omega)}.$$
The function $u:=\bar{u}+\hat{u}|_{\Omega}=\bar{u}+u_*$ is the desired unique solution of \eqref{eq:auxell10a}, satisfying the estimates
$$\|\nabla^2 u\|_{L_p(\Omega)}\le C(\|f\|_{L_p(\Omega)}+\|h_1\|_{W_p^{1-1/p}(S_1)}+
\|h_2\|_{W_p^{1-1/p}(S_2)}),$$
and
$$\|\nabla u\|_{L_p(\Omega)}\le C\|(f,h_1,h_2)\|_{\hat{W}_p^{-1}(\Omega)}.$$

If $\Omega$ is a bent quarter space, we will use change of coordinates and perturbation theory to prove the assertion in this case. We will give a detailed proof for the case of a bent two-phase half space below. The technique from this case carries over to the bent quarter space case.

(b) Let $\Omega=\mathbb{R}^3$ and $\Sigma=\mathbb{R}^2\times\{0\}$. Consider the problem
\begin{align}\label{eq:auxell11}
\begin{split}
-\Delta {u}&={f}\quad \mbox{ in } \; \Omega\backslash\Sigma,\\
[\![\rho {u}]\!]&=0\quad \mbox{ on } \;\Sigma,\\
[\![\partial_{3} {u}]\!]&={g}_2\quad \mbox{ on } \;\Sigma,
\end{split}
\end{align}
with $f\in L_p(\Omega)$, $g_2\in W_p^{1-1/p}(\Sigma)$ and $(f,g_2)\in \hat{W}_p^{-1}(\Omega)$.

By Lemma \ref{lem:appaux00} we may first solve the problem
\begin{align}\label{eq:auxell12}
\begin{split}
\lambda_0 \bar{u}-\Delta \bar{u}&={f}\quad \mbox{ in } \; \Omega\backslash\Sigma,\\
[\![\rho \bar{u}]\!]&=0\quad \mbox{ on } \;\Sigma,\\
[\![\partial_{3} \bar{u}]\!]&={g}_2\quad \mbox{ on } \;\Sigma,
\end{split}
\end{align}
where $\lambda_0>0$ is sufficiently large but fixed. This yields a unique solution $\bar{u}\in W_p^2(\Omega\backslash\Sigma)$. Next we consider the equation $-\Delta\tilde{u}=\tilde{f}$ in $\mathbb{R}^3$, where $\tilde{f}:=f+\Delta \bar{u}\in \hat{W}_p^{-1}(\mathbb{R}^3)\cap L_p(\mathbb{R}^3)$, since
$$\int_{\mathbb{R}^3}(f+\Delta\bar{u})\phi\ dx=-\int_\Sigma g_2\phi\ d\Sigma+
\int_{\mathbb{R}^3} f\phi\ dx-\int_{\mathbb{R}^3}\nabla\bar{u}\cdot\nabla\phi\ dx.$$
by what we have already shown, this full space problem admits a unique solution $\tilde{u}\in \dot{W}_p^1(\Omega)\cap\dot{W}_p^2(\Omega)$. Finally we study the problem
\begin{align}\label{eq:auxell13}
\begin{split}
-\Delta \hat{u}&=0\quad \mbox{ in } \; \Omega\backslash\Sigma,\\
[\![\rho \hat{u}]\!]&=\hat{g}_1\quad \mbox{ on } \;\Sigma,\\
[\![\partial_{3} \hat{u}]\!]&=0\quad \mbox{ on } \;\Sigma,
\end{split}
\end{align}
with $\hat{g}_1:=-\Jump{\rho\tilde{u}}\in \dot{W}_p^{1-1/p}(\Sigma)\cap\dot{W}_p^{2-1/p}(\Sigma)$. The solution is given in terms of the Poisson semigroup as follows.
$$\hat{u}(x_3)=\frac{1}{\rho_1+\rho_2}\begin{cases}
e^{-L x_3}\hat{g}_1,&\quad x_3\ge 0,\\
-e^{-L(-x_3)}\hat{g}_1,&\quad x_3<0,
\end{cases}$$
where $L:=(-\Delta_{x'})^{1/2}$. By semigroup theory it follows that $\hat{u}\in \dot{W}_p^{1}(\Omega\backslash\Sigma)\cap\dot{W}_p^{2}(\Omega\backslash\Sigma)$. Here we use the fact that
\begin{equation}\label{eq:equivnormhom}
\left(\int_0^\infty z^{(k-s)p}\|L^ke^{-Lz}g\|_{L_p(\Sigma)}^p \frac{dz}{z}\right)^{1/p}
\end{equation}
defines an equivalent norm in $\dot{W}_p^s(\Sigma)$ for $s>0$ and $k>s$ (if $s=j-1/p$, $j\in\{1,2\}$, we choose $k=j$). The function $u:=\bar{u}+\tilde{u}+\hat{u}$ is the unique solution of \eqref{eq:auxell11}, satisfying the estimates
$$\|\nabla^2 u\|_{L_p(\Omega)}\le C(\|f\|_{L_p(\Omega)}+\|g_2\|_{W_p^{1-1/p}(\Sigma)}),$$
and
$$\|\nabla u\|_{L_p(\Omega)}\le C\|(f,g_2)\|_{\hat{W}_p^{-1}(\Omega)}.$$
The uniqueness of the solution can be seen as follows. Let $u\in \dot{W}_p^1(\Omega\backslash\Sigma)\cap\dot{W}_p^2(\Omega\backslash\Sigma)$ be a solution of \eqref{eq:auxell11} with $f=g_2=0$. We want to show that $\nabla u=0$ in $\Omega\backslash\Sigma$. To this end we define two functions
$$v(x_1,x_2,x_3):=\rho_2 u_+(x_1,x_2,x_3)-\rho_1 u_-(x_1,x_2,-x_3),\ (x_1,x_2)\in\mathbb{R}^2,\ x_3>0,$$
and
$$w(x_1,x_2,x_3):=u_+(x_1,x_2,x_3)+u_-(x_1,x_2,-x_3),\ (x_1,x_2)\in\mathbb{R}^2,\ x_3>0,$$
where $u_\pm:=u|_{x_3\gtrless 0}$. It follows that $v$ and $w$ solve the half space problems
$$\Delta v=0,\ (x_1,x_2)\in\mathbb{R}^2,\ x_3>0,\quad v=0,\ (x_1,x_2)\in\mathbb{R}^2,\ x_3=0,$$
and
$$\Delta w=0,\ (x_1,x_2)\in\mathbb{R}^2,\ x_3>0,\quad \partial_3w=0,\ (x_1,x_2)\in\mathbb{R}^2,\ x_3=0,$$
respectively in $\dot{W}_p^1(\mathbb{R}_+^3)\cap\dot{W}_p^2(\mathbb{R}_+^3)$. Therefore $\nabla w=\nabla v=0$ by even or uneven reflection at $\{x_3=0\}$. This yields
\begin{align*}
0&=\rho_2\nabla u_+ +\rho_1\nabla u_-,\\
0&=\nabla u_+-\nabla u_-,
\end{align*}
wherefore $\nabla u_-=\nabla u_+=0$, hence $\nabla u=0$ in $\Omega\backslash\Sigma$.

Let now $\Omega=\mathbb{R}^{2}_+\times\mathbb{R}$ with $\Sigma=\{\mathbb{R}^{2}\times\{0\}\}\cap\Omega$. Here we have to solve the problem
\begin{align}\label{eq:auxell14}
\begin{split}
-\Delta u&=f\quad \mbox{ in } \; \Omega\backslash\Sigma,\\
[\![\rho u]\!]&=0\quad \mbox{ on } \;\Sigma,\\
[\![\partial_{3} u]\!]&=g_2\quad \mbox{ on } \;\Sigma,\\
\partial_{2}u&=h_1\quad \mbox{ on }\; S_1\backslash\partial\Sigma,
\end{split}
\end{align}
with $\Jump{\rho h_1}=0$ at $\partial\Sigma$. For some large $\lambda_0>0$, we first solve the problem
\begin{align}\label{eq:auxell15}
\begin{split}
\lambda_0\bar{u}-\Delta \bar{u}&=f\quad \mbox{ in } \; \Omega\backslash\Sigma,\\
[\![\rho \bar{u}]\!]&=0\quad \mbox{ on } \;\Sigma,\\
[\![\partial_{3} \bar{u}]\!]&=g_2\quad \mbox{ on } \;\Sigma,\\
\partial_{2}\bar{u}&=h_1\quad \mbox{ on }\; S_1\backslash\partial\Sigma,
\end{split}
\end{align}
by Lemma \ref{lem:appaux00}, to obtain a unique solution $\bar{u}\in W_p^2(\Omega\backslash\Sigma)$. Let $f_*:=f+\Delta \bar{u}$ and note that $f_*\in \hat{W}_p^{-1}(\Omega)\cap L_p(\Omega)$, which follows from integration by parts and from the assumption on $(f,g_2,h_1)$.
We extend $f_*$ with respect to $x_2$ by even reflection to some function
$$\tilde{f}(x):=\begin{cases}
f_*(x_1,x_2,x_3),&\quad x_2\ge 0,\\
f_*(x_1,-x_2,x_3),&\quad x_2<0.
\end{cases}$$
Then $\tilde{f}\in \hat{W}_p^{-1}(\mathbb{R}^3)\cap L_p(\mathbb{R}^3)$ and we may solve the full space problem $-\Delta \tilde{u}=\tilde{f}$ to obtain a unique solution $\tilde{u}\in \dot{W}_p^{1}(\mathbb{R}^3)\cap\dot{W}_p^{2}(\mathbb{R}^3)$ with the property $\tilde{u}(x_2)=\tilde{u}(-x_2)$, hence $\partial_2\tilde{u}=0$ at $S_1\backslash\partial\Sigma$. Consider next the problem
\begin{align}\label{eq:appauxlem0.1}
\begin{split}
\Delta \hat{u}&=0\quad \mbox{ in } \; \mathbb{R}^n\backslash\Sigma,\\
[\![\rho \hat{u}]\!]&=g\quad \mbox{ on } \;\Sigma,\\
[\![\partial_3 \hat{u}]\!]&=0\quad \mbox{ on } \;\Sigma,
\end{split}
\end{align}
where $g:=-\Jump{\rho\tilde{u}}\in \dot{W}_p^{1-1/p}(\Sigma)\cap \dot{W}_p^{2-1/p}(\Sigma)$. As in the previous case, the unique solution $\hat{u}\in \dot{W}_p^1(\Omega)\cap\dot{W}_p^1(\Omega)$ of \eqref{eq:appauxlem0.1} is given in terms of the Poisson semigroup.

Finally, since $\tilde{u}(x_2)=\tilde{u}(-x_2)$, it follows that $g(x_2)=g(-x_2)$, hence $\hat{u}(x_2)=\hat{u}(-x_2)$, by uniqueness, and therefore $\partial_2 \hat{u}=0$ at $S_1\backslash\partial\Sigma$. The function $u:=\bar{u}+\tilde{u}|_\Omega+\hat{u}|_\Omega$ is the unique solution of \eqref{eq:auxell14}, satisfying the estimates
$$\|\nabla^2 u\|_{L_p(\Omega)}\le C (\|f\|_{L_p(\Omega)}+\|h_1\|_{W_p^{1-1/p}(S_1\backslash\partial\Sigma)}+
\|g_2\|_{W_p^{1-1/p}(\Sigma)}),$$
and
$$\|\nabla u\|_{L_p(\Omega)}\le C\|(f,h_1,g_2)\|_{\hat{W}_p^{-1}(\Omega)}.$$

In a next step we consider the case of a bent two-phase half space. To be precise, we assume that
$$\Omega_\theta:=\{x\in\mathbb{R}^3:x_2>\theta(x_1)\},$$
where $\theta\in BC^3(\mathbb{R})$ with $\|\theta\|_\infty+\|\theta'\|_\infty<\eta$, and $\eta>0$ can be made as small as we wish. Let furthermore $S_{1,\theta}:=\{x\in\mathbb{R}^3:x_2=\theta(x_1)\}$ and $\Sigma_\theta:=\{\mathbb{R}^2\times\{0\}\}\cap\Omega_\theta$.
Consider the problem
\begin{align}\label{eq:auxell16}
\begin{split}
-\Delta u&=f\quad \mbox{ in } \; \Omega_\theta\backslash\Sigma_\theta,\\
[\![\rho u]\!]&=0\quad \mbox{ on } \;\Sigma_\theta,\\
[\![\partial_{3} u]\!]&=g_2\quad \mbox{ on } \;\Sigma_\theta,\\
\partial_{\nu_{\partial\Sigma_\theta}}u&=h_1\quad \mbox{ on }\; S_{1,\theta}\backslash\partial\Sigma_\theta,
\end{split}
\end{align}
where $f\in L_p(\Omega_\theta)$, $g_2\in W_p^{1-1/p}(\Sigma_\theta)$, $h_1\in W_p^{1-1/p}(S_{1,\theta}\backslash\partial\Sigma_\theta)$ and $(f,g_2,h_1)\in \hat{W}_p^{-1}(\Omega_\theta)$. Moreover, the compatibility condition $\Jump{\rho h_1}=0$ at $\partial\Sigma_\theta$ holds.

By means of Lemma \ref{lem:appaux00}, we may solve the problem
\begin{align}\label{eq:auxell17}
\begin{split}
\lambda_0\hat{u}-\Delta \hat{u}&={f}\quad \mbox{ in } \; \Omega_\theta\backslash\Sigma_\theta,\\
[\![\rho \hat{u}]\!]&=0\quad \mbox{ on } \;\Sigma_\theta,\\
[\![\partial_{3} \hat{u}]\!]&={g}_2\quad \mbox{ on } \;\Sigma_\theta,\\
\partial_{\nu_{\partial\Sigma_\theta}}\hat{u}&=h_1\quad \mbox{ on }\; S_{1,\theta}\backslash\partial\Sigma_\theta,
\end{split}
\end{align}
where $\lambda_0>0$ is large but fixed. This yields a unique solution $\hat{u}\in W_p^2(\Omega_\theta\backslash\Sigma_\theta)$. Let $\tilde{f}:=f+\Delta\hat{u}$ and consider
\begin{align}\label{eq:auxell18}
\begin{split}
-\Delta \tilde{u}&=\tilde{f}\quad \mbox{ in } \; \Omega_\theta\backslash\Sigma_\theta,\\
[\![\rho \tilde{u}]\!]&=0\quad \mbox{ on } \;\Sigma_\theta,\\
[\![\partial_{3} \tilde{u}]\!]&=0\quad \mbox{ on } \;\Sigma_\theta,\\
\partial_{\nu_{\partial\Sigma_\theta}}\tilde{u}&=0\quad \mbox{ on }\; S_{1,\theta}\backslash\partial\Sigma_\theta.
\end{split}
\end{align}
Observe that $\tilde{f}\in \hat{W}_p^{-1}(\Omega_\theta)\cap L_p(\Omega_\theta)$. We will now transform $\Omega_\theta$ to $\Omega_0$ by means of the coordinates $\bar{x}_1:=x_1$, $\bar{x}_2:=x_2-\theta (x_1)$ and $\bar{x}_3:=x_3$. Assume that $\tilde{u}$ solves \eqref{eq:auxell18} and define $\bar{u}(\bar{x}):=\tilde{u}(\bar{x}_1,\bar{x}_2+\theta(\bar{x}_1),\bar{x}_3)$. Then, the function $\bar{u}$ is a solution of the problem
\begin{align}\label{eq:auxell19}
\begin{split}
-\Delta \bar{u}&=\bar{f}+M_1(\theta,\bar{u})\quad \mbox{ in } \; \Omega\backslash\Sigma,\\
[\![\rho \bar{u}]\!]&=0\quad \mbox{ on } \;\Sigma,\\
[\![\partial_{3} \bar{u}]\!]&=0\quad \mbox{ on } \;\Sigma,\\
-\partial_{2}\bar{u}&=M_2(\theta,\bar{u})\quad \mbox{ on }\; S_{1}\backslash\partial\Sigma,
\end{split}
\end{align}
where $\bar{f}$ is the transformation of $\tilde{f}$,
$$M_1(\theta,\bar{u}):=-2\theta'(\bar{x}_1)\partial_{1}
\partial_{2}\bar{u}-\theta''(\bar{x}_1)\partial_{2}\bar{u}
+\theta'(\bar{x}_1)^2\partial_{2}^2\bar{u},$$
and $M_2(\theta,\bar{u}):=-\theta'(\bar{x}_1)\partial_1\bar{u}|_{S_1\backslash\partial\Sigma}+
\theta'(\bar{x}_1)^2\partial_2\bar{u}|_{S_1\backslash\partial\Sigma}$.

Define the function spaces
$$\mathbb{E}:=\{\nabla\bar{u}\in W_p^1(\Omega\backslash\Sigma):\Jump{\rho\bar{u}}=\Jump{\partial_3 \bar{u}}=0\ \text{on}\ \Sigma\},$$
with the equivalent norm
$\|\bar{u}\|_{\mathbb{E},\lambda}:=\lambda\|\nabla\bar{u}\|_{L_p}+\|\nabla^2\bar{u}\|_{L_p}$, $\lambda>0$, and let
\begin{multline*}
\mathbb{F}:=\{(f_1,f_2)\in L_p(\Omega)\times W_p^{1-1/p}(S_1\backslash\partial\Sigma):\Jump{\rho f_2}=0\ \text{at}\ \partial\Sigma\ \text{and}\ (f_1,f_2)\in \hat{W}_p^{-1}(\Omega)\}.
\end{multline*}
Moreover, we define a linear operator $L:\mathbb{E}\to\mathbb{F}$ by
$$L\bar{u}:=\begin{pmatrix}
\Delta\bar{u}\\
\partial_2\bar{u}|_{S_1\backslash\partial\Sigma}
\end{pmatrix}.
$$
It follows from our previous considerations that $L:\mathbb{E}\to\mathbb{F}$ is an isomorphism and, by the same arguments as in \cite[Section 3.1.1]{Lun95}, there exists $\lambda_0>0$ and a constant $C>0$ such that for all $\lambda\ge \lambda_0$ and $(f_1,f_2)\in\mathbb{F}$ the estimate
$$\|L^{-1}(f_1,f_2)\|_{\mathbb{E},\lambda}\le C\left(\|f_1\|_{L_p(\Omega)}+\lambda^{1/2}\|\tilde{f}_2\|_{L_p(\Omega)}+\|\nabla \tilde{f}_2\|_{L_p(\Omega)}\right)$$
holds, where $\tilde{f}_2$ is an extension of $f_2$ to $W_p^1(\Omega\backslash\Sigma)$.

Let $F:=(\bar{f},0)$ and $M(\theta,\bar{u}):=(M_1,M_2)(\theta,\bar{u})$. It follows that $F\in\mathbb{F}$, since
$$\int_{\Omega_\theta}\tilde{f}\phi\ dx=\int_\Omega\bar{f}\bar{\phi}\ d\bar{x},$$
with $\bar{\phi}(\bar{x}):=\phi(\bar{x}_1,\bar{x}_2-\theta(\bar{x}_1),\bar{x}_3)$ and $\phi\in C_c^\infty(\overline{\Omega_\theta})$. Furthermore, for each $\bar{u}\in\mathbb{E}$, we have $M(\theta,\bar{u})\in\mathbb{F}$. Indeed, as in the proof of Lemma \ref{lem:appaux00}, it can be readily checked that $M(\theta,\bar{u})\in L_p(\Omega)\times W_p^{1-1/p}(S_1\backslash\partial\Sigma)$ and $\Jump{\rho M_2(\theta,\bar{u})}=0$ at $\partial\Sigma$. It remains to verify the condition $(M_1,M_2)(\theta,\bar{u})\in \hat{W}_p^{-1}(\Omega)$. To this end, we integrate by parts to obtain the identity
\begin{multline*}
\int_\Omega M_1(\theta,\bar{u})\phi\ dx+\int_{S_1} M_2(\theta,\bar{u})\phi\ dS_1\\
=\int_\Omega\left(\theta'(\bar{x}_1)\partial_2\bar{u}\partial_1\phi+
\theta'(\bar{x}_1)\partial_1\bar{u}\partial_2\phi-\theta'(\bar{x}_1)^2
\partial_2\bar{u}\partial_2\phi\right)\ dx.
\end{multline*}
for each $\phi\in C_c^\infty(\overline{\Omega})$. This in turn yields the claim. We are now in a position to write \eqref{eq:auxell19} in the shorter form
$$\bar{u}=L^{-1}M(\theta,\bar{u})+L^{-1}F.$$
We may now follow the lines of the proof of Lemma \ref{lem:appaux00} to obtain a unique solution of \eqref{eq:auxell16}.

2. It follows from Lemma \ref{lem:appaux00} that the operator $Au:=-\Delta u$ with domain
$$D(A)=\{u\in W_p^{2}(\Omega\backslash\Sigma):[\![\rho u]\!]=[\![\partial_{\nu_\Sigma} u]\!]=0,\ \partial_{\nu_{S_j}}u=0\},$$
is closed. Since $D(A)$ is compactly embedded in $L_p(\Omega)$, the spectrum $\sigma(A)$ consists solely of isolated eigenvalues and $\lambda=0$ is a simple eigenvalue of $A$. Indeed, $N(A)=\lin{\mathds{1}_\rho}$, with
$$\mathds{1}_\rho:=\chi_{\Omega_1}+\frac{\rho_1}{\rho_2}\chi_{\Omega_2}.$$
Furthermore, $N(A^2)\subset N(A)$, since if $u\in N(A^2)$, then $v:=Au\in N(A)$. It follows that $v\in L_1(\Omega)$ and we may integrate $Au=v$ over $\Omega$ to obtain
$$\int_\Omega v\ dx=-\int_\Omega\Delta u\ dx=0,$$
hence $v=0$, since $\mathds{1}_\rho$ has a non-vanishing mean value.

In particular this yields $L_p(\Omega)=N(A)\oplus R(A)$ and it holds that $R(A)=L_p^{(0)}(\Omega)$. This can be seen as follows. Obviously one has the inclusion
$$R(A)\subset L_p^{(0)}(\Omega).$$
So, let $f\in L_p^{(0)}(\Omega)$. Then there exist unique $f_1\in N(A)$ and $f_2\in R(A)$ such that $f=f_1+f_2$. This in turn yields $f_1\in L_p^{(0)}(\Omega)$. Since $f_1=\alpha \mathds{1}_\rho$ for some $\alpha\in\mathbb{R}$ with
$$\mathds{1}_\rho:=\chi_{\Omega_1}+\frac{\rho_1}{\rho_2}\chi_{\Omega_2},$$
it follows that
$$(f_1|\mathds{1})=\alpha\left(|\Omega_1|+\frac{\rho_1}{\rho_2}|\Omega_2|\right),$$
hence $\alpha=0$ and therefore $f=f_2\in R(A)$, hence $L_p^{(0)}(\Omega)\subset R(A)$.
\end{proof}

We will also need an existence and uniqueness result for the weak version of \eqref{eq:appauxlem0} with $\lambda=0$. To be precise, we consider the problem
\begin{align}\label{eq:appauxlem0.01}
\begin{split}
(\nabla u|\nabla\phi)_2&=\langle f,\phi\rangle,\quad \phi\in W_{p'}^1(\Omega),\\
\Jump{\rho u}&=g,\quad \text{on} \; \Sigma.
\end{split}
\end{align}
Then we have the following result.
\begin{lem}\label{lem:appauxlemweak}
Let $\rho_j>0$, $n=2,3$, $p\ge 2$ and let $\Omega\subset\mathbb{R}^n$ satisfy condition (c) from above. Then there exists a unique solution $u\in \dot{W}_p^{1}(\Omega\backslash\Sigma)$ of
\eqref{eq:appauxlem0.01} if and only if $f\in \hat{W}_p^{-1}(\Omega)$ and $g\in W_p^{1-1/p}(\Sigma)$.
\end{lem}
\begin{proof}
Let $g\in W_p^{1-1/p}(\Sigma)$.  The Neumann Laplacian $\Delta_N$ in $L_p(\Sigma)$ with domain
$$D(\Delta_N)=\{u\in W_p^2(\Sigma):\partial_{\nu_{\partial G}} u=0\ \text{on}\ \partial\Sigma\}$$
generates an analytic semigroup. In particular, $D(\Delta_N)$ is dense in
$$W_p^{1-1/p}(\Sigma)=(L_p(\Sigma),D(\Delta_N))_{1/2-1/2p}=D_{\Delta_N}(1/2-1/2p,p).$$
Therefore, there exists $(g_n)_{n\in\mathbb{N}}\subset W_p^{2-1/p}(\Sigma)$ such that $\partial_{\nu_{\partial G}}g_n=0$ for each $n\in\mathbb{N}$ on $\partial\Sigma$ and $g_n\to g$ as $n\to\infty$ in $W_p^{1-1/p}(\Sigma)$. Denote by $u_n\in W_p^{2}(\Omega\backslash\Sigma)$ the solution of \eqref{eq:appauxlem0} with $f=g_2=h_1=h_2=0$, $g_1=g_n$ and a fixed $\lambda\ge\lambda_0$. Making use of local coordinates one can show that the estimate
$$\|u_n-u_m\|_{W_p^{1}(\Omega\backslash\Sigma)}\le C\|g_n-g_m\|_{W_p^{1-1/p}(\Sigma)}$$
is valid, with some constant $C>0$ which does not depend on $n$. Indeed, each of the local charts yields a transformed problem which is subject to one of the conditions in (a) and (b) above. We have already seen in the proof of Lemma \ref{lem:appaux00} that the two-phase half space and the quarter space can be drawn back to a two-phase full space and an ordinary half space, respectively, by means of reflection techniques. Making use of change of coordinates, perturbation theory and the results in \cite[Section 8]{KPW10} one obtains the desired estimate.

In particular, $(u_n)$ is a Cauchy sequence in $W_p^1(\Omega\backslash\Sigma)$ and therefore it has a limit point $u\in W_p^{1}(\Omega\backslash\Sigma)$. By trace theory it follows that $u$ satisfies the weak problem
\begin{align}\label{eq:appauxlem0.3}
\begin{split}
\lambda (u|\phi)_2+(\nabla u|\nabla\phi)_2&=0,\quad \phi\in W_{p'}^1(\Omega),\\
\Jump{\rho u}&=g,\quad \text{on} \; \Sigma,
\end{split}
\end{align}
for some fixed $\lambda\ge\lambda_0$.

Next, let $$a:\{u\in W_p^1(\Omega\backslash\Sigma):\Jump{\rho u}=0\ \text{on}\ \Sigma\}\times W_{p'}^1(\Omega)\to\mathbb{R},\quad a(u,\phi):=\int_\Omega\nabla u\cdot\nabla\phi dx,$$
and define an operator $B:W_p^1(\Omega\backslash\Sigma)\to (W_{p'}^1(\Omega))^*$
with domain
$$D(B)=\{u\in W_p^1(\Omega\backslash\Sigma):\Jump{\rho u}=0\ \text{on}\ \Sigma\},$$
by means of $\langle Bu,\phi\rangle:=a(u,\phi)$. It follows from integration by parts that the operator $A$ from the proof of the second assertion of Lemma \ref{lem:appaux0} is the part of $B$ in $L_p(\Omega)$. As in the proof of Lemma \ref{lem:appaux0} one can show that $\lambda=0$ is a simple eigenvalue of $B$. It follows that $(W_{p'}^1(\Omega))^*=N(B)\oplus R(B)$ and $W_p^1(\Omega\backslash\Sigma)=N(B)\oplus Y$, where $Y$ is a closed subspace of $W_p^1(\Omega\backslash\Sigma)$. Therefore there exists a unique solution $v\in Y$ of the equation $Bv=f$ if and only if $f\in R(B)$ or equivalently $\langle f,\mathds{1}\rangle=0$. It follows readily that $R(B)=\hat{W}_p^{-1}(\Omega)$. Indeed, the inclusion $\hat{W}_p^{-1}(\Omega)\subset R(B)$ is easy, since $\langle f,\mathds{1}\rangle=0$ for each $f\in \hat{W}_p^{-1}(\Omega)$ and the restriction of $f$ to $W_{p'}^1(\Omega)$ belongs to $(W_{p'}^1(\Omega))^*$. Let now $f\in R(B)$, i.e. $f\in (W_{p'}^1(\Omega))^*$ and $\langle f,\mathds{1}\rangle=0$. This yields
$$|\langle f,\phi\rangle|=|\langle f,\phi-\bar{\phi}\rangle|\le C\|\phi-\bar{\phi}\|_{W_{p'}^1(\Omega)}\le C\|\nabla\phi\|_{L_{p'}(\Omega)},$$
by the Poincar\'{e}-Wirtinger inequality and therefore $[\phi\mapsto\langle f,\phi\rangle]$ is continuous on $C_c^\infty(\overline{\Omega})$ with respect to the norm $\|\nabla\cdot\|_{L_{p'}(\Omega)}$.

Let $u\in W_p^1(\Omega\backslash\Sigma)$ denote the unique solution of \eqref{eq:appauxlem0.3} and let $v\in \dot{W}_p^1(\Omega\backslash\Sigma)$ denote the unique solution of
\begin{align*}
(\nabla v|\nabla\phi)_2&=\langle f,\phi\rangle-(\nabla u|\nabla\phi)_2,\quad \phi\in W_{p'}^1(\Omega),\\
\Jump{\rho v}&=0,\quad \text{on} \; \Sigma.
\end{align*}
It follows readily that the function $w:=v+u\in \dot{W}_p^1(\Omega\backslash\Sigma)$ is the unique solution of \eqref{eq:appauxlem0.01}.
\end{proof}

A final result in this subsection considers the system \eqref{eq:appauxlem0} with $\lambda=g_1=g_2=h_1=h_2=0$. We assume that the function $f$ depends on the spatial variable $x$ and on some parameter $t$, i.e. $f=f(t,x)$. In this case the solution $u=u(t,x)$ depends on $t$ as well. The following result contains some information about the regularity of $u$ with respect to $t$ and $x$.
\begin{lem}\label{lem:appauxhighreg}
Let $n=2,3$, $p\ge 2$, $J=[0,T]$ or $J=\mathbb{R}_+$ and $\lambda=g_1=g_2=h_1=h_2=0$. Then the following assertions are valid.
\begin{enumerate}
\item If $\Omega$ and $\Sigma$ satisfy one of the conditions in $(a),(b)$ above, then there exists a unique solution
    $$\nabla u\in\!_0W_p^1(J;{W}_p^{1}(\Omega\backslash\Sigma))\cap L_p(J;W_p^3(\Omega\backslash\Sigma))$$
    of \eqref{eq:appauxlem0}
    if and only if
    $$f\in\!_0W_p^1(J;\hat{W}_p^{-1}(\Omega)\cap L_p(\Omega))\cap L_p(J;W_p^2(\Omega\backslash\Sigma)).$$
\item If $\Omega$ and $\Sigma$ are subject to the condition $(c)$ above, then there exists a unique solution
    $$u\in\!_0W_p^1(J;W_p^1(\Omega\backslash\Sigma))\cap L_p(J;W_p^3(\Omega\backslash\Sigma))$$
    of \eqref{eq:appauxlem0} if and only if
    $$f\in\!_0W_p^1(J;\hat{W}_p^{-1}(\Omega))\cap L_p(J;W_p^1(\Omega\backslash\Sigma)).$$
\end{enumerate}
\end{lem}
\begin{proof}
(i) The regularity
$$\nabla u\in\!_0W_p^1(J;{W}_p^{1}(\Omega\backslash\Sigma))$$
in the first assertion and
$$u\in\!_0W_p^1(J;W_p^1(\Omega\backslash\Sigma))$$ in the second assertion is a direct consequence of Lemma \ref{lem:appaux0} and Lemma \ref{lem:appauxlemweak}, respectively.

(ii) Concerning the additional spatial regularity of $u$, one uses the fact that one already knows the unique solution $u$ of \eqref{eq:appauxlem0} with the regularity stated in Lemma \ref{lem:appaux0} and Lemma \ref{lem:appauxlemweak}. By means of local coordinates, one reduces each of the local problems to one of the model probolems in (a) and (b) above. In particular, the two-phase half space and the quarter space can be drawn back to a two-phase full space and an ordinary half space, respectively, by reflection techniques. The mapping behavior of the Laplacian and the Poisson semigroup in homogeneous Sobolev-Slobodeckii spaces, see \eqref{eq:equivnormhom}, yield the corresponding higher order estimates for the solution operators of the model problems. Therefore, the proof of the additional regularity of $u$ with respect to $x$ follows along the lines of \cite[Proof of Theorem 8.6]{KPW10}. We will not repeat the arguments.
\end{proof}

\subsection{Parabolic problems}

The following auxiliary lemma is concerned with the parabolic one-phase problem
\begin{align}
\begin{split}\label{eq:appauxlem1}
\partial_tu-\mu\Delta u&=f,\quad \text{in}\ \Omega,\\
P_{S_1}\left(\mu(\nabla u +\nabla u^{\sf T})\nu_{S_1}\right)&=P_{S_1}g_1,\quad \text{on}\ S_1,\\
u\cdot\nu_{S_1}&=g_2,\quad \text{on}\ S_1,\\
u&=g_3,\quad \text{on}\ S_2,\\
u(0)&=u_0,\quad \text{in}\ \Omega.
\end{split}
\end{align}
Again, we will concentrate on the case $n=3$. The results in this section remain true for the case $n=2$.
\begin{lem}\label{lem:appaux1}
Let $p>2$, $p\neq 3$, $\mu>0$, $T>0$ and $J=[0,T]$. Then there exists a unique solution
$$u\in H_p^1(J;L_p(\Omega)^3)\cap L_p(J;H_p^2(\Omega)^3)$$
of \eqref{eq:appauxlem1} if and only if the data are subject to the following regularity and compatibility conditions
\begin{enumerate}
\item $f\in L_p(J;L_p(\Omega)^3)$,
\item $g_1\in W_p^{1/2-1/2p}(J;L_p(S_1)^3)\cap L_p(J;W_p^{1-1/p}(S_1)^3)$,
\item $g_2\in W_p^{1-1/2p}(J;L_p(S_1))\cap L_p(J;W_p^{2-1/p}(S_1))$,
\item $g_3\in W_p^{1-1/2p}(J;L_p(S_2)^3)\cap L_p(J;W_p^{2-1/p}(S_2)^3)$,
\item $u_0\in W_p^{2-2/p}(\Omega)^3$,
\item $P_{S_1}\left(\mu(\nabla u_0 +\nabla u_0^{\sf T})\nu_{S_1}\right)=P_{S_1}g_1|_{t=0}$ $(p>3)$,
\item $u_0|_{S_1}\cdot\nu_{S_1}=g_2|_{t=0}$, $u_0|_{S_2}=g_3|_{t=0}$,
\item $g_3\cdot\nu_{S_1}=g_2$ at $\partial S_2$,
\item $P_{\partial G}\left(\mu(\nabla_{x'} g_3' +\nabla_{x'} (g_3')^{\sf T})\nu_{\partial S_2}\right)=P_{\partial G}g_1'$ at $\partial S_2$,
\item $\mu(\partial_{\nu_{S_1}}(g_3\cdot e_3)+\partial_3 g_2)=g_1\cdot e_3$ at $\partial S_2$,
\end{enumerate}
where $g_j':=\sum_{k=1}^2(g_j\cdot e_k)e_k$ for $j\in\{1,3\}$.

The result remains true for the case $J=\mathbb{R}_+$ if $\partial_t$ is replaced by $\partial_t+\omega$, with some sufficiently large $\omega>0$.
\end{lem}
\begin{proof}
1. Extend $u_0$ to some function $\tilde{u}_0\in W_p^{2-2/p}(\mathbb{R}^3)^3$ and solve the full space problem
\begin{align}
\begin{split}\label{eq:appauxlem1.1}
\partial_t\tilde{u}-\mu\Delta \tilde{u}&=0,\quad \text{in}\ \mathbb{R}^3,\\
\tilde{u}(0)&=\tilde{u}_0,\quad \text{in}\ \mathbb{R}^3,
\end{split}
\end{align}
to obtain a unique solution
$$\tilde{u}\in H_p^1(J;L_p(\mathbb{R}^3)^3)\cap L_p(J;H_p^2(\mathbb{R}^3)^3).$$
If $u$ is a solution of \eqref{eq:appauxlem1}, then $u-\tilde{u}|_{\Omega}$ solves \eqref{eq:appauxlem1} with $u_0=0$ and some modified data $(f,g_1,g_2,g_3)$ (not to be relabeled) having vanishing temporal trace at $t=0$, whenever it exists. Therefore, we may w.l.o.g. assume that $u_0=0$ in \eqref{eq:appauxlem1}.

Suppose that $u$ is a solution of \eqref{eq:appauxlem1} with $u_0=0$. We cover $\partial S_2$ by finitely many open balls $U_k:=B_{r}(x_k)$, $x_k\in \partial S_2$, $k=1,\ldots,N$. This way, we obtain $N$ bent quarter spaces with corresponding solution operators $\mathcal{S}_k$, which are well-defined, if $r>0$ is sufficiently small. Furthermore, by the results in Section \ref{sec:partitionone} there exist open sets $U_{N+j}$, $j=1,\ldots,3$ such that
\begin{itemize}
\item $U_{N+1}\subset\Omega$,
\item $U_{N+2}\cap S_1\neq\emptyset$, $U_{N+2}\cap S_2=\emptyset$,
\item $U_{N+3}\cap S_1=\emptyset$, $U_{N+3}\cap S_2\neq\emptyset$,
\item $\overline{\Omega}\subset\bigcup_{k=1}^{N+3} U_k$,
\end{itemize}
and a subordinated partition of unity $\{\varphi_k\}_{k=0}^N\subset C_c^3(\mathbb{R}^3;[0,1])$ with $\partial_{\nu_{\partial G}}\varphi_k=\partial_3\varphi_k=0$ at $\partial S_2$. Let $u_k:=u\varphi_k$, $f_k:=f\varphi_k$ and $g_j^k:=g_j\varphi_k$. Then $u_k$ solves the problem
\begin{align}
\begin{split}\label{eq:appauxlem1.2}
\partial_tu_k-\mu\Delta u_k&=F_k(u)+f_k,\quad \text{in}\ \Omega_k,\\
P_{S_1^k}\left(\mu(\nabla u_k +\nabla u_k^{\sf T})\nu_{S_1^k}\right)&=G_k(u)+P_{S_1^k}g_1^k,\quad \text{on}\ S_1^k,\\
u_k\cdot\nu_{S_1^k}&=g_2^k,\quad \text{on}\ S_1^k,\\
u_k&=g_3^k,\quad \text{on}\ S_2^k,\\
u_k(0)&=0,\quad \text{in}\ \Omega_k,
\end{split}
\end{align}
where $F_k(u):=-\mu[\Delta,\varphi_k]u$ and $G_k(u):=P_{S_1^k}\left(\mu(\nabla\varphi_k\otimes u+u\otimes\nabla\varphi_k)\nu_{S_1^k}\right)$.

Here $\Omega_{N+1}=\mathbb{R}^3$, $\Omega_{N+2}$ reduces to bent half-spaces with pure-slip boundary conditions, $\Omega_{N+3}$ is a half-space with Dirichlet boundary conditions and $\Omega_k$, $k=1,\ldots,N$ are bent quarter-spaces with pure-slip boundary conditions on one part of the boundary and Dirichlet boundary conditions on the other part. $S_j^k$ denote the corresponding parts of the boundary $\partial\Omega_k$ and $S_j^{N+1}=S_1^{N+3}=S_2^{N+2}=\emptyset$.

Denoting by $\mathcal{S}_k$ the corresponding solution operators to each of the $N+3$ problems, we obtain the representation
$$u_k=\mathcal{S}_k\left((f_k,g_1^k,g_2^k,g_3^k)+(F_k(u),G_k(u),0,0)\right).$$
Let $\{\psi_k\}_{k=0}^N\subset C_c^\infty(\mathbb{R}^3;[0,1])$ such that $\psi_k\equiv 1$ on $\operatorname{supp}\varphi_k$ and $\operatorname{supp}\psi_k\subset U_k$. Multiplying $u_k$ with $\psi_k$ and summing from $k=0$ to $N$ yields the identity
\begin{equation}\label{eq:appauxlem1.3}
u=\sum_{k=0}^N\psi_k\mathcal{S}_k\left((f_k,g_1^k,g_2^k,g_3^k)+(F_k(u),G_k(u),0,0)\right).
\end{equation}
Therefore, any solution to \eqref{eq:appauxlem1}, with $u_0=0$, necessarily satisfies \eqref{eq:appauxlem1.3}. The converse however is in general not true. This pathology stems from the compatibility conditions at $\partial S_2^k$ for the commutator term $G_k(u)$ in \eqref{eq:appauxlem1.2}. Thanks to Proposition \ref{app:propext} there exists an appropriate extension operator $\operatorname{ext}_{x_3,k}$ from
$$_0W_p^{1/2-1/p}(J;L_p(\partial S_2^k))\cap L_p(J;W_p^{1-2/p}(\partial S_2^k))$$
to
$$_0W_p^{1/2-1/2p}(J;L_p(\partial S_2^k\times\mathbb{R}_+))\cap L_p(J;W_p^{1-1/p}(\partial S_2^k\times\mathbb{R}_+)),$$
such that $[\operatorname{ext}_{x_3,k}v](0)=v$. Replace $G_k(u)$ by
$$
\tilde{G}_k(u,g_3):=G_k(u)-\operatorname{ext}_{x_3,k}\left(G_k(u)|_{x_3=H_j}
-G_k(g_3)|_{x_3=H_j}\right)=G_k^1(g_3)+G_k^2(u),
$$
where $G_k^1(g_3):=\operatorname{ext}_{x_3,k} G_k(g_3)|_{x_3=H_j}$.
We note on the go that $\tilde{G}_k(u,g_3)=G_k(u)$, if $u$ is a solution of \eqref{eq:appauxlem1}, since then $u=g_3$ at $\partial S_2$ and $g_3|_{S_1}\cdot\nu_{S_1}=g_2|_{S_2}$ at $\partial S_2$ by assumption.

Therefore we will henceforth work with the identity
\begin{equation}\label{eq:appauxlem1.4}
u=\sum_{k=0}^N\psi_k\left(\mathcal{S}_k(f_k,g_1^k+G_k^1(g_3),g_2^k,g_3^k)+
\mathcal{S}_k(F_k(u),G_k^2(u),0,0)\right).
\end{equation}
Let $_0\mathbb{E}(T):=\!_0 H_p^1(J;L_p(\Omega)^3)\cap L_p(J;H_p^2(\Omega)^3)$,
$$\mathbb{F}_1(T):=L_p(J\times\Omega)^3,$$
$$_0\mathbb{F}_2(T):=\!_0W_p^{1/2-1/2p}(J;L_p(S_1)^3)\cap L_p(J;W_p^{1-1/p}(S_1)^3),$$
$$_0\mathbb{F}_3(T):=\!_0W_p^{1-1/2p}(J;L_p(S_1))\cap L_p(J;W_p^{2-1/p}(S_1)),$$
$$_0\mathbb{F}_4(T):=\!_0W_p^{1-1/2p}(J;L_p(S_2)^3)\cap L_p(J;W_p^{2-1/p}(S_2)^3)$$
and
\begin{multline*}
_0\mathbb{F}(T):=\{(f,g_1,g_2,g_3)\in\mathbb{F}_1(T)
\times_{j=2}^4\left\{\,_0\mathbb{F}_j(T)\right\}:\\
(8)-(10)\ \text{in Lemma \ref{lem:appaux1} are satisfied}\}.
\end{multline*}
Since the terms involving $u$ on the right side of \eqref{eq:appauxlem1.4} are of lower order, it follows that there exists $\gamma>0$ such that the a priori estimate
$$\|u\|_{\mathbb{E}(T)}\le M\left(\|(f,g_1,g_2,g_3)\|_{\mathbb{F}(T)}+T^\gamma\|u\|_{\mathbb{E}(T)}\right),$$
holds for any solution $u$ of \eqref{eq:appauxlem1.4}. Therefore, if $T>0$ is sufficiently small, it follows that the operator $L:\!_0\mathbb{E}(T)\to\!_0\mathbb{F}(T)$ defined by the left side of \eqref{eq:appauxlem1} without the initial condition is injective and has closed range. This in turn implies that $L$ has a left-inverse.

Applying a Neumann series argument, we see that for each given set of data $(f,g_1,g_2,g_3)\in\!_0\mathbb{F}(T)$ there exists a unique solution $u$ of \eqref{eq:appauxlem1.4} on a (possibly) small time interval $[0,T]$. This follows as above by taking into account that the terms involving $u$ on the right side of \eqref{eq:appauxlem1.4} are linear and of lower order. Denote by $\mathcal{S}:\!_0\mathbb{F}(T)\to\!_0\mathbb{E}(T)$ the corresponding solution operator. It remains to prove the existence of a right inverse for $L$. Writing $u=\mathcal{S}(f,g_1,g_2,g_3)$, where $(f,g_1,g_2,g_3)\in\!_0\mathbb{F}(T)$, it follows that
\begin{equation}\label{eq:appauxlem1.5}
\mathcal{S}(f,g_1,g_2,g_3)=\sum_{k=0}^N\psi_k\Big(\mathcal{S}_k(f_k,g_1^k+G_k^1(g_3),g_2^k,g_3^k)
+
\mathcal{S}_k(F_k(u),G_k^2(u),0,0)\Big).
\end{equation}
Applying the operator $L$ to \eqref{eq:appauxlem1.5} we obtain
$$L\mathcal{S}(f,g_1,g_2,g_3)=(f,g_1,g_2,g_3)+R(f,g_1,g_2,g_3),$$
where the linear operator $R$ is given by
\begin{align*}
R(f,g_1,g_2,g_3)&:=\sum_{k=0}^N[L,\psi_k]\Big(\mathcal{S}_k(f_k,g_1^k+G_k^1(g_3),g_2^k,g_3^k)
+\mathcal{S}_k(F_k(u),G_k^2(u),0,0)\Big)\\
&+\sum_{k=0}^N(F_k(u),G_k(u,g_3),0,0)
\end{align*}
Since the commutator $[L,\psi_k]$ as well as $F_k(u)$ and $G_k(u,g_3)$ are of lower order compared to $L$, it follows that there exists $\gamma>0$ such that $R$ satisfies the estimate
$$\|R(f,g_1,g_2,g_3)\|_{\mathbb{F}(T)}\le MT^\gamma\|(f,g_1,g_2,g_3)\|_{\mathbb{F}(T)},$$
where $M>0$ does not depend on $T$. Therefore, a Neumann series argument implies that the right inverse for $L$ is given by the linear operator $\mathcal{S}(I-R)^{-1}$, provided that $T>0$ is sufficiently small. This implies that $L$ is boundedly invertible and the proof of the first assertion is complete.

2. Concerning the second assertion, we use local coordinates and make use of the fact that the corresponding local solution operators are bounded by $1/\omega$ in the norm of $\mathbb{F}$. By means of interpolation we are able to control all lower order terms by $C/\omega^a$ for some uniform $a>0$. Choosing $\omega>0$ large enough, the norms of the lower order terms will become small. This yields the invertibility of $L_\omega$ as above, where $L_\omega$ results from $L$ by replacing $\partial_t$ with $\partial_t+\omega$.
\end{proof}
We will also need a result on the well-posedness of the two-phase problem
\begin{align}
\begin{split}\label{eq:appauxlem2}
\partial_t(\rho {u})-\mu\Delta {u}&=f,\quad \text{in}\ \Omega\backslash\Sigma,\\
\Jump{\mu \partial_3 {v}}+\Jump{\mu\nabla_{x'} {w}}&=g_v,\quad \text{on}\ \Sigma,\\
\Jump{\mu \partial_3 {w}}&=g_w,\quad \text{on}\ \Sigma,\\
\Jump{{u}}&=u_\Sigma,\quad \text{on}\ \Sigma,\\
P_{S_1}\left(\mu(\nabla {u}+\nabla {u}^{\sf T})\nu_{S_1}\right)&=P_{S_1}g_1,\quad \text{on}\ S_1\backslash\partial\Sigma,\\
{u}\cdot\nu_{S_1}&=g_2,\quad \text{on}\ S_1\backslash\partial\Sigma,\\
{u}&=g_3,\quad \text{on}\ S_2,\\
{u}(0)&=u_0,\quad \text{in}\ \Omega\backslash\Sigma.
\end{split}
\end{align}
\begin{lem}\label{lem:appaux2}
Let $p>2$, $p\neq 3$, $\mu_j>0$, $\rho_j>0$, $T>0$ and $J=[0,T]$. Then there exists a unique solution
$$u\in H_p^1(J;L_p(\Omega)^3)\cap L_p(J;H_p^2(\Omega\backslash\Sigma)^3)$$
of \eqref{eq:appauxlem2} if and only if the data are subject to the following regularity and compatibility conditions
\begin{enumerate}
\item $f\in L_p(J;L_p(\Omega)^3)$,
\item $g_v\in W_p^{1/2-1/2p}(J;L_p(\Sigma)^2)\cap L_p(J;W_p^{1-1/p}(\Sigma)^2)$,
\item $g_w\in W_p^{1/2-1/2p}(J;L_p(\Sigma))\cap L_p(J;W_p^{1-1/p}(\Sigma))$,
\item $u_\Sigma=(v_\Sigma,w_\Sigma)\in W_p^{1-1/2p}(J;L_p(\Sigma)^3)\cap L_p(J;W_p^{2-1/p}(\Sigma)^3)$,
\item $g_1\in W_p^{1/2-1/2p}(J;L_p(S_1)^3)\cap L_p(J;W_p^{1-1/p}(S_1\backslash\partial\Sigma)^3)$,
\item $g_2\in W_p^{1-1/2p}(J;L_p(S_1))\cap L_p(J;W_p^{2-1/p}(S_1\backslash\partial\Sigma))$,
\item $g_3\in W_p^{1-1/2p}(J;L_p(S_2)^3)\cap L_p(J;W_p^{2-1/p}(S_2)^3)$,
\item $u_0=(v_0,w_0)\in W_p^{2-2/p}(\Omega\backslash\Sigma)^3$,
\item $P_{S_1}\left(\mu(\nabla u_0 +\nabla u_0^{\sf T})\nu_{S_1}\right)=P_{S_1}g_1|_{t=0}$, $\Jump{\mu\partial_3 v_0}+\Jump{\mu\nabla_{x'} w_0}=g_v|_{t=0}$ $(p>3)$,
\item $u_0|_{S_1}\cdot\nu_{S_1}=g_2|_{t=0}$, $u_0|_{S_2}=g_3|_{t=0}$, $\Jump{\mu\partial_3 w_0}=g_w|_{t=0}$, $\Jump{u_0}=u_\Sigma|_{t=0}$,
\item $g_3\cdot\nu_{S_1}=g_2$ at $\partial S_2$, $u_\Sigma\cdot\nu_{S_1}=\Jump{g_2}$ at $\partial\Sigma$,
\item $P_{\partial \Sigma}\left((\nabla_{x'} v_\Sigma +\nabla_{x'} v_\Sigma^{\sf T})\nu_{\partial \Sigma}\right)=P_{\partial \Sigma}\Jump{g_1'/\mu}$ at $\partial \Sigma$,
\item $\partial_{\nu_{S_1}}w_\Sigma=\Jump{(g_1\cdot e_3)/\mu-\partial_3 g_2}$, $(g_v|\nu_{\partial \Sigma})=\Jump{g_1\cdot e_3}$ at $\partial\Sigma$,
\item $P_{\partial G}\left(\mu(\nabla_{x'} {g}_3' +\nabla_{x'} ({g}_3')^{\sf T})\nu_{\partial S_2}\right)=P_{\partial G}g_1'$ at $\partial S_2$
\item $\mu(\partial_{\nu_{S_1}}(g_3\cdot e_3)+\partial_3 g_2)=g_1\cdot e_3$ at $\partial S_2$,
\end{enumerate}
where $g_j'=\sum_{k=1}^2({g}_j\cdot e_k)e_k$ for $j\in\{1,3\}$.

The result remains true for the case $J=\mathbb{R}_+$ if $\partial_t$ is replaced by $\partial_t+\omega$, with some sufficiently large $\omega>0$.
\end{lem}
\begin{proof}
1. Without loss of generality we may assume $u_0=0$. This can be seen as follows. Extend $u_0^+:=u_0|_{x_3\in (0,H_2)}\in W_p^{2-2/p}(G\times(0,H_2))^3$ first w.r.t. $x_3$, then w.r.t. $(x_1,x_2)$ to some $\tilde{u}_0^+\in W_p^{2-2/p}(\mathbb{R}^3)^3$ and solve the full space problem
\begin{align}
\begin{split}\label{eq:appauxlem2.1}
\partial_t\tilde{u}^+-\Delta \tilde{u}^+&=0,\quad \text{in}\ \mathbb{R}^3,\\
\tilde{u}^+(0)&=\tilde{u}_0^+,\quad \text{in}\ \mathbb{R}^3,
\end{split}
\end{align}
to obtain a unique solution
$$\tilde{u}^+\in H_p^1(J;L_p(\mathbb{R}^3)^3)\cap L_p(J;H_p^2(\mathbb{R}^3)^3).$$
Then we extend $u_0^-:=u_0|_{x_3\in (H_1,0)}\in W_p^{2-2/p}(G\times(H_1,0))^3$ first w.r.t. $x_3$, then w.r.t. $(x_1,x_2)$ to some $\tilde{u}_0^-\in W_p^{2-2/p}(\mathbb{R}^3)^3$ and solve \eqref{eq:appauxlem2.1} with $\tilde{u}_0^+$ replaced by $\tilde{u}_0^-$ to obtain a unique solution
$$\tilde{u}^-\in H_p^1(J;L_p(\mathbb{R}^3)^3)\cap L_p(J;H_p^2(\mathbb{R}^3)^3).$$
Define $\tilde{u}:=\tilde{u}^+\chi_{G\times(0,H_2)}+\tilde{u}^-\chi_{G\times (H_1,0)}$. If $u$ solves \eqref{eq:appauxlem2}, then $u-\tilde{u}$ solves \eqref{eq:appauxlem2} with $u_0=0$ and with some modified data $(f,g_j,u_\Sigma)$ (not to be relabeled). Note that the time traces of the modified data at $t=0$ are zero by construction, whenever they exist.

\textbf{Step 1:} In a first step we consider the case $\mu_j=\rho_j=1$. Extend
$$(g_v,g_w)\in\!_0W_p^{1/2-1/2p}(J;L_p(\Sigma)^3)\cap L_p(J;W_p^{1-1/p}(\Sigma)^3)$$
and
$$u_\Sigma\in\!_0W_p^{1-1/2p}(J;L_p(\Sigma)^3)\cap L_p(J;W_p^{2-1/p}(\Sigma)^3),$$
to some functions
$$(\tilde{g}_v,\tilde{g}_w)\in\!_0W_p^{1/2-1/2p}(J;L_p(\mathbb{R}^2)^3)\cap L_p(J;W_p^{1-1/p}(\mathbb{R}^2)^3)$$
and
$$\tilde{u}_\Sigma\in\!_0W_p^{1-1/2p}(J;L_p(\mathbb{R}^2)^3)\cap L_p(J;W_p^{2-1/p}(\mathbb{R}^2)^3)$$
respectively. Then we solve the following two-phase problem in $\dot{\mathbb{R}}^3:=\mathbb{R}^2\times \dot{\mathbb{R}}$.
\begin{align}
\begin{split}\label{eq:appauxlem2.2}
\partial_t\tilde{u}-\Delta \tilde{u}&=0,\quad \text{in}\ \dot{\mathbb{R}}^3,\\
\Jump{\partial_3 \tilde{v}}+\Jump{\nabla_{x'} \tilde{w}}&=\tilde{g}_v,\quad \text{on}\ \mathbb{R}^2\times\{0\},\\
\Jump{\partial_3 \tilde{w}}&=\tilde{g}_w,\quad \text{on}\ \mathbb{R}^2\times\{0\},\\
\Jump{\tilde{u}}&=\tilde{u}_\Sigma,\quad \text{on}\ \mathbb{R}^2\times\{0\},\\
\tilde{u}(0)&=0,\quad \text{in}\ \dot{\mathbb{R}}^3.
\end{split}
\end{align}
This yields the existence of a unique solution
$$\tilde{u}\in\!_0H_p^1(J;L_p(\mathbb{R}^3)^3)\cap L_p(J;H_p^2(\dot{\mathbb{R}}^3)^3).$$
If $u$ solves \eqref{eq:appauxlem2} with $u_0=0$, then $u-\tilde{u}|_{\Omega}$ solves \eqref{eq:appauxlem2} with $u_0=g_v=g_w=u_\Sigma=0$ and some modified data $(\hat{f},\hat{g}_1,\hat{g}_2,\hat{g}_3)$ in the right regularity classes and with vanishing trace at $t=0$ whenever it exists. Observe that the compatibility conditions on the modified data at $\partial\Sigma$ read as follows.
$$\Jump{\hat{g}_2}=\Jump{\partial_3\hat{g}_2}=0,\ \text{and}\ \Jump{P_{S_1}\hat{g}_1}=P_{S_1}\Jump{\hat{g}_1}=0$$
Note that this is in general not the case if $\Jump{\mu}\neq 0$. Therefore it follows that
$$P_{S_1}\hat{g}_1\in\!_0W_p^{1/2-1/2p}(J;L_p(S_1)^3)\cap L_p(J;W_p^{1-1/p}(S_1)^3),$$
and
$$\hat{g}_2\in\!_0W_p^{1-1/2p}(J;L_p(S_1))\cap L_p(J;W_p^{2-1/p}(S_1)).$$
Since the modified data $\hat{g}_j$ also satisfy the compatibility conditions at $\partial S_2$, we may solve \eqref{eq:appauxlem1} by Lemma \ref{lem:appaux1} with $\mu=1$, $f=\hat{f}$, $g_1=P_{S_1}\hat{g}_1$, $g_2=\hat{g}_2$, $g_3=\hat{g}_3$ and $u_0=0$. This in turn implies that problem \eqref{eq:appauxlem2} is well-posed, provided that $\mu_1=\mu_2=1$.

\textbf{Step 2:} In the second step we consider the case $\Jump{\rho}\neq 0$, $\Jump{\mu}\neq 0$. Let us first reduce \eqref{eq:appauxlem2} with $u_0=0$ to the case $g_1=g_2=g_3=0$. To this end will apply Lemma \ref{lem:appaux1} twice. First we extend $g_j^+:=g_j|_{x_3\in(0,H_2)}$ by some (higher order) reflections at $\{x_3=0\}$ to some functions
$$\tilde{g}_1^+\in\!_0W_p^{1/2-1/2p}(J;L_p(S_1)^3)\cap L_p(J;W_p^{1-1/p}(S_1)^3)$$
and
$$\tilde{g}_2^+\in\!_0W_p^{1-1/2p}(J;L_p(S_1))\cap L_p(J;W_p^{2-1/p}(S_1)),$$
such that $\tilde{g}_j^+|_{x_3=H_1}=0$. Then, we solve \eqref{eq:appauxlem1} with $\mu=\mu_2$, $f=0$, $g_1=P_{S_1}\tilde{g}_1^+$, $g_2=\tilde{g}_2^+$, $g_3|_{x_3=H_2}=g_3^+$ and $g_3|_{x_3=H_1}=0$ to obtain a unique solution
$$\tilde{u}^+\in\!_0H_p^1(J;L_p(\Omega)^3)\cap L_p(J;H_p^2(\Omega)^3).$$
Repeating the same procedure for $g_j^-:=g_j|_{x_3\in (H_1,0)}$ yields a unique solution
$$\tilde{u}^-\in\!_0H_p^1(J;L_p(\Omega)^3)\cap L_p(J;H_p^2(\Omega)^3).$$
Define $\tilde{u}:=\tilde{u}^+\chi_{G\times(0,H_2)}+\tilde{u}^-\chi_{G\times (H_1,0)}$. If $u$ solves \eqref{eq:appauxlem2} with $u_0=0$, then $u-\tilde{u}$ solves \eqref{eq:appauxlem2} with $u_0=0$, $g_1=0$, $g_2=0$ $g_3=0$ and some modified data $(\hat{f},\hat{g}_v,\hat{g}_w,\hat{u}_\Sigma)$ which are subject to the following compatibility conditions at $\partial\Sigma$:
\begin{equation}\label{eq:appauxlem2.3}\hat{u}_\Sigma\cdot\nu_{S_1}=0,\ \partial_{\nu_{S_1}}\hat{w}_\Sigma=0,\ \hat{g}_v\cdot\nu_{\partial \Sigma}=0
\end{equation}
and
\begin{equation}\label{eq:appauxlem2.4}P_{\partial \Sigma}\left((\nabla_{x'} \hat{v}_\Sigma +\nabla_{x'} \hat{v}_\Sigma^{\sf T})\nu_{\partial \Sigma}\right)=0.
\end{equation}
\textbf{Step 3:} Let $_0\mathbb{E}(T):=\!_0H_p^1(J;L_p(\Omega)^3)\cap L_p(J;H_p^2(\Omega\backslash\Sigma)^3)$ and denote by $_0\mathbb{F}(T)$ the space of data $(f,g_j,u_\Sigma)$, $j\in\{v,w,1,2,3\}$ such that the compatibility conditions (11)-(15) in Lemma \ref{lem:appaux2} are satisfied. Define $L:\!_0\mathbb{E}(T)\to\!_0\mathbb{F}(T)$ by the left side of \eqref{eq:appauxlem2} without the initial condition. By means of a localization procedure one can show that $L$ satisfies the a priori estimate
\begin{equation}\label{eq:appauxlem2.5}
\|u\|_{_0\mathbb{E}(T)}\le M\|Lu\|_{_0\mathbb{F}(T)}.
\end{equation}
This can be seen as in the proof of Lemma \ref{lem:appaux1}. Indeed, the charts which intersect $\partial S_2$ and $\partial\Sigma$ may be treated as in Sections \ref{QS} \& \ref{HS}, respectively, while the treatment of the remaining charts is well-known. Note that there is no need to carry any correction terms as in the proof of Lemma \ref{lem:appaux1}, since for the proof of \eqref{eq:appauxlem2.5} one already starts with a solution of \eqref{eq:appauxlem2}. Therefore, the compatibility conditions at $\partial S_2$ and $\partial\Sigma$ are necessarily satisfied.

Next, we set
\begin{multline*}
_0\tilde{\mathbb{E}}(T):=\{u\in \!_0H_p^1(J;L_p(\Omega)^3)\cap L_p(J;H_p^2(\Omega\backslash\Sigma)^3):\\
u|_{S_2}=0,\ u|_{S_1}\cdot\nu_{S_1}=0,\ P_{S_1}\left((\nabla {u}+\nabla {u}^{\sf T})\nu_{S_1}\right)=0\},
\end{multline*}
and denote by $_0\tilde{\mathbb{F}}(T)$ the space of data $(f,g_v,g_w,u_\Sigma)$ together with the compatibility conditions \eqref{eq:appauxlem2.3} \& \eqref{eq:appauxlem2.4} at $\partial\Sigma$. Note that
$$P_{S_1}\left((\nabla {u}+\nabla {u}^{\sf T})\nu_{S_1}\right)=0 \Leftrightarrow P_{S_1}\left(\mu(\nabla {u}+\nabla {u}^{\sf T})\nu_{S_1}\right)=0$$
at $S_1\backslash\partial\Sigma$.
Define $\tilde{L}:\!_0\tilde{\mathbb{E}}(T)\to\!_0\tilde{\mathbb{F}}(T)$ by
$$\tilde{L}u=
\begin{pmatrix}
\partial_t (\rho u)-\mu \Delta u\\
\Jump{\mu\partial_3 v}+\Jump{\mu\nabla_{x'} w}\\
\Jump{\mu\partial_3 w}\\
\Jump{u}
\end{pmatrix}.
$$
Since the norm in $_0\tilde{\mathbb{E}}(T)$ is the same as in $_0\mathbb{E}(T)$ and since
$$\|Lu\|_{_0\mathbb{F}(T)}=\|\tilde{L}u\|_{_0\tilde{\mathbb{F}}(T)}$$
for $u\in\!_0\tilde{\mathbb{E}}(T)$, it follows from \eqref{eq:appauxlem2.5} that $\tilde{L}$ is injective with closed range, i.e.\ $\tilde{L}$ is a semi-Fredholm operator. It is also crucial to observe that the constant $M>0$ is uniform on compact sets of $\mu>0$ and $\rho>0$, by continuity.

We replace the coefficients $(\rho_1,\rho_2,\mu_1,\mu_2)$ by $$(\rho_1^\tau,\rho_2^\tau,\mu_1^\tau,\mu_2^\tau):=
\tau(\rho_1,\rho_2,\mu_1,\mu_2)+(1-\tau)(1,1,1,1),\quad \tau\in[0,1],$$
and denote by $\tilde{L}_\tau:\!_0\tilde{\mathbb{E}}(J)\to\!_0\tilde{\mathbb{F}}(J)$ the corresponding operator which is induced by replacing $\rho$ and $\mu$ with $\rho^\tau$ and $\mu^\tau$, resectively. Note that $\tilde{L}_\tau$ satisfies the estimate
\begin{equation*}
\|u\|_{_0\tilde{\mathbb{E}}(T)}\le M\|\tilde{L}_\tau u\|_{_0\tilde{\mathbb{F}}(T)},
\end{equation*}
with some constant $M>0$ which is uniform with respect to $\tau\in[0,1]$. Hence $\tilde{L}_\tau$ is semi-Fredholm for each $\tau\in[0,1]$. By Step 1 of the proof, we already know that $L_0$ is a Fredholm operator with index zero. The continuity method for semi-Fredholm operators implies that $L_1$ is Fredholm with index zero as well. We remark that the reduction obtained in Step 2 of the proof is essential, since otherwise the viscosity coefficient $\mu$ appears in the definition of $\tilde{\mathbb{F}}(T)$. Replacing $\mu$ by $\mu^\tau$, it would follow that $\tilde{\mathbb{F}}(T)$ depends on $\tau$ as well.

2. The strategy for proof of the second assertion is the same as in the proof of Lemma \ref{lem:appaux1}. Will will not repeat the arguments.
\end{proof}

\section{The two-phase Stokes problem on the half line}

In this section we want to show that there exists $\omega_0>0$ such that for each $\omega\ge\omega_0$ the two-phase Stokes problem
\begin{align}\label{eq:AppStokes}
\begin{split}
\omega\rho u+\partial_t(\rho u)-\mu\Delta u+\nabla \pi&=f,\quad \text{in}\ \Omega\backslash\Sigma,\\
\div u&=f_d,\quad \text{in}\ \Omega\backslash\Sigma,\\
-\Jump{\mu \partial_3 v}-\Jump{\mu\nabla_{x'} w}&=g_v,\quad \text{on}\ \Sigma,\\
-2\Jump{\mu \partial_3 w}+\Jump{\pi}&=g_w,\quad \text{on}\ \Sigma,\\
\Jump{u}&=u_\Sigma,\quad \text{on}\ \Sigma,\\
P_{S_1}\left(\mu(\nabla u+\nabla u^{\sf T})\nu_{S_1}\right)&=P_{S_1}g_1,\quad \text{on}\ S_1\backslash\partial\Sigma,\\
u\cdot\nu_{S_1}&=g_2,\quad \text{on}\ S_1\backslash\partial\Sigma,\\
u&=g_3,\quad \text{on}\ S_2,\\
u(0)&=u_0,\quad \text{in}\ \Omega\backslash\Sigma,\\
\end{split}
\end{align}
has a unique solution $(u,\pi,\Jump{\pi})$ with maximal regularity of type $L_p$ on the half line $\mathbb{R}_+$. To this end we define
$$\mathbb{F}_1:=L_p(\mathbb{R}_+;L_p(\Omega)^3),\quad \mathbb{F}_2:=L_p(\mathbb{R}_+;H_p^1(\Omega\backslash\Sigma)),$$
$$\mathbb{F}_3:=W_p^{1/2-1/2p}(\mathbb{R}_+;L_p(\Sigma)^2)\cap L_p(\mathbb{R}_+;W_p^{1-1/p}(\Sigma)^2),$$
$$\mathbb{F}_4:=W_p^{1/2-1/2p}(\mathbb{R}_+;L_p(\Sigma))\cap L_p(\mathbb{R}_+;W_p^{1-1/p}(\Sigma)),$$
$$\mathbb{F}_5:=W_p^{1-1/2p}(\mathbb{R}_+;L_p(\Sigma)^3)\cap L_p(\mathbb{R}_+;W_p^{2-1/p}(\Sigma)^3),$$
$$\mathbb{F}_6:=W_p^{1/2-1/2p}(\mathbb{R}_+;L_p(S_1)^3)\cap L_p(\mathbb{R}_+;W_p^{1-1/p}(S_1\backslash\partial\Sigma)^3),$$
$$\mathbb{F}_7:=W_p^{1-1/2p}(\mathbb{R}_+;L_p(S_1))\cap L_p(\mathbb{R}_+;W_p^{2-1/p}(S_1\backslash\partial\Sigma)),$$
$$\mathbb{F}_8:=W_p^{1-1/2p}(\mathbb{R}_+;L_p(S_2))\cap L_p(\mathbb{R}_+;W_p^{2-1/p}(S_2)),$$
and $\tilde{\mathbb{F}}:=\times_{j=1}^{8}\mathbb{F}_j$ as well as
$$\mathbb{F}:=\{(f_1,\ldots,f_{8})\in \tilde{\mathbb{F}}:(f_2,f_5,f_7,f_{8})\in H_p^1(\mathbb{R}_+;\hat{H}_p^{-1}(\Omega))\}.$$
Furthermore, we set $X_{\gamma}:=W_p^{2-2/p}(\Omega\backslash\Sigma)^3$. Then we have the following result.
\begin{thm}\label{thm:AppStokes}
Let $\mu_j,\rho_j,H_j,\sigma>0$, $p>2$, $p\neq 3$ and let $G\in\mathbb{R}^{2}$ be open and bounded with $\partial G\in C^4$. Define $\Omega:=G\times (H_1,H_2)$ and let $\Sigma:=G\times\{0\}$. Let $S_1:=\partial G\times (H_1,H_2)$ and $S_2:=(G\times\{H_1\})\cup (G\times\{H_2\})$. Then there exists $\omega_0>0$ such that for each $\omega\ge\omega_0$ problem \eqref{eq:AppStokes} has a unique solution
$$u\in H_p^1(\mathbb{R}_+;L_p(\Omega)^3)\cap L_p(\mathbb{R}_+;H_p^2(\Omega\backslash\Sigma)^3),\quad \pi\in L_p(\mathbb{R}_+;\dot{H}_p^1(\Omega\backslash\Sigma)),$$
and
$$\Jump{\pi}\in W_p^{1/2-1/2p}(\mathbb{R}_+;L_p(\Sigma))\cap L_p(\mathbb{R}_+;W_p^{1-1/p}(\Sigma))$$
if and only if the data are subject to the following regularity and compatibility conditions.
\begin{enumerate}
\item $(f,f_d,g_v,g_w,u_\Sigma,g_1,g_2,g_3)\in\mathbb{F}$,
\item $u_0\in X_\gamma$,
\item $\div u_0=f_d|_{t=0}$, $-\Jump{\mu\nabla_{x'} w_0}-\Jump{\mu\partial_3 v_0}=g_v|_{t=0}$, $\Jump{u_0}=u_\Sigma|_{t=0}$,
\item $P_{S_1}(\mu(\nabla u_0+\nabla u_0^{\sf T})\nu_{S_1})=P_{S_1}g_1|_{t=0}$ $(p>3)$, $u_0\cdot \nu_{S_1}=g_2|_{t=0}$, $u_0=g_3|_{t=0}$,
\item $\Jump{g_2}=u_\Sigma\cdot \nu_{S_1}$,
\item $\Jump{(g_1\cdot e_3)/\mu-\partial_3 g_2}=\partial_{\nu_{S_1}}(u_\Sigma\cdot e_3),$
\item $P_{\partial \Sigma}[(D'v_\Sigma)\nu']=\Jump{P_{\partial \Sigma}g_1'/\mu},$
\item $(g_v|\nu_{S_1})=-\Jump{g_1\cdot e_3}$, $(g_3|\nu_{S_1})=g_2$,
\item $P_{\partial G}[\mu(D'g_3')\nu']=(P_{\partial G}{g}_{1}'),$
\item $\mu\partial_{\nu_{S_1}} (g_3\cdot e_3)+\mu\partial_3 g_2=g_1\cdot e_3$.
\end{enumerate}
\end{thm}
\begin{proof}
The proof is based on a localization procedure. However, in contrast to the proof of Theorem \ref{thm:linmaxreg} we are not able to control the commutator terms in the corresponding local problems which are of lower order by decreasing the length of the time interval. However, replacing the time derivative $\partial_t$ by $\omega+\partial_t$ in all auxiliary problems which were used in Chapter \ref{chptr:redmodprbl}, it follows that there exists $\omega_0>0$ such that each of these problems has maximal regularity of type $L_p$ on the half line $\mathbb{R}_+$, provided $\omega\ge \omega_0$. Indeed, this can be seen by studying the corresponding symbols $s(\lambda,\xi)$ of the differential operators. The parameter $\lambda$ is the Laplace transform of $\partial_t$, hence replacing $\lambda$ by $\omega+\lambda$ this yields the symbol $s_\omega(\lambda,\xi):=s(\omega+\lambda,\xi)$.

By means of interpolation and trace theory we are able to control all commutator terms which appear during the localization procedure by $C/\omega^a$ for some uniform $a\in (0,1)$ and some $C>0$ being independent of $\omega$. Choosing $\omega>0$ large enough, the norms of the lower order terms will become small. This yields the linear well-posedness of \eqref{eq:AppStokes} on the half line $\mathbb{R}_+$.
\end{proof}
As an immediate consequence of the last theorem, one obtains maximal regularity of type $L_p$ of \eqref{eq:AppStokes} in exponentially weighted spaces. To see this, we define $$e^{-\delta}\mathbb{F}_j:=\{f\in \mathbb{F}_j:[t\mapsto e^{\delta t} f(t)]\in \mathbb{F}_j\},$$
where $\delta\in\mathbb{R}$ and in the same way $e^{-\delta}\tilde{\mathbb{F}}$ and $e^{-\delta}\mathbb{F}$.

We write $\omega=\omega-\delta+\delta$ in \eqref{eq:AppStokes}, multiply each equation by $e^{\delta t}$ and use the formula $\partial_t (e^{\delta t}u(t))=e^{\delta t}(\delta u(t)+\partial_t u(t))$ to obtain the following result
\begin{cor}\label{cor:AppStokes}
Let the conditions of Theorem \ref{thm:AppStokes} be satisfied. Suppose that $\delta\in\mathbb{R}$ and let $\omega\ge \max\{\omega_0,\omega_0+\delta\}$. Then there exists a unique solution
$$u\in e^{-\delta}[H_p^1(\mathbb{R}_+;L_p(\Omega)^3)\cap L_p(\mathbb{R}_+;H_p^2(\Omega\backslash\Sigma)^3)],\quad \pi\in e^{-\delta}[L_p(\mathbb{R}_+;\dot{H}_p^1(\Omega\backslash\Sigma))],$$
and
$$\Jump{\pi}\in e^{-\delta}[W_p^{1/2-1/2p}(\mathbb{R}_+;L_p(\Sigma))\cap L_p(\mathbb{R}_+;W_p^{1-1/p}(\Sigma))]$$
of \eqref{eq:AppStokes} if and only if the data are subject to the conditions in Theorem \ref{thm:AppStokes} with $\mathbb{F}$ being replaced by $e^{-\delta}\mathbb{F}$.
\end{cor}

\section{Elliptic two-phase Stokes problems}

Let $\hat{f}\in L_p(\Omega)^3$, $\hat{f}_d\in H_p^1(\Omega\backslash\Sigma)$, $(\hat{g}_v,\hat{g}_w)\in W_p^{1-1/p}(\Sigma)^3$, $\hat{u}_\Sigma\in W_p^{2-1/p}(\Sigma)^3$, $\hat{g}_1\in W_p^{1-1/p}(S_1\backslash\partial\Sigma)$, $\hat{g}_2\in W_p^{2-1/p}(S_1\backslash\partial\Sigma)$ and $\hat{g}_3\in W_p^{2-1/p}(S_2)$ be given such that $(\hat{f}_d,\hat{u}_\Sigma,\hat{g}_2,\hat{g}_3)\in\hat{H}_p^{-1}(\Omega)$ and such that the compatibility conditions (5)-(10) in Theorem \ref{thm:AppStokes} are satisfied at $\partial S_1\cap\partial S_2$ and $S_1\cap\partial\Sigma$.

Define $f(t):=t e^{-t}\hat{f}$ and in the same way $f_d(t),u_\Sigma(t),g_j(t)$, $j\in\{v,w,1,2,3\}$. Then it holds that
$$(f,f_d,g_v,g_w,u_\Sigma,g_1,g_2,g_3)\in e^{-\delta}\mathbb{F}$$
for each $\delta\in (0,1)$ and the compatibility conditions (3)-(10) in Theorem \ref{thm:AppStokes} are satisfied with $u_0=0$. By Corollary \ref{cor:AppStokes} there exists a unique solution $(u,\pi,\Jump{\pi})$ of \eqref{eq:AppStokes} with $\omega\ge\omega_0+\delta$ such that
$$u\in e^{-\delta}[\, _0H_p^1(\mathbb{R}_+;L_p(\Omega)^3)\cap L_p(\mathbb{R}_+;H_p^2(\Omega\backslash\Sigma)^3)],\quad \pi\in e^{-\delta}[L_p(\mathbb{R}_+;\dot{H}_p^1(\Omega\backslash\Sigma))],$$
and
$$\Jump{\pi}\in e^{-\delta}[\,_0W_p^{1/2-1/2p}(\mathbb{R}_+;L_p(\Sigma))\cap L_p(\mathbb{R}_+;W_p^{1-1/p}(\Sigma))].$$
Therefore, the Laplace transform $\mathcal{L}$ of each term in \eqref{eq:AppStokes} is well defined. Observe that
$$
(\mathcal{L}f)(\lambda)=\int_0^\infty e^{-\lambda t} f(t)\ dt=\hat{f}\int_0^\infty te^{-(\lambda+1)t}\ dt=\frac{1}{(\lambda+1)^2}\hat{f},
$$
for $\Re\lambda>-1$, hence $(\mathcal{L}f)(0)=\hat{f}$. Doing the same for all the other data and defining $(\hat{u},\hat{\pi},\Jump{\hat{\pi}}):=\mathcal{L}(u,\pi,\Jump{\pi})$ we obtain that $(\hat{u},\hat{\pi},\Jump{\hat{\pi}})$ solves the elliptic problem
\begin{align}\label{eq:AppEllStokes}
\begin{split}
\omega\rho \hat{u}-\mu\Delta \hat{u}+\nabla \hat{\pi}&=\hat{f},\quad \text{in}\ \Omega\backslash\Sigma,\\
\div \hat{u}&=\hat{f}_d,\quad \text{in}\ \Omega\backslash\Sigma,\\
-\Jump{\mu \partial_3 \hat{v}}-\Jump{\mu\nabla_{x'} \hat{w}}&=\hat{g}_v,\quad \text{on}\ \Sigma,\\
-2\Jump{\mu \partial_3 \hat{w}}+\Jump{\hat{\pi}}&=\hat{g}_w,\quad \text{on}\ \Sigma,\\
\Jump{\hat{u}}&=\hat{u}_\Sigma,\quad \text{on}\ \Sigma,\\
P_{S_1}\left(\mu(\nabla \hat{u}+\nabla \hat{u}^{\sf T})\nu_{S_1}\right)&=P_{S_1}\hat{g}_1,\quad \text{on}\ S_1\backslash\partial\Sigma,\\
\hat{u}\cdot\nu_{S_1}&=\hat{g}_2,\quad \text{on}\ S_1\backslash\partial\Sigma,\\
\hat{u}&=\hat{g}_3,\quad \text{on}\ S_2,
\end{split}
\end{align}
whenever $\omega\ge\omega_0+\delta$. Let $Au:=(\mu/\rho)\Delta u-(1/\rho)\nabla\pi$ with domain
\begin{multline*}
D(A)=\{u\in H_p^2(\Omega\backslash\Sigma)^3\cap L_{p,\sigma}(\Omega):\Jump{\mu\partial_3 v}+\Jump{\mu\nabla_{x'}w}=0,\ \Jump{u}=0,\\
P_{S_1}(\mu (Du)\nu_{S_1})=0,\ u\cdot\nu_{S_1}=0,\ u|_{S_2}=0\},
\end{multline*}
and $\pi\in\dot{W}_p^1(\Omega\backslash\Sigma)$ is the unique solution of the weak transmission problem
\begin{align*}
\left(\frac{1}{\rho}\nabla \pi|\nabla\phi\right)_{L_2(\Omega)}&=\left(\frac{\mu}{\rho}\Delta u|\nabla\phi\right)_{L_2(\Omega)},\quad\phi\in W_{p'}^1(\Omega),\\
\Jump{\pi}&=2\Jump{\mu\partial_3 w},\quad \text{on}\ \Sigma.
\end{align*}
Since $A$ has a compact resolvent, the spectrum $\sigma(A)$ of $A$ consists solely of isolated eigenvalues having a finite multiplicity. Furthermore it holds that $\Re\sigma(A)=\sigma(A)\subset (-\infty,0)$ by Korn's inequality. Indeed, multiplying the eigenvalue problem $Au=\lambda u$ by $u$ and integrating by parts, we obtain
the identity
$$\lambda\|u\|_{L_2(\Omega)}^2=-\|\mu^{1/2}Du\|_{L_2(\Omega)}^2.$$
This yields the following result.
\begin{thm}\label{thm:AppEllStokes}
Let $\omega\ge 0$, $\mu_j,\rho_j,\sigma>0$, $p>2$, $p\neq 3$ and let $\Omega$ and $\Sigma$ as in Theorem \ref{thm:AppStokes}. Then there exists a unique solution $(\hat{u},\hat{\pi},\Jump{\hat{\pi}})$ with
$$\hat{u}\in H_p^2(\Omega\backslash\Sigma)^3,\quad \hat{\pi}\in \dot{H}_p^1(\Omega\backslash\Sigma),\quad\Jump{\hat{\pi}}\in W_p^{1-1/p}(\Sigma)$$
of \eqref{eq:AppEllStokes} if and only if the data are subject to the following regularity and compatibility conditions.
\begin{enumerate}
\item $\hat{f}\in L_p(\Omega)^3$, $\hat{f}_d\in H_p^1(\Omega\backslash\Sigma)$,
\item $(\hat{g}_v,\hat{g}_w)\in W_p^{1-1/p}(\Sigma)^3$, $\hat{u}_\Sigma\in W_p^{2-1/p}(\Sigma)^3$,
\item $\hat{g}_1\in W_p^{1-1/p}(S_1\backslash\partial\Sigma)$, $\hat{g}_2\in W_p^{2-1/p}(S_1\backslash\partial\Sigma)$,
\item $\hat{g}_3\in W_p^{2-1/p}(S_2)$, $(\hat{f}_d,\hat{u}_\Sigma,\hat{g}_2,\hat{g}_3)\in\hat{H}_p^{-1}(\Omega)$,
\item $\Jump{\hat{g}_2}=\hat{u}_\Sigma\cdot \nu_{S_1}$,
\item $\Jump{(\hat{g}_1\cdot e_3)/\mu-\partial_3 \hat{g}_2}=\partial_{\nu_{S_1}}(\hat{u}_\Sigma\cdot e_3),$
\item $P_{\partial \Sigma}[(D'\hat{v}_\Sigma)\nu']=\Jump{P_{\partial \Sigma}\hat{g}_1'/\mu},$
\item $(\hat{g}_v|\nu_{S_1})=-\Jump{\hat{g}_1\cdot e_3}$, $(\hat{g}_3|\nu_{S_1})=\hat{g}_2$,
\item $P_{\partial G}[\mu(D'\hat{g}_3')\nu']=(P_{\partial G}\hat{g}_{1}'),$
\item $\mu\partial_{\nu_{S_1}} (\hat{g}_3\cdot e_3)+\mu\partial_3 \hat{g}_2=\hat{g}_1\cdot e_3$,
\end{enumerate}
where $\nu'=\nu_{\partial G}$.
\end{thm}

\section{Miscellaneous results}

Let $G\subset\mathbb{R}^{n-1}$, $n\in\{2,3\}$ be a bounded domain with boundary $\partial G\in C^1$ and define $\Omega:=G\times (H_1,H_2)$, with $H_1<0<H_2$. Furthermore, let $\Sigma:=G\times\{0\}$, $S_1:=\partial G\times (H_1,H_2)$ and $S_2:=\bigcup_{j=1}^2\{G\times\{H_j\}\}$.
Define $x'=(x_1,\ldots,x_{n-1})^{\sf T}$ and $x=(x',x_n)^{\sf T}$. Assume that $h:G\to (H_1,H_2)$ is continuous and set
$$\Gamma:=\{x=(x',x_n)\in\Omega:x_n=h(x'),\ x'\in G\},$$
that is, $\Gamma$ is an $(n-1)$-dimensional manifold in $\Omega$ which is given as the graph of the height function $h$ over $\Sigma$.
\begin{prop}[Divergence theorem in cylindrical domains]\label{pro:divthmcyldom}
For each $u\in H_2^1(\Omega\backslash\Sigma)^n$ the following identity holds.
$$\int_{\Omega}\div u\ dx=\int_{S_1}u|_{S_1}\cdot\nu_{S_1}\ dS_1+\int_{S_2}u|_{S_2}\cdot\nu_{S_2}\ dS_2-\int_\Gamma\Jump{u}\nu_{\Gamma}\ d\Gamma,$$
where $\nu_{S_j}$ are the outer unit normals on $S_j$ and $\nu_{\Gamma}$ is the normal on $\Gamma$ pointing from
$$\Omega_1:=\{x=(x',x_n)\in\Omega:x_n<h(x'),\ x'\in G\}$$
to $\Omega_2:=\Omega\backslash\overline{\Omega_1}$.
\end{prop}
\begin{proof}
The proof follows from the fact that $\Omega_j$ are both Lipschitz domains. Indeed, it is well-known that the divergence theorem is valid for Lipschitz domains, see for example \cite[Section 4.3]{EG92}.
\end{proof}

For $u\in H_2^1(\Omega)^n$, let $Du:=\nabla u+\nabla u^{\sf T}$. The following result is well-known: There exists a constant $C>0$ such that
$$\|u\|_{H_2^1(\Omega)}\le C\|Du\|_{L_2(\Omega)}$$
for all $u\in H_2^1(\Omega)^n$ such that $u=0$ on $\partial\Omega$ (in the sense of traces). The proof of this inequality relies on integration by parts. We will show that the estimate remains true, if $u=0$ on some subset of $\partial\Omega$ having a positive $(n-1)$-dimensional Hausdorff measure.
\begin{thm}[Korn's inequality]\label{thm:korn}
Let $\Omega\subset\mathbb{R}^n$, $n=2,3$, be a connected, bounded Lipschitz domain. Then there exists $C>0$, which does only depend on $\Omega$ such that the estimate
\begin{equation}\label{Kornineq}
\|\nabla u\|_{L_2(\Omega)}\le C\|Du\|_{L_2(\Omega)}
\end{equation}
holds for each $u\in H_2^1(\Omega)^n$ with $u=0$ on some subset $\partial_D\Omega$ of the boundary $\partial\Omega$ of $\Omega$ such that $\mathcal{H}^{n-1}(\partial_{D}\Omega)>0$, where $\mathcal{H}^d$ denotes the $d$-dimensional Hausdorff measure.
\end{thm}
\begin{proof}
Let us first show that we have some kind of Poincar\'{e} type estimate, that is, there exists a constant $C>0$ such that the estimate
$$\|u\|_{L_2(\Omega)}\le C\|Du\|_{L_2(\Omega)},$$
holds for all $u\in H_2^1(\Omega)^n$ with $u=0$ on some subset $\partial_D\Omega$ of the boundary $\partial\Omega$ of $\Omega$ such that $
\mathcal{H}^{n-1}(\partial_{D}\Omega)>0$.

Assume on the contrary that for each $m\in\mathbb{N}$ there exists $u_m\in H_2^1(\Omega)^n$ with $u_m=0$ on $\partial_D\Omega$ and $\|u_m\|_{L_2(\Omega)}=1$ such that
$$\|u_m\|_{L_2(\Omega)}\ge m\|Du_m\|_{L_2(\Omega)}.$$
It follows that $D u_m\to 0$ in $L_2(\Omega)$ as $m\to\infty$. By Korn's inequality for functions in $H_2^1(\Omega)^n$ (see \cite{Necas80}) we obtain
\begin{equation}\label{genKornineq}
\|u_m\|_{H_2^1(\Omega)}\le C_0(\|D u_m\|_{L_2(\Omega)}+\|u_m\|_{L_2(\Omega)}),
\end{equation}
for some constant $C_0>0$. It follows that $(u_m)\subset H_2^1(\Omega)^n$ is bounded. By Rellich's theorem, there exists a subsequence $(u_{m_k})$ such that $u_{m_k}\to u_*$ in $L_2(\Omega)$. Then $\|u_*\|_{L_2(\Omega)}=1$ and by trace theory it holds that $u_*(x)=0$ for a.e.\ $x\in\partial_D\Omega$. We make use of \eqref{genKornineq} one more time to conclude that $(u_{m_k})$ is a Cauchy sequence in $H_2^1(\Omega)^n$, since $D u_{m_k}\to 0$ in $L_2(\Omega)$. Therefore we obtain $u_{m_k}\to u_*$ even in $H_2^1(\Omega)$. Since
$$\|D u_{m_k}-D u_*\|_{L_2(\Omega)}\le C\|\nabla u_{m_k}-\nabla u_*\|_{L_2(\Omega)}\to 0$$
as $k\to\infty$ it follows readily that $D u_*=0$.

Therefore there exists a skew-symmetric matrix $A\in\mathbb{R}^{n\times n}$ and some $b\in \mathbb{R}^n$ such that $u_*(x)=Ax+b$ for a.e.\ $x\in\Omega$ (see \cite{Necas80}). Define $U:=\{x\in\mathbb{R}^n:Ax+b=0\}$. Then $U$ is an $(n-1)$-dimensional affine subspace of $\mathbb{R}^n$, since $\partial_{D}\Omega\subset U$ and $u_*\not\equiv 0$. Fix any $x_0\in U$ and define
$$U_0:=U-x_0:=\{x-x_0:x\in U\}.$$
It follows that $\dim U_0=n-1$ and $Ax=0$ for each $x\in U_0$. Let $U_0^\perp$ be the orthogonal complement of $U_0$ and let $y\in U_0^\perp$. Then $(x|Ay)=-(Ax|y)=0$ for each $x\in U_0$, since $A$ is skew-symmetric, wherefore $Ay\in U_0^\perp$. Furthermore we have $(Ay|y)=0$, since $A$ is skew-symmetric, hence $Ay\in (U_0^\perp)^\perp=U_0$ and therefore $Ay=0$ for each $y\in U_0^\perp$. But this means that $Ax=0$ for each $x\in\mathbb{R}^n$, since $\mathbb{R}^n=U_0\oplus U_0^\perp$. Thus, we have shown that $A=0$, hence $u_*(x)=b$ for some $b\in\mathbb{R}^n$. Since $\|u_*\|_{L_2(\Omega)}=1$ and $u_*(x)=0$ for a.e.\ $x\in\partial_D\Omega$ we have a contradiction.

Finally, the assertion of the proposition follows from the Poincar\'{e} type estimate combined with Korn's inequality for functions in $H_2^1(\Omega)^n$.
\end{proof}

Last but not least, we need an auxiliary result which is crucial for the proof of local well-posedness in Chapter \ref{chptr:LWP}.
\begin{prop}\label{prop:appaux}
Let $p>2$, $G\subset\mathbb{R}^2$ be a bounded domain with boundary $\partial G\in C^2$ and outer unit normal vector field $\nu$ which is $C^1$ in a neighborhood of $\partial G$. If $v\in W_p^{2}(G;\mathbb{R}^2)$ and $h\in W_p^{3-1/p}(G)$ such that $(v|\nu)=\partial_\nu h=0$ and $P_{\partial G}[(Dv)\nu]=0$, then $\partial_\nu (v|\nabla h)=0$.
\end{prop}
\begin{proof}
An easy computation shows that
$$\partial_\nu (v|\nabla h)=(\partial_\nu v|\nabla h)+(v|\nabla^2 h\nu),$$
where $\partial_\nu v:=\nabla v^{\sf T}\nu$.

Note that $P_{\partial G}[(Dv)\nu]=0$ implies that $((Dv)\nu|\nabla h)=0$, since by assumption $\partial_\nu h=0$. This in turn yields $(\partial_\nu v|\nabla h)=-(\nabla v\nu|\nabla h)$. Making use of the representation $\nabla h=\tau\partial_\tau h+\nu\partial_\nu h=\tau\partial_\tau h$, where $\tau\in\mathbb{R}^2$ with $|\tau|=1$ and $(\tau|\nu)=0$, we obtain
$$(\nabla v\nu|\nabla h)=((\nabla h\cdot\nabla)v|\nu)=\partial_\tau h(\partial_\tau v|\nu)=-\partial_\tau h(v|\partial_\tau\nu).$$
Here we made use of the assumption $(v|\nu)=0$ and $\nabla h\cdot\nabla:=\sum_{j=1}^2\partial_j h\partial_j$.

Concentrating on the term $(v|\nabla^2 h\nu)$, we obtain
\begin{align*}
(v|\nabla^2 h\nu)&=\sum_{i,j=1}^2v_i\partial_i\partial_j h\nu_j=\sum_{i,j=1}^2[v_i\partial_i(\partial_jh \nu_j)-v_i\partial_jh\partial_i \nu_j]\\
&=(v\cdot\nabla)\partial_\nu h-\sum_{i,j=1}^2v_i\partial_jh\partial_i \nu_j=(v|\tau)\partial_\tau\partial_\nu h+(v|\nu)\partial_\nu^2 h-\sum_{i,j=1}^2v_i\partial_jh\partial_i \nu_j\\
&=-\sum_{i,j=1}^2v_i\partial_jh\partial_i \nu_j,
\end{align*}
since $(v|\nu)=\partial_\nu h=0$. Here it is important to observe that $\partial_\tau\partial_\nu h=0$, whenever $\partial_\nu h=0$, since $\partial_\tau$ denotes the derivative in tangential direction.

Note that $$\sum_{i,j=1}^2v_i\partial_jh\partial_i \nu_j=((v\cdot\nabla)\nu|\nabla h)=(v|\tau)(\partial_\tau\nu|\nabla h)=(v|\tau)\partial_\tau h(\partial_\tau\nu|\tau)=\partial_\tau h(\partial_\tau\nu|v),$$
since $v=\tau(v|\tau)+\nu(v|\nu)=\tau(v|\tau)$ and $\nabla h=\tau\partial_\tau h$. Finally, this yields
$$\partial_\nu (v|\nabla h)=\partial_\tau h[(\partial_\tau\nu|v)-(\partial_\tau\nu|v)]=0.
$$
The proof is complete.

\end{proof}


\backmatter
\bibliographystyle{amsalpha}

\bibliographystyle{amsalpha}


\printindex

\end{document}